\documentclass[a4paper]{amsart}
\usepackage{indentfirst}
\usepackage{amsmath}
\usepackage{amsthm}
\usepackage{amssymb}
\usepackage{tikz}
\usepackage{adjustbox}
\usepackage{hyperref}
\usepackage{cleveref}
\usepackage{tikz-cd}

\usepackage{stackengine} 

\usepackage[T1]{fontenc}% http://ctan.org/pkg/fontenc
\usepackage[outline]{contour}% http://ctan.org/pkg/contour
\usepackage{xcolor}% http://ctan.org/pkg/xcolor

\contourlength{0.3pt}
\contournumber{15}

\usepackage{scalerel}
\usepackage{dsfont}
\usepackage{txfonts} %double inclined arrows
\usepackage{bbm} %lowercase bb
\usepackage{bbold} %numbers bb
\usepackage{mathrsfs} %scr
\usepackage[scr=boondox]{mathalfa} %lowercase scr
\DeclareFontFamily{U}{mathc}{}
\DeclareFontShape{U}{mathc}{m}{it}%
{<->s*[1.03] mathc10}{}
\DeclareMathAlphabet{\mathcal}{U}{mathc}{m}{it}

\DeclareMathAlphabet{\matheus}{U}{eus}{m}{n}
\SetMathAlphabet{\matheus}{bold}{U}{eus}{b}{n}

\DeclareFontFamily{U}{dmjhira}{}
\DeclareFontShape{U}{dmjhira}{m}{n}{ <-> dmjhira }{}

\newtheorem{defi}{Definition}[subsection]

\newtheorem{theorem}[defi]{Theorem}

\AfterEndEnvironment{maintheorem}{\noindent\ignorespaces}
\newtheorem{coroll}[defi]{Corollary}
\newtheorem{pro}[defi]{Proposition}
\newtheorem{lemma}[defi]{Lemma}
\theoremstyle{definition}
\newtheorem{construct}[defi]{Construction}
\newtheorem{example}[defi]{Example}
\newtheorem{innercustomgeneric}{\customgenericname}
\providecommand{\customgenericname}{}
\newcommand{\newcustomtheorem}[2]{%
	\newenvironment{#1}[1]
	{%
		\renewcommand\customgenericname{#2}%
		\renewcommand\theinnercustomgeneric{##1}%
		\innercustomgeneric
	}
	{\endinnercustomgeneric}
}
\newcustomtheorem{customdef}{Definition}
\newcustomtheorem{customthm}{Theorem}
\newcustomtheorem{custompro}{Proposition}
\newcustomtheorem{customlem}{Lemma}
\newcustomtheorem{custommainthm}{Main Theorem}
\newcustomtheorem{customeg}{Example}
\newcustomtheorem{customconjecture}{Conjecture}
\newcustomtheorem{customidea}{Idea}

\theoremstyle{remark}
\newtheorem{rk}[defi]{Remark}
\newtheorem{nota}[defi]{Notation}
\newcustomtheorem{customrk}{Remark}

\newcommand{\bbA}{\mathbbm{A}}

\newcommand{\calA}{\mathcal{A}}
\newcommand{\twoA}{\mathcal{A}}
\newcommand{\bbB}{\mathbbm{B}}

\newcommand{\calB}{\mathcal{B}}
\newcommand{\twoB}{\mathcal{B}}
\newcommand{\N}{\mathbbm{N}}

\newcommand{\bbK}{\mathbbm{K}}
\newcommand{\calK}{\mathcal{K}}

\newcommand{\twoC}{\mathcal{C}}

\newcommand{\F}{\mathscr{F}}
\newcommand{\bbF}{\mathbbm{F}}

\newcommand{\V}{\mathcal{V}}

\newcommand*{\Set}{\mathbf{Set}}
\newcommand*{\Cat}{\mathcal{Cat}}

\newcommand*{\FCat}{\mathbf{\F}\text{-}\mathcal{Cat}}
\newcommand*{\FDeltaCat}{\mathbf{\F_\Delta}\text{-}\mathcal{Cat}}
\newcommand*{\VCat}{\mathcal{V}\text{-}\mathcal{Cat}}

\newcommand*{\sSet}{\mathbf{Set}_\Delta}
\newcommand*{\Kan}{\mathcal{Kan}}
\newcommand*{\qCat}{Q\text{-}\mathcal{Cat}}

\newcommand*{\JLim}{J\text{-}\mathbb{Lim}}
\newcommand*{\calJLim}{J\text{-}\mathcal{Lim}}
\newcommand*{\JColim}{J\text{-}\mathbb{Colim}}
\newcommand*{\calJColim}{J\text{-}\mathcal{Colim}}
\newcommand*{\Cart}{\mathbb{Cart}}
\newcommand*{\calCart}{\mathcal{Cart}}
\newcommand*{\CoCart}{\mathbb{CoCart}}
\newcommand*{\calCoCart}{\mathcal{CoCart}}

\newcommand*{\Lali}{\mathbb{Lali}}
\newcommand*{\La}{\mathbb{La}}
\newcommand*{\Ra}{\mathbb{Ra}}
\newcommand*{\calLa}{\mathcal{La}}
\newcommand*{\calRa}{\mathcal{Ra}}
\newcommand*{\Rali}{\mathbb{Rali}}
\newcommand*{\calLali}{\mathcal{Lali}}
\newcommand*{\calRali}{\mathcal{Rali}}

\newcommand*{\Adj}{\mathbf{Adj}}
\newcommand*{\Mnd}{\mathbf{Mnd}}

\DeclareMathOperator{\im}{im\;}
\DeclareMathOperator{\res}{res}
\DeclareMathOperator{\Dom}{dom\;}
\DeclareMathOperator{\Cod}{cod\;}
\DeclareMathOperator{\inc}{inc}

\DeclareMathOperator{\ev}{ev\;}
\DeclareMathOperator{\ins}{ins}

\newcommand*{\TAlg}{\mathrm{T}\text{-}\mathrm{Alg}}

\newcommand{\rins}[3]{\ins^{#1}_{#2}({#3})}

\newcommand*{\shortbar}{\scalebox{1.6}[1]{-}}

\definecolor{turquoise}{HTML}{00B4CE}

\usepackage{lmodern}
\usepackage{mathtools,enumitem}
\usepackage{microtype}
\usepackage{graphicx}

\usetikzlibrary{arrows,calc,decorations.pathmorphing}
\tikzset{
	marked/.style = {decoration = {markings, mark = at position 0.5 with { 
				\node[transform shape, xscale = .8, yscale=.4] {/}; } }, postaction = 
		{decorate} },
	circled/.style = {decoration = {markings, mark = at position 0.5 with { 
				\node[transform shape, scale = .7] {$\circ$}; } }, postaction = {decorate} 
	},
	loose/.style ={->,
		decorate,
		decoration={snake,amplitude=.4mm,segment length=2mm,post length=1mm}},
}

\tikzcdset{scale cd/.style={every label/.append style={scale=#1},
		cells={nodes={scale=#1}}}}
\tikzcdset{every label/.append style = {font = \footnotesize}} % bigger labels

\usetikzlibrary{decorations.markings,decorations.pathmorphing,shapes}
\tikzset{double line with arrow/.style 
	args={#1,#2}{decorate,decoration={markings,%
			mark=at position 0 with {\coordinate (ta-base-1) at (0,1pt);
				\coordinate (ta-base-2) at (0,-1pt);},
			mark=at position 1 with {\draw[#1] (ta-base-1) -- (0,1pt);
				\draw[#2] (ta-base-2) -- (0,-1pt);
}}}}

\tikzset{dar/.style={double,double equal sign distance,-implies}}%style for 
\tikzset{mid/.style={anchor=mid}} % put labels on the arrow

\makeatletter
\newbox\dottedarrow@box
\setbox\dottedarrow@box\hbox
{%
	\begin{tikzpicture}
		\draw[shorten <=0.2cm, shorten >= 0.15cm , dotted,->] (-0.1,0) -- 
		(1.8em,0);
	\end{tikzpicture}%
}
\newcommand*\dottedarrow
{\relax\ifmmode\expandafter\dottedarrow@m\else\expandafter\dottedarrow@t\fi}
\newcommand*\dottedarrow@t[1][1.3em]
{\resizebox{#1}{!}{\raisebox{.5ex}{\usebox\dottedarrow@box}}}
\newcommand*\dottedarrow@m[1][]
{%
	\if\relax\detokenize{#1}\relax
	\mathchoice% values are trial and error based\ldots
	{\dottedarrow@t}
	{\dottedarrow@t}
	{\dottedarrow@t[0.5em]}
	{\dottedarrow@t[0.5em]}%
	\else
	\dottedarrow@t[#1]%
	\fi
}

\newcommand*\Neginternal[3]{\mathpalette\Neg@{{#1}{#2}{#3}}}
\newcommand*\Neg@[2]{\Neg@@{#1}#2}
\newcommand*\Neg@@[4]{%
	\mathrel{\ooalign{%
			$\m@th#1#4$\cr
			\hidewidth$\m@th#3{#1}\mkern\muexpr#2*2$\hidewidth\cr
	}}%
}
\newcommand*\mynegslash[1]{\rotatebox[origin=c]{60}{$\m@th#1\shortbar$}}
\newcommand*\amsnegslash[1]{\rotatebox[origin=c]{60}{$\m@th#1-$}}

\makeatother

\newcounter{sarrow}
\newcommand\xleadsto[1]{%
	\stepcounter{sarrow}%
	\mathrel{\begin{tikzpicture}[decoration=snake]
			\node (\thesarrow) {$\scriptstyle #1$};
			\draw[->,decorate] (\thesarrow.south west) -- (\thesarrow.south 
			east);
		\end{tikzpicture}
	}
}

\renewcommand{\leadsto}{\xleadsto{\mkern9mu}}

\usepackage{graphicx}

\newcommand{\xrightarrowdbl}[2][]{%
	\xrightarrow[#1]{#2}\mathrel{\mkern-14mu}\rightarrow
}

\newcommand{\isof}{\rotatebox[origin=c]{270}{\scalebox{0.7}{$\xrightarrowdbl{}$}}}

\makeatletter
\providecommand*{\twoheadrightarrowfill@}{%
	\arrowfill@\relbar\relbar\twoheadrightarrow
}
\providecommand*{\twoheadleftarrowfill@}{%
	\arrowfill@\twoheadleftarrow\relbar\relbar
}
\providecommand*{\xtwoheadrightarrow}[2][]{%
	\ext@arrow 0579\twoheadrightarrowfill@{#1}{#2}%
}
\providecommand*{\xtwoheadleftarrow}[2][]{%
	\ext@arrow 5097\twoheadleftarrowfill@{#1}{#2}%
}
\makeatother

\newcommand{\hooktwoheadrightarrow}{%
	\hookrightarrow\mathrel{\mspace{-15mu}}\rightarrow
}

\usetikzlibrary{graphs,graphs.standard,quotes}

\setlist[enumerate,1]{label=\textup{(\arabic*)}}
\setlist[enumerate,2]{label=\textup{(\alph*)}}

\usepackage{thmtools}

\newcommand{\longref}[2]{\hyperref[#2]{#1~\textup{\ref*{#2}}}}
\newcommand{\eqtref}[2]{\hyperref[#2]{#1~\textup{(\ref*{#2})}}}
\newcommand{\longcite}[1]{\textup{\cite{#1}}}

\newcommand*{\op}{\mathrm{op}}%opposite
\newcommand*{\co}{\mathrm{co}}%co
\newcommand*{\chor}{\mathrm{chordate}}

\newcommand*{\id}{\mathrm{id}}% identity map
\DeclareMathOperator{\ob}{ob}

\numberwithin{equation}{section}
\newcounter{subsubsubsection}[subsubsection]

\makeatletter
\renewcommand\paragraph{\@startsection{paragraph}{5}{\z@}%
	{3.25ex \@plus1ex \@minus.2ex}%
	{-1em}%
	{\normalfont\normalsize\bfseries}}
\renewcommand\subparagraph{\@startsection{subparagraph}{6}{\parindent}%
	{3.25ex \@plus1ex \@minus .2ex}%
	{-1em}%
	{\normalfont\normalsize\bfseries}}
\def\toclevel@subsubsubsection{4}
\def\toclevel@paragraph{5}
\def\toclevel@paragraph{6}
\def\l@subsubsubsection{\@dottedtocline{4}{7em}{4em}}
\def\l@paragraph{\@dottedtocline{5}{10em}{5em}}
\def\l@subparagraph{\@dottedtocline{6}{14em}{6em}}
\makeatother

\usepackage[margin=22.2mm]{geometry} 
\title[Limits of $(\infty, 1)$-categories with structure and their lax morphisms]{Limits of $(\infty, 1)$-categories with structure \\ and their lax morphisms}
\author{Joanna Ko}
\email{joanna.ko.maths@gmail.com}
\address{Department of Mathematics and Statistics, Masarykova univerzita, 
	Kotlářská 2, Brno 61137, Czech Republic}

\begin{document}
	
	\begin{abstract}
		%		Riehl and Verity have established that for a quasi-category $A$ that admits limits, and a homotopy coherent monad on $A$ which does not preserve limits, the Eilenberg-Moore object still admits limits; this can be seen as a completeness result involving lax morphisms. We investigate limits of lax morphisms in the $\infty$-categorical setting, and generalise their result to different models for $(\infty, 1)$-categories, with an abundant varieties of structure. For instance, $(\infty, 1)$-categories with limits, (discrete) Cartesian fibrations between $(\infty, 1)$-categories, and adjunctions between $(\infty, 1)$-categories. In addition, we show that these $(\infty, 1)$-categories with structure in fact possess other kinds of limits of lax morphisms, such as $\infty$-categorical versions of inserters and equifiers, when only one morphism in the diagram is structure-preserving. Our approach provides a minimal requirement and a transparent explanation for several kinds of limits of $(\infty, 1)$-categories and their lax morphisms to exist. 
		%
		Riehl and Verity have established that for a quasi-category $A$ that admits limits, and a homotopy coherent monad on $A$ which does not preserve limits, the Eilenberg-Moore object still admits limits; this can be interpreted as a completeness result involving lax morphisms. We generalise their result to different models for $(\infty, 1)$-categories, with an abundant variety of structures. For instance, $(\infty, 1)$-categories with limits, Cartesian fibrations between $(\infty, 1)$-categories, and adjunctions between $(\infty, 1)$-categories. In addition, we show that these $(\infty, 1)$-categories with structure in fact possess an important class of limits of lax morphisms, including $\infty$-categorical versions of inserters and equifiers, when only one morphism in the diagram is required to be structure-preserving. Our approach provides a minimal requirement and a transparent explanation for several kinds of limits of $(\infty, 1)$-categories and their lax morphisms to exist. 

	\end{abstract}
	\maketitle
	\tableofcontents
	
	\section{Introduction}
	\label{sec:intro}
	In \cite{RV:2015}, Riehl and Verity have shown that for a quasi-category $A$ that admits limits, and a (homotopy coherent) monad $T$ on $A$ which does \emph{not} preserve limits, the Eilenberg-Moore object over $T$ is also a quasi-category with limits, and that the forgetful functor from the Eilenberg-Moore object to $A$ preserves limits.
	
	Whether the above theorem could be generalised to \emph{different models} for $(\infty, 1)$-categories is of great interest. One framework that allows working on $(\infty, 1)$-category theory more model-independently is given by the theory of \emph{$\infty$-cosmoi}, as developed by Riehl and Verity in \cite{book:RV:2022}.
	
	Generally speaking, $\infty$-cosmoi are quasi-categorically enriched categories that satisfy certain nice properties resembling enriched categories of fibrant objects. The theory of $\infty$-cosmoi provides a setting for understanding $(\infty, 1)$-categories with structure and the \emph{pseudo} morphisms between them, i.e. morphisms that preserve structure up to isomorphism, model-independently. For instance, we have the $\infty$-cosmos of $(\infty, 1)$-categories admitting $J$-shaped limits for a simplicial set $J$ and the functors which preserve limits, and also the $\infty$-cosmos of Cartesian fibrations between $(\infty, 1)$-categories and the Cartesian functors.
	
	In particular, the theory of $\infty$-cosmoi enables the understanding of limits of $(\infty, 1)$-categories with structure and their pseudo morphisms. Riehl and Verity have established in \cite{book:RV:2022} that any $\infty$-cosmos admits all \emph{flexible weighted limits}, which are simplicially enriched limits that are analogous to \emph{PIE limits} in $2$-category theory. Intuitively, flexible weighted limits in an $\infty$-cosmos are $\infty$-categorical limits that do not impose strict equations; just as PIE limits are $2$-dimensional limits that do not impose equations between $1$-morphisms. Examples of flexible weighted limits include products, inserters, and comma objects. Indeed, in the seminal paper \cite{BKP:1989} by Blackwell, Kelly, and Power, it is shown that for any $2$-monad, the $2$-category $\TAlg_p$ of $T$-algebras together with their pseudo $T$-morphisms admits all PIE limits. This hints that the phenomena happening in an $\infty$-cosmos should be similar to those in the $2$-categorical situation.
	
	In \cite{Lack:2005}, Lack has established the existence of several $2$-dimensional limits involving \emph{lax} morphisms. For instance, the dual version of \cite[Proposition 4.4]{Lack:2005} implies that in the $2$-category of categories with $J$-shaped limits for a category $J$ and the functors that do \emph{not} necessarily preserve limits, comma objects exist even when \emph{only one} of the $1$-morphisms in the diagram preserves limits; similarly, \cite[Proposition 4.11]{Lack:2002} implies that the Eilenberg-Moore object over a monad which does \emph{not} necessarily preserves limits still has $J$-shaped limits.
	
	These phenomena led Lack and Shulman to introduce \emph{enhanced $2$-category theory} in \cite{LS:2012}. Roughly speaking, an enhanced $2$-category is a $2$-category with \emph{two} types of $1$-morphisms: the tight ones and the loose ones, in which every tight $1$-morphism is also loose. For example, categories admitting $J$-shaped limits form an enhanced $2$-category, with the tight $1$-morphisms given by the functors that preserve limits, whereas the loose $1$-morphisms given by just the functors, and the $2$-morphisms given by the natural transformations. By formulating enhanced $2$-categorical notions via enriched category theory, Lack and Shulman have successfully studied the subtle behaviours of limits involving lax morphisms.
	
	Seeing that the phenomena in an $\infty$-cosmos are similar to those in the $2$-categorical situation, we aim to recover similar classical results in the $\infty$-categorical world. More precisely, we would like to show that in many cases, knowing that \emph{only one} of the lax morphisms is pseudo, is already sufficient to deduce the existence of several kinds of limits of $(\infty, 1)$-categories with structure and their lax morphisms.
	
	For this, we take inspiration from the work \cite{LS:2012} by Lack and Shulman. We introduce the notion of \emph{enhanced simplicial categories}, which are basically simplicially enriched categories with \emph{two} types of $0$-arrows. We proceed to this concept via enriched category theory by introducing first a category $\F_\Delta$, and then define enhanced simplicial categories as categories enriched in $\F_\Delta$. Roughly speaking, the category $\F_\Delta$ consists of \emph{enhanced simplicial sets}, which are simplicial sets with special vertices. Therefore, enriching in $\F_\Delta$ gives rise to a simplicial category with two types of $0$-arrows. When every $0$-arrow are tight, we call it a \emph{chordate} enhanced simplicial category; when only the identities are tight, we call it an \emph{inchordate} enhanced simplicial category. Very often, we are interested in the case when the tight part constitutes an $\infty$-cosmos, so as to establish model-independent results for $(\infty, 1)$-categories. Examples include $(\infty, 1)$-categories possessing limits together with their pseudo and their lax morphisms, and also Cartesian fibrations between $(\infty, 1)$-categories together with the Cartesian functors and the functors between them.
	
	To explore limits that exist in these enhanced simplicial categories, we would like to propose a new notion that captures \emph{rigged PIE limits} in enhanced $2$-category theory, which are shown to exist in many enhanced $2$-categorical structures in \cite{LS:2012}. We then introduce the notion of \emph{rigged $n$-inserters}, which happens to be a crucial concept that generalises (rigged) products, inserters, comma objects, and equifiers in (enhanced) $2$-category theory to the $\infty$-categorical world, and is also the key to the construction of Eilenberg-Moore objects. In other words, rigged $n$-inserters are a basic and fundamental class of $\F_\Delta$-weighted limits in enhanced simplicial categories, which serve as a main component in generating some more advanced $\F_\Delta$-weighted limits.
	
	Particularly, when the tight part of an enhanced simplicial category is an $\infty$-cosmos, there is a convenient description of rigged $n$-inserters.
	\begin{custompro}{\hspace{-1mm}}[Rigged $n$-inserters]
		Let $\mathbb{K}$ be an enhanced simplicial category that has simplicial powers by inchordate enhanced simplicial sets, and $S, T$ be objects in $\mathbb{K}$.  Given a loose $0$-arrow $D \colon S \leadsto T^{\partial \Delta^n}$ in $\bbK$, where the composite
		\[ S \xleadsto{D} T^{\partial \Delta^n} \xrightarrow{\ev[[{n}]]} T \]
		of $D$ with the evaluation at the terminal vertex is a {tight} $0$-arrow,  the \emph{terminally rigged $n$-inserter} $\rins{n}{[[n]]}{D}$ of $D$ is the pullback
		\begin{center}
			\begin{tikzcd}
				{\rins{n}{[[n]]}{D}} \arrow[dr, phantom, "\lrcorner", very near start] \ar[r, "\phi", loose] \ar[d, "p"'] & T^{\Delta^n} \ar[d,  "{T^{\partial}}"]
				\\
				S \ar[r, "D"', loose]  & T^{\partial\Delta^n}
			\end{tikzcd}
		\end{center} 
		of $D$ along the restriction $T^\partial \colon T^{\Delta_n} \to T^{\partial \Delta_n}$.
		
		Dually, if the composite
		\[ S \xleadsto{D} T^{\partial \Delta^n} \xrightarrow{\ev[[{0}]]} T \]
		of $D$ with the evaluation at the initial vertex is a {tight} $0$-arrow, then the \emph{initially rigged $n$-inserter} $\rins{n}{[[[0]]}{D}$ of $D$ is the pullback
		of $D$ along the restriction $T^\partial \colon T^{\Delta_n} \to T^{\partial \Delta_n}$.
	\end{custompro}
	In simple words, rigged $n$-inserters are limits of diagrams in the shape of the boundary of a standard $n$-simplex in a hom-object, where \emph{only one} of the legs in the diagram is tight.
	
	Apart from rigged $n$-inserters, it is important to know that \emph{cosmological limits}, i.e. simplicial limits which naturally exist in any $\infty$-cosmos, are upgraded automatically to $\F_\Delta$-weighted limits in the corresponding enhanced simplicial category. Such limits are called \emph{tight cosmological limits}.
	
	We then establish most of our completeness results in different settings via a more unified approach. Instead of proving the existence of rigged $n$-inserters and tight cosmological limits in different enhanced simplicial categories directly case-by-case, we establish first in the more fundamental enhanced simplicial categories, which are $\Lali(\calK)$ of \emph{lali-isofibrations} of an $\infty$-cosmos $\calK$, i.e., isofibrations which are left adjoints left inverses, and dually $\Rali(\calK)$ of \emph{rali-isofibrations}, i.e., isofibrations which are right adjoints left inverses. The tight $0$-arrows are morphisms of left adjoints, and respectively right adjoints, whereas the loose $0$-arrows are simply morphisms of isofibrations.
	\begin{customthm}{\hspace{-1mm}}
		The enhanced simplicial category $\Lali(\calK)$ of lali-isofibrations of an $\infty$-cosmos $\calK$ admits tight cosmological limits and terminally rigged $n$-inserters.
		
		Dually, the enhanced simplicial category $\Rali(\calK)$ of rali-isofibrations of an $\infty$-cosmos $\calK$ admits tight cosmological limits and initially rigged $n$-inserters.
	\end{customthm}
	%	\begin{customthm}{}
		%		The enhanced simplicial category $\Rali(\calK)$ of rali-isofibrations of an $\infty$-cosmos $\calK$ admits initially rigged $n$-inserters.
		%	\end{customthm}
	%	\begin{customthm}{}
		%		The enhanced simplicial categories $\Lali(\calK)$ and $\Rali(\calK)$ both admit tight cosmological limits.
		%	\end{customthm}
	In fact, (co)Eilenberg-Moore objects over loose (co)monads can be constructed from rigged $n$-inserters. In $2$-category theory, the Eilenberg-Moore object $\TAlg$ over a loose monad $T$ with unit $\eta$ and multiplication $\mu$ on an object $A$ could be constructed in three steps. First, equip $A$ with structure morphisms, which could be done by taking the inserter
	\begin{center}
		\begin{tikzcd}[row sep=3]
			& A \ar[dr, loose, "T"] \ar[Rightarrow, from=1-2, to=3-2, shorten=2mm, "\phi"] &
			\\
			L \ar[ur,  "p"] \ar[dr, "p"'] & & A
			\\
			& A \ar[ur, equal] &
		\end{tikzcd}
	\end{center}
	of $T$ along the identity. Second, impose associativity by considering the equifier
	\begin{center}
		\begin{tikzcd}
			E \ar[r, "q"] & L \ar[rrr, bend left, loose, "TTp"{name=U}] \ar[rrr, bend right, "p"'{name=D}] \ar[Rightarrow, shift right=2ex, from=U, to=D, "\phi \cdot T\phi"', shorten=2mm] \ar[Rightarrow, shift left=2ex, from=U, to=D, "\phi \cdot \mu p", shorten=2mm] & & & A
		\end{tikzcd}.
	\end{center}
	And finally, impose unitality by considering the equifier on the right below.
	\begin{center}
		\begin{tikzcd}
			\TAlg \ar[r] & E \ar[rrr, bend left, "pq"{name=U}] \ar[rrr, bend right, "pq"'{name=D}] \ar[Rightarrow, shift right=2ex, from=U, to=D, "\phi \cdot \eta q"', shorten=2mm] \ar[Rightarrow, equal, shift left=2ex, from=U, to=D, "\id", shorten=2.5mm] & & & A
		\end{tikzcd}.
	\end{center}
	So, altogether, the Eilenberg-Moore object $\TAlg$ can be built as a chain
	$${\TAlg}  \to E \to L \to A$$
	of $c$-rigged inserter and equifiers introduced in \cite{LS:2012}. As $c$-rigged inserters and equifiers are special cases of terminally rigged $n$-inserters, it is expected that Eilenberg-Moore objects over loose monads on $(\infty, 1)$-categories could be constructed from the latter. By looking into the work \cite{RV:2015} of Riehl and Verity, we present how Eilenberg-Moore objects over \emph{loose} monads are built from as a countable chain of terminally rigged $n$-inserters. Dually, coEilenberg-Moore objects over loose comonads are constructed out of initially rigged $n$-inserters.
	%	\begin{customlem}{}
		%		Let $\bbK$ be an enhanced simplicial category, where the tight part $\calK_\tau$ is an $\infty$-cosmos. Suppose $\bbK$ admits limits of a countable chain of tight isofibrations, simplicial powers by inchordate enhanced simplicial sets, and terminally rigged $n$-inserters, where the projection of a terminally rigged $n$-inserter is always a tight isofibration of $\bbK$.
		%		
		%		Then, for any loose monad $T \colon \Mnd \to \bbK$ on an object $A$ in $\bbK$, the Eilenberg-Moore object over $T$ exists, and that the forgetful functor from the Eilenberg-Moore object to $A$ is a tight isofibration, which also reflects tightness.
		%	\end{customlem}
	As a consequence, we are able to establish further completeness results in $\Lali(\calK)$ or $\Rali(\calK)$.
	\begin{customthm}{\hspace{-1mm}}
		The enhanced simplicial category $\Lali(\calK)$ of lali-isofibrations of an $\infty$-cosmos $\calK$ admits Eilenberg-Moore objects over loose monads.
		
		Dually, the enhanced simplicial category $\Rali(\calK)$ of rali-isofibrations of an $\infty$-cosmos $\calK$ admits coEilenberg-Moore objects over loose comonads.
		
		In both cases, the forgetful functor is tight and reflects tightness.
	\end{customthm}
	%	\begin{customthm}{}
		%		The enhanced simplicial category $\Rali(\calK)$ of rali-isofibrations of an $\infty$-cosmos $\calK$ admits coEilenberg-Moore objects over loose comonads, and that the forgetful functor is tight and reflects tightness.
		%	\end{customthm}
	
	The $\infty$-cosmos $\calLali(\calK)$ of lali-isofibrations of an arbitrary $\infty$-cosmos $\calK$ is the source of many other important $\infty$-cosmoi, as shown in \cite{book:RV:2022}, including $(\infty, 1)$-categories possessing limits, and Cartesian fibrations. In addition, we show that left adjoints which are isofibrations of an $\infty$-cosmos also form an $\infty$-cosmos $\calLa(\calK)$, which can actually be obtained as the pullback of $\calLali(\calK)$. Since enhanced simplicial categories are formulated as $\F_\Delta$-categories, and rigged $n$-inserters are $\F_\Delta$-weighted limits, consequently, by applying tools from enriched category theory, our theorem in $\Lali(\calK)$ then implies that the enhanced simplicial category of $(\infty, 1)$-categories possessing limits, that of Cartesian fibrations between $(\infty, 1)$-categories, and that of left adjoints which are isofibrations between $(\infty, 1)$-categories, all admit terminally rigged $n$-inserters. As corollaries, the aforementioned enhanced simplicial categories admit inserters, comma $(\infty, 1)$-categories, and equifiers of \emph{loose} $0$-arrows when \emph{only} their codomains are tight, and also Eilenberg-Moore objects over \emph{loose} monads. The dual versions of these results are also established accordingly, by considering \emph{initially} rigged $n$-inserters $\rins{n}{[[0]]}{D}$. Dually, the $\infty$-cosmos of rali-isofibrations is the source of the dual versions of the aforementioned structures, such as $(\infty, 1)$-categories possessing colimits, coCartesian fibrations, and right adjoints which are isofibrations, hence our theorem in $\Rali(\calK)$ then implies that the enhanced simplicial category of $(\infty, 1)$-categories possessing colimits, that of coCartesian fibrations between $(\infty, 1)$-categories, and that of right adjoints which are isofibrations between $(\infty, 1)$-categories, all admit initially rigged $n$-inserters; consequently, they all admit inserters, comma $(\infty, 1)$-categories, and equifiers of \emph{loose} $0$-arrows, when \emph{only} the domains are tight, and coEilenberg-Moore objects over \emph{loose} comonads.
	
	In short, we have generalised the result by Riehl and Verity on quasi-categories with limits in \cite{RV:2015}. We successfully established a {model-independent} version of the theorem; furthermore, we include many more structures other than $(\infty, 1)$-categories with limits, and provide a better understanding on the phenomenon of limits of $(\infty, 1)$-categories with structure and their lax morphisms by introducing the very central notion of rigged $n$-inserters, which helps us deduce the existence of substantially more limits aside from Eilenberg-Moore objects over loose monads. In addition, our approach of formulating the enhanced structures via enriched category theory allows us to deduce an \emph{extra} property of the limit projections, namely, we show that they \emph{jointly reflect} tightness as well.
	
	After that, we also discuss the existence of some $\F_\Delta$-weighted limits in certain enhanced simplicial categories, but not in the enhanced simplicial category $\Lali(\calK)$ of lali-isofibrations or $\Rali(\calK)$ of rali-isofibrations. We show that the enhanced simplicial categories $\La(\calK)$ of left adjoints which are isofibrations between $(\infty, 1)$-categories, $\Ra(\calK)$ of right adjoints which are isofibrations, $\JLim(\calK)$ of $(\infty, 1)$-categories with $J$-shaped limits, and $\JColim(\calK)$ od $(\infty, 1)$-categories that have $J$-shaped colimits all admit pullbacks of a \emph{loose} $0$-arrow along a tight Cartesian fibration. Our arguments indicate that  $\Lali(\calK)$ or $\Rali(\calK)$ might not possess this kind of limits. In other words, pullbacks of a \emph{loose} $0$-arrow along a tight Cartesian fibration would be an example of $\F_\Delta$-weighted limits that exist in several interesting enhanced simplicial categories, which is, however, not inherited from $\Lali(\calK)$ or $\Rali(\calK)$. This surprising phenomenon could be further investigated and explored carefully, so as to develop and formalise a theory of \emph{enhanced $\infty$-cosmoi} and \emph{enhanced $(\infty, 2)$-category theory} in the future.
	
	%	{\hl rewrite, emphasise new concepts, new techniques, comparison with old result, improvement on understanding the rigged phenonmenon e.g. EM-object}
	
	The article is outlined as follows. In \longref{Section}{sec:enhanced}, we present the basic notions in enhanced $2$-category theory, and introduce the concept of enhanced simplicial categories as enriched categories in $\F_\Delta$. We show several important examples of enhanced simplicial categories. After that, we explore $\F_\Delta$-weighted limits, and understand a few basic examples. In  \longref{Section}{sec:inserter}, we introduce the notion of rigged $n$-inserters, which is the heart of this article. We then demonstrate with a few crucial examples of limits which are subsumed by our notion of rigged $n$-inserters, including products, inserters, comma objects, and equifiers. In \longref{Section}{sec:complete}, we begin by proving our main completeness results in $\Lali(\calK)$ and $\Rali(\calK)$. Subsequently, we provide a few useful lemmata with the help of enriched category theory, which allow us to transfer the $\F_\Delta$-weighted limits in $\Lali(\calK)$ or $\Rali(\calK)$ to other enhanced simplicial categories. We then show numerous completeness results in many different kinds of enhanced simplicial categories. Lastly, in \longref{Section}{sec:not_in_Lali}, we discuss the existence of certain $\F_\Delta$-weighted limits involving lax morphisms in some enhanced simplicial categories, which are not shown to exist in  $\Lali(\calK)$ or $\Rali(\calK)$.

	\section*{Acknowledgement}
	The author would like to thank {John Bourke} for his kind guidance and endless encouragement. The author is grateful to Richard Garner and Stephen Lack for the fruitful and inspiring discussions during her stay in Macquarie University. The author would also like to thank Dennis-Charles Cisinski and Johannes Gloßner for the enlightening conversations when she visited Higher Invariants in Universität Regensburg.
	
	The author acknowledges the support of Masarykova univerzita under the grant MUNI/A/1457/2023.
	
	\section{Enhanced simplicial categories}
	\label{sec:enhanced}
	
	\subsection{Preliminaries on enhanced $2$-category theory}
	
	We recall the basics of enhanced $2$-category theory, which was first 
	proposed by Lack and Shulman in 
	\longcite{LS:2012}. An introduction can also be found in
	\longcite{Bourke:2014} by Bourke.
	
	\begin{defi}
		Denote by $\F$ the full sub-category of the arrow category 
		$\mathbf{Cat}^\mathbf{2}$ 
		of the $1$-category	$\mathbf{Cat}$ of categories, determined by the fully faithful and injective-on-objects 
		functors, i.e., the \emph{full embeddings}.
	\end{defi}
	
	In other words, an object of $\F$ is a full embedding
	\begin{equation*}
		A_\tau \xhookrightarrow{j_A} A_\lambda,
	\end{equation*}
	a morphism $f$ from $j_A$ to $j_B$ in $\F$ is given by two functors 
	$f_\tau \colon A_\tau \to B_\tau$ and 
	$f_\lambda \colon A_\lambda \to B_\lambda$ 
	making the following square commute
	\begin{center}
		\begin{tikzcd}
			A_\tau \ar[r, hook, "j_A"] \ar[d, "f_\tau"'] & A_\lambda \ar[d, 
			"f_\lambda"]
			\\
			B_\tau \ar[r, hook, "j_B"'] & B_\lambda
		\end{tikzcd}.
	\end{center}
	
	We call $A_\tau$ the \emph{tight} part of $A$, and $A_\lambda$ the 
	\emph{loose} part of 
	$A$; similarly, we apply this terminology to $f$.
	
	\begin{rk}
		$\F$ is (co)complete and Cartesian closed.
	\end{rk}
	
	\begin{defi}
		An \emph{enhanced $2$-category} $\bbA$ is a category $\bbA$ enriched in $\F$. 
	\end{defi}
	
	This 
	means $\bbA$ has
	\begin{enumerate}
		\item [$\bullet$] objects $x, y, \cdots$;
		\item [$\bullet$] hom-objects $\bbA(x, y)$ in $\F$, each consists 
		of a full embedding $\bbA(x, y)_\tau \hookrightarrow \bbA(x, 
		y)_\lambda$ of the tight part into the loose part.
	\end{enumerate}
	
	We can form a $2$-category $\calA_\tau$ as follows:
	
	\begin{enumerate}
		\item [$\bullet$] $\calA_\tau$ has all the objects of $\bbA$;
		\item [$\bullet$] the hom-categories $\calA_\tau (x, y)$ for any 
		objects $x, y$ are $\bbA(x, y)_\tau$.
	\end{enumerate}
	
	Similarly, we can form a $2$-category $\calA_\lambda$ by setting the 
	hom-categories $\calA_\lambda (x, y)$  as $\bbA(x, y)_\lambda$ for any 
	objects $x, y$.
	
	Since for each pair of objects $x, y$, $\calA_\tau (x, y) \hookrightarrow 
	\calA_\lambda (x, y)$ is a full-embedding, we obtain a $2$-functor
	
	\begin{align*}
		J_\bbA \colon \calA_\tau &\to \calA_\lambda.
	\end{align*}
	
	By construction, $J_\bbA$ is identity-on-objects, faithful, and locally 
	fully faithful.
	
	\begin{rk}
		We may identify an $\F$-category $\bbA$ with $J_\bbA$. Indeed, any 
		$2$-functor which is identity-on-objects, faithful, and locally 
		fully faithful uniquely determined an $\F$-category.
	\end{rk}
	
	The morphisms in $\calA_\tau$ are called the \emph{tight} morphisms, 
	whereas 
	those in $\calA_\lambda$ are called \emph{loose}.
	
	\begin{nota}
		We write $A \to B$ for a tight morphism from $A$ to $B$, and $A 
		\leadsto B$ for a loose morphism from $A$ to $B$.
	\end{nota}
	
	Let $\bbA$ and $\bbB$ be two $\F$-categories. An \emph{$\F$-functor} 
	$F\colon 
	\bbA \to \bbB$ is a functor enriched in $\F$, which precisely means that 
	$F$ consists 
	of $2$-functors $F_\tau \colon \calA_\tau \to 
	\calB_\tau$ and $F_\lambda \colon \calA_\lambda \to \calB_\lambda$ 
	making the following diagram commute
	
	\begin{equation}
		\label{diag:F_functor}
		\begin{tikzcd}
			\calA_\tau \ar[r, hook, "J_\bbA"] \ar[d, "F_\tau"'] & \calA_\lambda 
			\ar[d, 
			"F_\lambda"]
			\\
			\calB_\tau \ar[r, hook, "J_\bbB"'] & \calB_\lambda
		\end{tikzcd}.
	\end{equation}
	
	\begin{rk}
		$F_\tau$ is uniquely determined by $F_\lambda$: an $\F$-functor 
		$F\colon 
		\bbA \to \bbB$ is a $2$-functor $F_\lambda \colon 
		\calA_\lambda \to \calB_\lambda$ which preserves tightness, i.e., 
		$F_\lambda$ sends a tight morphism in $\bbA$ to a tight morphism in 
		$\bbB$.
	\end{rk}
	
	Let $F, G\colon \bbA \rightrightarrows \bbB$ be two $\F$-functors. An 
	\emph{{$\F$}-natural 
		transformation}	$\alpha \colon F \to G$ consists of $2$-natural 
	transformations $\alpha_\tau \colon F_\tau \to G_\tau$ and 
	$\alpha_\lambda \colon F_\lambda \to G_\lambda$ 
	making the 
	following diagram of $2$-morphisms commute
	
	\begin{equation}
		\label{diag:F_nat_tran}
		\begin{tikzcd}
			\calA_\tau \ar[r, hook, "J_\bbA"] \ar[d, "G_\tau"', bend right] 
			\ar[d, 
			"F_\tau", 
			bend left] \ar[d, near end, phantom, 
			"{\substack{\Leftarrow \\ \alpha_\tau}}"] & \calA_\lambda \ar[d, 
			"F_\lambda", bend left] \ar[d, "G_\lambda"', bend right] \ar[d, 
			near 
			end, phantom, 
			"{\substack{\Leftarrow \\ \alpha_\lambda}}"]
			\\
			\calB_\tau \ar[r, hook, "J_\bbB"'] & \calB_\lambda
		\end{tikzcd}.
	\end{equation}
	
	\begin{rk}
		\label{rk:F-nat-tran}
		The existence of $\alpha_\tau$ can be seen as 
		the condition that  $\alpha_\lambda$ has tight components.
	\end{rk}
	
	%	\subsection{\texorpdfstring{$\F_\Delta$}{FΔ}-categories}
	\subsection{Enhanced simplicial categories, enhanced simplicial functors, and \texorpdfstring{$\F_\Delta$}{FΔ}-weighted limits}
	
	We now introduce the concept of enhanced simplicial categories and relevant notions. These notions capture the situation of having two types of morphisms in a simplicially enriched category.
	
	\begin{nota}
		Denote by $\sSet$ the category of simplicial sets, and by  $\sSet^\mathbb{2}$ the arrow $1$-category of $\sSet$.
	\end{nota}
	
	Recall that a simplicial map $f \colon S \to T$ is said to be \emph{fully faithful}, precisely when for any vertices $x, y \in S_0$, the induced map
	$$f_{x, y} \colon S(x, y) \to T(f(x), f(y))$$
	on the hom-objects is an isomorphism of simplicial sets.
	
	\begin{defi}
		Denote by $\F_\Delta$ the full sub-category of $\sSet^\mathbb{2}$  spanned by the fully faithful and injective-on-vertices simplicial maps.
		An \emph{enhanced simplicial set} is an object in $\F_\Delta$.
	\end{defi}
	
	Basically, an enhanced simplicial set $S$ is a simplicial set having a sub-simplicial set
	
	\begin{equation*}
		S_\tau \xhookrightarrow{j_S} S_\lambda,
	\end{equation*}
	
	and a morphism $f$ from $j_S$ to $j_T$ in $\F_\Delta$ is given by two simplicial maps 
	$f_\tau \colon S_\tau \to T_\tau$ and 
	$f_\lambda \colon S_\lambda \to T_\lambda$ 
	making the following square commute
	\begin{center}
		\begin{tikzcd}
			S_\tau \ar[r, hook, "j_S"] \ar[d, "f_\tau"'] & S_\lambda \ar[d, 
			"f_\lambda"]
			\\
			T_\tau \ar[r, hook, "j_B"'] & T_\lambda
		\end{tikzcd}.
	\end{center}
	
	We call $S_\tau$ the \emph{tight} part of $S$, and $S_\lambda$ the 
	\emph{loose} part of 
	$S$; similarly, we apply this terminology to $f$.
	
	\begin{nota}
		Let $J$ be a simplicial set.
		%		 Denote by $J_{[j], \cdots , [k]}$ the enhanced simplicial set in which only the $j^{\mathrm{th}}, \cdots, k^{\mathrm{th}}$ vertices in $J$ are taken to be tight.
		%		
		Indeed, there are two extreme ways to turn a simplicial set into an enhanced simplicial set: any simplicial set can be viewed as a \emph{chordate} enhanced simplicial set in which every vertex is tight; or an \emph{inchordate} enhanced simplicial set with no tight vertices. In the chordate case, we denote by $J_\chor$; in the inchordate case, we denote simply by $J$.
	\end{nota}

	\begin{pro}
		$\F_\Delta$ is (co)complete and Cartesian closed.
	\end{pro}
	
	\begin{proof}
		We first show that the fully faithful and injective-on-vertices simplicial maps form a right class of a factorisation system on $\sSet$.
		
		Let $f \colon S \to T$ be an arbitrary simplicial map. Write $\im f$ as the simplicial subset of $T$ spanned by the vertices $f(x)$ for any $x \in S_0$. Then $f$ restricts onto $\im f$ as a surjective-on-vertices simplicial map
		$$l_f \colon S \twoheadrightarrow \im f$$
		and that there is an inclusion
		$$r_f \colon S \hookrightarrow T$$
		which is fully faithful, because
		$$\im f(f(x), f(y)) \to T(f(x), f(y))$$
		is an isomorphism of simplicial sets.
		
		In other words, $f$ factorises into a surjective-on-vertices simplicial map followed by a fully faithful and injective-on-vertices simplicial map.
		
		Now write $\twoheadrightarrow$ for a surjective-on-vertices simplicial map, and $\rightarrowtail$ for a fully faithful and injective-on-vertices simplicial map.
		
		Suppose we are given a commutative diagram
		
		% https://q.uiver.app/#q=WzAsNixbMCwwLCJBIl0sWzEsMCwiWCJdLFsyLDAsIkIiXSxbMCwxLCJBJyJdLFsxLDEsIlgnIl0sWzIsMSwiQiciXSxbMCwzLCJmIiwyXSxbMiw1LCJnIl0sWzAsMSwibCIsMCx7InN0eWxlIjp7ImhlYWQiOnsibmFtZSI6ImVwaSJ9fX1dLFsxLDIsInIiLDAseyJzdHlsZSI6eyJ0YWlsIjp7Im5hbWUiOiJtb25vIn19fV0sWzMsNCwibCciLDIseyJzdHlsZSI6eyJoZWFkIjp7Im5hbWUiOiJlcGkifX19XSxbNCw1LCJyJyIsMix7InN0eWxlIjp7InRhaWwiOnsibmFtZSI6Im1vbm8ifX19XV0=
		\[\begin{tikzcd}[ampersand replacement=\&]
			A \& X \& B \\
			{A'} \& {X'} \& {B'}
			\arrow["l", two heads, from=1-1, to=1-2]
			\arrow["f"', from=1-1, to=2-1]
			\arrow["r", tail, from=1-2, to=1-3]
			\arrow["g", from=1-3, to=2-3]
			\arrow["{l'}"', two heads, from=2-1, to=2-2]
			\arrow["{r'}"', tail, from=2-2, to=2-3]
		\end{tikzcd}\]
		
		in $\sSet$. We can define a simplicial map $h \colon X \to X'$ as follows. Fix any vertex $x \in X_0$, we have $gr(x) \in B'_0$. Since $r'$ is injective-on-vertices, there is a unique vertex $x' \in X'_0$ such that $r'(x') = gr(x)$. We then define $h(x) := x'$. Next, fix any $n$-simplex $\chi \in X$ for $n \ge 1$. Since $r'$ is fully faithful, there exists a unique $n$-simplex $\chi' \in X'$ such that $r'(\chi') = gr(\chi)$. We then define $h(\chi) := \chi'$.
		
		Now for a simplex $a \in A$, we have $hl(a) = (r')^{-1}(grl(a)) = (r')^{-1}(r'l'f(a)) = l'f(a),$
		%		\begin{align*}
			%			hl(a) &= (r')^{-1}(grl(a))
			%			\\
			%			&= (r')^{-1}(r'l'f(a))
			%			\\
			%			&= l'f(a)
			%		\end{align*}
		and therefore, $l'f = hl$.
		
		Now suppose there is a simplicial map $k \colon X \to X'$ such that $l'f = kl$ and $r'k = gr$. This means that $hl = kl$ and $r'h = r'k$. As $l$ is surjective-on-vertices, for any vertex $x \in X_0$, there exists a unique $a \in A_0$ such that $x = l(a)$. Then 	$h(x) = hl(a) = kl(a) = k(x)$,
		%		\begin{align*}
			%			h(x) &= hl(a)
			%			\\
			%			&= kl(a)
			%			\\
			%			&= k(x)
			%		\end{align*}
		and so, $h|_{X_0} = k|_{X_0}$. Since $r'$ is fully faithful, for any $n$-simplex $\chi \in X$ with $n \ge 1$, if $r'h(\chi) = r'k(\chi)$, then $h(\chi) = k(\chi)$.  Altogether, we have $h = k$, proving the uniqueness of $h$.
		
		The above shows that the fully faithful and injective-on-vertices simplicial map form a right class of a factorisation system on $\sSet$.
		
		By \cite[Theorem 5.10]{IK:1986}, it then follows that the assignment $L \colon f \mapsto r_f$ for any simplicial map $f \colon S \to T$ constitutes a left adjoint $L$ to the inclusion
		%		Next, we proceed to show that the inclusion
		$$\inc \colon \F_\Delta \hookrightarrow \sSet^{\mathbb{2}}.$$

		Now since $\sSet^{\mathbb{2}}$ is (co)complete, we conclude that $\F_\Delta$ is also (co)complete.
		
		We move on to show that $\F_\Delta$ is Cartesian-closed.
		
		Indeed, it is known that $\sSet^\mathbb{2} \cong [(\Delta \times \mathbb{2})^\op, \Set]$ is Cartesian-closed. Now, note that the product of two surjective-on-vertices maps is also surjective-on-vertices, so in this case, both the left and the right classes of the factorisation system are closed under products, by \cite[Theorem 3.10]{Street:1980}, we deduce that $\F_\Delta$ is Cartesian-closed.
	\end{proof}
	
	\begin{rk}
		More explicitly, the internal hom $[S, T]$ of enhanced simplicial sets $S$ and $T$ in $\F_\Delta$ is an enhanced simplicial set with loose part $[S, T]_\lambda$ given by the usual internal hom $\sSet(S_\lambda, T_\lambda)$ in $\sSet$, and tight part $[S, T]_\tau$ given by the pullback
		\[\begin{tikzcd}[ampersand replacement=\&]
			{[S, T]_\tau} \& {[S, T]_\lambda} \\
			\& {\sSet(S_\lambda, T_\lambda)} \\
			{\sSet(S_\tau, T_\tau)} \& {\sSet(S_\tau, T_\lambda)}
			\arrow["{j_{[S, T]}}", hook, from=1-1, to=1-2]
			\arrow[from=1-1, to=3-1]
			\arrow[from=1-1, to=3-2, phantom, "\lrcorner", very near start]
			\arrow[equal, from=1-2, to=2-2]
			\arrow["{{j_S}^*}", from=2-2, to=3-2]
			\arrow["{{j_T}_*}"', hook, from=3-1, to=3-2]
		\end{tikzcd}.\]
		In particular, the unit for the monoidal structure on $\F_\Delta$ is given by $\Delta^0_{[[0]]}$.
	\end{rk}
	
	The interaction between $\F_\Delta$ and $\F$ is very similar to that between $\sSet$ and $\Cat$, as expected.

	\begin{pro}
		\label{pro:adj1}
		The nerve-realisation adjunction between the nerve functor $N \colon \mathbf{Cat} \hookrightarrow \sSet$ and the homotopy category functor $h \colon \sSet \to \mathbf{Cat}$ which preserves finite products
		% https://q.uiver.app/#q=WzAsMixbMCwwLCJcXENhdCJdLFsyLDAsIlxcc1NldCJdLFsxLDAsImgiLDIseyJvZmZzZXQiOjJ9XSxbMCwxLCJOIiwyLHsib2Zmc2V0IjoyLCJzdHlsZSI6eyJ0YWlsIjp7Im5hbWUiOiJob29rIiwic2lkZSI6InRvcCJ9fX1dLFsyLDMsIiIsMix7ImxldmVsIjoxLCJzdHlsZSI6eyJuYW1lIjoiYWRqdW5jdGlvbiJ9fV1d
		\[\begin{tikzcd}[ampersand replacement=\&]
			\mathbf{Cat} \&\& \sSet
			\arrow[""{name=0, anchor=center, inner sep=0}, "N"', shift right=2, hook, from=1-1, to=1-3]
			\arrow[""{name=1, anchor=center, inner sep=0}, "h"', shift right=2, from=1-3, to=1-1]
			\arrow["\dashv"{anchor=center, rotate=-90}, draw=none, from=1, to=0]
		\end{tikzcd}\]
		lifts to an adjunction
		% https://q.uiver.app/#q=WzAsMixbMCwwLCJcXEYiXSxbMiwwLCJcXEZfXFxEZWx0YSJdLFsxLDAsImhee1xcbWF0aGJiezJ9fSIsMix7Im9mZnNldCI6Mn1dLFswLDEsIk5ee1xcbWF0aGJiezJ9fSIsMix7Im9mZnNldCI6Miwic3R5bGUiOnsidGFpbCI6eyJuYW1lIjoiaG9vayIsInNpZGUiOiJ0b3AifX19XSxbMiwzLCIiLDIseyJsZXZlbCI6MSwic3R5bGUiOnsibmFtZSI6ImFkanVuY3Rpb24ifX1dXQ==
		\[\begin{tikzcd}[ampersand replacement=\&]
			\F \&\& {\F_\Delta}
			\arrow[""{name=0, anchor=center, inner sep=0}, "{N^{\mathbb{2}}}"', shift right=2, hook, from=1-1, to=1-3]
			\arrow[""{name=1, anchor=center, inner sep=0}, "{h^{\mathbb{2}}}"', shift right=2, from=1-3, to=1-1]
			\arrow["\dashv"{anchor=center, rotate=-90}, draw=none, from=1, to=0]
		\end{tikzcd},\]
		in which $h^{\mathbb{2}}$ preserves finite products.
	\end{pro}
	
	\begin{proof}
		Since $[\mathbb{2}, -]$ is a $2$-functor, the adjunction
		\[\begin{tikzcd}[ampersand replacement=\&]
			\mathbf{Cat} \&\& \sSet
			\arrow[""{name=0, anchor=center, inner sep=0}, "N"', shift right=2, hook, from=1-1, to=1-3]
			\arrow[""{name=1, anchor=center, inner sep=0}, "h"', shift right=2, from=1-3, to=1-1]
			\arrow["\dashv"{anchor=center, rotate=-90}, draw=none, from=1, to=0]
		\end{tikzcd}\]
		lifts to the adjunction
		\[\begin{tikzcd}[ampersand replacement=\&]
			\mathbf{Cat}^\mathbb{2} \&\& \sSet^\mathbb{2}
			\arrow[""{name=0, anchor=center, inner sep=0}, "N^\mathbb{2}"', shift right=2, hook, from=1-1, to=1-3]
			\arrow[""{name=1, anchor=center, inner sep=0}, "h^\mathbb{2}"', shift right=2, from=1-3, to=1-1]
			\arrow["\dashv"{anchor=center, rotate=-90}, draw=none, from=1, to=0]
		\end{tikzcd}\]
		in which $h^\mathbb{2}$ also preserves finite products.
		
		It then remains to check that $N^\mathbb{2}$ sends full embeddings to fully faithful and injective-on-vertices simplicial maps, whereas  $h^\mathbb{2}$ does the opposite.
		
		%		Recall that, for each $1$-category $C$, the image $N(C)$ is the homotopy coherent nerve; for each functor $F \colon C \to D$ in $\Cat$, the image $N(F)\colon N(C) \to N(D)$ is the simplicial map that sends each $0$- and, respectively, $1$-simplices of $N(C)$ just as sending objects and, respectively, morphisms under $F$. The higher dimensional simplicial simplices are composites with identities, so they are mapped correspondingly.
		%		
		%		Then $N$ extends to $N^\mathbb{2}$ as follows.
		
		For a full embedding $j_A \colon A_\tau \hookrightarrow A_\lambda$ in $\F$, we have $N^\mathbb{2}(j_A) = N(j_A)$. Indeed, as $j_A$ is injective-on-vertices, $N(j_A)$ is clearly also injective-on-vertices. Since $j_A$ is fully faithful, so $1$-simplices in $N(A_\tau)$ correspond bijectively to $1$-simplices in $N(A_\lambda)$. And since all the higher simplices are just composites of morphisms with identitiesm altogether, $N(j_A)$ is fully faithful.

		For a fully faithful and injective-on-vertices simplicial map $j_S \colon S_\tau \hookrightarrow S_\lambda$ in $\F_\Delta$, we have $h^\mathbb{2}(j_S) := h(j_S)$, which is cleearly a full embedding.
	\end{proof}
	
	%		{\hl define enhanced simplicial categories + functors + natural transformations here}
	
	Now, we are ready to do enriched category theory in $\F_\Delta$.
	
	\begin{defi}
		An \emph{enhanced simplicial category} is a category enriched in $\F_\Delta$. Similarly, an \emph{enhanced simplicial functor} is an $\F_\Delta$-functor, whereas an \emph{enhanced simplicial natural transformation} is an $\F_\Delta$-natural transformation.
	\end{defi}
	
	Unravelling, this means that for any objects $a, b \in \ob\mathbb{A}$, the mapping space amounts to a fully faithful and injective-on-vertices simplicial map
	\[\mathbb{A}(a, b)_\tau \hookrightarrow \mathbb{A}(a, b)_\lambda.\]
	In other words, the \emph{tight} $0$-arrows in $\mathbb{A}(a, b)_\tau$ are embedded into the \emph{loose} $0$-arrows in $\mathbb{A}(a, b)_\lambda$.
	
	\begin{nota}
		Adapting the convention in enhanced $2$-category theory, we write $a \to b$ for a tight $0$-arrow, and $a 
		\leadsto b$ for a loose $0$-arrow.
	\end{nota}
	
	Just like an enhanced $2$-category, an enhanced simplicial category $\bbA$ can be seen as an identity-on-objects, faithful, and locally fully faithful simplicial functor
	\begin{align*}
		J_\bbA \colon \calA_\tau &\to \calA_\lambda,
	\end{align*}
	in which $\calA_\tau$ is a simplicial category with only tight $0$-arrows, and $\calA_\lambda$ is the ambient simplicial category with all the (loose) $0$-arrows.
	
	\begin{nota}
		Similarly, there are two extreme ways to view a simplicial category $\calA$ as an enhanced simplicial category. We could take all the $0$-arrows to be tight, which gives us a \emph{chordate} enhanced simplicial category, denoted by $\calA_\chor$; or we could only take the identities to be tight, which gives us an \emph{inchordate} enhanced simplicial category, denoted simply by $\calA$.
	\end{nota}
	
	An {enhanced simplicial functor} 
	$F\colon 
	\bbA \to \bbB$  consists 
	of simplicial functors $F_\tau \colon \calA_\tau \to 
	\calB_\tau$ and $F_\lambda \colon \calA_\lambda \to \calB_\lambda$ 
	making \longref{Diagram}{diag:F_functor}
	commute.
	
	Let $F, G\colon \bbA \rightrightarrows \bbB$ be two enhanced simplicial functors. An 
	{enhanced simplicial natural 
		transformation}	$\alpha \colon F \to G$ consists of simplicial natural 
	transformations $\alpha_\tau \colon F_\tau \to G_\tau$ and 
	$\alpha_\lambda \colon F_\lambda \to G_\lambda$ making \longref{Diagram}{diag:F_nat_tran} commute.
	
	\begin{rk}
		Similar to the case in enhanced $2$-category theory, the tight part $F_\tau$ of an enhanced simplicial functor $F$ is uniquely determined by the loose part $F_\lambda$, in which $F_\lambda$ has to also preserve tight $0$-arrows.
		
		As for any enhanced simplicial natural transformations $\alpha$, the existence of $\alpha_\tau$ can be seen as 
		the extra condition on $\alpha_\lambda$ that the components are tight $0$-arrows.
	\end{rk}
	
	Examples of enhanced simplicial categories are abundant.
	
	\begin{example}[$\JLim(\qCat)$]
		\label{eg:JLimQCat}
		Quasi-categories admitting $J$-shaped limits for a simplicial set $J$ form an enhanced simplicial category, with tight $0$-arrows given by the functors that preserve $J$-shaped limits, and loose $0$-arrows given by the functors (that do not necessarily preserve $J$-shaped limits).
	\end{example}
	
	\begin{example}[$\Cart(\qCat)$]
		\label{eg:CartQCat}
		Cartesian fibrations between quasi-categories form an enhanced simplicial category, with tight $0$-arrows given by the Cartesian functors, i.e., commutative squares that preserve Cartesian lifts, and loose $0$-arrows given by plain commutative squares.
	\end{example}
	
	Indeed, with the development of the theory of $\infty$-cosmoi by Riehl and Verity in \cite{book:RV:2022}, the above examples can be considered in a more general context.
	
	\begin{defi}[{\cite[Definition 1.2.1]{book:RV:2022}}]
		\label{def:cosmos}
		An $\infty$-cosmos $\calK$ is a category enriched over quasi-categories, whose objects and morphisms are called \emph{$\infty$-categories} and \emph{$\infty$-functors}, respectively, and that for any $\infty$-categories $A$ and $B$, an $n$-simplex in the quasi-category $\calK(A, B)$ is called an \emph{$n$-arrow from $A$ to $B$}, that is also equipped with a specified collection of maps which are called \emph{isofibrations} and denoted by $\twoheadrightarrow$, satisfying the following two axioms.
		\begin{enumerate}
			\item[$\bullet$] The quasi-categorically enriched category $\calK$ possesses a terminal object, small products, pullbacks of isofibrations, limits of a countable chain of isofibrations, and powers with simplicial sets, each satisfying a universal property that is enriched in $\sSet$.
			\item[$\bullet$] The isofibrations contain all isomorphisms and any $0$-arrow whose codomain is the terminal object; are closed under composition, product, pullback, limit of chains, and Leibniz exponential with a monomorphism of simplicial sets; and have the property that if $f \colon A \twoheadrightarrow B$ is an isofibration, then $\calK(X, f) \colon \calK(X, A) \twoheadrightarrow  \calK(X, B)$ is an isofibration of quasi-categories, for any object $X \in \calK$.
		\end{enumerate}
	\end{defi}
	
	\begin{rk}
		Indeed, for a general simplicial category $\calK$, we may also call an  $n$-simplex in $\calK(A, B)$ as an \emph{$n$-arrow from $A$ to $B$}.
	\end{rk}
	
	\begin{defi}[{\cite[Definition 1.3.1]{book:RV:2022}}]
		A \emph{cosmological functor} between two $\infty$-cosmoi is a simplicial functor that preserves the isofibrations and all of the cosmological limits described in \longref{Definition}{def:cosmos}.
	\end{defi}
	
	\begin{rk}
		Cosmological limits are in particular $\sSet$-weighted limits.
		
		The axiom for isofibrations in \longref{Definition}{def:cosmos} implies that the limit projections of any cosmological limits are isofibrations.
	\end{rk}
	
	\begin{defi}[{\cite[Definition 1.2.2]{book:RV:2022}}]
		A $0$-arrow $f \colon A \to B$ in an $\infty$-cosmos is called
		\begin{enumerate}
			\item[$\bullet$] an \emph{equivalence} precisely when the induced map $f_* \colon \calK(X, A) \to \calK(X, B)$  is an equivalence of quasi-categories for any $\infty$-category $X$;
			\item[$\bullet$] a \emph{trivial fibration} precisely when it is both an isofibration and an equivalence.
		\end{enumerate}
	\end{defi}
	
	\begin{example}
		Most models for $(\infty, 1)$-categories gives an $\infty$-cosmos; for instance, quasi-categories form an $\infty$-cosmos $\qCat$.  
		
		Particularly, $1$-categories and Kan complices also form an $\infty$-cosmos $\Cat$ and $\Kan$, respectively.
	\end{example}
	
	The theory of $\infty$-cosmoi provides a setting to understand $(\infty, 1)$-categories with structure and the pseudo morphisms, i.e., morphisms that preserve the structure up to isomorphism, model-independently.
	
	\begin{example}[$(\infty, 1)$-categories with structure]
		The $\infty$-categories in an $\infty$-cosmos $\calK$ admitting $J$-shaped limits form the $\infty$-cosmos $\calJLim(\calK)$, with $0$-arrows given by the limit-preserving $\infty$-functors. Similarly, Cartesian fibrations between $\infty$-categories in an $\infty$-cosmos $\calK$ and the Cartesian $\infty$-functors between them form the $\infty$-cosmos $\calCart(\calK)$.
		
		Dually, we have the $\infty$-cosmos $\calJColim(\calK)$ of $\infty$-categories in an $\infty$-cosmos $\calK$ admitting $J$-shaped colimits, and the $\infty$-cosmos $\calCoCart(\calK)$ of coCartesian fibrations.
	\end{example}
	
	\begin{example}[$\calK^{\isof}$]
		\label{eg:isof_of_cosmos}
		The isofibrations of an $\infty$-cosmos $\calK$ form an $\infty$-cosmos $\calK^{\isof}$, which has
		\begin{enumerate}
			\item[$\bullet$] objects the isofibrations $p \colon E \twoheadrightarrow B$ of $\calK$;
			\item[$\bullet$] hom-object from $p_1 \colon E_1 \twoheadrightarrow B_1$ to $p_2 \colon E_2 \twoheadrightarrow B_2$ given by the pullback
			% https://q.uiver.app/#q=WzAsNCxbMCwwLCJcXGNhbEtee1xcaXNvZn0ocF8xLCBwXzIpIl0sWzAsMSwiXFxjYWxLKEJfMSwgQl8yKSJdLFsxLDAsIlxcY2FsSyhFXzEsIEVfMikiXSxbMSwxLCJcXGNhbEsoRV8xLCBCXzIpIl0sWzAsMl0sWzEsMywicF8xXioiLDJdLFswLDEsIiIsMix7InN0eWxlIjp7ImhlYWQiOnsibmFtZSI6ImVwaSJ9fX1dLFsyLDMsIntwXzJ9XyoiLDAseyJzdHlsZSI6eyJoZWFkIjp7Im5hbWUiOiJlcGkifX19XSxbMCwzLCIiLDEseyJzdHlsZSI6eyJuYW1lIjoiY29ybmVyIn19XV0=
			\[\begin{tikzcd}[ampersand replacement=\&]
				{\calK^{\isof}(p_1, p_2)} \& {\calK(E_1, E_2)} \\
				{\calK(B_1, B_2)} \& {\calK(E_1, B_2)}
				\arrow[from=1-1, to=1-2]
				\arrow[two heads, from=1-1, to=2-1]
				\arrow["\lrcorner"{anchor=center, pos=0.125}, draw=none, from=1-1, to=2-2]
				\arrow["{{p_2}_*}", two heads, from=1-2, to=2-2]
				\arrow["{p_1^*}"', from=2-1, to=2-2]
			\end{tikzcd}  ;\]
			\item[$\bullet$] isofibrations given by commutative squares
			% https://q.uiver.app/#q=WzAsNSxbMCwwLCJFXzEiXSxbMiwwLCJFXzIiXSxbMCwyLCJCXzEiXSxbMiwyLCJCXzIiXSxbMSwxLCJcXGJ1bGxldCJdLFswLDEsImUiLDAseyJzdHlsZSI6eyJoZWFkIjp7Im5hbWUiOiJlcGkifX19XSxbMiwzLCJiIiwyLHsic3R5bGUiOnsiaGVhZCI6eyJuYW1lIjoiZXBpIn19fV0sWzAsMiwicF8xIiwyLHsic3R5bGUiOnsiaGVhZCI6eyJuYW1lIjoiZXBpIn19fV0sWzEsMywicF8yIl0sWzQsMSwiIiwwLHsic3R5bGUiOnsiYm9keSI6eyJuYW1lIjoiZGFzaGVkIn0sImhlYWQiOnsibmFtZSI6ImVwaSJ9fX1dLFs0LDIsIiIsMCx7InN0eWxlIjp7ImJvZHkiOnsibmFtZSI6ImRhc2hlZCJ9LCJoZWFkIjp7Im5hbWUiOiJlcGkifX19XSxbNCwzLCIiLDAseyJzdHlsZSI6eyJuYW1lIjoiY29ybmVyIn19XSxbMCw0LCIiLDAseyJzdHlsZSI6eyJib2R5Ijp7Im5hbWUiOiJkYXNoZWQifSwiaGVhZCI6eyJuYW1lIjoiZXBpIn19fV1d
			\[\begin{tikzcd}[ampersand replacement=\&]
				{E_1} \&\& {E_2} \\
				\& \bullet \\
				{B_1} \&\& {B_2}
				\arrow["e", two heads, from=1-1, to=1-3]
				\arrow[dashed, two heads, from=1-1, to=2-2]
				\arrow["{p_1}"', two heads, from=1-1, to=3-1]
				\arrow["{p_2}", from=1-3, to=3-3]
				\arrow[dashed, two heads, from=2-2, to=1-3]
				\arrow[dashed, two heads, from=2-2, to=3-1]
				\arrow["\lrcorner"{anchor=center, pos=0.125}, draw=none, from=2-2, to=3-3]
				\arrow["b"', two heads, from=3-1, to=3-3]
			\end{tikzcd}\]
			in which the horizontal maps and the induced map from the initial vertex to the pullback of the cospans are isofibrations of $\calK$;
		\end{enumerate}
		where limits are computed pointwise in $\calK$. Furthermore, a $0$-arrow $(e, b) \colon p_1 \to p_2$
		% https://q.uiver.app/#q=WzAsNCxbMCwwLCJFXzEiXSxbMSwwLCJFXzIiXSxbMCwxLCJCXzEiXSxbMSwxLCJCXzIiXSxbMCwxLCJlIiwwLHsic3R5bGUiOnsiaGVhZCI6eyJuYW1lIjoiZXBpIn19fV0sWzIsMywiYiIsMix7InN0eWxlIjp7ImhlYWQiOnsibmFtZSI6ImVwaSJ9fX1dLFswLDIsInBfMSIsMix7InN0eWxlIjp7ImhlYWQiOnsibmFtZSI6ImVwaSJ9fX1dLFsxLDMsInBfMiJdXQ==
		\[\begin{tikzcd}[ampersand replacement=\&]
			{E_1} \& {E_2} \\
			{B_1} \& {B_2}
			\arrow["e", from=1-1, to=1-2]
			\arrow["{p_1}"', two heads, from=1-1, to=2-1]
			\arrow["{p_2}", from=1-2, to=2-2]
			\arrow["b"',  from=2-1, to=2-2]
		\end{tikzcd}\]
		is an equivalence if and only if both $e$ and $b$ are equivalences in $\calK$.
	\end{example}
	
	The next example is crucial to the study of $(\infty, 1)$-categories with structure, because many of the above $(\infty, 1)$-categories with structure can actually be recovered from the following.
	
	\begin{example}[$\calLali(\calK)$ and $\calRali(\calK)$]
		\label{eg:calLali}
		Recall that an isofibration $p \colon E \twoheadrightarrow B$ of an $\infty$-cosmos $\calK$ is said to be a \emph{lali} precisely when it is a lali (left-adjoint-left-inverse) in the homotopy $2$-category $h\calK$, i.e., it admits a right adjoint and that the counit is an isomorphism; a morphism of isofibrations
		% https://q.uiver.app/#q=WzAsNCxbMCwwLCJFXzEiXSxbMCwxLCJCXzEiXSxbMSwwLCJFXzIiXSxbMSwxLCJCXzIiXSxbMCwxLCJwXzEiLDIseyJzdHlsZSI6eyJoZWFkIjp7Im5hbWUiOiJlcGkifX19XSxbMiwzLCJwXzIiLDAseyJzdHlsZSI6eyJoZWFkIjp7Im5hbWUiOiJlcGkifX19XSxbMCwyLCJlIl0sWzEsMywiYiIsMV1d
		\[\begin{tikzcd}[ampersand replacement=\&]
			{E_1} \& {E_2} \\
			{B_1} \& {B_2}
			\arrow["e", from=1-1, to=1-2]
			\arrow["{p_1}"', two heads, from=1-1, to=2-1]
			\arrow["{p_2}", two heads, from=1-2, to=2-2]
			\arrow["b"', from=2-1, to=2-2]
		\end{tikzcd}\]
		is a \emph{morphism of lalis} precisely if its mate is an isomorphism, i.e. the composite
		% https://q.uiver.app/#q=WzAsNixbMSwwLCJFXzEiXSxbMSwxLCJCXzEiXSxbMiwwLCJFXzIiXSxbMiwxLCJCXzIiXSxbMCwxLCJCXzEiXSxbMywwLCJFXzIiXSxbMCwxLCJwXzEiLDIseyJzdHlsZSI6eyJoZWFkIjp7Im5hbWUiOiJlcGkifX19XSxbMiwzLCJwXzIiLDAseyJzdHlsZSI6eyJoZWFkIjp7Im5hbWUiOiJlcGkifX19XSxbMCwyLCJlIl0sWzEsMywiYiIsMV0sWzQsMSwiIiwxLHsic3R5bGUiOnsiaGVhZCI6eyJuYW1lIjoibm9uZSJ9fX1dLFsyLDUsIiIsMSx7InN0eWxlIjp7ImhlYWQiOnsibmFtZSI6Im5vbmUifX19XSxbNCwwLCJyXzEiXSxbMyw1LCJyXzIiLDJdLFsxMiwxMCwiXFxlcHNpbG9uXzEiLDAseyJzaG9ydGVuIjp7InNvdXJjZSI6MjAsInRhcmdldCI6MjB9fV0sWzExLDEzLCJcXGV0YV8yIiwyLHsic2hvcnRlbiI6eyJzb3VyY2UiOjIwLCJ0YXJnZXQiOjIwfX1dXQ==
		\[\begin{tikzcd}[ampersand replacement=\&]
			\& {E_1} \& {E_2} \& {E_2} \\
			{B_1} \& {B_1} \& {B_2}
			\arrow["e", from=1-2, to=1-3]
			\arrow["{p_1}"', two heads, from=1-2, to=2-2]
			\arrow[""{name=0, anchor=center, inner sep=0}, no head, equal, from=1-3, to=1-4]
			\arrow["{p_2}", two heads, from=1-3, to=2-3]
			\arrow[""{name=1, anchor=center, inner sep=0}, "{r_1}", from=2-1, to=1-2]
			\arrow[""{name=2, anchor=center, inner sep=0}, no head, equal, from=2-1, to=2-2]
			\arrow["b"', from=2-2, to=2-3]
			\arrow[""{name=3, anchor=center, inner sep=0}, "{r_2}"', from=2-3, to=1-4]
			\arrow["{\eta_2}"', shorten <=2pt, shorten >=2pt, Rightarrow, from=0, to=3]
			\arrow["{\epsilon_1}", shorten <=2pt, shorten >=2pt, Rightarrow, from=1, to=2]
		\end{tikzcd}\]
		is invertible as a $2$-morphism in $h\calK$, where $r_i$ denotes the right adjoint for $p_i$, and that $\epsilon_i$ and $\eta_i$ denote the corresponding counit and unit, respectively, for $i = 1, 2$.
		
		The isofibrations of $\calK$ which are lalis form an $\infty$-cosmos $\calLali(\calK)$, with $0$-arrows given by the morphisms of lalis.
		
		Dually,  an isofibration $p \colon E \twoheadrightarrow B$ is said to be a \emph{rali} just when it is a rali (right-adjoint-left-inverse) in the homotopy $2$-category $h\calK$, i.e., it admits a left adjoint and that the unit is an isomorphism; a morphism of isofibrations
		% https://q.uiver.app/#q=WzAsNCxbMCwwLCJFXzEiXSxbMCwxLCJCXzEiXSxbMSwwLCJFXzIiXSxbMSwxLCJCXzIiXSxbMCwxLCJwXzEiLDIseyJzdHlsZSI6eyJoZWFkIjp7Im5hbWUiOiJlcGkifX19XSxbMiwzLCJwXzIiLDAseyJzdHlsZSI6eyJoZWFkIjp7Im5hbWUiOiJlcGkifX19XSxbMCwyLCJlIl0sWzEsMywiYiIsMV1d
		\[\begin{tikzcd}[ampersand replacement=\&]
			{E_1} \& {E_2} \\
			{B_1} \& {B_2}
			\arrow["e", from=1-1, to=1-2]
			\arrow["{p_1}"', two heads, from=1-1, to=2-1]
			\arrow["{p_2}", two heads, from=1-2, to=2-2]
			\arrow["b"', from=2-1, to=2-2]
		\end{tikzcd}\]
		is a \emph{morphism of ralis} precisely if its mate is an isomorphism, i.e. the composite
		% https://q.uiver.app/#q=WzAsNixbMSwwLCJFXzEiXSxbMSwxLCJCXzEiXSxbMiwwLCJFXzIiXSxbMiwxLCJCXzIiXSxbMCwxLCJCXzEiXSxbMywwLCJFXzIiXSxbMCwxLCJwXzEiLDIseyJzdHlsZSI6eyJoZWFkIjp7Im5hbWUiOiJlcGkifX19XSxbMiwzLCJwXzIiLDAseyJzdHlsZSI6eyJoZWFkIjp7Im5hbWUiOiJlcGkifX19XSxbMCwyLCJlIl0sWzEsMywiYiIsMl0sWzQsMSwiIiwxLHsic3R5bGUiOnsiaGVhZCI6eyJuYW1lIjoibm9uZSJ9fX1dLFsyLDUsIiIsMSx7InN0eWxlIjp7ImhlYWQiOnsibmFtZSI6Im5vbmUifX19XSxbNCwwLCJsXzEiXSxbMyw1LCJsXzIiLDJdLFsxMCwxMiwiXFxldGFfMSIsMix7InNob3J0ZW4iOnsic291cmNlIjoyMCwidGFyZ2V0IjoyMH19XSxbMTMsMTEsIlxcZXBzaWxvbl8yIiwyLHsic2hvcnRlbiI6eyJzb3VyY2UiOjIwLCJ0YXJnZXQiOjIwfX1dXQ==
		\[\begin{tikzcd}[ampersand replacement=\&]
			\& {E_1} \& {E_2} \& {E_2} \\
			{B_1} \& {B_1} \& {B_2}
			\arrow["e", from=1-2, to=1-3]
			\arrow["{p_1}"', two heads, from=1-2, to=2-2]
			\arrow[""{name=0, anchor=center, inner sep=0}, equal, from=1-3, to=1-4]
			\arrow["{p_2}", two heads, from=1-3, to=2-3]
			\arrow[""{name=1, anchor=center, inner sep=0}, "{l_1}", from=2-1, to=1-2]
			\arrow[""{name=2, anchor=center, inner sep=0}, equal, from=2-1, to=2-2]
			\arrow["b"', from=2-2, to=2-3]
			\arrow[""{name=3, anchor=center, inner sep=0}, "{l_2}"', from=2-3, to=1-4]
			\arrow["{\eta_1}"', shorten <=2pt, shorten >=2pt, Rightarrow, from=2, to=1]
			\arrow["{\epsilon_2}", shorten <=2pt, shorten >=2pt, Rightarrow, from=3, to=0]
		\end{tikzcd}\]
		is invertible, where $l_i$ denotes the left adjoint for $p_i$, and that $\epsilon_i$ and $\eta_i$ denote the corresponding counit and unit, respectively, for $i = 1, 2$.
		
		The isofibrations of $\calK$ which are ralis form an $\infty$-cosmos $\calRali(\calK)$, with $0$-arrows given by the morphisms of ralis.
	\end{example}
	
	Due to the importance of these concepts, we decide to give them a name.
	
	\begin{defi}
		A \emph{lali-isofibration} of an $\infty$-cosmos $\calK$ is an isofibration of $\calK$ which is also a lali.
		
		A \emph{rali-isofibration} of an $\infty$-cosmos $\calK$ is an isofibration of $\calK$ which is also a rali.
	\end{defi}
	
	We can then consider the previous examples \longref{Example}{eg:JLimQCat} and \longref{Example}{eg:CartQCat} of enhanced simplicial categories model-independently using the language of $\infty$-cosmoi. We will also provide the characterisations of different $(\infty, 1)$-categories with structure in terms of lali-isofibrations or rali-isofibrations below.
	
	\begin{example}[$\JLim(\calK)$ and $\JColim(\calK)$]
		\label{eg:JLimK}
		For any $\infty$-cosmos $\calK$, the $\infty$-categories in $\calK$ that have $J$-shaped limits for a simplicial set $J$ form an enhanced simplicial category $\JLim(\calK)$. The loose $0$-arrows are $\infty$-functors, whereas the tight $0$-arrows are those limit-preserving $\infty$-functors.
		
		A characterisation of $\infty$-categories with $J$-limits is given as follows. Consider the cone over $J$, denoted by $J_\triangleleft$, which is given by the join $\Delta^0 \star J$ of the simplicial sets $\Delta^0$ and $J$. Then, as illustrated in \cite[Chapter 4.2]{book:RV:2022}, the inclusion $J \hookrightarrow J_\triangleleft$ induces a restriction map
		$$\res \colon A^{J_\triangleleft} \twoheadrightarrow A^J,$$
		which is an isofibration of $\calK$. By \cite[Corollary 4.3.5]{book:RV:2022}, $A$ admits $J$-shaped limits if and only if $\res \colon A^{J_\triangleleft} \twoheadrightarrow A^J$ is a lali of $\calK$; moreover, for two $\infty$-categories $A$ and $B$ that admit $J$-shaped limits, a $0$-arrow $f \colon A \to B$ preserves limits precisely when the square
		% https://q.uiver.app/#q=WzAsNCxbMCwwLCJBXntKX1xcdHJpYW5nbGVsZWZ0fSJdLFsxLDAsIkJee0pfXFx0cmlhbmdsZWxlZnR9Il0sWzAsMSwiQV5KIl0sWzEsMSwiQl5KIl0sWzAsMiwiXFxyZXNfQSIsMix7InN0eWxlIjp7ImhlYWQiOnsibmFtZSI6ImVwaSJ9fX1dLFsxLDMsIlxccmVzX0IiLDAseyJzdHlsZSI6eyJoZWFkIjp7Im5hbWUiOiJlcGkifX19XSxbMCwxLCJmXntKX1xcdHJpYW5nbGVsZWZ0fSJdLFsyLDMsImZeSiIsMl1d
		\begin{equation}
			\label{diag:res}
			\begin{tikzcd}[ampersand replacement=\&]
				{A^{J_\triangleleft}} \& {B^{J_\triangleleft}} \\
				{A^J} \& {B^J}
				\arrow["{f^{J_\triangleleft}}", from=1-1, to=1-2]
				\arrow["{\res_A}"', two heads, from=1-1, to=2-1]
				\arrow["{\res_B}", two heads, from=1-2, to=2-2]
				\arrow["{f^J}"', from=2-1, to=2-2]
			\end{tikzcd}
		\end{equation}
		is a morphism of lalis.
		
		Dually, the $\infty$-categories in $\calK$ that have $J$-shaped colimits form an enhanced simplicial category $\JColim(\calK)$. The loose $0$-arrows are $\infty$-functors, whereas the tight $0$-arrows are those colimit-preserving $\infty$-functors.
		
		Similarly, consider this time the cone under $J$, denoted by $J_\triangleright$, which is given by the join $J \star \Delta^0$. Again, the inclusion $J \hookrightarrow J_\triangleright$ induces an isofibration
		$$\res \colon A^{J_\triangleright} \twoheadrightarrow A^J$$
		of $\calK$. The $\infty$-category $A$ admits $J$-shaped colimits precisely when this restriction map is a rali, and that a $0$-arrow $f \colon A \to B$ between $\infty$-categories with $J$-shaped colimits preserves colimits just when the evident square is a morphism of ralis.
	\end{example}
	
	\begin{example}[$\Cart(\calK)$ and $\CoCart(\calK)$]
		\label{eg:CartK}
		Similarly, the Cartesian fibrations between $\infty$-categories in an $\infty$-cosmos $\calK$ form an enhanced simplicial category $\Cart(\calK)$. The loose $0$-arrows are the morphisms of isofibrations, whereas the tight $0$-arrows are the Cartesian functors.
		
		As formulated in \cite[Theorem 5.2.8, Theorem 5.3.4]{book:RV:2022}, an isofibration $p \colon E \twoheadrightarrow B$ of $\calK$ is a Cartesian fibration precisely when its Leibniz exponential $i_1 \hat{\pitchfork} p$ with $i_1 \colon \Delta^0 \hookrightarrow \Delta^1$ which picks out the target is a lali; a commutative square
		% https://q.uiver.app/#q=WzAsNCxbMCwwLCJFXzEiXSxbMCwxLCJCXzEiXSxbMSwwLCJFXzIiXSxbMSwxLCJCXzIiXSxbMCwyLCJlIl0sWzEsMywiYiIsMl0sWzAsMSwicF8xIiwyLHsic3R5bGUiOnsiaGVhZCI6eyJuYW1lIjoiZXBpIn19fV0sWzIsMywicF8yIiwwLHsic3R5bGUiOnsiaGVhZCI6eyJuYW1lIjoiZXBpIn19fV1d
		\[\begin{tikzcd}[ampersand replacement=\&]
			{E_1} \& {E_2} \\
			{B_1} \& {B_2}
			\arrow["e", from=1-1, to=1-2]
			\arrow["{p_1}"', two heads, from=1-1, to=2-1]
			\arrow["{p_2}", two heads, from=1-2, to=2-2]
			\arrow["b"', from=2-1, to=2-2]
		\end{tikzcd}\]
		where both $p_1$ and $p_2$ are Cartesian fibrations is a Cartesian functor if and only if the square
		% https://q.uiver.app/#q=WzAsNCxbMCwwLCJFXzFee1xcRGVsdGFeMX0iXSxbMCwxLCJCXzFcXGRvd25hcnJvdyBwXzEiXSxbMSwwLCJFXzJee1xcRGVsdGFeMX0iXSxbMSwxLCJCXzIgXFxkb3duYXJyb3cgcF8yIl0sWzAsMiwiZV57XFxEZWx0YV4xfSJdLFsxLDMsImIgXFxkb3duYXJyb3cgZSIsMl0sWzAsMSwiaV8xIFxcaGF0e1xccGl0Y2hmb3JrfXBfMSIsMix7InN0eWxlIjp7ImhlYWQiOnsibmFtZSI6ImVwaSJ9fX1dLFsyLDMsImlfMSBcXGhhdHtcXHBpdGNoZm9ya31wXzIiLDAseyJzdHlsZSI6eyJoZWFkIjp7Im5hbWUiOiJlcGkifX19XV0=
		\begin{equation}
			\label{diag:Leibniz}
			\begin{tikzcd}[ampersand replacement=\&]
				{E_1^{\Delta^1}} \& {E_2^{\Delta^1}} \\
				{B_1\downarrow p_1} \& {B_2 \downarrow p_2}
				\arrow["{e^{\Delta^1}}", from=1-1, to=1-2]
				\arrow["{i_1 \hat{\pitchfork}p_1}"', two heads, from=1-1, to=2-1]
				\arrow["{i_1 \hat{\pitchfork}p_2}", two heads, from=1-2, to=2-2]
				\arrow["{b \downarrow e}"', from=2-1, to=2-2]
			\end{tikzcd}
		\end{equation}		
		is a morphism of lalis. Here, $B_i \downarrow p_i$ denotes the comma $\infty$-category from  the identity $1_{B_i}$ to $p_i$, for $i = 1, 2$, and $b \downarrow p$ denotes the canonical map induced by the universal property of comma $\infty$-categories.
		
		Dually, the coCartesian fibrations between $\infty$-categories in an $\infty$-cosmos $\calK$ form an enhanced simplicial category $\CoCart(\calK)$, whose loose $0$-arrows are simply morphisms of isofibrations, whereas the tight $0$-arrows are coCartesian functors.
		
		An isofibration $p \colon E \twoheadrightarrow B$ of $\calK$ is a coCartesian fibration precisely when its Leibniz exponential $i_0 \hat{\pitchfork} p$ with $i_0 \colon \Delta^0 \hookrightarrow \Delta^1$ which picks out the source is a rali; a coCartesian functor between coCartesian fibrations $p_1$ and $p_2$ consists of a square
		\[\begin{tikzcd}[ampersand replacement=\&]
			{E_1} \& {E_2} \\
			{B_1} \& {B_2}
			\arrow["e", from=1-1, to=1-2]
			\arrow["{p_1}"', two heads, from=1-1, to=2-1]
			\arrow["{p_2}", two heads, from=1-2, to=2-2]
			\arrow["b"', from=2-1, to=2-2]
		\end{tikzcd}\]
		where
		% https://q.uiver.app/#q=WzAsNCxbMCwwLCJFXzFee1xcRGVsdGFeMX0iXSxbMCwxLCJwXzFcXGRvd25hcnJvdyBCXzEiXSxbMSwwLCJFXzJee1xcRGVsdGFeMX0iXSxbMSwxLCJwXzIgXFxkb3duYXJyb3cgQl8yIl0sWzAsMiwiZV57XFxEZWx0YV4xfSJdLFsxLDMsImUgXFxkb3duYXJyb3cgYiIsMl0sWzAsMSwiaV8wIFxcaGF0e1xccGl0Y2hmb3JrfXBfMSIsMix7InN0eWxlIjp7ImhlYWQiOnsibmFtZSI6ImVwaSJ9fX1dLFsyLDMsImlfMCBcXGhhdHtcXHBpdGNoZm9ya31wXzIiLDAseyJzdHlsZSI6eyJoZWFkIjp7Im5hbWUiOiJlcGkifX19XV0=
		\[\begin{tikzcd}[ampersand replacement=\&]
			{E_1^{\Delta^1}} \& {E_2^{\Delta^1}} \\
			{p_1\downarrow B_1} \& {p_2 \downarrow B_2}
			\arrow["{e^{\Delta^1}}", from=1-1, to=1-2]
			\arrow["{i_0 \hat{\pitchfork}p_1}"', two heads, from=1-1, to=2-1]
			\arrow["{i_0 \hat{\pitchfork}p_2}", two heads, from=1-2, to=2-2]
			\arrow["{e \downarrow b}"', from=2-1, to=2-2]
		\end{tikzcd}\]
		is a morphism of ralis.
	\end{example}
	
	%	\begin{example}[$\DisCart(\calK)$ and $\DisCoCart(\calK)$]
		%		\label{eg:DisCartK}
		%		Recall from \cite[Definition 5.5.3]{book:RV:2022} that a Cartesian fibration $p \colon E \twoheadrightarrow B$ is a \emph{discrete Cartesian fibration} precisely when it is a discrete object of the sliced $\infty$-cosmos $\calK_{/B}$, i.e., the hom-space $\calK_{/B}(x, p)$ is a Kan complex for any $x \in \calK_{/B}$. A characterisation of discrete Carteisian fibrations is given by \cite[Proposition 5.5.8]{book:RV:2022}, which states that an isofibration $p \colon E \twoheadrightarrow$ of $\calK$ is a discrete Cartesian fibration precisely when the Leibniz exponential described in \longref{Example}{eg:CartK} is a trivial fibration. Clearly, they form an enhanced simplicial sub-category $\DisCart(\calK)$ of $\Cart(\calK)$.
		%		
		%		Dually, a coCartesian fibration $p \colon E \twoheadrightarrow B$ is a \emph{discrete coCartesian fibration} precisely when it is a discrete object of $\calK_{/B}$. Again, an isofibration is a discrete coCartesian fibration precisely when its Leibniz exponential with $i_0$ is actually a trivial fibration. They form an enhanced simplicial sub-category $\DisCoCart(\calK)$ of $\CoCart(\calK)$.
		%	\end{example}
	
	The next example serves as the key in this article that enables us to deduce numerous results regarding $(\infty, 1)$-categories with structure and their lax morphisms.
	
	\begin{example}[$\Lali(\calK)$ and $\Rali(\calK)$]
		\label{eg:bbLali}
		We see in \longref{Example}{eg:calLali} that lali-isofibrations of an $\infty$-cosmos form an $\infty$-cosmos $\calLali(\calK)$. We can then take the morphisms of isofibrations (which are not necessarily morphisms of lalis) as loose $0$-arrows, and obtain an enhanced simplicial category $\Lali(\calK)$.
		
		Dually, we could form an enhanced simplicial category $\Rali(\calK)$ from the $\infty$-cosmos $\calRali(\calK)$, in the same way.
		
		We will see in \longref{Section}{sec:complete} that $\Lali(\calK)$ and $\Rali(\calK)$ produce a lot of interesting examples of enhanced simplicial categories, including the above, and that limits in $\Lali(\calK)$ and $\Rali(\calK)$, respectively, are inherited by them.
	\end{example}
	
	We continue with other fundamental examples of enhanced simplicial categories. Enriched category theory provides the notion of functor categories in the enriched context.
	
	\begin{example}[Functor $\F_\Delta$-categories]
		Let $\bbA$ and $\bbB$ be enhanced simplicial categories. The functor category $[\bbA, \bbB]$ has
		\begin{enumerate}
			\item[$\bullet$] objects given by the enhanced simplicial functors $\bbA \to \bbB$;
			\item[$\bullet$] hom-object $[\bbA, \bbB](F, G)$ from $F$ to $G$ given by the enhanced simplicial set whose loose part is the simplicial set $[\calA_\lambda, \calB_\lambda](F_\lambda, G_\lambda)$, and the tight vertices given by  simplicial natural transformations $F_\lambda \to G_\lambda$ that restrict to enhanced simplicial natural transformations.
		\end{enumerate}
		More explicitly, a loose $0$-arrow from $F$ to $G$ in $[\bbA, \bbB]$ is a simplicial map $\Delta^0 \to [\calA_\lambda, \calB_\lambda](F_\lambda, G_\lambda)$, which is simply a simplicial natural transformation between loose parts; a tight $0$-arrow is an enhanced simplicial natural transformation. For $i > 0$, an $i$-arrow from $F$ to $G$ is a simplicial map $\Delta^i \to [\calA_\lambda, \calB_\lambda](F_\lambda, G_\lambda)$.
	\end{example}
	
	Finally, we note that $\F_\Delta$ can also be viewed as an enhanced simplicial category.
	
	\begin{example}[$\bbF_\Delta$]
		Since $\F_\Delta$ is Cartesian-closed, it is actually self-enriched, and we denote it by $\bbF_\Delta$ when it is considered as an $\F_\Delta$-category.
		
		An object of $\bbF_\Delta$ is an object of $\F_\Delta$. A tight $0$-arrow in $\bbF_\Delta$ is a $0$-arrow in $\F_\Delta$, whereas a loose $0$-arrow in $\bbF_\Delta$ is simply a simplicial map between the loose parts. The higher arrows in $\bbF_\Delta$ coincide with those in $\F_\Delta$.
	\end{example}

	Enhanced simplicial categories are basically the extension of enhanced $2$-categories to the simplicial setting. Indeed, the following result tells us that any enhanced simplicial category has an enhanced homotopy $2$-category.
	
	\begin{coroll}
		There is a change-of-base adjunction
		% https://q.uiver.app/#q=WzAsMixbMCwwLCJcXEZDYXQiXSxbMiwwLCJcXEZEZWx0YUNhdCJdLFsxLDAsImhee1xcbWF0aGJiezJ9fV8qIiwyLHsib2Zmc2V0IjoyfV0sWzAsMSwiTl57XFxtYXRoYmJ7Mn19XyoiLDIseyJvZmZzZXQiOjIsInN0eWxlIjp7InRhaWwiOnsibmFtZSI6Imhvb2siLCJzaWRlIjoidG9wIn19fV0sWzIsMywiIiwyLHsibGV2ZWwiOjEsInN0eWxlIjp7Im5hbWUiOiJhZGp1bmN0aW9uIn19XV0=
		\[\begin{tikzcd}[ampersand replacement=\&]
			{\FCat} \&\& {\FDeltaCat}
			\arrow[""{name=0, anchor=center, inner sep=0}, "{N^{\mathbb{2}}_*}"', shift right=2, hook, from=1-1, to=1-3]
			\arrow[""{name=1, anchor=center, inner sep=0}, "{h^{\mathbb{2}}_*}"', shift right=2, from=1-3, to=1-1]
			\arrow["\dashv"{anchor=center, rotate=-90}, draw=none, from=1, to=0]
		\end{tikzcd}\]
		between the $2$-category of enhanced $2$-categories and the $2$-category of enhanced simplicial categories.
		Therefore, an enhanced simplicial category, i.e., an $\F_\Delta$-category, has an enhanced homotopy $2$-category, i.e., a homotopy $\F$-category. Moreover, the nerve embedding ${N^{\mathbb{2}}_*}$ is fully faithful.
		
		In other words, $\FCat$ is a reflective sub-$2$-category of $\FDeltaCat$.
	\end{coroll}
	
	\begin{proof}
		Since both the left and the right adjoints in \longref{Proposition}{pro:adj1} preserve finite products, the adjunction lifts to one as claimed.
	\end{proof}
	
	\begin{rk}
		For an $\F_\Delta$-category $\bbA$, the underlying $2$-category of $h^\mathbb{2}_*(\bbA)$ is the homotopy $2$-category of the underlying simplicial category of $\bbA$.
		
		Indeed, for objects $a, b \in \bbA_0$, the hom-object $\bbA(a, b)$ consists of a fully faithful and injective-on-vertices simplicial map
		$$j_\bbA \colon \bbA(a, b)_\tau \hookrightarrow \bbA(a, b)_\lambda.$$
		The action of $h^\mathbb{2}_*$ on $\bbA$ takes $\bbA$ to an $\F$-category $h\bbA$ whose object set is given by $\bbA_0$, whereas the hom-object $h\bbA(a, b)$ consists of a full embedding
		$$h\bbA(a, b)_\tau \hookrightarrow h\bbA(a, b)_\lambda,$$
		where $h\bbA(a, b)_\tau$ is the homotopy category of the simplicial set $\bbA(a, b)_\tau$, and $h\bbA(a, b)_\lambda$ is the homotopy category of the simplicial set $\bbA(a, b)_\lambda$.
	\end{rk}
	
	Now, we look into $\F_\Delta$-weights and limits weighted by them.
	
	\begin{defi}
		An \emph{$\F_\Delta$-weight} is an enhanced simplicial functor $W \colon \bbA \to \bbF_\Delta$ from any small enhanced simplicial category $\bbA$.
	\end{defi}
	
	More explicitly, an $\F_\Delta$-weight $W \colon \bbA \to \bbF_\Delta$ consists of a pair of simplicial functors $W_\tau \colon \calA_\tau \to \F_\Delta$ and $W_\lambda \colon \calA_\lambda \to \sSet$ satisfying
	\begin{center}
		\begin{tikzcd}
			\calA_\tau \ar[r, hook, "J_\bbA"] \ar[d, "W_\tau"'] & \calA_\lambda 
			\ar[d, 
			"W_\lambda"]
			\\
			\F_\Delta \ar[r, hook, "J_{\bbF_\Delta}"'] & \sSet
		\end{tikzcd}.
	\end{center}
	So for each object $S \in \calA_\tau$, the image $W_\tau(S)$ is an enhanced simplicial set, and that $W_\lambda J_\bbA(S)$ is only a simplicial set. And for each tight $0$-arrow $f \in \calA_\tau$, its image $W_\tau(f)$ is a simplicial map between enhanced simplicial sets that preserves tight vertices, whereas $W_\lambda J_\bbA(f)$ is only a simplicial map.
	
	Altogether, the $\F_\Delta$-weight $W$ then amounts to a pair of simplicial functors $\Phi_\tau \colon \calA_\tau \to \sSet$ and $\Phi_\lambda \colon \calA_\lambda \to \sSet$, together with a simplicial natural transformation $\varphi \colon \Phi_\tau \to \Phi_\lambda J_\bbA$ whose $1$-components are fully faithful and injective-on-vertices simplicial map. We can rewrite the $\F_\Delta$-weight as a triple $\Phi := (\Phi_\tau, \Phi_\lambda, \varphi)$.
	
	If $\Phi_1 := (\Phi_{1\tau}, \Phi_{1\lambda}, \varphi_1)$ and $\Phi_2 := (\Phi_{2\tau}, \Phi_{2\lambda}, \varphi_2)$ are two $\F_\Delta$-weights, then an $\F_\Delta$-natural transformation $\beta \colon \Phi_1 \to \Phi_2$ amounts to a pair of simplicial natural transformations $\beta_\tau \colon \Phi_{1\tau} \to \Phi_{2\tau}$ and $\beta_\lambda \colon \Phi_{1\lambda} \to \Phi_{2\lambda}$, satisfying $\beta_\lambda J_\bbA \varphi_1 = \varphi_2 \beta_\tau$.
	
	Now, let $F \colon \bbA \to \bbB$ be an enhanced simplicial functor. A \emph{cone over $F$} consists of an object $Z \in \ob\bbB$, together with an enhanced simplicial natural transformation $\zeta \colon \Phi \to \bbB(Z, F-)$. A \emph{limit cone over $F$} is a cone $\eta \colon \Phi \to \bbB(\{\Phi, F\}, F-)$ over $F$, such that the enhanced simplicial map
	\begin{equation*}
		\label{def:sigmalim}
		\bbB(B, \{\Phi, F\}) \cong [\bbA, \bbF_\Delta](\Phi, 
		\bbB(B, F-))
	\end{equation*}
	induced by applying Yoneda's Lemma to each component $\eta_A$ at  $A \in \ob\bbA$ of the limit cone is an isomorphism in $\F_\Delta$. The object $\{\Phi, F\}$ in $\bbB$ is called the \emph{weighted limit of $F$ by $\Phi$}. For any vertex $v \in \Phi(A)$, the image $\eta_A(v) \colon \{\Phi, F\} \to F(A)$ is called a \emph{limit projection}. In particular, if $v$ is a tight vertex, i.e. $v \in \Phi_\tau(A)$, then the corresponding limit projection is a tight $0$-arrow; moreover, the universal property implies that these limit projections jointly reflect tightness, i.e. a $0$-arrow $b \colon B \leadsto \{\Phi, F\}$ is tight if and only if $\eta_A(v) \cdot b$ is tight for any $A$ and $v \in \Phi_\tau(A)$.

	%	
	%	Then the \emph{weighted limit $\{\Phi, F\}$ of $F$ by $\Phi$} is characterised by an isomorphism
	%	\begin{equation*}
		%		\label{def:sigmalim}
		%		\bbB(B, \{\Phi, F\}) \cong [\bbA, \bbF_\Delta](\Phi, 
		%		\bbB(B, F-))
		%	\end{equation*}
	%	in $\F_\Delta$, which is natural in $B \in \ob\bbB$. 
	
	\begin{example}[Tight limits]
		\label{eg:tight_lim}
		When the indexing enhanced simplicial category for an $\F_\Delta$-weighted limit is in fact chordate, the limit is called a \emph{tight limit}.
		
		More precisely, consider an $\F_\Delta$-weight $\Phi := (\Phi_\tau, \Phi_\lambda, \varphi) \colon \twoA \to \bbF_\Delta$, in which $\twoA$ is a simplicial category, viewed as a chordate enhanced simplicial category. In this case, $\Phi_\tau = \Phi_\lambda$, and $\varphi = 1$. For an enhanced simplicial functor $F \colon \twoA \to \bbB$, which becomes a simplicial functor $F_\tau \colon \twoA \to \twoB_\tau$, the $\F_\Delta$-weighted limit $\{\Phi, F\}$ is a tight limit. Indeed, it is given by the $\sSet$-weighted limit $\{\Phi_\tau, F_\tau\}$ in $\twoB_\tau$, which is preserved by ${J_\bbB} \colon \twoB_\tau \to \twoB_\lambda$; in other words, $\{\Phi, F\}$ can be understood as a $\sSet$-weighted limit in $\twoB_\lambda$, in addition, the limit projections are tight, and they jointly reflect tightness, which precisely means that any arbitrary $0$-arrow $X \leadsto \{\Phi_\tau, F_\tau\}$ of $\bbB$  is tight if and only if its composites with the limit projections are all tight.
		
		An important example of tight limits is \emph{simplicial powers by inchordate enhanced simplicial sets}. Given any simplicial set $J$, we can view $J$ as an inchordate enhanced simplicial set in $\F_\Delta$. Then, for any object $A$ in an enhanced simplicial category $\bbK$, the simplicial power $A^J$ of $A$ by $J$ is given by the usual power in enriched category theory. In this case, the projections $\ev[[n]] \colon A^J \to A$ for all $n$ are tight and jointly reflect tightness. 
		%		Moreover, powering by an inchordate enhanced simplicial set is actually an enhanced simplicial functor. Let $f \colon A \to B$ be a tight $0$-arrow in an enhanced simplicial category $\bbK$; in other words, $f$ amounts to an enhanced simplicial map $\Delta^0_{[[0]]} \to \bbK(A, B)$. Then $f^J$ is constructed as follows. Note that $\ev\cdot \colon J \to \bbK(A^J, A)$, which sends any $[[j]]$ to the evaluation at $[[j]]$, is an enhanced simplicial map, because every evaluation is tight. Also, since $f$ is tight, the post-composition $\bbK(A^J, f) \colon \bbK(A^J, A) \to \bbK(A^J, B)$ by $f$ is also an enhanced simplicial map, i.e., it preserves tight vertices. Altogether, we obtain an enhanced simplicial map
		%		$$J \xrightarrow{\ev\cdot} \bbK(A^J, A) \xrightarrow{\bbK(A^J, f)} \bbK(A^J, B),$$
		%		which then corresponds to an enhanced simplicial map $\Delta^0_{[[0]]} \to \bbK(A^J, B)^J \cong \bbK(A^J, B^J)$. As a consequence, we obtain a tight $0$-arrow $f^J \colon A^J \to B^J$ in $\bbK$.
		
		Another important class of tight limits is \emph{limits of a countable chain of tight $0$-arrows}. Suppose $\bbK$ is an enhanced simplicial category. 
		%		where the tight part $\calK_\tau$ is indeed an $\infty$-cosmos. 
		Then, for any countable chain of tight $0$-arrows, 
		%		i.e. $0$-arrows which are isofibrations in $\calK_\tau$,
		its limit is given by the usual one in  $\calK_\tau$, with the extra property that the projections are tight and they jointly reflect tightness.
		
		When $\calK_\tau$ is an $\infty$-cosmos, then a cosmological limit in $\calK_\tau$ which is preserved by ${J_\bbK} \colon \calK_\tau \to \calK_\lambda$ is called a \emph{tight cosmological limit}. In this case, the limit projections are tight isofibrations, and they jointly reflect tightness.
	\end{example}
	
	\begin{nota}
		\label{nt:co}
		Let $\bbK$ be an enhanced simplicial category. We denote by $\bbK^\co$ the enhanced simplicial category which has the same objects as $\bbK$; but hom-object
		$\bbK^\co(A, B) := \bbK(A, B)^\op$
		for any $A, B$ in $\bbK^\co$ as the opposite simplicial set of the hom-object $\bbK(A, B)$ in $\bbK$. 
		%		
		%		There is a canonical enhanced simplicial functor
		%		$$\bbK^\co \xrightarrow{\cong} \bbK$$
		%		which is identity-on-objects and identity-on-$n$-arrows for $n \ne 1$. For $1$-arrows, it maps them to the corresponding reversed $1$-arrows. For any enhanced simplicial functor $F \colon \bbA \to \bbB$, we denote by 
		%		$$F_\co \colon \bbA^\co \to \bbB$$
		%		the composite $\bbA^\co \xrightarrow{\cong} \bbA \xrightarrow{F} \bbB$.
		
		For an enhanced simplicial functor $F \colon \bbA \to \bbB$, there is an associated functor $F^\co \colon \bbA^\co \to \bbB^\co$. On objects, $F^\co$ is the same as $F$; on any hom-object $\bbA^\co(x, y) = \bbA(x, y)^\op$, $(F^\co)_{x, y}$ is simply given by $(F_{x, y})^\op$.
		
		Note that there is an enhanced simplicial functor $\cdot\;^\op \colon \bbF_\Delta^\co \to \bbF_\Delta$, which sends every enhanced simplicial set $S$ to its opposite $S^\op$, every enhanced simplicial map $f \colon S \to T$ to the associated map $f^\op \colon S^\op \to T^\op$, but every $\Delta^1 \times S \xrightarrow{(f_1, f_2)} T$ to $(\Delta^1)^\op \times S \xrightarrow{(f_1^\op, f_2^\op)} T$ which becomes $\Delta^1 \times S^\op \xrightarrow{(f_2^\op, f_1^\op)} T^\op$; the higher dimensional maps are mapped similarly, in which the directions of simplicial maps are reversed.
		
		Then, for any enhanced simplicial functor $G \colon \bbA \to \bbF_\Delta$, we denote by $G_\co$ the composite
		$$\bbA^\co \xrightarrow{G^\co} \bbF_\Delta^\co \xrightarrow{\cdot\;^\op} \bbF_\Delta$$
		of $\cdot\;^\op$ and $G^\co$.
	\end{nota}
	
	As expected, completeness of $\bbK$ and of $\bbK^\co$ are indeed dual to each other.
	
	\begin{lemma}
		\label{lem:dual_fun}
		Let $\bbA$ be an enhanced simplicial category. Given two enhanced simplicial functors $G, H \colon \bbA \rightrightarrows \bbF_\Delta$, we always have an isomorphism
		$$[\bbA, \bbF_\Delta](G, H)^\op \cong [\bbA^\co, \bbF_\Delta](G_\co, H_\co)$$
		in $\F_\Delta$.
	\end{lemma}
	
	\begin{proof}
		Let $n \in \N.$ An $n$-simplex in $[\bbA, \bbF_\Delta](G, H)$ is a simplicial natural transformation $\Delta^n \times G \xrightarrow{\alpha} H$, whose component at $A \in \ob\bbA$ is a simplicial map $\Delta^n \times G(A) \xrightarrow{\alpha_A} H(A)$. So, an $n$-simplex in $[\bbA, \bbF_\Delta](G, H)^\op$ has component given by $(\Delta^n)^\op \times G(A) \xrightarrow{\overline{\alpha_A}} H(A)$. Applying the functor $\cdot\;^\op$, this becomes a simplicial map $\Delta^n \times G(A)^\op \xrightarrow{\overline{\alpha_A}^\op} H(A)^\op$, which is equivalently $\Delta^n \times (G^\co(A))^\op \xrightarrow{\overline{\alpha_A}^\op} (H^\co(A))^\op$. By construction, this is indeed $\Delta^n \times G_\co(A) \xrightarrow{\overline{\alpha_A}^\op} H_\co(A)$, which is clearly the component of an $n$-simplex in $[\bbA^\co, \bbF_\Delta](G_\co, H_\co)$ at $A$.
		
		Now, we consider the $0$-simplices. It is clear that an ordinary simplicial map $G(A) \to H(A)$ corresponds bijectively to a simplicial map $G_\co(A) \to H_\co(A)$, following the argument above. Suppose that $G \xrightarrow{\alpha} H$ is an enhanced simplicial natural transformation. This means that the component $G(A) \to H(A)$ at $A \in \bbA$ is tight, which means it is an enhanced simplicial map. Since $\cdot\;^\op \colon \bbF_\Delta^\co \to \bbF_\Delta$ is an enhanced simplicial functor, it sends enhanced simplicial maps to enhanced simplicial maps, therefore, $\alpha$ corresponds to an enhanced simplicial natural transformation $G_\co \to H_\co$. This proves that the isomorphism
		$$[\bbA, \bbF_\Delta](G, H)^\op \cong [\bbA^\co, \bbF_\Delta](G_\co, H_\co)$$
		is in $\F_\Delta$.
	\end{proof}
	
	\begin{pro}
		\label{pro:K^co_has_dual_lim}
		Let $\bbK$ be an enhanced simplicial category. Let $\Phi \colon \bbA \to \bbF_\Delta$ be an $\F_\Delta$-weight, and $F \colon \bbA \to \bbK$ be an enhanced simplicial functor. Then, $\bbK$ admits the weighted limit $\{\Phi, F\}$ of $F$ by $\Phi$ if and only if $\bbK^\co$ admits the weighted limit $\{\Phi_\co, F^\co\}$ of $F^\co$ by $\Phi_\co$.
	\end{pro}
	
	\begin{proof}
		To establish
		$$\bbK^\co(X, \{\Phi_\co, F^\co\}) \cong \bbK(X, \{\Phi, F\})$$
		for any $X \in \ob\bbK$, it suffices to show that
		$$[\bbA, \bbF_\Delta](\Phi, \bbK(X, F))^\op \cong [\bbA^\co, \bbF_\Delta](W_\co, \bbK^\co(X, F^\co))$$
		in $\F_\Delta$. Indeed, $\bbK^\co(X, \cdot)$ is equal to the composite
		$$\bbK^\co \xrightarrow{\bbK(X, \cdot)^\co} \bbF_\Delta^\co \xrightarrow{\cdot\;^\op} \bbF_\Delta,$$
		which implies that $\bbK^\co(X, F^\co) = \bbK^\co(X, \cdot) \cdot F^\co$ is equal to
		$(\bbK(X, F)^\co)^\op =: \bbK(X, F)_\co$. Therefore, by \longref{Lemma}{lem:dual_fun}, we are done.
	\end{proof}
	
	%	\begin{lemma}
		%		\label{lem:tight_lim_dual}
		%		Let $\bbK$ be an enhanced simplicial category which admits tight limits. Then $\bbK^\co$ also admits tight limits.
		%	\end{lemma}
	%	
	%	\begin{proof}
		%		Let $W \colon \twoA \to \bbF_\Delta$ be an $\F_\Delta$-weight, and $F \colon \twoA \to \bbK^\co$ be an enhanced simplicial functor, where $\twoA$ is chordate. By assumption, $\{W_\co, F^\co\}$ exists in $\bbK$.
		%		
		%		We check that $\{W, F\} \cong \{W_\co, F^\co\}$.
		%	\end{proof}
	
	A very important example of $\F_\Delta$-weighted limits for us is \emph{Eilenberg-Moore objects over loose monads}.
	
	\begin{nota}
		We write $\Delta_{\emptyset}$ for the category of finite ordinals and the order-preserving maps. This category is generated by
		% https://q.uiver.app/#q=WzAsNSxbMCwwLCJbLTFdIl0sWzEsMCwiWzBdIl0sWzIsMCwiWzFdIl0sWzMsMCwiWzJdIl0sWzQsMCwiXFxjZG90cyJdLFswLDFdLFsyLDEsIlxcc2lnbWFfMCIsMV0sWzEsMiwiXFxkZWx0YV8wIiwxLHsib2Zmc2V0IjotM31dLFsxLDIsIlxcZGVsdGFfMSIsMSx7Im9mZnNldCI6M31dLFsyLDMsIlxcZGVsdGFfMSIsMV0sWzMsMiwiXFxzaWdtYV8xIiwxLHsib2Zmc2V0IjotM31dLFszLDIsIlxcc2lnbWFfMCIsMSx7Im9mZnNldCI6M31dLFsyLDMsIlxcZGVsdGFfMCIsMSx7Im9mZnNldCI6LTV9XSxbMiwzLCJcXGRlbHRhXzIiLDEseyJvZmZzZXQiOjV9XV0=
		\[\begin{tikzcd}[ampersand replacement=\&, scale cd=0.7]
			{[-1]} \& {[0]} \& {[1]} \& {[2]\cdots\cdots} 
			\arrow[from=1-1, to=1-2]
			\arrow["{\delta_0}"{description}, shift left=3, from=1-2, to=1-3]
			\arrow["{\delta_1}"{description}, shift right=3, from=1-2, to=1-3]
			\arrow["{\sigma_0}"{description}, from=1-3, to=1-2]
			\arrow["{\delta_1}"{description}, from=1-3, to=1-4]
			\arrow["{\delta_0}"{description}, shift left=5, from=1-3, to=1-4]
			\arrow["{\delta_2}"{description}, shift right=5, from=1-3, to=1-4]
			\arrow["{\sigma_1}"{description}, shift left=3, from=1-4, to=1-3]
			\arrow["{\sigma_0}"{description}, shift right=3, from=1-4, to=1-3]
		\end{tikzcd}\]
		subject to the simplicial relations, where $[-1]$ denotes the empty ordinal.
		
		Denote by $\Delta$ the full sub-category of $\Delta_{\emptyset}$ with non-empty ordinals, which is usually called the \emph{simplex category}. This means that $\Delta$ is generated by
		% https://q.uiver.app/#q=WzAsNCxbMCwwLCJbMF0iXSxbMSwwLCJbMV0iXSxbMiwwLCJbMl0iXSxbMywwLCJcXGNkb3RzIl0sWzEsMCwiXFxzaWdtYV8wIiwxXSxbMCwxLCJcXGRlbHRhXzAiLDEseyJvZmZzZXQiOi0zfV0sWzAsMSwiXFxkZWx0YV8xIiwxLHsib2Zmc2V0IjozfV0sWzEsMiwiXFxkZWx0YV8xIiwxXSxbMiwxLCJcXHNpZ21hXzEiLDEseyJvZmZzZXQiOi0zfV0sWzIsMSwiXFxzaWdtYV8wIiwxLHsib2Zmc2V0IjozfV0sWzEsMiwiXFxkZWx0YV8wIiwxLHsib2Zmc2V0IjotNX1dLFsxLDIsIlxcZGVsdGFfMiIsMSx7Im9mZnNldCI6NX1dXQ==
		\[\begin{tikzcd}[ampersand replacement=\&, scale cd = 0.7]
			{[0]} \& {[1]} \& {[2]\cdots\cdots}  
			\arrow["{\delta_0}"{description}, shift left=3, from=1-1, to=1-2]
			\arrow["{\delta_1}"{description}, shift right=3, from=1-1, to=1-2]
			\arrow["{\sigma_0}"{description}, from=1-2, to=1-1]
			\arrow["{\delta_1}"{description}, from=1-2, to=1-3]
			\arrow["{\delta_0}"{description}, shift left=5, from=1-2, to=1-3]
			\arrow["{\delta_2}"{description}, shift right=5, from=1-2, to=1-3]
			\arrow["{\sigma_1}"{description}, shift left=3, from=1-3, to=1-2]
			\arrow["{\sigma_0}"{description}, shift right=3, from=1-3, to=1-2]
		\end{tikzcd}\]
		subject to the simplicial relations.
		
		Denote by $\Delta_\top$ the wide sub-category of $\Delta$ with top-preserving maps. So, $\Delta_\top$ is generated by
		% https://q.uiver.app/#q=WzAsMyxbMCwwLCJbMF0iXSxbMSwwLCJbMV0iXSxbMiwwLCJbMl1cXGNkb3RzXFxjZG90cyJdLFsxLDAsIlxcc2lnbWFfMCIsMSx7Im9mZnNldCI6LTF9XSxbMCwxLCJcXGRlbHRhXzAiLDEseyJvZmZzZXQiOi0xfV0sWzEsMiwiXFxkZWx0YV8xIiwxLHsib2Zmc2V0IjoxfV0sWzIsMSwiXFxzaWdtYV8xIiwxLHsib2Zmc2V0IjotNH1dLFsyLDEsIlxcc2lnbWFfMCIsMSx7Im9mZnNldCI6MX1dLFsxLDIsIlxcZGVsdGFfMCIsMSx7Im9mZnNldCI6LTR9XV0=
		\[\begin{tikzcd}[ampersand replacement=\&, scale cd = 0.7]
			{[0]} \& {[1]} \& {[2]\cdots\cdots}
			\arrow["{\delta_0}"{description}, shift left, from=1-1, to=1-2]
			\arrow["{\sigma_0}"{description}, shift left, from=1-2, to=1-1]
			\arrow["{\delta_1}"{description}, shift right, from=1-2, to=1-3]
			\arrow["{\delta_0}"{description}, shift left=4, from=1-2, to=1-3]
			\arrow["{\sigma_1}"{description}, shift left=4, from=1-3, to=1-2]
			\arrow["{\sigma_0}"{description}, shift right, from=1-3, to=1-2]
		\end{tikzcd}\]
		subject to the simplicial relations.
	\end{nota}
	
	\begin{example}[(Co)Eilenberg-Moore objects over loose (co)monads]
		\label{eg:EM}
		Recall that in \cite{RV:2016}, the free (homotopy coherent) adjunction $\Adj$ is a simplicial category with two objects ${+}$ and ${-}$, which is equivalently the  $2$-category of a free adjunction by Schanuel and Street in \cite{SS:1986}; then ${\Mnd}$ denotes the full simplicial sub-category of ${\Adj}$ determined by ${+}$, with ${\Mnd}(+, +) := N(\Delta_{\emptyset})$, the nerve of the category of finite ordinals and the order-preserving maps.
		
		We then view $\Mnd$ as an inchordate enhanced simplicial category.
		
		Now, since $\Mnd$ is inchordate, an $\F_\Delta$-weight $\Mnd \to \bbF_\Delta$ is simply a simplicial functor $\Mnd \to {\F_\Delta}_\lambda$, where ${\F_\Delta}_\lambda$ is regarded as a simplicial category by forgetting the tight $0$-arrows in $\bbF_\Delta$. Any simplicial functor $F \colon \Mnd \to {\F_\Delta}_\lambda$ amounts to an enhanced  simplicial set $F(+)$, together with a left action 
		$$N(\Delta_{\emptyset}) \to {\F_\Delta}_\lambda(F(+), F(+))$$
		by $N(\Delta_{\emptyset})$. 
		
		Let $W \colon \Mnd \to \bbF_\Delta$ be the $\F_\Delta$-weight, which sends the unique object $+$ to the enhanced simplicial set $N(\Delta_{\top})_{[0]}$, where only the terminal vertex $[0]$ is tight, with left action
		$$
		\begin{aligned}
			N(\Delta_{\emptyset}) &\to {\F_\Delta}_\lambda(N(\Delta_{\top})_{[0]}, N(\Delta_{\top})_{[0]})
			\\
			[n] &\mapsto {[n] \oplus \cdot}
		\end{aligned}
		$$
		given by the usual ordinal sum. Note that ${[n] \oplus \cdot}$ is not required to preserve tight vertices, because we are considering the hom in ${\F_\Delta}_\lambda$, whose $0$-arrows are just simplicial maps (between enhanced simplicial sets).
		
		A \emph{loose (homotopy coherent) monad} on an object $A$ in an enhanced simplicial category $\bbK$ consists of an enhanced simplicial functor $T \colon \Mnd \to \bbK$, which sends $+$ to $A$, with left action
		$$
		\begin{aligned}
			N(\Delta_{\emptyset}) &\to \bbK(A, A)
			\\
			[n] &\mapsto T([0])^n
		\end{aligned}
		$$
		given by the $n$-fold composition of $T([0])$.
		
		The \emph{Eilenberg-Moore object over $T$} is defined as the $\F_\Delta$-weighted limit $\{W, T\}$.
		
		Dually, consider the opposite $\Mnd^\co$ of $\Mnd$, where $\Mnd^\co(+, +) := N(\Delta_{\emptyset}^\op)$, the nerve of the opposite category of $\Delta_{\emptyset}$.
		
		A \emph{loose (homotopy coherent) comonad} $C$ on an object $A$ in an enhanced simplicial category $\bbK$ is a monad on $A$ in $\bbK^\co$. In other words, $C = T^\co$, for some loose monad $T$ in $\bbK$.
		
		More precisely, the comonad $C$ consists of an enhanced simplicial functor $C \colon \Mnd^\co \to \bbK$, which sends $+$ to $A$, with left action
		$$
		\begin{aligned}
			N(\Delta_{\emptyset}^\op) &\to \bbK(A, A)
			\\
			[n] &\mapsto C([0])^n
		\end{aligned}
		$$
		given by the $n$-fold composition of $C([0])$.	
		
		The \emph{coEilenberg-Moore object over $C$} is defined as the $\F_\Delta$-weighted limit $\{W_\co, C\}$.
		
		%		Write $W_\co \colon \Mnd^\co \to \bbF_\Delta$ for the enhanced simplicial functor that sends $+$ to $N(\Delta_{\top}^\op)_{[0]}$, the nerve of the opposite category of $\Delta_{\top}$, in which only the initial vertex $[0]$ is tight, with left action
		%		$$
		%		\begin{aligned}
			%			N(\Delta_{\emptyset}^\op) &\to {\F_\Delta}_\lambda(N(\Delta_{\top}^\op)_{[0]}, N(\Delta_{\top}^\op)_{[0]})
			%			\\
			%			[n] &\mapsto {[n] \oplus \cdot}
			%		\end{aligned}
		%		$$
		%		given by the usual ordinal sum.
		
		In fact, the coEilenberg-Moore object over a loose comonad $C$ is exactly the Eilenberg-Moore object over the loose monad $C^\co$.
		
		As (co)Eilenberg-Moore objects over loose (co)monads are formulated as $\F_\Delta$-weighted limits, the forgetful functor from a (co)Eilenberg-Moore object over a loose (co)monad on $A$ to $A$ is always tight and reflects tightness.
	\end{example}
	
	We will see in \longref{Section}{sec:inserter} that (co)Eilenberg-Moore objects over loose (co)monads can be constructed from simpler limits.
	
	\section{Rigged $n$-inserters}
	\label{sec:inserter}
	In this section, we present a kind of $\F_\Delta$-weighted limits, called \emph{rigged $n$-inserters}, which can actually be seen as the $\infty$-categorical versions of inserters and equifiers in $2$-category theory.
	
	\subsection{Terminally rigged $n$-inserters}
	\begin{nota}
		Let $n \in \N_0$. Denote by $\mathbb{D}^n_{[[n]]}$ the enhanced simplicial category with two distinct objects $x$ and $y$, where 
		$$
		\begin{aligned}
			&\mathbb{D}^n_{[[n]]}(x, y) := \partial \Delta^n_{[[n]]},
			\\
			&\mathbb{D}^n_{[[n]]}(x, x) := \mathbb{D}^n_{[[n]]}(y, y) := \Delta^0_{[[0]]},
			\\
			&\mathbb{D}^n_{[[n]]}(y, x) := \emptyset.
		\end{aligned}
		$$
	\end{nota}
	
	\begin{rk}
		The only non-identity tight $0$-arrow is given by the terminal vertex $[[n]]$ in $\mathbb{D}^n_{[[n]]}(x, y) := \partial \Delta^n_{[[n]]}$.
	\end{rk}
	
	\begin{example}
		Set $n = 2$. Then $\mathbb{D}^2_{[[2]]}$ is depicted as
		
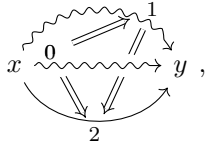
\begin{center}
			\begin{tikzcd}[ampersand replacement=\&]
				x \ar[rr, loose, bend left = 60, "1"{name=B}, near end] \ar[rr, loose, "0"{name=A}, very near start, crossing over] \ar[Rightarrow, from=A, to=B, shorten=1.5mm]  \ar[rr, bend right = 60, "2"'{name=C}] \ar[Rightarrow, from=A, to=C, shorten=0.5mm] \ar[Rightarrow,  from=B, to=C, shorten=1mm]  \& \& y \ar[from=1-1, to=1-3, loose, crossing over, "0", very near start] 
			\end{tikzcd},
		\end{center}
		in which $0, 1, 2$ are non-degenerate $0$-arrows from $x$ to $y$, and that only $2$ is tight; moreover, we have non-degenerate $1$-arrows from $0$ to $1$, $1$ to $2$, and $0$ to $2$, respectively.
	\end{example}
	
	\begin{construct}
		\label{construct:weight}
		Let $\Psi \colon \mathbb{D}^n_{[[n]]} \to \bbF_\Delta$ be an $\F_\Delta$-weight $(\Psi_\tau, \Psi_\lambda, \psi)$ as follows.
		
		Define $\Psi_\lambda \colon {\mathbb{D}^n_{[[n]]} }_\lambda  \to \sSet$ such that $\Psi_\lambda x = \Delta^0$ and $\Psi_\lambda y = \Delta^n$, so that
		$${\mathbb{D}^n_{[[n]]}}_\lambda(x, y) \to \sSet(\Delta^0, \Delta^n)$$
		is given by the canonical inclusion $\partial \Delta^n \hookrightarrow \Delta^n$.
		
		Define $\Psi_\tau \colon {\mathbb{D}^n_{[[n]]} }_\tau  \to \sSet$ such that $\Psi_\tau x = \Psi_\tau y = \Delta^0$.
		
		Define the simplicial natural transformation $\psi \colon \Psi_\tau \to \Psi_\lambda  J_{\mathbb{D}^n_{[[n]]}}$ such that the $1$-components $\psi_x \colon \Psi_\tau x \hookrightarrow \Psi_\lambda x$ and $\psi_y \colon \Psi_\tau y \hookrightarrow \Psi_\lambda y$ are given by the identity at $\Delta^0$ and the map that picks out the terminal vertex $[[n]] \colon \Delta^0 \to \Delta^n$, respectively.
	\end{construct}
	
	\begin{defi}
		\label{def:rigged_weight}
		Let $F \colon \mathbb{D}^n_{[[n]]} \to \bbK$ be an  enhanced simplicial functor. The \emph{terminally rigged $n$-inserter of $F$} is defined as the limit $\{\Psi, F\}$ of $F$ weighted by $\Psi$ constructed in \longref{Construction}{construct:weight}. The \emph{projection of the terminally rigged $n$-inserter} is defined as the unique limit projection $\{\Psi, F\} \to Fx$. 
	\end{defi}
	
	\begin{rk}
		\label{rk:term_flexible}
		It is not hard to see that terminally rigged $n$-inserters in a chordate enhanced simplicial category, i.e., one that every $0$-arrow is tight, are \emph{flexible weighted limits}, introduced in \cite[Chapter 6.2]{book:RV:2022}. For, the $\F_\Delta$-weight $\Psi$ can be built by attaching a projective $0$-cell at the terminal vertex, followed by a projective $n$-cell at the terminal vertex:
		% https://q.uiver.app/#q=WzAsNyxbMCwwLCJcXHBhcnRpYWxcXERlbHRhXjAgXFx0aW1lcyBcXG1hdGhiYnsyfV97XFxwYXJ0aWFsIHtcXERlbHRhXm59fSh4LCBcXGNkb3QpIl0sWzEsMCwiXFxlbXB0eXNldCJdLFswLDEsIlxcRGVsdGFeMCBcXHRpbWVzIFxcbWF0aGJiezJ9X3tcXHBhcnRpYWwge1xcRGVsdGFebn19KHgsIFxcY2RvdCkiXSxbMSwyLCJcXFBoaV8wIl0sWzAsMywiXFxwYXJ0aWFsXFxEZWx0YV5uIFxcdGltZXMgXFxtYXRoYmJ7Mn1fe1xccGFydGlhbCB7XFxEZWx0YV5ufX0oeSwgXFxjZG90KSJdLFswLDQsIlxcRGVsdGFebiBcXHRpbWVzIFxcbWF0aGJiezJ9X3tcXHBhcnRpYWwge1xcRGVsdGFebn19KHksIFxcY2RvdCkiXSxbMSw0LCJcXFBoaSJdLFswLDFdLFsxLDNdLFswLDJdLFsyLDNdLFs0LDNdLFszLDZdLFs0LDVdLFs1LDZdLFszLDAsIiIsMix7InN0eWxlIjp7Im5hbWUiOiJjb3JuZXIifX1dLFs2LDQsIiIsMSx7InN0eWxlIjp7Im5hbWUiOiJjb3JuZXIifX1dXQ==
		\[\begin{tikzcd}[ampersand replacement=\&, row sep=small]
			{\partial\Delta^0 \times \mathbb{D}^n_{[[n]]}(x, \cdot)} \& \emptyset \\
			{\Delta^0 \times \mathbb{D}^n_{[[n]]}(x, \cdot)} \\
			\& {\Psi_0} \\
			{\partial\Delta^n \times \mathbb{D}^n_{[[n]]}(y, \cdot)} \\
			{\Delta^n \times \mathbb{D}^n_{[[n]]}(y, \cdot)} \& \Psi
			\arrow[from=1-1, to=1-2]
			\arrow[from=1-1, to=2-1]
			\arrow[from=1-2, to=3-2]
			\arrow[from=2-1, to=3-2]
			\arrow["\lrcorner"{anchor=center, pos=0.125, rotate=180}, draw=none, from=3-2, to=1-1]
			\arrow[from=3-2, to=5-2]
			\arrow[from=4-1, to=3-2]
			\arrow[from=4-1, to=5-1]
			\arrow[from=5-1, to=5-2]
			\arrow["\lrcorner"{anchor=center, pos=0.125, rotate=180}, draw=none, from=5-2, to=4-1]
		\end{tikzcd}.\]
	\end{rk}
	
	\begin{rk}
		\label{rk:F_as_D}
		Let $\bbK$ be an enhanced simplicial category. Then an enhanced simplicial functor $F \colon \mathbb{D}^n_{[[n]]} \to \bbK$ amounts to a pair of objects $Fx$ and $Fy$ of $\bbK$, and an enhanced simplicial map
		$$F_{x, y} \colon \partial\Delta^n_{[[n]]} \to \bbK(Fx, Fy),$$
		which can be equivalently seen as a simplicial map $\partial \Delta^n \xrightarrow{F_{x, y}} \bbK(Fx, Fy)$ together with an enhanced simplicial map $\Delta^0_{[[0]]} \xrightarrow{[[n]]} \partial \Delta^n \xrightarrow{F_{x, y}} \bbK(Fx, Fy)$. If $\bbK$ admits simplicial powers by inchordate enhanced simplicial sets, then the transpose of $F_{x, y}$ is a (loose) $0$-arrow $\overline{F_{x, y}} \colon Fx \leadsto (Fy)^{\partial \Delta^n}$ in $\bbK$, where the composite
		$$Fx \xleadsto{\overline{F_{x, y}}} (Fy)^{\partial \Delta^n} \xrightarrow{\ev [[n]]} Fy$$
		of $\overline{F_{x, y}}$ with the evaluation at the terminal vertex is a tight $0$-arrow in $\bbK$.	
		
		In other words, if $\bbK$ is an enhanced simplicial category which admits simplicial powers by inchordate enhanced simplicial sets, then a diagram $\mathbb{D}^n_{[[n]]} \to \bbK$ precisely consists of two objects $S$ and $T$ of $\bbK$, and a loose $0$-arrow $D \colon S \leadsto T^{\partial \Delta^n}$, where the composite $\ev[[n]] \cdot D$ is tight.
	\end{rk}
	
	Alternatively, we can describe terminally rigged $n$-inserters via pullbacks.

	\begin{pro}
		\label{pro:as_pullbacks}
		Let $\bbK$ be an enhanced simplicial category with simplicial powers by inchordate enhanced simplicial sets. The terminally rigged $n$-inserter $\{\Psi, F\}$ of $F\colon \mathbb{D}^n_{[[n]]} \to \bbK$ in \longref{Definition}{def:rigged_weight} can be equivalently described as the pullback
		\begin{equation}
			\label{diag:pullback}
			\begin{tikzcd}
				{\rins{n}{[[n]]}{\overline{F_{x, y}}}} \arrow[dr, phantom, "\lrcorner", very near start] \ar[r, "\phi", loose] \ar[d,  "p"'] & (Fy)^{\Delta^n } \ar[d,  "{(Fy)^{\partial}}"]
				\\
				Fx \ar[r, "\overline{F_{x, y}}"', loose]  & (Fy)^{\partial\Delta^n }
			\end{tikzcd}
		\end{equation} 		
		of the transpose $\overline{F_{x, y}}$ of the action $F_{x, y}$ of $F$ along the restriction $(Fy)^\partial \colon (Fy)^{\Delta^n} \twoheadrightarrow (Fy)^{\partial \Delta^n}$ in $\calK_\lambda$, in which the projection $p \colon \rins{n}{[[n]]}{\overline{F_{x, y}}} \to Fx$ is tight and reflects tightness.
	\end{pro}
	
	\begin{proof}
		Note that \longref{Diagram}{diag:pullback} is a pullback implies that for any (positive) $i \in \N$,
		\begin{equation}
			\label{diag:pullback_i}
			\begin{tikzcd}
				{{\rins{n}{[[n]]}{\overline{F_{x, y}}}}^{\Delta^i}} \arrow[dr, phantom, "\lrcorner", very near start] \ar[r, "\phi^{\Delta^i}", loose] \ar[d,  "p^{\Delta^i}"'] & (Fy)^{\Delta^n \times \Delta^i} \ar[d,  "{{(Fy)^{{\Delta^i}\partial}} = {(Fy)^\partial}^{\Delta^i}}"]
				\\
				(Fx)^{\Delta^i} \ar[r, "\overline{F_{x, y}}^{\Delta^i}"', loose]  & (Fy)^{\partial\Delta^n \times {\Delta^i}}
			\end{tikzcd}
		\end{equation} 
		is a pullback, since pullbacks commute with simplicial powers.
		
		We shall verify that there is an isomorphism
		$$\bbK(X, {\rins{n}{[[n]]}{\overline{F_{x, y}}}} ) \cong [\mathbb{D}^n_{[[n]]}, \bbF_\Delta](\Psi, \bbK(X, F))$$
		of enhanced simplicial sets, natural in any object $X$ of $\bbK$.
		
		Consider a loose $0$-arrow in $[\mathbb{D}^n_{[[n]]}, \bbF_\Delta](\Psi, \bbK(X, F))$, which is a simplicial natural transformation between the loose parts
		$$\alpha \colon \Psi_\lambda \to \calK_{\lambda}(X, F_\lambda).$$
		It has components
		$$
		\begin{aligned}
			&\alpha_x \colon \Delta^0 \to \calK_{\lambda}(X, Fx),
			&\alpha_y \colon \Delta^n \to \calK_{\lambda}(X, Fy),
		\end{aligned}
		$$
		whose transposes are (loose) $0$-arrows
		$$
		\begin{aligned}
			&\overline{\alpha_x} \colon X \leadsto Fx,
			&\overline{\alpha_y} \colon X \leadsto (Fy)^{\Delta^n},
		\end{aligned}
		$$
		in $\calK_\lambda$, respectively. The naturality of $\alpha$ says that for any map $[[j]] \colon \Delta^0 \to \Delta^n$ that picks out the $j^\mathrm{th}$ vertex, the square
		% https://q.uiver.app/#q=WzAsNCxbMCwwLCJcXERlbHRhXjBfe1tbMF1dfSJdLFswLDEsIlxcRGVsdGFebl97W1tuXV19Il0sWzEsMCwiXFxjYWxLX1xcbGFtYmRhKFgsIEZ4KSJdLFsxLDEsIlxcY2FsS19cXGxhbWJkYShYLCBGeSkiXSxbMCwyLCJcXGFscGhhX3giXSxbMSwzLCJcXGFscGhhX3kiLDJdLFswLDEsIltbal1dIiwyXSxbMiwzLCJcXGNhbEtfXFxsYW1iZGEoWCwgXFxldltbal1dIFxcY2RvdCBcXG92ZXJsaW5le0Zfe3gsIHl9fSkiXV0=
		\[\begin{tikzcd}[ampersand replacement=\&]
			{\Delta^0} \& {\calK_\lambda(X, Fx)} \\
			{\Delta^n} \& {\calK_\lambda(X, Fy)}
			\arrow["{\alpha_x}", from=1-1, to=1-2]
			\arrow["{[[j]]}"', from=1-1, to=2-1]
			\arrow["{\calK_\lambda(X, \ev[[j]] \cdot \overline{F_{x, y}})}", from=1-2, to=2-2]
			\arrow["{\alpha_y}"', from=2-1, to=2-2]
		\end{tikzcd}\]
		commutes. Similarly, we have the naturality with respect to the higher simplices in $[\Delta^0, \Delta^n]$.
		
		Unravelling, the commutative square above tells us that the composite
		$$X \xleadsto{\overline{\alpha_y}} (Fy)^{\Delta^n} \xrightarrow{\ev[[j]]} Fy$$
		is equal to
		$$X \xleadsto{\overline{\alpha_x}} Fx \xleadsto{\overline{F_{x, y}}} (Fy)^{\partial \Delta^n} \xrightarrow{\ev[[j]]} Fy$$
		for any $j \in \{0, 1, \cdots, n\}$. This is equivalent to say that $(Fy)^\partial \cdot \overline{\alpha_y} = \overline{F_{x, y}} \cdot \overline{\alpha_x}$, where $(Fy)^\partial$ is the restriction $(Fy)^{\Delta^n} \to (Fy)^{\partial \Delta^n}$. So, by the universal property of the pullback \longref{Diagram}{diag:pullback}, there is unique $0$-arrow $u^0 \colon \Delta^0 \to \calK(X,  \rins{n}{[[n]]}{\overline{F_{x, y}}})$ of $\calK_\lambda$, satisfying $p \cdot \overline{u^0} = \overline{\alpha_x}$ and $\phi \cdot \overline{u^i} = \overline{\alpha_y}$, where $\overline{u^0} \colon X \leadsto \rins{n}{[[n]]}{\overline{F_{x, y}}}$ is the transpose of $u^0$.
		
		Moreover, for any (positive) $i \in \N$, an $i$-arrow in $[{\mathbb{D}^n_{[[n]]}}_\lambda, \sSet](\Psi_\lambda, \calK_\lambda(X, F_\lambda))$ is given by a simplicial natural transformation
		$$\alpha^{\Delta^i} \colon  \Psi_\lambda \to \calK_{\lambda}(X, F_\lambda)^{\Delta^i},$$
		which has components
		$$
		\begin{aligned}
			&\alpha_x^{\Delta^i} \colon \Delta^i \to \calK_{\lambda}(X, Fx),
			&\alpha_y^{\Delta^i} \colon \Delta^n \times \Delta^i \to \calK_{\lambda}(X, Fy),
		\end{aligned}
		$$
		whose transposes are, respectively, given by
		$$
		\begin{aligned}
			&\overline{\alpha_x}^{\Delta^i} \colon X \leadsto (Fx)^{\Delta^i},
			&\overline{\alpha_y}^{\Delta^i} \colon X \leadsto (Fy)^{\Delta^n \times \Delta^i}
		\end{aligned}
		$$
		in $\calK_{\lambda}$, such that the composite
		$$X \xleadsto{\overline{\alpha_y}^{\Delta^i}} (Fy)^{\Delta^n \times \Delta^i} \xrightarrow{\ev[[j]]} (Fy)^{\Delta^i}$$
		is equal to the composite
		$$X \xleadsto{\overline{\alpha_x}^{\Delta^i}} (Fx)^{\Delta^i} \xleadsto{\overline{F_{x, y}}^{\Delta^i}} (Fy)^{\partial \Delta^n \times \Delta^i} \xrightarrow{\ev[[j]]} (Fy)^{\Delta^i}$$
		for any $j \in \{0, 1, \cdots, n\}$. In other words, we have $(Fy)^{{\Delta^i}\partial} \cdot \overline{\alpha_y}^{\Delta^i} = \overline{F_{x, y}}^{\Delta^i} \cdot \overline{\alpha_x}^{\Delta^i}$.
		
		According to the universal property of the pullback \longref{Diagram}{diag:pullback_i}, the above data then correspond bijectively to a unique $i$-arrow $u^i \colon \Delta^i \to \calK_\lambda(X, \rins{n}{[[n]]}{\overline{F_{x, y}}})$ of $\calK_\lambda$, satisfying $p^{\Delta^i} \cdot \overline{u^i} = \overline{\alpha_x}^{\Delta^i}$ and $\phi^{\Delta^i} \cdot \overline{u^i} = \overline{\alpha_y}^{\Delta^i}$, where $\overline{u^i}$ is the transpose of $u^i$. That is to say, we have an isomorphism
		$$\calK_\lambda(X, {\rins{n}{[[n]]}{\overline{F_{x, y}}}} ) \cong [{\mathbb{D}^n_{[[n]]} }_\lambda, \sSet](\Psi_\lambda, \calK_\lambda(X, F_\lambda))$$
		of simplicial sets, natural in any object $X$ of $\calK_\lambda$.

		Now, if $\alpha \colon \Psi \to \bbK(X, F)$ is an enhanced simplicial natural transformation, then we have a simplicial natural transformation
		$$\alpha_\tau \colon \Psi_\tau \to \calK_\tau(X, F_\tau),$$
		whose components are a pair of simplicial maps
		$$
		\begin{aligned}
			&{\alpha_\tau}_x \colon \Delta^0 \to \calK_\tau(X, Fx),
			&{\alpha_\tau}_y \colon \Delta^0 \to \calK_\tau(X, Fy).
		\end{aligned}
		$$
		Denote by $\alpha_\lambda$ the simplicial natural transformation $\Psi_\lambda \to \calK_\lambda(X, F_\lambda)$ considered above. We then have 
		$$\Delta^0 \xrightarrow{{\alpha_\tau}_y} \calK_\tau(X, Fy) \xhookrightarrow{J_\bbK} \calK_\lambda(X, Fy)$$
		being equal to 
		$$\Delta^0 \xrightarrow{[[n]]} \Delta^n \xrightarrow{{\alpha_\lambda}_y} \calK_\lambda(X, Fy),$$
		and that ${\alpha_\lambda}_x = {\alpha_\tau}_x$.
		This means that 
		$$
		\begin{aligned}
			& X \xrightarrow{\overline{{\alpha_\lambda}_x}} Fx,
			& X \xleadsto{\overline{{\alpha_\lambda}_y}} (Fy)^{\Delta^n} \xrightarrow{\ev[[n]]} Fy,
		\end{aligned}
		$$
		are tight $0$-arrows in $\bbK$. Indeed, the former is tight already implies that the latter is also tight; for, $\ev[[n]] \cdot \overline{{\alpha_\lambda}_y} = \ev[[n]] \cdot  (Fy)^\partial \cdot \overline{{\alpha_\lambda}_y} = \ev[[n]] \cdot \overline{F_{x, y}} \cdot \overline{{\alpha_\lambda}_x}$.
		
		Therefore, the additional information that we obtain by considering $\alpha$ as a tight $0$-arrow in $[\mathbb{D}^n_{[[n]]}, \bbF_\Delta]$ is that the component $\alpha_x$ becomes a tight $0$-arrow in $\bbK$, i.e., $\overline{\alpha_x} \colon X \to Fx$ is tight. As we required the projection $p \colon \rins{n}{[[n]]}{\overline{F_{x, y}}} \to Fx$ in \longref{Diagram}{diag:pullback} to reflect tightness, the induced $0$-arrow $\overline{u^0} \colon X \to \rins{n}{[[n]]}{\overline{F_{x, y}}}$ then becomes tight. 
		
		Altogether, we have shown that there is an isomorphism
		$$\bbK(X, {\rins{n}{[[n]]}{\overline{F_{x, y}}}} ) \cong [\mathbb{D}^n_{[[n]]}, \bbF_\Delta](\Psi, \bbK(X, F))$$
		in $\F_\Delta$, natural in any object $X$ of $\bbK$.
	\end{proof}
	
	\begin{nota}
		Using the description in \longref{Remark}{rk:F_as_D}, we denote by ${\rins{n}{[[n]]}{D}}$ the terminally rigged $n$-inserter of $D$.
	\end{nota}
	
	We now demonstrate how the notion of terminally rigged $n$-inserters subsumes several $2$-categorical limits, such as products, $c$-rigged inserters and equifiers discussed in \cite{LS:2012}.
	
	Let $\bbK$ be an enhanced simplicial category with simplicial powers by inchordate simplicial sets.
	
	\begin{example}[Products]
		Let $n = 0$. Then $T^{\partial \Delta^0}$ is the terminal $\infty$-category $1$. Therefore, the rigged $0$-inserter amounts to the pullback diagram
		\begin{center}
			\begin{tikzcd}
				{\rins{0}{0}{!}} \arrow[dr, phantom, "\lrcorner", very near start] \ar[r, "\phi"] \ar[d,  "p"'] & T \ar[d,  "!"]
				\\
				S \ar[r, "!"']  & 1
			\end{tikzcd}
		\end{center} 
		in the tight part $\calK_\tau$ of an enhanced simplicial category $\bbK$. And this is equivalently the product $S \times T$, with the extra property that the projections are tight, and they jointly reflect tightness.
	\end{example}
	
	\begin{example}[$c$-rigged inserters] 
		Consider when $n = 1$. In this case, the restriction $T^\partial \colon T^{\Delta^1} \to T^{\partial \Delta^1}$ combines the domain and codomain projections. A loose $0$-arrow $D \colon S \leadsto T \times T$ consists of a pair  of $0$-arrows $f \colon S \leadsto T$ and $g \colon S \to T$ in $\bbK$, where $g$ has to be tight, because the composite $D \cdot \ev[[1]]$ is assumed to be a tight $0$-arrow. The terminally rigged $1$-inserter $\rins{1}{1}{f, g}$ of $f$ and $g$ is given by the pullback diagram
		\begin{center}
			\begin{tikzcd}
				{\rins{1}{1}{f, g}} \arrow[dr, phantom, "\lrcorner", very near start] \ar[r, "\phi",  loose] \ar[d,  "p"'] & {T^{\Delta^1}} \ar[d,  "{(\Dom, \Cod)}"]
				\\
				S \ar[r, "{(f, g)}"', loose]  & T \times T
			\end{tikzcd},
		\end{center}
		which says that $\phi$ is indeed depicted as
		\begin{center}
			\begin{tikzcd}[row sep=3]
				& S \ar[dr, "f", loose] \ar[Rightarrow, from=1-2, to=3-2, shorten=2mm, "\phi"] &
				\\
				{\rins{1}{1}{f, g}} \ar[ur,  "p"] \ar[dr,  "p"'] & & T
				\\
				& S \ar[ur, "g"'] &
			\end{tikzcd},
		\end{center}
		which would be a $c$-rigged inserter in enhanced $2$-category theory.
	\end{example}
	
	\begin{example}[$c$-rigged comma objects] 
		Consider again when $n = 1$. Let $D \colon A \times B \leadsto T \times T$ be a loose $0$-arrow in $\bbK$, which is given by a pair of $0$-arrows $f \colon A \leadsto T$ and $g \colon B \to T$. Again, we need $g$ to be tight, because the composite $D \cdot \ev[[1]]$ is tight. The terminally rigged $1$-inserter $\rins{1}{1}{f, g}$ of $f$ and $g$ is then given by the pullback diagram
		\begin{center}
			\begin{tikzcd}
				{f \downarrow g} \arrow[dr, phantom, "\lrcorner", very near start] \ar[r, "\phi"] \ar[d,  "{p = (p_1, p_2)}"'] & {T^{\Delta^1}} \ar[d,  "{(\Dom, \Cod)}"]
				\\
				{A \times B} \ar[r, "{(f, g)}"', loose]  & T \times T
			\end{tikzcd},
		\end{center}
		which says that $\phi$ is indeed depicted as
		\begin{center}
			\begin{tikzcd}
				{f \downarrow g} \ar[r,  "p_1"] \ar[d,  "p_2"'] & A \ar[d, loose, "f"] \ar[Rightarrow, from=1-2, to=2-1, "\phi", shorten=4mm]
				\\
				B \ar[r, "g"']  &  T
			\end{tikzcd}.
		\end{center}
	\end{example}
	
	\begin{example}[$c$-rigged equifiers] 
		Consider when $n = 2$. Then a loose $0$-arrow $D \colon S \leadsto T^{\partial\Delta^2}$ amounts to a simplicial map $\partial\Delta^2 \to \calK(S, T)$.
		
		Let $a \colon S \leadsto T$ and $b \colon S \to T$  be $0$-arrows, where only $b$ is tight. Let $f$ and $g$ be $1$-arrows from $a$ to $b$. Let $D := (f, g, \deg) \colon S \leadsto T^{\partial\Delta^2}$. This means that $D$ is given by the diagram
		\begin{center}
			\begin{tikzcd}[ampersand replacement=\&]
				{S} \ar[rr, bend left = 60, "b"{name=B}, near end] \ar[rr, loose, "a"{name=A}, very near start, crossing over] \ar[Rightarrow, from=A, to=B, shorten=1.5mm, "f"description]  \ar[rr, bend right = 60, "b"'{name=C}] \ar[Rightarrow, from=A, to=C, shorten=0.5mm, "g"description] \ar[Rightarrow, equal, from=B, to=C, shorten=1mm]  \& \& {T} \ar[from=1-1, to=1-3, loose, crossing over, "a", very near start] 
			\end{tikzcd}
		\end{center}
		in $\bbK$. The terminally rigged $2$-inserter $\rins{2}{2}{f, g, \deg}$ is then given by the pullback diagram
		\begin{center}
			\begin{tikzcd}[ampersand replacement=\&]
				{\rins{2}{2}{f, g, \deg}} \arrow[dr, phantom, "\lrcorner", very near start] \ar[r, "\phi", loose] \ar[d,  "p"'] \& {T^{\Delta^2}} \ar[d,  "{T^\partial}"]
				\\
				S \ar[r, "{(f, g, \deg)}"', loose]  \& {T^{\partial\Delta^2}}
			\end{tikzcd},
		\end{center}
		which means $\phi$ is depicted as
		\begin{center}
			\begin{tikzcd}[ampersand replacement=\&]
				{\rins{2}{2}{f, g, \deg}} \ar[r, "p"] \& S \ar[rr, bend left = 60, "b"{name=B}, near end] \ar[rr, loose, "a"{name=A}, very near start, crossing over] \ar[Rightarrow, from=A, to=B, shorten=1.5mm, "f"description]  \ar[rr, bend right = 60, "b"'{name=C}] \ar[Rightarrow, from=A, to=C, shorten=0.5mm, "g"description] \ar[Rightarrow, equal, from=B, to=C, shorten=1mm]  \& \& T \ar[from=1-2, to=1-4, loose, crossing over, "a", very near start] 
			\end{tikzcd}.
		\end{center}
		If we are actually in an enhanced $2$-category, the triangle spanned by $f, g, \deg$ would collapse, and we would obtain the $c$-rigged equifier of $f$ and $g$. 
	\end{example}
	
	\subsection{Initially rigged $n$-inserters}
	We have the dual notion of terminally rigged $n$-inserters, that generalises $l$-rigged limits in \cite{LS:2012}.

	\begin{defi}
		\label{def:initial}
		Let $F \colon \mathbb{D}^n_{[[n]]} \to \bbK$ be an enhanced simplicial functor. The \emph{initially rigged $n$-inserter of $F^\co$} in $\bbK^\co$ is defined as the limit $\{\Phi_\co, F^\co\}$ of $F^\co$ weighted by $\Phi_\co$ constructed in \longref{Construction}{construct:weight}. The \emph{projection of the initially rigged $n$-inserter} is defined as the unique limit projection $\{\Phi_\co, F^\co\} \to F^\co x$. 
	\end{defi}
	
	\begin{pro}
		\label{pro:terminal_initial}
		Let $\bbK$ be an enhanced simplicial category. Then, $\bbK^\co$ admits terminally rigged $n$-inserters if and only if $\bbK$ admits initially rigged $n$-inserters.
		
		In other words, an initially rigged $n$-inserter in an enhanced simplicial category $\bbK$ is simply a terminally rigged $n$-inserter in $\bbK^\co$.
	\end{pro}
	
	\begin{proof}
		%		By \longref{Lemma}{lem:composite_tight}, the pullback diagram
		%		\begin{center}
			%			\begin{tikzcd}
				%				{\rins{n}{[[n]]}{D}} \arrow[dr, phantom, "\lrcorner", very near start] \ar[r, "\phi", loose] \ar[d,  "p"'] & T^{\Delta^n } \ar[d,  "{T^{\partial}}"]
				%				\\
				%				S \ar[r, "D"', loose]  & T^{\partial\Delta^n }
				%			\end{tikzcd}
			%		\end{center}
		%		of a (loose) $0$-arrow $D \colon S \leadsto T^{\partial \Delta^n}$ along $T^\partial$ in $\bbK$, is equivalently a pullback diagram in $\bbK^\co$.
		%		
		%		In terms of $\F_\Delta$weighted limits, it is clear that the $\F_\Delta$weighted limit $\{\Psi, F\}$ in $\bbK$ defined in \longref{Definition}{def:rigged_weight} is the $\F_\Delta$weighted limit $\{\psI, F^\co\}$ in $\bbK^\co$, where $\psI \colon \mathbb{D}^n_{[[0]]} \to \bbK^\co$ is the $\F_\Delta$-weight constructed in \longref{Construction}{construct:weight_initial}, because 
		%		
		By \longref{Proposition}{pro:K^co_has_dual_lim}, the terminally rigged $n$-inserter $\{\Phi, F\}$ in $\bbK$ corresponds to the initially rigged $n$-inserter $\{\Phi_\co, F^\co\}$ in $\bbK^\co$. Replacing $\bbK$ with $\bbK^\co$ yields the desired statement.
	\end{proof}
	
	%	\begin{defi}
		%		\label{def:initial}
		%		Let $F \colon \mathbb{D}^n_{[[0]]} \to \bbK$ be an  enhanced simplicial functor. The \emph{initially rigged $n$-inserter of $F$} is defined as the limit $\{\psI, F\}$ of $F$ weighted by $\psI$ constructed in \longref{Construction}{construct:weight_initial}. The \emph{projection of the initially rigged $n$-inserter} is defined as the unique limit projection $\{\psI, F\} \to Fx$. 
		%	\end{defi}
	
	\begin{rk}
		\label{rk:initial_flexible}
		By \longref{Remark}{rk:term_flexible}, initially rigged $n$-inserters in a chordate enhanced simplicial category are {flexible weighted limits}.
	\end{rk}
	
	\begin{rk}
		\label{rk:F_as_D_initial}
		Following \longref{Remark}{rk:F_as_D}, when an enhanced simplicial category $\bbK$ admits simplicial powers by inchordate enhanced simplicial sets,  a diagram ${\mathbb{D}^n_{[[n]]}}^\co \to \bbK$ is specified by two objects $S$ and $T$ of $\bbK$, and a loose $0$-arrow $D \colon S \leadsto T^{\partial \Delta^n}$, where the composite
		\[ S \xleadsto{D} T^{\partial \Delta^n} \xrightarrow{\ev[[0]]} T \]
		of $D$ with the evaluation at the initial vertex is a tight $0$-arrow.
	\end{rk}

	\subsection{Constructing (co)Eilenberg-Moore objects over loose (co)monads from rigged $n$-inserters}
	Rigged $n$-inserters are significant, not only because it subsumes a wide class of limits; in fact, they serve as the building blocks for (co)Eilenberg-Moore objects over loose (co)monads. The concept of rigged $n$-inserters allows a better understanding and analysis on (co)Eilenberg-Moore objects over loose (co)monads.
	
	\begin{lemma}
		\label{lem:EM}
		Let $\bbK$ be an enhanced simplicial category, where the tight part $\calK_\tau$ is an $\infty$-cosmos. Suppose $\bbK$ admits limits of a countable chain of tight isofibrations, simplicial powers by inchordate enhanced simplicial sets, and terminally rigged $n$-inserters, where the projection of a terminally rigged $n$-inserter is always a tight isofibration of $\bbK$.
		
		Then, for any loose monad $T \colon \Mnd \to \bbK$ on an object $A$ in $\bbK$, the Eilenberg-Moore object over $T$ exists, and that the forgetful functor from the Eilenberg-Moore object to $A$ is a tight isofibration.
	\end{lemma}
	
	\begin{proof}
		Following \cite[Lemma 5.10]{RV:2015} and \cite[Corollary 5.12]{RV:2015}, we refine the constructions and adapt to the enhanced setting.
		
		For any $k \in \N_0$, an $\F_\Delta$-weight $W_k \colon \Mnd \to \bbF_\Delta$ amounts to an enhanced simplicial set $W_k(+)$, which is equipped with a left action by $N(\Delta_{\emptyset})$. We define $W_k(+)$ to be an enhanced simplicial set whose loose part is the same as the simplicial set $W_k(+)$ in \cite[Lemma 5.10]{RV:2015}, where only the terminal vertex $[0]$ is tight. It is clear that $W_k(+)$ is an enhanced simplicial subset of $W(+)$ in \longref{Example}{eg:EM}, and the inclusion preserves tight vertices. The left action for $W_k$ is then inherited from $W$. To conclude, each $W_k$ is a sub-$\F_\Delta$-weight of $W$. In particular, $W_0$ is the representable $\Mnd(+, \cdot) \colon \Mnd \to \bbF_\Delta$.
		
		We then obtain a sequence of $\F_\Delta$-weights
		$$W_0 \hookrightarrow W_1 \hookrightarrow \cdots \hookrightarrow W_k \hookrightarrow \cdots$$
		whose colimit is the $\F_\Delta$-weight $W \colon \Mnd \to \bbF_\Delta$, where each $W_{k + 1}$ can be constructed from $W_k$ as a pushout
		\begin{center}
			\begin{tikzcd}[ampersand replacement=\&]
				{\partial \Delta^{m_k} \times {\Mnd(+, \cdot)}}  \ar[rr, "\scriptscriptstyle{i^{m_k} \times {\Mnd(+, \cdot)}}", hook] \ar[d,   "d_k"'] \& \phantom{} \& {\Delta^{m_k} \times {\Mnd(+, \cdot)}} \ar[d]
				\\
				W_{k} \ar[rr, hook]  \& \phantom{} \& {W_{k + 1}}\arrow[ul, phantom, "\ulcorner", very near start]
			\end{tikzcd}
		\end{center}
		in the enhanced simplicial category $[\Mnd, \bbF_\Delta]$, for some $m_k \ge 1$.
		
		We proceed by induction to show the existence of each $\{W_k, T\}$ in $\bbK$. 
		
		When $k = 0$, it is clear that $\{{\Mnd(+, \cdot)}, T\} \cong A$ exists, by Yoneda's Lemma. Assume that $\{W_k, T\}$ exists. The enhanced simplicial functor $\{\cdot, T\} \colon [\Mnd, \bbF_\Delta] \to \bbK$ sends the span in the above pushout diagram to a cospan
		% https://q.uiver.app/#q=WzAsMyxbMCwxLCJcXHtXX2ssIFRcXH0iXSxbMSwxLCJcXHtcXHBhcnRpYWxcXERlbHRhXnttX2t9XFx0aW1lcyBXXzAsIFRcXH0iXSxbMSwwLCJcXHtcXERlbHRhXnttX2t9XFx0aW1lcyBXXzAsIFRcXH0iXSxbMCwxLCJEOj0gXFx7ZCwgVFxcfSIsMl0sWzIsMSwiXFx7XFxpb3RhIFxcdGltZXMgV18wLCBUXFx9Il1d
		\[\begin{tikzcd}[ampersand replacement=\&]
			\& {\{\Delta^{m_k}\times {\Mnd(+, \cdot)}, T\}} \\
			{\{W_k, T\}} \& {\{\partial\Delta^{m_k}\times {\Mnd(+, \cdot)}, T\}}
			\arrow["{\{i^{m_k} \times {\Mnd(+, \cdot)}, T\}}", from=1-2, to=2-2, loose]
			\arrow["{D_k:= \{d_k, T\}}"', from=2-1, to=2-2, loose]
		\end{tikzcd}\]
		in $\bbK$. We show that the pullback $\{W_{k + 1}, T\}$ of this cospan exists, and is indeed a terminally rigged $n$-inserter.
		
		For any simplicial set $J$, and for any $X \in \ob\bbK$, we have a chain of isomorphisms
		$$
		\begin{aligned}
			[\Mnd, \bbF_\Delta](J \times {\Mnd(+, \cdot)}, \bbK(X, T)) & \cong [\Mnd, \bbF_\Delta]( {\Mnd(+, \cdot)}, \bbK(X, T))^J
			\\
			&\cong \bbK(X, \{{\Mnd(+, \cdot)}, T\})^J \cong \bbK(X, \{{\Mnd(+, \cdot)}, T\}^J)
		\end{aligned}
		$$
		in $\F_\Delta$, and therefore, $\{J \times {\Mnd(+, \cdot)}, T\} \cong \{{\Mnd(+, \cdot)}, T\}^J \cong A^J$. Similarly, for any simplicial map $j \colon J_1 \to J_2$, the image $\{j \times {\Mnd(+, \cdot)}, T\} \colon \{J_2 \times {\Mnd(+, \cdot)}, T\} \to \{J_1 \times W_0, T\}$ of $j \times {\Mnd(+, \cdot)} \colon J_1 \times {\Mnd(+, \cdot)} \to J_2 \times {\Mnd(+, \cdot)}$ under $\{\cdot, T\}$ is defined to be the universal map associated to the composite of 
		$j \times {\Mnd(+, \cdot)}$ with the unit $J_2 \times \Mnd(+, \cdot) \to \bbK(A^{J_2}, T)$ for $\{J_2 \times {\Mnd(+, \cdot)}, T\}$, which is simply given by the pre-composition $A^j\colon A^{J_2} \to A^{J_1}$ by $j$.
		
		Moreover, we check that 
		$$\{W_k, T\} \xleadsto{D := \{d, T\}} A^{\partial \Delta^{m_k}} \xtwoheadrightarrow{\ev[[m_k]]} A$$
		is a tight $0$-arrow in $\bbK$. From \cite[Observation 5.15]{RV:2015}, we have a commutative triangle
		% https://q.uiver.app/#q=WzAsMyxbMCwxLCJcXERlbHRhXjBfe1tbMF1dfSBcXHRpbWVzIHtcXE1uZCgrLCBcXGNkb3QpfSJdLFsxLDAsIlxccGFydGlhbFxcRGVsdGFee21fa30gXFx0aW1lcyB7XFxNbmQoKywgXFxjZG90KX0iXSxbMSwxLCJXX2siXSxbMCwxLCJbW21fa11dIFxcdGltZXMgMSJdLFsxLDIsImRfayJdLFswLDIsIlswXSBcXG9wbHVzIFxcY2RvdCIsMix7InN0eWxlIjp7InRhaWwiOnsibmFtZSI6Imhvb2siLCJzaWRlIjoidG9wIn19fV1d
		\[\begin{tikzcd}[ampersand replacement=\&]
			\& {\partial\Delta^{m_k} \times {\Mnd(+, \cdot)}} \\
			{\Delta^0_{[[0]]} \times {\Mnd(+, \cdot)}} \& {W_k}
			\arrow["{d_k}", from=1-2, to=2-2]
			\arrow["{[[m_k]] \times 1}", from=2-1, to=1-2]
			\arrow["{[0] \oplus \cdot}"', hook, from=2-1, to=2-2]
		\end{tikzcd}.\]
		By Yoneda's Lemma, for any enhanced simplicial set $S$, we have a natural isomorphism
		$$[\Mnd, \bbF_\Delta](S \times \Mnd(+, \cdot), W_k) \cong \bbF_\Delta(S, W_k(+)),$$
		which particularly sends $\Delta^0_{[[0]]} \times \Mnd(+, \cdot) \xhookrightarrow{[0] \oplus \cdot} W_k$ to its component
		$$\Delta^0_{[[0]]} \xrightarrow{[0] \oplus [-1]} W_k(+)$$
		at the identity $[-1]$. So altogether, the above commutative triangle is translated to
		% https://q.uiver.app/#q=WzAsMyxbMCwxLCJcXERlbHRhXjBfe1tbMF1dfSAiXSxbMSwwLCJcXHBhcnRpYWxcXERlbHRhXnttX2t9ICJdLFsxLDEsIldfaygrKSJdLFswLDEsIltbbV9rXV0gIl0sWzEsMl0sWzAsMiwiWzBdICIsMl1d
		\[\begin{tikzcd}[ampersand replacement=\&]
			\& {\partial\Delta^{m_k} } \\
			{\Delta^0_{[[0]]} } \& {W_k(+)}
			\arrow[from=1-2, to=2-2]
			\arrow["{[[m_k]] }", from=2-1, to=1-2]
			\arrow["{[0] }"', from=2-1, to=2-2]
		\end{tikzcd}.\]
		Since $\Delta^0_{[[0]]} \xrightarrow{[0]} W_k(+)$ is clearly an enhanced simplicial map, we conclude that
		$$\Delta^0_{[[0]]} \times \Mnd(+, \cdot) \xrightarrow{[[m_k]] \times 1} \partial \Delta^{m_k} \times \Mnd(+, \cdot) \xrightarrow{d_k} W_k$$
		is an enhanced simplicial natural transformation, which means it is a tight $0$-arrow in $[\Mnd, \bbF_\Delta]$. Now, the contravariant $\F_\Delta$-functor $\{\cdot, T\}$ takes this tight $0$-arrow to a tight $0$-arrow
		$$\{W_k, T\} \xleadsto{D_k := \{d_k, T\}} \{\partial \Delta^{m_k} \times \Mnd(+, \cdot), T\} \cong A^{\partial \Delta^{m_k}} \xrightarrow{\{[[m_k]] \times \Mnd(+, \cdot), T\}} \{\Mnd(+, \cdot), T\} \cong A$$
		in $\bbK$, where ${\{[[m_k]] \times \Mnd(+, \cdot), T\}}$ is indeed $\ev[[m_k]]$.
		
		Altogether, the pullback of the above cospan is given by
		\begin{center}
			\begin{tikzcd}[ampersand replacement=\&]
				{\{W_{k + 1}, T\}} \arrow[dr, phantom, "\lrcorner", very near start] \ar[r, "", loose] \ar[d, two heads,  ""'] \& {A^{\Delta^{m_k}}} \ar[d, two heads, "{A^\partial}"]
				\\
				\{W_{k}, T\} \ar[r, "D"', loose]  \& {A^{\partial\Delta^{m_k}}}
			\end{tikzcd},
		\end{center}
		which is a terminally rigged $n$-inserter, where the projection $\{W_{k + 1}, T\} \twoheadrightarrow \{W_k, T\}$ is a tight isofibration.
		
		In conclusion, the $\F_\Delta$-functor $\{\cdot, T\} \colon \Mnd \to \bbF_\Delta$ sends the sequence to a countable chain of tight isofibrations
		$$\cdots \twoheadrightarrow \{W_k, T\} \twoheadrightarrow \cdots \twoheadrightarrow \{\Mnd(+, \cdot), T\} \cong A$$
		of rigged $n$-inserters, whose limit is the Eilenberg-Moore object over $T$.
		
		The forgetful functor is a tight isofibration which also reflects tightness, because the Eilenberg-Moore object over $T$ is constructed as a limit of a countable chain of tight isofibrations.
	\end{proof}
	
	%	\begin{example}
		%		It is well-known that in $2$-category theory, the Eilenberg-Moore object over a (loose) monad in an (enhanced) $2$-category can be built from a (rigged) inserter and equifiers. Taking certain inserter equip the structure morphisms for $T$-algebras, and taking certain equifiers impose the associativity and the unitality, respectively.
		%		
		%		{\hl more expxplictly}
		%	\end{example}
	
	And we have the dual version as follows.
	
	\begin{lemma}
		\label{lem:coEM}
		Let $\bbK$ be an enhanced simplicial category, where the tight part $\calK_\tau$ is an $\infty$-cosmos. Suppose $\bbK$ admits limits of a countable chain of tight isofibrations, simplicial powers by inchordate enhanced simplicial sets, and initially rigged $n$-inserters, where the projection of an initially rigged $n$-inserter is always a tight isofibration of $\bbK$.
		
		Then, for any loose comonad $C \colon \Mnd \to \bbK$ on an object $A$ in $\bbK$, the coEilenberg-Moore object over $C$ exists, and that the forgetful functor from the coEilenberg-Moore object to $A$ is a tight isofibration.
	\end{lemma}
	
	%	\begin{proof}
		%		From \longref{Example}{eg:EM}, we note that the coEilenberg-Moore object $\{W_\co, C\}$ over $C$ is indeed the Eilenberg-Moore object $\{W, C^\co\}$ over $C^\co$.
		%		
		%		Following \longref{Lemma}{lem:EM}, we conclude that the Eilenberg-Moore object $\{W, C^\co\}$ is the limit of the countable chain
		%		$$\cdots \twoheadrightarrow \{W_k, C^\co\} \twoheadrightarrow \cdots \twoheadrightarrow \{\Mnd(+, \cdot), C^\co\} \cong A$$
		%		of tight isofibrations, where each weighted limit $\{W_k, C^\co\}$ in the chain is a terminally rigged inserter in $\bbK$, which is equivalently an innitially rigged inserter in $\bbK^\co$, according to \longref{Proposition}{pro:terminal_initial}.
		%	\end{proof}
	
	\section{Completeness results for $(\infty, 1)$-categories with structure and their lax morphisms}
	\label{sec:complete}
	We are going to establish our main results in this section.
	
	\subsection{\texorpdfstring{$\Lali(\calK)$}{Lali(K)}}
	Recall in \longref{Example}{eg:bbLali} that $\Lali(\calK)$ denotes the enhanced simplicial category of lali-isofibrations of an $\infty$-cosmos $\calK$. 
	
	First of all, it is important that $\Lali(\calK)$ has tight cosmological limits.
	
	\begin{pro}
		\label{pro:tight_lim}
		Tight cosmological limits exist in the enhanced simplicial category $\Lali(\calK)$ of lali-isofibrations of an $\infty$-cosmos $\calK$, and are preserved by the inclusion
		$$\Lali(\calK) \hookrightarrow \calK^{\isof}_\chor.$$
	\end{pro}
	
	\begin{proof}
		As explained in \longref{Example}{eg:tight_lim}, a tight limit in an enhanced simplicial category $\bbK$ is a $\sSet$-weighted limit in $\calK_\tau$, which is also a $\sSet$-weighted limit in $\calK_\lambda$ with the additional universal property that the projections jointly reflect tightness. In our case, a tight cosmological limit in $\Lali(\calK)$ amounts to a $\sSet$-weighted limit in the $\infty$-cosmos $\calLali(\calK)$, in addition, the limit projections jointly reflect morphisms of lalis. In short, our goal is to show that for an existing $\sSet$-weighted limit in $\Lali(\calK)$, the limit projections jointly reflect morphisms of lalis.
		
		More precisely, let $W \colon \mathcal{I} \to \sSet$ be a  $\sSet$-weight, and $F \colon \mathcal{I} \to \calLali(\calK)$ be a simplicial functor. For each object $i$ in $\mathcal{I}$, $Fi \colon (Fi)_1 \twoheadrightarrow (Fi)_2$ is a lali-isofibration with right adjoint $r_i \colon (Fi)_2 \to (Fi)_1$. The corresponding unit in the homotopy $2$-category for the lali-isofibration $Fi$ is denoted by
		% https://q.uiver.app/#q=WzAsMyxbMCwwLCIoRmkpXzEiXSxbMiwwLCIoRmkpXzEiXSxbMSwxLCIoRmkpXzIiXSxbMCwxLCIiLDIseyJzdHlsZSI6eyJoZWFkIjp7Im5hbWUiOiJub25lIn19fV0sWzAsMiwiRmkiLDJdLFsyLDEsInJfaSIsMl0sWzMsMiwiXFxldGFfaSIsMCx7InNob3J0ZW4iOnsic291cmNlIjo0MCwidGFyZ2V0IjozMH19XV0=
		\[\begin{tikzcd}[ampersand replacement=\&]
			{(Fi)_1} \&\& {(Fi)_1} \\
			\& {(Fi)_2}
			\arrow[""{name=0, anchor=center, inner sep=0}, equal, from=1-1, to=1-3]
			\arrow["Fi"', two heads, from=1-1, to=2-2]
			\arrow["{r_i}"', from=2-2, to=1-3]
			\arrow["{\eta_i}", shorten <=6pt, shorten >=5pt, Rightarrow, from=0, to=2-2]
		\end{tikzcd}.\]
		Suppose $\{W, F\} \colon \{W, F\}_1 \twoheadrightarrow \{W, F\}_2$ exists in $\calLali(\calK)$, whose right adjoint is given by $R \colon \{W, F\}_2 \to \{W, F\}_1$. Denote by $p_{i} \colon \{W, F\} \to Fi$ a limit projection to $Fi$, which amounts to a commutative square
		% https://q.uiver.app/#q=WzAsNCxbMCwwLCJcXHtXLCBGXFx9XzEiXSxbMSwwLCIoRmkpXzEiXSxbMCwxLCJcXHtXLCBGXFx9XzIiXSxbMSwxLCIoRmkpXzIiXSxbMCwxLCJ7cF9pfV8xIl0sWzIsMywie3BfaX1fMiIsMl0sWzAsMiwiXFx7VywgRlxcfSIsMix7InN0eWxlIjp7ImhlYWQiOnsibmFtZSI6ImVwaSJ9fX1dLFsxLDMsIkZpIiwwLHsic3R5bGUiOnsiaGVhZCI6eyJuYW1lIjoiZXBpIn19fV1d
		\[\begin{tikzcd}[ampersand replacement=\&]
			{\{W, F\}_1} \& {(Fi)_1} \\
			{\{W, F\}_2} \& {(Fi)_2}
			\arrow["{{p_i}_1}", two heads, from=1-1, to=1-2]
			\arrow["{\{W, F\}}"', two heads, from=1-1, to=2-1]
			\arrow["Fi", two heads, from=1-2, to=2-2]
			\arrow["{{p_i}_2}"', two heads, from=2-1, to=2-2]
		\end{tikzcd},\] 
		for any $i$ in $\mathcal{I}$. By construction, the $\eta_i$'s induce the unit
		% https://q.uiver.app/#q=WzAsMyxbMCwwLCJcXHtXLCBGXFx9XzEiXSxbMiwwLCJcXHtXLCBGXFx9XzEiXSxbMSwxLCJcXHtXLCBGXFx9XzIiXSxbMCwxLCIiLDIseyJzdHlsZSI6eyJoZWFkIjp7Im5hbWUiOiJub25lIn19fV0sWzAsMiwiXFx7VywgRlxcfSIsMl0sWzIsMSwiUiIsMl0sWzMsMiwiXFxldGEiLDAseyJzaG9ydGVuIjp7InNvdXJjZSI6NDAsInRhcmdldCI6MzB9fV1d
		\[\begin{tikzcd}[ampersand replacement=\&]
			{\{W, F\}_1} \&\& {\{W, F\}_1} \\
			\& {\{W, F\}_2}
			\arrow[""{name=0, anchor=center, inner sep=0}, equal, from=1-1, to=1-3]
			\arrow["{\{W, F\}}"', two heads, from=1-1, to=2-2]
			\arrow["R"', from=2-2, to=1-3]
			\arrow["\eta", shorten <=6pt, shorten >=5pt, Rightarrow, from=0, to=2-2]
		\end{tikzcd}\]
		in the homotopy $2$-category for $\{W, F\}$, which satisfies the equation
		\begin{equation}
			\label{eqt:eta_proj}
			\begin{tikzcd}[ampersand replacement=\&, scale cd = 0.9]
				{\{W, F\}_1} \&\& {\{W, F\}_1} \& {(Fi)_1} \\
				\& {\{W, F\}_2}
				\arrow[""{name=0, anchor=center, inner sep=0}, equal, from=1-1, to=1-3]
				\arrow["{\{W, F\}}"', two heads, from=1-1, to=2-2]
				\arrow["{{p_i}_1}", two heads, from=1-3, to=1-4]
				\arrow["R"', from=2-2, to=1-3]
				\arrow["\eta", shorten <=6pt, shorten >=5pt, Rightarrow, from=0, to=2-2]
			\end{tikzcd}
			=
			\begin{tikzcd}[ampersand replacement=\&, scale cd = 0.9]
				{\{W, F\}_1} \& {(Fi)_1} \&\& {(Fi)_1} \\
				\&\& {(Fi)_2}
				\arrow["{{p_i}_1}", two heads, from=1-1, to=1-2]
				\arrow[""{name=0, anchor=center, inner sep=0}, equal, from=1-2, to=1-4]
				\arrow["Fi"', two heads, from=1-2, to=2-3]
				\arrow["{r_i}"', from=2-3, to=1-4]
				\arrow["{\eta_i}", shorten <=6pt, shorten >=5pt, Rightarrow, from=0, to=2-3]
			\end{tikzcd},
		\end{equation}
		for each object $i$.
		
		Now, let $C \colon C_1 \twoheadrightarrow C_2$ be a lali-isofibration of $\calK$, whose right adjoint is given by $R_C \colon C_2 \to C_1$ and counit is denoted by $\epsilon_C$. Let $q \colon C \to \{W, F\}$ be a morphism of isofibrations, i.e., a commutative square
		% https://q.uiver.app/#q=WzAsNCxbMSwwLCJcXHtXLCBGXFx9XzEiXSxbMSwxLCJcXHtXLCBGXFx9XzIiXSxbMCwwLCJDXzEiXSxbMCwxLCJDXzIiXSxbMCwxLCJcXHtXLCBGXFx9IiwwLHsic3R5bGUiOnsiaGVhZCI6eyJuYW1lIjoiZXBpIn19fV0sWzIsMywiQyIsMix7InN0eWxlIjp7ImhlYWQiOnsibmFtZSI6ImVwaSJ9fX1dLFsyLDAsInFfMSJdLFszLDEsInFfMiIsMl1d
		\[\begin{tikzcd}[ampersand replacement=\&]
			{C_1} \& {\{W, F\}_1} \\
			{C_2} \& {\{W, F\}_2}
			\arrow["{q_1}", from=1-1, to=1-2]
			\arrow["C"', two heads, from=1-1, to=2-1]
			\arrow["{\{W, F\}}", two heads, from=1-2, to=2-2]
			\arrow["{q_2}"', from=2-1, to=2-2]
		\end{tikzcd}.\]
		From \longref{Equation}{eqt:eta_proj}, we deduce that
		% https://q.uiver.app/#q=WzAsOCxbMiwwLCJcXHtXLCBGXFx9XzEiXSxbMiwxLCJcXHtXLCBGXFx9XzIiXSxbMSwwLCJDXzEiXSxbMSwxLCJDXzIiXSxbMywxLCIoRmkpXzIiXSxbMywwLCIoRmkpXzEiXSxbMCwxLCJDXzIiXSxbNCwwLCIoRmkpXzEiXSxbMCwxLCJcXHtXLCBGXFx9IiwxLHsic3R5bGUiOnsiaGVhZCI6eyJuYW1lIjoiZXBpIn19fV0sWzIsMywiQyIsMCx7InN0eWxlIjp7ImhlYWQiOnsibmFtZSI6ImVwaSJ9fX1dLFsyLDAsInFfMSJdLFszLDEsInFfMiIsMl0sWzEsNCwie3BfaX1fMiIsMl0sWzAsNSwie3BfaX1fMSJdLFs1LDQsIkZpIiwxLHsic3R5bGUiOnsiaGVhZCI6eyJuYW1lIjoiZXBpIn19fV0sWzYsMywiIiwwLHsic3R5bGUiOnsiaGVhZCI6eyJuYW1lIjoibm9uZSJ9fX1dLFs2LDIsIlJfQyJdLFs0LDcsInJfaSIsMl0sWzUsNywiIiwyLHsic3R5bGUiOnsiaGVhZCI6eyJuYW1lIjoibm9uZSJ9fX1dLFsxNiwxNSwiXFxlcHNpbG9uX0MiLDAseyJzaG9ydGVuIjp7InNvdXJjZSI6MjAsInRhcmdldCI6MjB9fV0sWzE4LDE3LCJcXGV0YV9pIiwyLHsic2hvcnRlbiI6eyJzb3VyY2UiOjIwLCJ0YXJnZXQiOjIwfX1dXQ==
		\[\begin{tikzcd}[ampersand replacement=\&, scale cd = 0.85]
			\& {C_1} \& {\{W, F\}_1} \& {(Fi)_1} \& {(Fi)_1} \\
			{C_2} \& {C_2} \& {\{W, F\}_2} \& {(Fi)_2}
			\arrow["{q_1}", from=1-2, to=1-3]
			\arrow["C", two heads, from=1-2, to=2-2]
			\arrow["{{p_i}_1}", from=1-3, to=1-4]
			\arrow["{\{W, F\}}"{description}, two heads, from=1-3, to=2-3]
			\arrow[""{name=0, anchor=center, inner sep=0}, equal, from=1-4, to=1-5]
			\arrow["Fi"{description}, two heads, from=1-4, to=2-4]
			\arrow[""{name=1, anchor=center, inner sep=0}, "{R_C}", from=2-1, to=1-2]
			\arrow[""{name=2, anchor=center, inner sep=0}, equal, from=2-1, to=2-2]
			\arrow["{q_2}"', from=2-2, to=2-3]
			\arrow["{{p_i}_2}"', from=2-3, to=2-4]
			\arrow[""{name=3, anchor=center, inner sep=0}, "{r_i}"', from=2-4, to=1-5]
			\arrow["{\eta_i}"', shorten <=2pt, shorten >=2pt, Rightarrow, from=0, to=3]
			\arrow["{\epsilon_C}", shorten <=2pt, shorten >=2pt, Rightarrow, from=1, to=2]
		\end{tikzcd}
		=
		\begin{tikzcd}[ampersand replacement=\&, scale cd = 0.85]
			\& {C_1} \& {\{W, F\}_1} \& {\{W, F\}_1} \& {(Fi)_1} \\
			{C_2} \& {C_2} \& {\{W, F\}_2}
			\arrow["{q_1}", from=1-2, to=1-3]
			\arrow["C", two heads, from=1-2, to=2-2]
			\arrow[""{name=0, anchor=center, inner sep=0}, equal, from=1-3, to=1-4]
			\arrow["{\{W, F\}}"{description}, two heads, from=1-3, to=2-3]
			\arrow["{{p_i}_1}", two heads, from=1-4, to=1-5]
			\arrow[""{name=1, anchor=center, inner sep=0}, "{R_C}", from=2-1, to=1-2]
			\arrow[""{name=2, anchor=center, inner sep=0}, equal, from=2-1, to=2-2]
			\arrow["{q_2}"', from=2-2, to=2-3]
			\arrow[""{name=3, anchor=center, inner sep=0}, "R"', from=2-3, to=1-4]
			\arrow["\eta"', shorten <=2pt, shorten >=2pt, Rightarrow, from=0, to=3]
			\arrow["{\epsilon_C}", shorten <=2pt, shorten >=2pt, Rightarrow, from=1, to=2]
		\end{tikzcd},
		\]
		for each $i \in \mathcal{I}$. By the universal property of $\{W, F\}$, the right hand side is an isomorphism for every $i$ if and only if
		% https://q.uiver.app/#q=WzAsNixbMiwwLCJcXHtXLCBGXFx9XzEiXSxbMiwxLCJcXHtXLCBGXFx9XzIiXSxbMSwwLCJDXzEiXSxbMSwxLCJDXzIiXSxbMCwxLCJDXzIiXSxbMywwLCJcXHtXLCBGXFx9XzEiXSxbMCwxLCJcXHtXLCBGXFx9IiwxLHsic3R5bGUiOnsiaGVhZCI6eyJuYW1lIjoiZXBpIn19fV0sWzIsMywiQyIsMCx7InN0eWxlIjp7ImhlYWQiOnsibmFtZSI6ImVwaSJ9fX1dLFsyLDAsInFfMSJdLFszLDEsInFfMiIsMl0sWzQsMywiIiwwLHsic3R5bGUiOnsiaGVhZCI6eyJuYW1lIjoibm9uZSJ9fX1dLFs0LDIsIlJfQyJdLFswLDUsIiIsMix7InN0eWxlIjp7ImhlYWQiOnsibmFtZSI6Im5vbmUifX19XSxbMSw1LCJSIiwyXSxbMTEsMTAsIlxcZXBzaWxvbl9DIiwwLHsic2hvcnRlbiI6eyJzb3VyY2UiOjIwLCJ0YXJnZXQiOjIwfX1dLFsxMiwxMywiXFxldGEiLDIseyJzaG9ydGVuIjp7InNvdXJjZSI6MjAsInRhcmdldCI6MjB9fV1d
		\[\begin{tikzcd}[ampersand replacement=\&]
			\& {C_1} \& {\{W, F\}_1} \& {\{W, F\}_1} \\
			{C_2} \& {C_2} \& {\{W, F\}_2}
			\arrow["{q_1}", from=1-2, to=1-3]
			\arrow["C", two heads, from=1-2, to=2-2]
			\arrow[""{name=0, anchor=center, inner sep=0}, equal, from=1-3, to=1-4]
			\arrow["{\{W, F\}}"{description}, two heads, from=1-3, to=2-3]
			\arrow[""{name=1, anchor=center, inner sep=0}, "{R_C}", from=2-1, to=1-2]
			\arrow[""{name=2, anchor=center, inner sep=0}, equal, from=2-1, to=2-2]
			\arrow["{q_2}"', from=2-2, to=2-3]
			\arrow[""{name=3, anchor=center, inner sep=0}, "R"', from=2-3, to=1-4]
			\arrow["\eta"', shorten <=2pt, shorten >=2pt, Rightarrow, from=0, to=3]
			\arrow["{\epsilon_C}", shorten <=2pt, shorten >=2pt, Rightarrow, from=1, to=2]
		\end{tikzcd}\]
		is an isomorphism. In other words, the mates of $p_i \cdot q$ are invertible if and only if the mate of $q$ is invertible; henceforth, every $p_i \cdot q$ is a morphism of lalis precisely when $q$ is a morphism of lalis.
	\end{proof}
	
	In particular, when $\calK = \qCat$, we recover the isofibrations of quasi-categories which are also lalis.
	
	In order to establish the existence of terminally rigged $n$-inserters in $\Lali(\calK)$ for an arbitrary $\calK$, our strategy is to first approach the problem in $\qCat$, specifically. In fact, our proof relies on the notion of universal elements of lalis of quasi-categories introduced by Bourke and Lack.
	
	\begin{defi}[{\cite[Definition 5.4]{BL:2023}}]
		\label{def:5.4}
		Let $p \colon E \twoheadrightarrow B$ be an isofibration of $\qCat$. A vertex $x$ in the quasi-category $E$ is said to be a \emph{$p$-universal element} if and only if for every diagram in the solid part of
		% https://q.uiver.app/#q=WzAsNSxbMiwwLCJFIl0sWzIsMSwiQiJdLFsxLDAsIlxccGFydGlhbFxcRGVsdGFebiJdLFsxLDEsIlxcRGVsdGFebiJdLFswLDAsIlxcRGVsdGFeMCJdLFswLDEsInAiLDAseyJzdHlsZSI6eyJoZWFkIjp7Im5hbWUiOiJlcGkifX19XSxbMiwwXSxbMywxXSxbNCwyLCJbW25dXSJdLFs0LDMsIltbbl1dIiwyXSxbMiwzLCJpXm4iLDIseyJzdHlsZSI6eyJ0YWlsIjp7Im5hbWUiOiJob29rIiwic2lkZSI6InRvcCJ9fX1dLFszLDAsImwiLDIseyJzdHlsZSI6eyJib2R5Ijp7Im5hbWUiOiJkYXNoZWQifX19XSxbNCwwLCJlIiwwLHsiY3VydmUiOi00fV1d
		\[\begin{tikzcd}[ampersand replacement=\&]
			{\Delta^0} \& {\partial\Delta^n} \& E \\
			\& {\Delta^n} \& B
			\arrow["{[[n]]}", from=1-1, to=1-2]
			\arrow["x", bend left = 30, from=1-1, to=1-3]
			\arrow["{[[n]]}"', from=1-1, to=2-2]
			\arrow[from=1-2, to=1-3]
			\arrow["{i^n}"', hook, from=1-2, to=2-2]
			\arrow["p", two heads, from=1-3, to=2-3]
			\arrow["l"', dashed, from=2-2, to=1-3]
			\arrow[from=2-2, to=2-3]
		\end{tikzcd},\]
		there is a lift $l \colon \Delta^n \to E$ making the whole diagram commute.
	\end{defi}
	
	\begin{rk}
		Universal elements in a quasi-category are isomorphic; conversely, if a vertex is isomorphic to a universal element, then it is also a universal element.
	\end{rk}
	
	\begin{pro}[{\cite[Proposition 5.5]{BL:2023}}]
		\label{pro:5.5}
		Consider a morphism of isofibrations
		\[\begin{tikzcd}[ampersand replacement=\&]
			{E_1} \& {E_2} \\
			{B_1} \& {B_2}
			\arrow["e", from=1-1, to=1-2]
			\arrow["{p_1}"', two heads, from=1-1, to=2-1]
			\arrow["{p_2}", two heads, from=1-2, to=2-2]
			\arrow["b"', from=2-1, to=2-2]
		\end{tikzcd}\]
		of quasi-categories. Then,
		\begin{enumerate}
			\item[$\bullet$] $p_1$ is a lali if and only if for any $y \colon \Delta^0 \to B_1$, there exists a $p_1$-universal element $x \colon \Delta^0 \to E_1$ satisfying $p_1 x = y$;
			\item[$\bullet$] if both $p_1$ and $p_2$ are lalis, then the square is a morphism of lalis if and only if $e$ sends $p_1$-universal elements to $p_2$-universal elements.
		\end{enumerate}
	\end{pro}
	
	The following result is the technical heart of this article.
	
	\begin{pro}
		\label{pro:Lali_qCat}
		The enhanced simplicial category $\Lali(\qCat)$ of lali-isofibrations of $\qCat$ admits terminally rigged $m$-inserters, which are preserved by the inclusion
		$$\Lali(\qCat) \hookrightarrow \qCat^{\isof}_\chor,$$
		moreover, the projection is an isofibration of $\Lali(\qCat)$. 
		
		In elementary terms, let $A \colon A_1 \twoheadrightarrow A_2$ and $B \colon B_1 \twoheadrightarrow B_2$ be lali-isofibrations of $\qCat$. Suppose $D \colon A \to B^{\partial \Delta^m}$ is a diagram in $\qCat^{\isof}$, where the composite
		$$A \xrightarrow{D} B^{\partial \Delta^m} \xtwoheadrightarrow{\ev[[m]]} B$$
		amounts to a morphism of lalis, i.e., $D$ consists of a square
		% https://q.uiver.app/#q=WzAsNCxbMCwwLCJBXzEiXSxbMSwwLCJCXzFee1xccGFydGlhbCBcXERlbHRhXm19Il0sWzAsMSwiQV8yIl0sWzEsMSwiQl8yXntcXHBhcnRpYWwgXFxEZWx0YV5tfSJdLFswLDEsIkRfMSJdLFsyLDMsIkRfMiIsMl0sWzAsMiwiQSIsMix7InN0eWxlIjp7ImhlYWQiOnsibmFtZSI6ImVwaSJ9fX1dLFsxLDMsIkJee1xccGFydGlhbCBcXERlbHRhXm19IiwwLHsic3R5bGUiOnsiaGVhZCI6eyJuYW1lIjoiZXBpIn19fV1d
		\[\begin{tikzcd}[ampersand replacement=\&]
			{A_1} \& {B_1^{\partial \Delta^m}} \\
			{A_2} \& {B_2^{\partial \Delta^m}}
			\arrow["{D_1}", from=1-1, to=1-2]
			\arrow["A"', two heads, from=1-1, to=2-1]
			\arrow["{B^{\partial \Delta^m}}", two heads, from=1-2, to=2-2]
			\arrow["{D_2}"', from=2-1, to=2-2]
		\end{tikzcd}\]
		in $\qCat$, where 
		% https://q.uiver.app/#q=WzAsNixbMCwwLCJBXzEiXSxbMSwwLCJCXzFee1xccGFydGlhbCBcXERlbHRhXm19Il0sWzAsMSwiQV8yIl0sWzEsMSwiQl8yXntcXHBhcnRpYWwgXFxEZWx0YV5tfSJdLFsyLDAsIkJfMSJdLFsyLDEsIkJfMiJdLFswLDEsIkRfMSJdLFsyLDMsIkRfMiIsMl0sWzAsMiwiQSIsMix7InN0eWxlIjp7ImhlYWQiOnsibmFtZSI6ImVwaSJ9fX1dLFsxLDQsIlxcZXZbW21dXV8xIl0sWzMsNSwiXFxldltbbV1dXzIiLDJdLFs0LDUsIkIiLDAseyJzdHlsZSI6eyJoZWFkIjp7Im5hbWUiOiJlcGkifX19XV0=
		\[\begin{tikzcd}[ampersand replacement=\&]
			{A_1} \& {B_1^{\partial \Delta^m}} \& {B_1} \\
			{A_2} \& {B_2^{\partial \Delta^m}} \& {B_2}
			\arrow["{D_1}", from=1-1, to=1-2]
			\arrow["A"', two heads, from=1-1, to=2-1]
			\arrow["{\ev[[m]]_1}", from=1-2, to=1-3]
			\arrow["B", two heads, from=1-3, to=2-3]
			\arrow["{D_2}"', from=2-1, to=2-2]
			\arrow["{\ev[[m]]_2}"', from=2-2, to=2-3]
		\end{tikzcd}\]
		is a morphism of lalis. Then, the pullback
		% https://q.uiver.app/#q=WzAsNCxbMCwwLCJcXHJpbnN7bn17bn17RH0iXSxbMSwwLCJCXntcXERlbHRhXm19Il0sWzAsMSwiQSJdLFsxLDEsIkJee1xccGFydGlhbFxcRGVsdGFebX0iXSxbMCwxLCJcXHBoaSJdLFsyLDMsIkQiLDJdLFswLDIsInAiLDIseyJzdHlsZSI6eyJoZWFkIjp7Im5hbWUiOiJlcGkifX19XSxbMSwzLCJCXntcXHBhcnRpYWx9IiwwLHsic3R5bGUiOnsiaGVhZCI6eyJuYW1lIjoiZXBpIn19fV0sWzAsMywiIiwxLHsic3R5bGUiOnsibmFtZSI6ImNvcm5lciJ9fV1d
		\[\begin{tikzcd}[ampersand replacement=\&]
			{\rins{m}{[[m]]}{D}} \& {B^{\Delta^m}} \\
			A \& {B^{\partial\Delta^m}}
			\arrow["\phi", from=1-1, to=1-2]
			\arrow["p"', two heads, from=1-1, to=2-1]
			\arrow["\lrcorner"{anchor=center, pos=0.125}, draw=none, from=1-1, to=2-2]
			\arrow["{B^{\partial}}", two heads, from=1-2, to=2-2]
			\arrow["D"', from=2-1, to=2-2]
		\end{tikzcd}\]
		of $D$ along $B^\partial$ is also a lali-isofibration, and that the projection $p$, which is the square
		% https://q.uiver.app/#q=WzAsNCxbMCwwLCJcXHJpbnN7bn17bn17RH1fMSJdLFsxLDAsIkFfMSJdLFswLDEsIlxccmluc3tufXtufXtEfV8yIl0sWzEsMSwiQV8yIl0sWzAsMSwicF8xIiwwLHsic3R5bGUiOnsiaGVhZCI6eyJuYW1lIjoiZXBpIn19fV0sWzIsMywicF8yIiwyXSxbMCwyLCJcXHJpbnN7bn17bn17RH0iLDIseyJzdHlsZSI6eyJoZWFkIjp7Im5hbWUiOiJlcGkifX19XSxbMSwzLCJBIiwwLHsic3R5bGUiOnsiaGVhZCI6eyJuYW1lIjoiZXBpIn19fV0sWzAsMywiIiwxLHsic3R5bGUiOnsibmFtZSI6ImNvcm5lciJ9fV1d
		\[\begin{tikzcd}[ampersand replacement=\&]
			{\rins{m}{[[m]]}{D}_1} \& {A_1} \\
			{\rins{m}{[[m]]}{D}_2} \& {A_2}
			\arrow["{p_1}", two heads, from=1-1, to=1-2]
			\arrow["{\rins{m}{[[m]]}{D}}"', two heads, from=1-1, to=2-1]
			\arrow[""{anchor=center, pos=0.125}, draw=none, from=1-1, to=2-2]
			\arrow["A", two heads, from=1-2, to=2-2]
			\arrow["{p_2}"', from=2-1, to=2-2]
		\end{tikzcd}\]	
		where $\rins{m}{[[m]]}{D}_i := \rins{m}{[[m]]}{D_i}$ for $i = 1, 2$, is a morphism of lalis and reflects morphisms of lalis. Moreover, the projection $p$
		is an isofibration.
	\end{pro}
	
	\begin{proof}
		We first prove the following. Given a commutative diagram as in the solid part of
		% https://q.uiver.app/#q=WzAsNyxbMSwwLCJcXHBhcnRpYWwgXFxEZWx0YV5uIl0sWzIsMCwiXFxyaW5ze219e1tbbV1dfXtEfV8xIl0sWzEsMSwiXFxEZWx0YV5uIl0sWzIsMSwiQV8xIl0sWzMsMCwiXFxyaW5ze219e1tbbV1dfXtEfV8yIl0sWzMsMSwiQV8yIl0sWzAsMSwiXFxEZWx0YV4wIl0sWzAsMSwiZyJdLFswLDIsImlebiIsMix7InN0eWxlIjp7InRhaWwiOnsibmFtZSI6Imhvb2siLCJzaWRlIjoidG9wIn19fV0sWzIsMywiYV8xIiwyXSxbNCw1LCJwXzIiXSxbMSwzLCJwXzEiLDEseyJsYWJlbF9wb3NpdGlvbiI6MzB9XSxbMSw0LCJcXHJpbnN7bX17W1ttXV19e0R9IiwwLHsic3R5bGUiOnsiaGVhZCI6eyJuYW1lIjoiZXBpIn19fV0sWzMsNSwiQSIsMix7InN0eWxlIjp7ImhlYWQiOnsibmFtZSI6ImVwaSJ9fX1dLFsyLDQsImQiLDIseyJsYWJlbF9wb3NpdGlvbiI6ODB9XSxbNiwyLCJbW25dXSIsMl0sWzIsMSwibCIsMCx7InN0eWxlIjp7ImJvZHkiOnsibmFtZSI6ImRhc2hlZCJ9fX1dLFs2LDMsInVfe0FfMX0iLDIseyJjdXJ2ZSI6M31dXQ==
		\begin{equation}
			\label{diag:new}
			\begin{tikzcd}[ampersand replacement=\&]
				\& {\partial \Delta^n} \& {\rins{m}{[[m]]}{D}_1} \& {\rins{m}{[[m]]}{D}_2} \\
				{\Delta^0} \& {\Delta^n} \& {A_1} \& {A_2}
				\arrow["g", from=1-2, to=1-3]
				\arrow["{i^n}"', hook, from=1-2, to=2-2]
				\arrow["{\rins{m}{[[m]]}{D}}", two heads, from=1-3, to=1-4]
				\arrow["{p_1}"{description, pos=0.3}, from=1-3, to=2-3]
				\arrow["{p_2}", from=1-4, to=2-4]
				\arrow["{[[n]]}"', from=2-1, to=2-2]
				\arrow["{u_{A_1}}"', bend right, from=2-1, to=2-3]
				\arrow["l", dashed, from=2-2, to=1-3]
				\arrow["d"'{pos=0.8}, from=2-2, to=1-4]
				\arrow["{a_1}"', from=2-2, to=2-3]
				\arrow["A"', two heads, from=2-3, to=2-4]
			\end{tikzcd},
		\end{equation}
		where $u_{A_1}$ is an $A$-universal element satisfying
		% https://q.uiver.app/#q=WzAsNSxbMCwxLCJcXERlbHRhXjAiXSxbMSwxLCJcXERlbHRhXm4iXSxbMiwxLCJcXHJpbnN7bX17W1ttXV19e0R9XzIiXSxbMywxLCJBXzIiXSxbMywwLCJBXzEiXSxbNCwzLCJBIiwwLHsic3R5bGUiOnsiaGVhZCI6eyJuYW1lIjoiZXBpIn19fV0sWzAsMSwiW1tuXV0iLDJdLFsxLDIsImQiLDJdLFsyLDMsInBfMiIsMl0sWzAsNCwidV97QV8xfSJdXQ==
		\begin{equation}
			\label{diag:u_A_1}
			\begin{tikzcd}[ampersand replacement=\&]
				\&\&\& {A_1} \\
				{\Delta^0} \& {\Delta^n} \& {\rins{m}{[[m]]}{D}_2} \& {A_2}
				\arrow["A", two heads, from=1-4, to=2-4]
				\arrow["{u_{A_1}}", from=2-1, to=1-4]
				\arrow["{[[n]]}"', from=2-1, to=2-2]
				\arrow["d"', from=2-2, to=2-3]
				\arrow["{p_2}"', from=2-3, to=2-4]
			\end{tikzcd},
		\end{equation}
		there exists a simplicial map $l \colon \Delta^n \to \rins{m}{[[m]]}{D}_1$ making \longref{Diagram}{diag:new} commute.

		First, note that $D_1 \cdot a_1 \cdot i^n = D_1 \cdot p_1 \cdot g = B_1^\partial \cdot \phi_1 \cdot g$, therefore, their transposes, $\overline{D_1 \cdot a_1}$ and $\overline{\phi_1 \cdot g}$, respectively, would make the solid part of the diagram
		% https://q.uiver.app/#q=WzAsNSxbMCwwLCJcXHBhcnRpYWwgXFxEZWx0YV5uIFxcdGltZXMgXFxwYXJ0aWFsIFxcRGVsdGFebSJdLFsxLDAsIlxccGFydGlhbCBcXERlbHRhXm4gXFx0aW1lcyBcXERlbHRhXm0iXSxbMCwxLCJcXERlbHRhXm4gXFx0aW1lcyBcXHBhcnRpYWwgXFxEZWx0YV5tIl0sWzEsMSwiKFxcRGVsdGFebiBcXHRpbWVzIFxccGFydGlhbCBcXERlbHRhXm0pIFxcY29wcm9kX3tcXHBhcnRpYWwgXFxEZWx0YV5uIFxcdGltZXMgXFxwYXJ0aWFsIFxcRGVsdGFebX0gKFxccGFydGlhbCBcXERlbHRhXm4gXFx0aW1lcyAgXFxEZWx0YV5tKSJdLFszLDIsIkJfMSJdLFswLDEsIiIsMix7InN0eWxlIjp7InRhaWwiOnsibmFtZSI6Imhvb2siLCJzaWRlIjoidG9wIn19fV0sWzIsMywial8xIiwwLHsic3R5bGUiOnsidGFpbCI6eyJuYW1lIjoiaG9vayIsInNpZGUiOiJ0b3AifX19XSxbMCwyLCIiLDEseyJzdHlsZSI6eyJ0YWlsIjp7Im5hbWUiOiJob29rIiwic2lkZSI6InRvcCJ9fX1dLFsxLDMsImpfMiJdLFsxLDQsIlxcb3ZlcmxpbmV7XFxwaGlfMSBcXGNkb3QgZ30iXSxbMiw0LCJcXG92ZXJsaW5le0RfMSBcXGNkb3QgYV8xfSIsMl0sWzMsNCwiXFxleGlzdHMgISBoIiwxLHsic3R5bGUiOnsiYm9keSI6eyJuYW1lIjoiZGFzaGVkIn19fV0sWzMsNSwiIiwwLHsibGV2ZWwiOjEsInN0eWxlIjp7Im5hbWUiOiJjb3JuZXIifX1dXQ==
		\begin{equation}
			\label{eqt:h}
			\begin{tikzcd}[ampersand replacement=\&]
				{\partial \Delta^n \times \partial \Delta^m} \& {\partial \Delta^n \times \Delta^m} \\
				{\Delta^n \times \partial \Delta^m} \& {\partial(\Delta^n \times \Delta^m)} \\
				\&\&\& {B_1}
				\arrow[""{name=0, anchor=center, inner sep=0}, hook, from=1-1, to=1-2]
				\arrow[hook, from=1-1, to=2-1]
				\arrow["{j_2}", from=1-2, to=2-2]
				\arrow["{\overline{\phi_1 \cdot g}}", near end, from=1-2, to=3-4]
				\arrow["{j_1}", hook, from=2-1, to=2-2]
				\arrow["{\overline{D_1 \cdot a_1}}"', from=2-1, to=3-4]
				\arrow["{\exists ! h}"{description}, dashed, from=2-2, to=3-4]
				\arrow["\lrcorner"{anchor=center, pos=0.125, rotate=180}, draw=none, from=2-2, to=0]
			\end{tikzcd}
		\end{equation}
		commute, where ${\partial(\Delta^n \times \Delta^m)} := {(\Delta^n \times \partial \Delta^m) \coprod_{\partial \Delta^n \times \partial \Delta^m} (\partial \Delta^n \times  \Delta^m)}$. By the universal property of the pushout ${\partial(\Delta^n \times \Delta^m)}$, there is a unique map $h \colon {\partial(\Delta^n \times \Delta^m)} \to B_1$ making the above diagram commute. Similarly, we have a unique map $\iota \colon {\partial(\Delta^n \times \Delta^m)}  \hookrightarrow \Delta^n \times \Delta^m$ making
		% https://q.uiver.app/#q=WzAsNSxbMCwwLCJcXHBhcnRpYWwgXFxEZWx0YV5uIFxcdGltZXMgXFxwYXJ0aWFsIFxcRGVsdGFebSJdLFsxLDAsIlxccGFydGlhbCBcXERlbHRhXm4gXFx0aW1lcyBcXERlbHRhXm0iXSxbMCwxLCJcXERlbHRhXm4gXFx0aW1lcyBcXHBhcnRpYWwgXFxEZWx0YV5tIl0sWzEsMSwiKFxcRGVsdGFebiBcXHRpbWVzIFxccGFydGlhbCBcXERlbHRhXm0pIFxcY29wcm9kX3tcXHBhcnRpYWwgXFxEZWx0YV5uIFxcdGltZXMgXFxwYXJ0aWFsIFxcRGVsdGFebX0gKFxccGFydGlhbCBcXERlbHRhXm4gXFx0aW1lcyAgXFxEZWx0YV5tKSJdLFszLDIsIlxcRGVsdGFebiBcXHRpbWVzIFxcRGVsdGFebSJdLFswLDEsIiIsMix7InN0eWxlIjp7InRhaWwiOnsibmFtZSI6Imhvb2siLCJzaWRlIjoidG9wIn19fV0sWzIsMywial8xIiwwLHsic3R5bGUiOnsidGFpbCI6eyJuYW1lIjoiaG9vayIsInNpZGUiOiJ0b3AifX19XSxbMCwyLCIiLDEseyJzdHlsZSI6eyJ0YWlsIjp7Im5hbWUiOiJob29rIiwic2lkZSI6InRvcCJ9fX1dLFsxLDMsImpfMiJdLFsxLDQsIiIsMCx7InN0eWxlIjp7InRhaWwiOnsibmFtZSI6Imhvb2siLCJzaWRlIjoidG9wIn19fV0sWzIsNCwiIiwyLHsic3R5bGUiOnsidGFpbCI6eyJuYW1lIjoiaG9vayIsInNpZGUiOiJ0b3AifX19XSxbMyw0LCJcXGV4aXN0cyAhIFxccGkiLDEseyJzdHlsZSI6eyJib2R5Ijp7Im5hbWUiOiJkYXNoZWQifX19XSxbMyw1LCIiLDIseyJsZXZlbCI6MSwic3R5bGUiOnsibmFtZSI6ImNvcm5lciJ9fV1d
		\[\begin{tikzcd}[ampersand replacement=\&]
			{\partial \Delta^n \times \partial \Delta^m} \& {\partial \Delta^n \times \Delta^m} \\
			{\Delta^n \times \partial \Delta^m} \& {\partial(\Delta^n \times \Delta^m)} \\
			\&\&\& {\Delta^n \times \Delta^m}
			\arrow[""{name=0, anchor=center, inner sep=0}, hook, from=1-1, to=1-2]
			\arrow[hook, from=1-1, to=2-1]
			\arrow["{j_2}", from=1-2, to=2-2]
			\arrow[hook, from=1-2, to=3-4]
			\arrow["{j_1}", hook, from=2-1, to=2-2]
			\arrow[hook, from=2-1, to=3-4]
			\arrow["{\exists ! \iota}"{description}, hook, dashed, from=2-2, to=3-4]
			\arrow["\lrcorner"{anchor=center, pos=0.125, rotate=180}, draw=none, from=2-2, to=0]
		\end{tikzcd}\]
		commute. Now, we verify that
		% https://q.uiver.app/#q=WzAsNCxbMCwwLCJ7KFxcRGVsdGFebiBcXHRpbWVzIFxccGFydGlhbCBcXERlbHRhXm0pIFxcY29wcm9kX3tcXHBhcnRpYWwgXFxEZWx0YV5uIFxcdGltZXMgXFxwYXJ0aWFsIFxcRGVsdGFebX0gKFxccGFydGlhbCBcXERlbHRhXm4gXFx0aW1lcyAgXFxEZWx0YV5tKX0gIl0sWzEsMCwiQl8xIl0sWzAsMSwiXFxEZWx0YV5uIFxcdGltZXMgXFxEZWx0YV5tIl0sWzEsMSwiQl8yIl0sWzEsMywiQiIsMCx7InN0eWxlIjp7ImhlYWQiOnsibmFtZSI6ImVwaSJ9fX1dLFswLDEsImgiXSxbMiwzLCJcXG92ZXJsaW5le1xccGhpXzIgXFxjZG90IGR9IiwyXSxbMCwyLCJcXGlvdGEiLDIseyJzdHlsZSI6eyJ0YWlsIjp7Im5hbWUiOiJob29rIiwic2lkZSI6InRvcCJ9fX1dXQ==
		\begin{equation}
			\label{eqt:h_iota}
			\begin{tikzcd}[ampersand replacement=\&]
				{{\partial(\Delta^n \times \Delta^m)} } \& {B_1} \\
				{\Delta^n \times \Delta^m} \& {B_2}
				\arrow["h", from=1-1, to=1-2]
				\arrow["\iota"', hook, from=1-1, to=2-1]
				\arrow["B", two heads, from=1-2, to=2-2]
				\arrow["{\overline{\phi_2 \cdot d}}"', from=2-1, to=2-2]
			\end{tikzcd},
		\end{equation}
		i.e., $B \cdot h = \overline{\phi_2 \cdot d} \cdot \iota$, where ${\overline{\phi_2 \cdot d}}$ denotes the transpose of the composite
		$$\Delta^n \xrightarrow{d} \rins{m}{[[m]]}{D} \xrightarrow{\phi_2} B_2^{\Delta^m}.$$
		For this, we make use of the universal property of the pushout, and check that $\overline{\phi_2 \cdot d} \cdot (1 \times i^m) = B \cdot \overline{D_1 \cdot a_1}$ and $\overline{\phi_2 \cdot d} \cdot (i^m \times 1) = B \cdot \overline{\phi_1 \cdot g}$. Note that $\phi_2 \cdot d \cdot i^n = \phi_2 \cdot \rins{m}{[[m]]}{D} \cdot g = B^{\Delta^m} \cdot \phi_1 \cdot g$, where $B^{\Delta^m} \colon B_1^{\Delta^m} \twoheadrightarrow B_2^{\Delta^m}$ is the simplicial power of $B$ by $\Delta^m$;  also, $B_2^\partial \cdot \phi_2 \cdot d = D_2 \cdot p_2 \cdot d = D_2 \cdot A \cdot a_1 = B^{\partial \Delta^m} \cdot D_1 \cdot a_1$, where $B_2^\partial \colon B_2^{\Delta^m} \twoheadrightarrow B_2^{\partial\Delta^m}$ is the restriction for $B_2$, and $B^{\partial\Delta^m} \colon B_1^{\partial\Delta^m} \twoheadrightarrow B_2^{\partial\Delta^m}$ is the simplicial power of $B$ by $\partial\Delta^m$, hence the two equations are verified.
		%		{\hl may add a 3d diagram}
		
		We now define $u_{B_1}$ and $b_1$ in terms of $u_{A_1}$ and $a_1$ via the following commutative diagram
		
		%		Since $\ev[[m]] \cdot D$ is assumed to be a morphism of lalis, by \cite[Proposition 5.5]{BL:2023}, we have a commutative diagram
		% https://q.uiver.app/#q=WzAsNSxbMiwwLCJBXzEiXSxbMywwLCJCXzFee1xccGFydGlhbCBcXERlbHRhXm19Il0sWzMsMSwiQl8xIl0sWzEsMCwiXFxEZWx0YV5uIl0sWzAsMCwiXFxEZWx0YV4wIl0sWzAsMSwiRF8xIl0sWzEsMiwiXFxldltbbV1dXzEiLDAseyJzdHlsZSI6eyJoZWFkIjp7Im5hbWUiOiJlcGkifX19XSxbNCwzLCJbW25dXSJdLFszLDAsImFfMSJdLFszLDIsImJfMSIsMV0sWzQsMiwidV97Ql8xfSIsMl0sWzQsMCwidV97QV8xfSIsMCx7ImN1cnZlIjotM31dXQ==
		\begin{equation}
			\label{eqt:star}
			\begin{tikzcd}[ampersand replacement=\&]
				{\Delta^0} \& {\Delta^n} \& {A_1} \& {B_1^{\partial \Delta^m}} \\
				\&\&\& {B_1}
				\arrow["{[[n]]}", from=1-1, to=1-2]
				\arrow["{u_{A_1}}", bend left = 30, from=1-1, to=1-3]
				\arrow["{u_{B_1}}"', from=1-1, to=2-4]
				\arrow["{a_1}", from=1-2, to=1-3]
				\arrow["{b_1}"{description}, from=1-2, to=2-4]
				\arrow["{D_1}", from=1-3, to=1-4]
				\arrow["{\ev[[m]]_1}", two heads, from=1-4, to=2-4]
			\end{tikzcd}.
		\end{equation}
		%		in other words, $\ev[[m]]_1 \cdot D_1$ sends $u_{A_1}$ to $u_{B_1}$. This tells us that $\overline{D_1 \cdot a_1}$ sends $([[n]], [[m]])$ in $\Delta^n \times \partial \Delta^m$ to the $B$-universal element $u_{B_1}$ of $B_1$. 
		Note that since $\ev[[m]] \cdot D$ is assumed to be a morphism of lalis, by \longref{Proposition}{pro:5.5}, $u_{B_1}$ is $B$-universal. Then from the bottom triangle in \longref{Diagram}{eqt:h}, we see that $h$ sends $([[n]], [[m]])$ to $u_{B_1}$, together with \longref{Diagram}{eqt:h_iota}, this means
		%		Moreover, we note that the $j$-simplices of $\Delta^n \times \Delta^m$ for any $j > 0$ which are not in ${(\Delta^n \times \partial \Delta^m) \coprod (\partial \Delta^n \times  \Delta^m)}$ must have $([[n]], [[m]])$ as the terminal vertex.
		% https://q.uiver.app/#q=WzAsNSxbMCwwLCJcXERlbHRhXjAiXSxbMiwwLCJ7XFxwYXJ0aWFsKFxcRGVsdGFebiBcXHRpbWVzIFxcRGVsdGFebSl9Il0sWzMsMCwiQl8xIl0sWzIsMSwiXFxEZWx0YV5uIFxcdGltZXMgXFxEZWx0YV5tIl0sWzMsMSwiQl8yIl0sWzEsMiwiaCIsMl0sWzAsMSwiKFtbbl1dLCBbW21dXSkiLDJdLFswLDIsInVfe0JfMX0iLDAseyJjdXJ2ZSI6LTR9XSxbMiw0LCJCIiwwLHsic3R5bGUiOnsiaGVhZCI6eyJuYW1lIjoiZXBpIn19fV0sWzEsMywiXFxpb3RhIiwyLHsic3R5bGUiOnsidGFpbCI6eyJuYW1lIjoiaG9vayIsInNpZGUiOiJ0b3AifX19XSxbMyw0LCJcXG92ZXJsaW5le1xccGhpXzIgXFxjZG90IGR9IiwyXV0=
		\[\begin{tikzcd}[ampersand replacement=\&]
			{\Delta^0} \&\& {{\partial(\Delta^n \times \Delta^m)}} \& {B_1} \\
			\&\& {\Delta^n \times \Delta^m} \& {B_2}
			\arrow["{([[n]], [[m]])}"', from=1-1, to=1-3]
			\arrow["{u_{B_1}}", bend left, from=1-1, to=1-4]
			\arrow["h"', from=1-3, to=1-4]
			\arrow["\iota"', hook, from=1-3, to=2-3]
			\arrow["B", two heads, from=1-4, to=2-4]
			\arrow["{\overline{\phi_2 \cdot d}}"', from=2-3, to=2-4]
		\end{tikzcd}\]
		commutes. Our goal is then to obtain a lift $\Delta^n \times \Delta^m \to B_1$ for this commutative diagram.
		
		Note that since $\iota$ is injective, by \cite[Proposition 3.2.2]{book:Hovey:1999}, $\iota$ can be written as a countable composite
		$$\partial(\Delta^n \times \Delta^m) =: X_0 \xhookrightarrow{\iota_0} X_1 \hookrightarrow \cdots \hookrightarrow X_j \xhookrightarrow{\iota_j} X_{j + 1} \hookrightarrow \cdots \hookrightarrow \Delta^n \times \Delta^m$$
		of injective simplicial maps, where each $X_{j + 1}$ is the pushout
		% https://q.uiver.app/#q=WzAsNCxbMSwxLCJYX3tqICsgMX0iXSxbMSwwLCJYX2oiXSxbMCwxLCJcXGNvcHJvZF97fFNfanx9ICBcXERlbHRhXmoiXSxbMCwwLCJcXGNvcHJvZF97fFNfanx9ICBcXHBhcnRpYWxcXERlbHRhXmoiXSxbMywyLCJcXGNvcHJvZCBpXmoiLDIseyJzdHlsZSI6eyJ0YWlsIjp7Im5hbWUiOiJob29rIiwic2lkZSI6InRvcCJ9fX1dLFsyLDBdLFsxLDAsIlxcaW90YV9qIiwwLHsic3R5bGUiOnsidGFpbCI6eyJuYW1lIjoiaG9vayIsInNpZGUiOiJ0b3AifX19XSxbMywxXSxbMCwzLCIiLDEseyJzdHlsZSI6eyJuYW1lIjoiY29ybmVyIn19XV0=
		\[\begin{tikzcd}[ampersand replacement=\&]
			{\coprod_{|S_j|}  \partial\Delta^j} \& {X_j} \\
			{\coprod_{|S_j|}  \Delta^j} \& {X_{j + 1}}
			\arrow[from=1-1, to=1-2]
			\arrow["{\coprod i^j}"', hook, from=1-1, to=2-1]
			\arrow["{\iota_j}", hook, from=1-2, to=2-2]
			\arrow[from=2-1, to=2-2]
			\arrow["\lrcorner"{anchor=center, pos=0.125, rotate=180}, draw=none, from=2-2, to=1-1]
		\end{tikzcd}\]
		of the $|S_j|$-folded coproduct of $\Delta^j$ with $X_j$, where $S_j$ is the set of $j$-simplices of $\Delta^n \times \Delta^m$ which are not in the image of $\iota \colon \partial(\Delta^n \times \Delta^m) \hookrightarrow \Delta^n \times \Delta^m$.
		
		When $j = 0$, we have a simplicial map $h_0 := h$ making the diagram
		% https://q.uiver.app/#q=WzAsNSxbMCwwLCJcXHBhcnRpYWwoXFxEZWx0YV5uIFxcdGltZXMgXFxEZWx0YV5tKSJdLFsxLDAsIkJfMSJdLFswLDEsIlhfMCJdLFswLDIsIlxcRGVsdGFebiBcXHRpbWVzIFxcRGVsdGFebSJdLFsxLDIsIkJfMiJdLFsxLDQsIkIiLDAseyJzdHlsZSI6eyJoZWFkIjp7Im5hbWUiOiJlcGkifX19XSxbMCwxLCJoIl0sWzAsMiwiIiwyLHsibGV2ZWwiOjIsInN0eWxlIjp7ImhlYWQiOnsibmFtZSI6Im5vbmUifX19XSxbMiwzLCJcXGlvdGEiLDIseyJzdHlsZSI6eyJ0YWlsIjp7Im5hbWUiOiJob29rIiwic2lkZSI6InRvcCJ9fX1dLFsyLDEsImhfMCIsMix7InN0eWxlIjp7ImJvZHkiOnsibmFtZSI6ImRhc2hlZCJ9fX1dLFszLDQsIlxcb3ZlcmxpbmV7XFxwaGlfMiBcXGNkb3QgZH0iLDJdXQ==
		\[\begin{tikzcd}[ampersand replacement=\&]
			{\partial(\Delta^n \times \Delta^m)} \& {B_1} \\
			{X_0} \\
			{\Delta^n \times \Delta^m} \& {B_2}
			\arrow["h", from=1-1, to=1-2]
			\arrow[equals, from=1-1, to=2-1]
			\arrow["B", two heads, from=1-2, to=3-2]
			\arrow["{h_0}"', dashed, from=2-1, to=1-2]
			\arrow["\iota"', hook, from=2-1, to=3-1]
			\arrow["{\overline{\phi_2 \cdot d}}"', from=3-1, to=3-2]
		\end{tikzcd}\]
		commute. We then proceed by induction. Suppose we have a simplicial map $h_j \colon X_j \to B_1$ in the commutative diagram
		% https://q.uiver.app/#q=WzAsNixbMCwwLCJcXHBhcnRpYWwoXFxEZWx0YV5uIFxcdGltZXMgXFxEZWx0YV5tKSJdLFsxLDAsIkJfMSJdLFswLDMsIlxcRGVsdGFebiBcXHRpbWVzIFxcRGVsdGFebSJdLFsxLDMsIkJfMiJdLFswLDEsIlhfaiJdLFswLDIsIlhfe2ogKyAxfSJdLFsxLDMsIkIiLDAseyJzdHlsZSI6eyJoZWFkIjp7Im5hbWUiOiJlcGkifX19XSxbMCwxLCJoIl0sWzIsMywiXFxvdmVybGluZXtcXHBoaV8yIFxcY2RvdCBkfSIsMl0sWzAsNCwiXFxpb3RhX3tqIC0gMX0gXFxjZG90IFxcY2RvdHMgXFxjZG90IFxcaW90YV8wIiwyLHsic3R5bGUiOnsidGFpbCI6eyJuYW1lIjoiaG9vayIsInNpZGUiOiJ0b3AifX19XSxbNCw1LCJcXGlvdGFfaiIsMix7InN0eWxlIjp7InRhaWwiOnsibmFtZSI6Imhvb2siLCJzaWRlIjoidG9wIn19fV0sWzUsMiwiIiwyLHsic3R5bGUiOnsidGFpbCI6eyJuYW1lIjoiaG9vayIsInNpZGUiOiJ0b3AifX19XSxbNCwxLCJoX2oiLDJdXQ==
		\[\begin{tikzcd}[ampersand replacement=\&]
			{\partial(\Delta^n \times \Delta^m)} \& {B_1} \\
			{X_j} \\
			{X_{j + 1}} \\
			{\Delta^n \times \Delta^m} \& {B_2}
			\arrow["h", from=1-1, to=1-2]
			\arrow["{\iota_{j - 1} \cdot \cdots \cdot \iota_0}"', hook, from=1-1, to=2-1]
			\arrow["B", two heads, from=1-2, to=4-2]
			\arrow["{h_j}"', from=2-1, to=1-2]
			\arrow["{\iota_j}"', hook, from=2-1, to=3-1]
			\arrow[hook, from=3-1, to=4-1]
			\arrow["{\overline{\phi_2 \cdot d}}"', from=4-1, to=4-2]
		\end{tikzcd}.\]
		Consider a $j$-simplex in $S_j$, which by definition is not in the image of $\iota$, then we have a commutative diagram
		% https://q.uiver.app/#q=WzAsNyxbMSwxLCJYX3tqICsgMX0iXSxbMSwwLCJYX2oiXSxbMCwxLCIgXFxEZWx0YV5qIl0sWzAsMCwiIFxccGFydGlhbFxcRGVsdGFeaiJdLFszLDAsIkJfMSJdLFszLDEsIkJfMiJdLFsyLDEsIlxcRGVsdGFebiBcXHRpbWVzIFxcRGVsdGFebSJdLFszLDIsImleaiIsMix7InN0eWxlIjp7InRhaWwiOnsibmFtZSI6Imhvb2siLCJzaWRlIjoidG9wIn19fV0sWzIsMCwiXFxzaWdtYSIsMl0sWzEsMCwiXFxpb3RhX2oiLDAseyJzdHlsZSI6eyJ0YWlsIjp7Im5hbWUiOiJob29rIiwic2lkZSI6InRvcCJ9fX1dLFszLDFdLFs0LDUsIkIiLDAseyJzdHlsZSI6eyJoZWFkIjp7Im5hbWUiOiJlcGkifX19XSxbMSw0LCJoX2oiXSxbMCw2LCIiLDAseyJzdHlsZSI6eyJ0YWlsIjp7Im5hbWUiOiJob29rIiwic2lkZSI6InRvcCJ9fX1dLFs2LDUsIlxcb3ZlcmxpbmV7XFxwaGlfMiBcXGNkb3QgZH0iLDJdXQ==
		\[\begin{tikzcd}[ampersand replacement=\&]
			{ \partial\Delta^j} \& {X_j} \&\& {B_1} \\
			{ \Delta^j} \& {X_{j + 1}} \& {\Delta^n \times \Delta^m} \& {B_2}
			\arrow[from=1-1, to=1-2]
			\arrow["{i^j}"', hook, from=1-1, to=2-1]
			\arrow["{h_j}", from=1-2, to=1-4]
			\arrow["{\iota_j}", hook, from=1-2, to=2-2]
			\arrow["B", two heads, from=1-4, to=2-4]
			\arrow["\sigma"', from=2-1, to=2-2]
			\arrow[hook, from=2-2, to=2-3]
			\arrow["{\overline{\phi_2 \cdot d}}"', from=2-3, to=2-4]
		\end{tikzcd}.\]
		From the proof of \cite[Proposition 5.17]{RV:2015}, every simplex of $\Delta^n \times \Delta^m$ which is not in the image of $\iota \colon \partial(\Delta^n \times \Delta^m) \hookrightarrow \Delta^n \times \Delta^m$
		has terminal vertex $([[n]], [[m]])$. Thus, we have a commutative diagram
		% https://q.uiver.app/#q=WzAsOCxbMiwxLCJYX3tqICsgMX0iXSxbMiwwLCJYX2oiXSxbMSwxLCIgXFxEZWx0YV5qIl0sWzEsMCwiIFxccGFydGlhbFxcRGVsdGFeaiJdLFs0LDAsIkJfMSJdLFs0LDEsIkJfMiJdLFszLDEsIlxcRGVsdGFebiBcXHRpbWVzIFxcRGVsdGFebSJdLFswLDAsIlxcRGVsdGFeMCJdLFszLDIsImleaiIsMix7InN0eWxlIjp7InRhaWwiOnsibmFtZSI6Imhvb2siLCJzaWRlIjoidG9wIn19fV0sWzIsMCwiXFxzaWdtYSIsMl0sWzEsMCwiXFxpb3RhX2oiLDAseyJzdHlsZSI6eyJ0YWlsIjp7Im5hbWUiOiJob29rIiwic2lkZSI6InRvcCJ9fX1dLFszLDFdLFs0LDUsIkIiLDAseyJzdHlsZSI6eyJoZWFkIjp7Im5hbWUiOiJlcGkifX19XSxbMSw0LCJoX2oiXSxbMCw2LCIiLDAseyJzdHlsZSI6eyJ0YWlsIjp7Im5hbWUiOiJob29rIiwic2lkZSI6InRvcCJ9fX1dLFs2LDUsIlxcb3ZlcmxpbmV7XFxwaGlfMiBcXGNkb3QgZH0iLDJdLFs3LDMsIltbal1dIiwyXSxbNywxLCIoW1tuXV0sIFtbbV1dKSIsMCx7ImN1cnZlIjotMn1dLFs3LDQsInVfe0JfMX0iLDAseyJjdXJ2ZSI6LTV9XSxbMiw0LCJcXGV0YV9qIiwyLHsibGFiZWxfcG9zaXRpb24iOjcwLCJzdHlsZSI6eyJib2R5Ijp7Im5hbWUiOiJkYXNoZWQifX19XV0=
		\[\begin{tikzcd}[ampersand replacement=\&]
			{\Delta^0} \& { \partial\Delta^j} \& {X_j} \&\& {B_1} \\
			\& { \Delta^j} \& {X_{j + 1}} \& {\Delta^n \times \Delta^m} \& {B_2}
			\arrow["{[[j]]}"', from=1-1, to=1-2]
			\arrow["{([[n]], [[m]])}",bend left=30, from=1-1, to=1-3]
			\arrow["{u_{B_1}}", bend left = 60, from=1-1, to=1-5]
			\arrow[from=1-2, to=1-3]
			\arrow["{i^j}"', hook, from=1-2, to=2-2]
			\arrow["{h_j}", from=1-3, to=1-5]
			\arrow["{\iota_j}", hook, from=1-3, to=2-3]
			\arrow["B", two heads, from=1-5, to=2-5]
			\arrow["{\eta_j}"'{pos=0.7}, dashed, from=2-2, to=1-5]
			\arrow["\sigma"', from=2-2, to=2-3]
			\arrow[hook, from=2-3, to=2-4]
			\arrow["{\overline{\phi_2 \cdot d}}"', from=2-4, to=2-5]
		\end{tikzcd},\]
		which then gives us a simplicial map $\eta_j \colon \Delta^j \to B_1$, by the universality of $u_{B_1}$ from \longref{Definition}{def:5.4}. As $\sigma \in S_j$ is arbitrary, by the universal property of coproducts, we then obtain a map $\eta \colon \coprod_{|S_j|} \Delta^j \to B_1$ in the commutative diagram
		% https://q.uiver.app/#q=WzAsNyxbMSwxLCJYX3tqICsgMX0iXSxbMSwwLCJYX2oiXSxbMCwxLCJcXGNvcHJvZF97fFNfanx9IFxcRGVsdGFeaiJdLFswLDAsIiBcXGNvcHJvZF97fFNfanx9XFxwYXJ0aWFsXFxEZWx0YV5qIl0sWzMsMCwiQl8xIl0sWzMsMSwiQl8yIl0sWzIsMSwiXFxEZWx0YV5uIFxcdGltZXMgXFxEZWx0YV5tIl0sWzMsMiwiXFxjb3Byb2QgaV5qIiwyLHsic3R5bGUiOnsidGFpbCI6eyJuYW1lIjoiaG9vayIsInNpZGUiOiJ0b3AifX19XSxbMiwwLCJcXHNpZ21hIiwyXSxbMSwwLCJcXGlvdGFfaiIsMCx7InN0eWxlIjp7InRhaWwiOnsibmFtZSI6Imhvb2siLCJzaWRlIjoidG9wIn19fV0sWzMsMV0sWzQsNSwiQiIsMCx7InN0eWxlIjp7ImhlYWQiOnsibmFtZSI6ImVwaSJ9fX1dLFsxLDQsImhfaiJdLFswLDYsIiIsMCx7InN0eWxlIjp7InRhaWwiOnsibmFtZSI6Imhvb2siLCJzaWRlIjoidG9wIn19fV0sWzYsNSwiXFxvdmVybGluZXtcXHBoaV8yIFxcY2RvdCBkfSIsMl0sWzIsNCwiXFxldGEiLDAseyJsYWJlbF9wb3NpdGlvbiI6NDAsInN0eWxlIjp7ImJvZHkiOnsibmFtZSI6ImRhc2hlZCJ9fX1dLFswLDQsImhfe2ogKyAxfSIsMix7ImxhYmVsX3Bvc2l0aW9uIjo3MCwic3R5bGUiOnsiYm9keSI6eyJuYW1lIjoiZGFzaGVkIn19fV0sWzAsMywiIiwwLHsic3R5bGUiOnsibmFtZSI6ImNvcm5lciJ9fV1d
		\[\begin{tikzcd}[ampersand replacement=\&]
			{ \coprod_{|S_j|}\partial\Delta^j} \& {X_j} \&\& {B_1} \\
			{\coprod_{|S_j|} \Delta^j} \& {X_{j + 1}} \& {\Delta^n \times \Delta^m} \& {B_2}
			\arrow[from=1-1, to=1-2]
			\arrow["{\coprod i^j}"', hook, from=1-1, to=2-1]
			\arrow["{h_j}", from=1-2, to=1-4]
			\arrow["{\iota_j}", hook, from=1-2, to=2-2]
			\arrow["B", two heads, from=1-4, to=2-4]
			\arrow["\eta"{pos=0.4}, dashed, from=2-1, to=1-4]
			\arrow["\sigma"', from=2-1, to=2-2]
			\arrow["\lrcorner"{anchor=center, pos=0.125, rotate=180}, draw=none, from=2-2, to=1-1]
			\arrow["{h_{j + 1}}"'{pos=0.7}, dashed, from=2-2, to=1-4]
			\arrow[hook, from=2-2, to=2-3]
			\arrow["{\overline{\phi_2 \cdot d}}"', from=2-3, to=2-4]
		\end{tikzcd},\]
		and then by the universal property of pushouts, we have the desired simplicial map $h_{j + 1} \colon X_{j + 1} \to B_1$ making the above diagram commute.
		
		So, by induction, there is a simplicial map $\overline{f}\colon \Delta^n\times\Delta^m \to B_1$ such that the diagrams
		% https://q.uiver.app/#q=WzAsNCxbMCwwLCJ7KFxcRGVsdGFebiBcXHRpbWVzIFxccGFydGlhbCBcXERlbHRhXm0pIFxcY29wcm9kX3tcXHBhcnRpYWwgXFxEZWx0YV5uIFxcdGltZXMgXFxwYXJ0aWFsIFxcRGVsdGFebX0gKFxccGFydGlhbCBcXERlbHRhXm4gXFx0aW1lcyAgXFxEZWx0YV5tKX0gIl0sWzEsMCwiQl8xIl0sWzAsMSwiXFxEZWx0YV5uIFxcdGltZXMgXFxEZWx0YV5tIl0sWzEsMSwiQl8yIl0sWzEsMywiQiIsMCx7InN0eWxlIjp7ImhlYWQiOnsibmFtZSI6ImVwaSJ9fX1dLFswLDEsImgiXSxbMiwzLCJcXG92ZXJsaW5le1xccGhpXzIgXFxjZG90IGR9IiwyXSxbMCwyLCJcXGlvdGEiLDIseyJzdHlsZSI6eyJ0YWlsIjp7Im5hbWUiOiJob29rIiwic2lkZSI6InRvcCJ9fX1dLFsyLDEsIlxcb3ZlcmxpbmV7Zn0iLDJdXQ==
		\[\begin{tikzcd}[ampersand replacement=\&]
			{{\partial(\Delta^n \times \Delta^m)} } \& {B_1} \\
			{\Delta^n \times \Delta^m} \& {B_2}
			\arrow["h", from=1-1, to=1-2]
			\arrow["\iota"', hook, from=1-1, to=2-1]
			\arrow["B", two heads, from=1-2, to=2-2]
			\arrow["{\overline{f}}"', from=2-1, to=1-2]
			\arrow["{\overline{\phi_2 \cdot d}}"', from=2-1, to=2-2]
		\end{tikzcd},\]
		% https://q.uiver.app/#q=WzAsNCxbMCwxLCJ7KFxcRGVsdGFebiBcXHRpbWVzIFxccGFydGlhbCBcXERlbHRhXm0pIFxcY29wcm9kX3tcXHBhcnRpYWwgXFxEZWx0YV5uIFxcdGltZXMgXFxwYXJ0aWFsIFxcRGVsdGFebX0gKFxccGFydGlhbCBcXERlbHRhXm4gXFx0aW1lcyAgXFxEZWx0YV5tKX0gIl0sWzEsMSwiQl8xIl0sWzAsMiwiXFxEZWx0YV5uIFxcdGltZXMgXFxEZWx0YV5tIl0sWzAsMCwiXFxEZWx0YV5uIFxcdGltZXMgXFxwYXJ0aWFsIFxcRGVsdGFebSJdLFswLDEsImgiLDFdLFswLDIsIlxcaW90YSIsMix7InN0eWxlIjp7InRhaWwiOnsibmFtZSI6Imhvb2siLCJzaWRlIjoidG9wIn19fV0sWzIsMSwiXFxvdmVybGluZXtmfSIsMl0sWzMsMCwial8xIiwyLHsic3R5bGUiOnsidGFpbCI6eyJuYW1lIjoiaG9vayIsInNpZGUiOiJ0b3AifX19XSxbMywxLCJcXG92ZXJsaW5le0RfMSBcXGNkb3QgYV8xfSJdXQ==
		\[\begin{tikzcd}[ampersand replacement=\&]
			{\Delta^n \times \partial \Delta^m} \\
			{{\partial(\Delta^n \times \Delta^m)} } \& {B_1} \\
			{\Delta^n \times \Delta^m}
			\arrow["{j_1}"', hook, from=1-1, to=2-1]
			\arrow["{\overline{D_1 \cdot a_1}}", from=1-1, to=2-2]
			\arrow["h"{description}, from=2-1, to=2-2]
			\arrow["\iota"', hook, from=2-1, to=3-1]
			\arrow["{\overline{f}}"', from=3-1, to=2-2]
		\end{tikzcd}\]
		both commute. Denote by $f \colon \Delta^n \to B_1^{\Delta^m}$ the transpose of $\overline{f}$. Then the above commutative diagrams say that $B^{\Delta^m} \cdot f = \phi_2 \cdot d$ and $B_1^\partial \cdot f = D_1 \cdot a_1$.
		
		Now, we can apply the universal property of the pullback $\rins{m}{[[m]]}{D}_1$ to obtain the lift $l \colon \Delta^n \to \rins{m}{[[m]]}{D}_1$ satisfying
		% https://q.uiver.app/#q=WzAsNSxbMSwxLCJcXHJpbnN7bX17bX17RH1fMSJdLFsxLDIsIkFfMSJdLFsyLDIsIkJfMV57XFxwYXJ0aWFsXFxEZWx0YV5tfSJdLFsyLDEsIkJfMV57XFxEZWx0YV5tfSJdLFswLDAsIlxcRGVsdGFebiJdLFswLDMsIlxccGhpXzEiXSxbMSwyLCJEXzEiLDJdLFszLDIsIkJfMV5cXHBhcnRpYWwiLDAseyJzdHlsZSI6eyJoZWFkIjp7Im5hbWUiOiJlcGkifX19XSxbMCwxLCJwXzEiLDIseyJzdHlsZSI6eyJoZWFkIjp7Im5hbWUiOiJlcGkifX19XSxbMCwyLCIiLDEseyJzdHlsZSI6eyJuYW1lIjoiY29ybmVyIn19XSxbNCwxLCJhXzEiLDIseyJjdXJ2ZSI6M31dLFs0LDMsImYiLDAseyJjdXJ2ZSI6LTN9XSxbNCwwLCJcXGV4aXN0cyEgbCIsMCx7InN0eWxlIjp7ImJvZHkiOnsibmFtZSI6ImRhc2hlZCJ9fX1dXQ==
		\[\begin{tikzcd}[ampersand replacement=\&]
			{\Delta^n} \\
			\& {\rins{m}{[[m]]}{D}_1} \& {B_1^{\Delta^m}} \\
			\& {A_1} \& {B_1^{\partial\Delta^m}}
			\arrow["{\exists! l}", dashed, from=1-1, to=2-2]
			\arrow["f", bend left, from=1-1, to=2-3]
			\arrow["{a_1}"', bend right, from=1-1, to=3-2]
			\arrow["{\phi_1}", from=2-2, to=2-3]
			\arrow["{p_1}"', two heads, from=2-2, to=3-2]
			\arrow["\lrcorner"{anchor=center, pos=0.125}, draw=none, from=2-2, to=3-3]
			\arrow["{B_1^\partial}", two heads, from=2-3, to=3-3]
			\arrow["{D_1}"', from=3-2, to=3-3]
		\end{tikzcd}.\]
		
		Moreover, from \longref{Diagram}{eqt:h}, we have
		% https://q.uiver.app/#q=WzAsNCxbMCwxLCJ7KFxcRGVsdGFebiBcXHRpbWVzIFxccGFydGlhbCBcXERlbHRhXm0pIFxcY29wcm9kX3tcXHBhcnRpYWwgXFxEZWx0YV5uIFxcdGltZXMgXFxwYXJ0aWFsIFxcRGVsdGFebX0gKFxccGFydGlhbCBcXERlbHRhXm4gXFx0aW1lcyAgXFxEZWx0YV5tKX0gIl0sWzEsMSwiQl8xIl0sWzAsMiwiXFxEZWx0YV5uIFxcdGltZXMgXFxEZWx0YV5tIl0sWzAsMCwiXFxEZWx0YV5uIFxcdGltZXMgXFxwYXJ0aWFsIFxcRGVsdGFebSJdLFswLDEsImgiLDFdLFswLDIsIlxcaW90YSIsMix7InN0eWxlIjp7InRhaWwiOnsibmFtZSI6Imhvb2siLCJzaWRlIjoidG9wIn19fV0sWzIsMSwiXFxvdmVybGluZXtmfSIsMl0sWzMsMCwial8yIiwyLHsic3R5bGUiOnsidGFpbCI6eyJuYW1lIjoiaG9vayIsInNpZGUiOiJ0b3AifX19XSxbMywxLCJcXG92ZXJsaW5le1xccGhpXzEgXFxjZG90IGd9Il1d
		\[\begin{tikzcd}[ampersand replacement=\&]
			{\Delta^n \times \partial \Delta^m} \\
			{{\partial(\Delta^n \times \Delta^m)} } \& {B_1} \\
			{\Delta^n \times \Delta^m}
			\arrow["{j_2}"', hook, from=1-1, to=2-1]
			\arrow["{\overline{\phi_1 \cdot g}}", from=1-1, to=2-2]
			\arrow["h"{description}, from=2-1, to=2-2]
			\arrow["\iota"', hook, from=2-1, to=3-1]
			\arrow["{\overline{f}}"', from=3-1, to=2-2]
		\end{tikzcd},\]
		which gives $f \cdot i^n = \phi_1 \cdot g$ by transposing. We check that $p_1 \cdot (l \cdot i^n) = a_1 \cdot i^n = p_1 \cdot g$, and that $\phi_1 \cdot (l \cdot i^n) = f\cdot i^n = \phi_1 \cdot g$. Thus, by the universal property of the pullback $\rins{m}{[[m]]}{D}_1$, we get $g = l \cdot i^n$. Also, we check that $p_2 \cdot (\rins{m}{[[m]]}{D} \cdot l) = A \cdot p_1 \cdot l = A \cdot a_1 = p_2 \cdot d$, and that $\phi_2 \cdot (\rins{m}{[[m]]}{D} \cdot l) = B^{\Delta^m} \cdot \phi_1 \cdot l = B^{\Delta^m} \cdot f = \phi_2 \cdot d$. Hence, $d = \rins{m}{[[m]]}{D} \cdot l$.
		
		Concluding, we obtain a lift $l$ such that \longref{Diagram}{diag:new} commutes, as desired.
		
		Now, take $g$ as the unique simplicial map $\emptyset \to \rins{m}{[[m]]}{D}_1$, and $a_1$ as any $A$-universal element $u_{A_1} \colon \Delta^0 \to A_1$ satisfying \longref{Diagram}{diag:u_A_1}, which exists since $A$ is a lali. For any $d \colon \Delta^0 \to \rins{m}{[[m]]}{D}_1$, we clearly have a commutative diaagram
		% https://q.uiver.app/#q=WzAsNyxbMSwwLCJcXGVtcHR5c2V0Il0sWzIsMCwiXFxyaW5ze219e1tbbV1dfXtEfV8xIl0sWzEsMSwiXFxEZWx0YV4wIl0sWzIsMSwiQV8xIl0sWzMsMCwiXFxyaW5ze219e1tbbV1dfXtEfV8yIl0sWzMsMSwiQV8yIl0sWzAsMSwiXFxEZWx0YV4wIl0sWzAsMSwiISJdLFswLDIsIiEiLDIseyJzdHlsZSI6eyJ0YWlsIjp7Im5hbWUiOiJob29rIiwic2lkZSI6InRvcCJ9fX1dLFsyLDMsInVfe0FfMX0iLDJdLFs0LDUsInBfMiJdLFsxLDMsInBfMSIsMSx7ImxhYmVsX3Bvc2l0aW9uIjozMH1dLFsxLDQsIlxccmluc3ttfXtbW21dXX17RH0iLDAseyJzdHlsZSI6eyJoZWFkIjp7Im5hbWUiOiJlcGkifX19XSxbMyw1LCJBIiwyLHsic3R5bGUiOnsiaGVhZCI6eyJuYW1lIjoiZXBpIn19fV0sWzIsNCwiZCIsMix7ImxhYmVsX3Bvc2l0aW9uIjo4MH1dLFs2LDIsIiIsMix7ImxldmVsIjoyLCJzdHlsZSI6eyJoZWFkIjp7Im5hbWUiOiJub25lIn19fV0sWzIsMSwibCIsMCx7InN0eWxlIjp7ImJvZHkiOnsibmFtZSI6ImRhc2hlZCJ9fX1dXQ==
		\[\begin{tikzcd}[ampersand replacement=\&]
			\& \emptyset \& {\rins{m}{[[m]]}{D}_1} \& {\rins{m}{[[m]]}{D}_2} \\
			{\Delta^0} \& {\Delta^0} \& {A_1} \& {A_2}
			\arrow["{!}", from=1-2, to=1-3]
			\arrow["{!}"', hook, from=1-2, to=2-2]
			\arrow["{\rins{m}{[[m]]}{D}}", two heads, from=1-3, to=1-4]
			\arrow["{p_1}"{description, pos=0.3}, from=1-3, to=2-3]
			\arrow["{p_2}", from=1-4, to=2-4]
			\arrow[equals, from=2-1, to=2-2]
			\arrow["l", dashed, from=2-2, to=1-3]
			\arrow["d"'{pos=0.8}, from=2-2, to=1-4]
			\arrow["{u_{A_1}}"', from=2-2, to=2-3]
			\arrow["A"', two heads, from=2-3, to=2-4]
		\end{tikzcd}.\] From the above, we then obtain $l \colon \Delta^0 \to \rins{m}{[[m]]}{D}_1$ such that
		% https://q.uiver.app/#q=WzAsMyxbMCwxLCJcXERlbHRhXjAiXSxbMSwxLCJcXHJpbnN7bX17bX17RH1fMiJdLFsxLDAsIlxccmluc3ttfXttfXtEfV8xIl0sWzIsMSwiXFxyaW5ze219e219e0R9IiwwLHsic3R5bGUiOnsiaGVhZCI6eyJuYW1lIjoiZXBpIn19fV0sWzAsMiwidSIsMCx7InN0eWxlIjp7ImJvZHkiOnsibmFtZSI6ImRhc2hlZCJ9fX1dLFswLDEsImQiLDJdXQ==
		\[\begin{tikzcd}[ampersand replacement=\&]
			\& {\rins{m}{[[m]]}{D}_1} \\
			{\Delta^0} \& {\rins{m}{[[m]]}{D}_2}
			\arrow["{\rins{m}{[[m]]}{D}}", two heads, from=1-2, to=2-2]
			\arrow["l", dashed, from=2-1, to=1-2]
			\arrow["d"', from=2-1, to=2-2]
		\end{tikzcd}.\]
		commutes. By \longref{Proposition}{pro:5.5}, it remains to verify that $l$ is $\rins{m}{[[m]]}{D}$-universal. Consider any commutative diagram
		% https://q.uiver.app/#q=WzAsNSxbMiwwLCJcXHJpbnN7bX17W1ttXV19e0R9XzEiXSxbMiwxLCJcXHJpbnN7bX17W1ttXV19e0R9XzIiXSxbMSwwLCJcXHBhcnRpYWxcXERlbHRhXm4iXSxbMSwxLCJcXERlbHRhXm4iXSxbMCwwLCJcXERlbHRhXjAiXSxbMCwxLCJcXHJpbnN7bX17W1ttXV19e0R9IiwwLHsic3R5bGUiOnsiaGVhZCI6eyJuYW1lIjoiZXBpIn19fV0sWzIsMCwiXFxkZWx0YV8xIl0sWzMsMSwiXFxkZWx0YV8yIiwyXSxbNCwyLCJbW25dXSJdLFs0LDMsIltbbl1dIiwyXSxbMiwzLCJpXm4iLDIseyJzdHlsZSI6eyJ0YWlsIjp7Im5hbWUiOiJob29rIiwic2lkZSI6InRvcCJ9fX1dLFs0LDAsImwiLDAseyJjdXJ2ZSI6LTR9XV0=
		\[\begin{tikzcd}[ampersand replacement=\&]
			{\Delta^0} \& {\partial\Delta^n} \& {\rins{m}{[[m]]}{D}_1} \\
			\& {\Delta^n} \& {\rins{m}{[[m]]}{D}_2}
			\arrow["{[[n]]}", from=1-1, to=1-2]
			\arrow["l", bend left, from=1-1, to=1-3]
			\arrow["{[[n]]}"', from=1-1, to=2-2]
			\arrow["{\delta_1}", from=1-2, to=1-3]
			\arrow["{i^n}"', hook, from=1-2, to=2-2]
			\arrow["{\rins{m}{[[m]]}{D}}", two heads, from=1-3, to=2-3]
			\arrow["{\delta_2}"', from=2-2, to=2-3]
		\end{tikzcd},\]
		we then compose with the projection $p$ to a commutative diagram given as the solid part of
		% https://q.uiver.app/#q=WzAsNyxbMiwwLCJcXHJpbnN7bX17W1ttXV19e0R9XzEiXSxbMiwxLCJcXHJpbnN7bX17W1ttXV19e0R9XzIiXSxbMSwwLCJcXHBhcnRpYWxcXERlbHRhXm4iXSxbMSwxLCJcXERlbHRhXm4iXSxbMCwwLCJcXERlbHRhXjAiXSxbMywwLCJBXzEiXSxbMywxLCJBXzIiXSxbMCwxLCJcXHJpbnN7bX17W1ttXV19e0R9IiwyLHsic3R5bGUiOnsiaGVhZCI6eyJuYW1lIjoiZXBpIn19fV0sWzIsMCwiXFxkZWx0YV8xIl0sWzMsMSwiXFxkZWx0YV8yIiwyXSxbNCwyLCJbW25dXSJdLFs0LDMsIltbbl1dIiwyXSxbMiwzLCJpXm4iLDIseyJzdHlsZSI6eyJ0YWlsIjp7Im5hbWUiOiJob29rIiwic2lkZSI6InRvcCJ9fX1dLFs0LDAsImwiLDAseyJjdXJ2ZSI6LTR9XSxbNSw2LCJBIiwwLHsic3R5bGUiOnsiaGVhZCI6eyJuYW1lIjoiZXBpIn19fV0sWzAsNSwicF8xIl0sWzEsNiwicF8yIiwyXSxbNCw1LCJ1X3tBXzF9IiwwLHsiY3VydmUiOi01fV0sWzMsNSwiayIsMix7ImxhYmVsX3Bvc2l0aW9uIjo4MCwic3R5bGUiOnsiYm9keSI6eyJuYW1lIjoiZGFzaGVkIn19fV1d
		\[\begin{tikzcd}[ampersand replacement=\&]
			{\Delta^0} \& {\partial\Delta^n} \& {\rins{m}{[[m]]}{D}_1} \& {A_1} \\
			\& {\Delta^n} \& {\rins{m}{[[m]]}{D}_2} \& {A_2}
			\arrow["{[[n]]}", from=1-1, to=1-2]
			\arrow["l", bend left = 30, from=1-1, to=1-3]
			\arrow["{u_{A_1}}", bend left = 60, from=1-1, to=1-4]
			\arrow["{[[n]]}"', from=1-1, to=2-2]
			\arrow["{\delta_1}", from=1-2, to=1-3]
			\arrow["{i^n}"', hook, from=1-2, to=2-2]
			\arrow["{p_1}", from=1-3, to=1-4]
			\arrow["{\rins{m}{[[m]]}{D}}"', two heads, from=1-3, to=2-3]
			\arrow["A", two heads, from=1-4, to=2-4]
			\arrow["k"'{pos=0.8}, dashed, from=2-2, to=1-4]
			\arrow["{\delta_2}"', from=2-2, to=2-3]
			\arrow["{p_2}"', from=2-3, to=2-4]
		\end{tikzcd},\]
		because $u_{A_1} = a_1 = p_1 \cdot l$. Since $u_{A_1}$ is $A$-universal, there exists a lift $k \colon \Delta^n \to A_1$ making the diagram commute. In particular, $k \cdot [[n]] = u_{A_1}$, which means we have a commutative diagram
		% https://q.uiver.app/#q=WzAsNyxbMSwwLCJcXHBhcnRpYWwgXFxEZWx0YV5uIl0sWzIsMCwiXFxyaW5ze219e1tbbV1dfXtEfV8xIl0sWzEsMSwiXFxEZWx0YV5uIl0sWzIsMSwiQV8xIl0sWzMsMCwiXFxyaW5ze219e1tbbV1dfXtEfV8yIl0sWzMsMSwiQV8yIl0sWzAsMSwiXFxEZWx0YV4wIl0sWzAsMSwiXFxkZWx0YV8xIl0sWzAsMiwiaV5uIiwyLHsic3R5bGUiOnsidGFpbCI6eyJuYW1lIjoiaG9vayIsInNpZGUiOiJ0b3AifX19XSxbMiwzLCJrIiwyXSxbNCw1LCJwXzIiXSxbMSwzLCJwXzEiLDEseyJsYWJlbF9wb3NpdGlvbiI6MzB9XSxbMSw0LCJcXHJpbnN7bX17W1ttXV19e0R9IiwwLHsic3R5bGUiOnsiaGVhZCI6eyJuYW1lIjoiZXBpIn19fV0sWzMsNSwiQSIsMix7InN0eWxlIjp7ImhlYWQiOnsibmFtZSI6ImVwaSJ9fX1dLFsyLDQsIlxcZGVsdGFfMiIsMix7ImxhYmVsX3Bvc2l0aW9uIjo4MH1dLFs2LDIsIltbbl1dIiwyXSxbMiwxLCJcXGxhbWJkYSIsMCx7InN0eWxlIjp7ImJvZHkiOnsibmFtZSI6ImRhc2hlZCJ9fX1dLFs2LDMsInVfe0FfMX0iLDIseyJjdXJ2ZSI6M31dXQ==
		\[\begin{tikzcd}[ampersand replacement=\&]
			\& {\partial \Delta^n} \& {\rins{m}{[[m]]}{D}_1} \& {\rins{m}{[[m]]}{D}_2} \\
			{\Delta^0} \& {\Delta^n} \& {A_1} \& {A_2}
			\arrow["{\delta_1}", from=1-2, to=1-3]
			\arrow["{i^n}"', hook, from=1-2, to=2-2]
			\arrow["{\rins{m}{[[m]]}{D}}", two heads, from=1-3, to=1-4]
			\arrow["{p_1}"{description, pos=0.3}, from=1-3, to=2-3]
			\arrow["{p_2}", from=1-4, to=2-4]
			\arrow["{[[n]]}"', from=2-1, to=2-2]
			\arrow["{u_{A_1}}"', bend right, from=2-1, to=2-3]
			\arrow["\lambda", dashed, from=2-2, to=1-3]
			\arrow["{\delta_2}"'{pos=0.8}, from=2-2, to=1-4]
			\arrow["k"', from=2-2, to=2-3]
			\arrow["A"', two heads, from=2-3, to=2-4]
		\end{tikzcd},\]
		which is simply \longref{Diagram}{diag:new}, by taking $g = \delta_1$, $a_1 = k$, and $d = \delta_2$. Therefore, by the previous argument, we obtain a lift $\lambda \colon \Delta^n \to \rins{m}{[[m]]}{D}_1$. This proves the universality of $l$.

		It remains to verify that $p \colon \rins{m}{[[m]]}{D} \twoheadrightarrow A$ is a morphism of lalis and reflects morphisms of lalis.
		
		Indeed, since $u_{A_1} = p_1 \cdot g \cdot [[n]] = p_1 \cdot u$, by \longref{Proposition}{pro:5.5}, we see that
		\[\begin{tikzcd}[ampersand replacement=\&]
			{\rins{m}{[[m]]}{D}_1} \& {A_1} \\
			{\rins{m}{[[m]]}{D}_2} \& {A_2}
			\arrow["{p_1}", two heads, from=1-1, to=1-2]
			\arrow["{\rins{m}{[[m]]}{D}}"', two heads, from=1-1, to=2-1]
			\arrow[""{anchor=center, pos=0.125}, draw=none, from=1-1, to=2-2]
			\arrow["A", two heads, from=1-2, to=2-2]
			\arrow["{p_2}"', from=2-1, to=2-2]
		\end{tikzcd}\]	
		is a morphism of lalis.
		
		Next, suppose $C \colon C_1 \twoheadrightarrow C_2$ is a lali and that $q = (q_1, q_2) \colon C \to \rins{m}{[[m]]}{D}$ is a morphism in $\qCat^{\isof}$. Assume that $p \cdot q$, depicted as
		% https://q.uiver.app/#q=WzAsNixbMSwwLCJcXHJpbnN7bX17W1ttXV19e0R9XzEiXSxbMiwwLCJBXzEiXSxbMSwxLCJcXHJpbnN7bX17W1ttXV19e0R9XzIiXSxbMiwxLCJBXzIiXSxbMCwwLCJDXzEiXSxbMCwxLCJDXzIiXSxbMCwxLCJwXzEiLDAseyJzdHlsZSI6eyJoZWFkIjp7Im5hbWUiOiJlcGkifX19XSxbMiwzLCJwXzIiLDJdLFswLDIsIlxccmluc3ttfXtbW21dXX17RH0iLDIseyJzdHlsZSI6eyJoZWFkIjp7Im5hbWUiOiJlcGkifX19XSxbMSwzLCJBIiwwLHsic3R5bGUiOnsiaGVhZCI6eyJuYW1lIjoiZXBpIn19fV0sWzQsNSwiQyIsMix7InN0eWxlIjp7ImhlYWQiOnsibmFtZSI6ImVwaSJ9fX1dLFs0LDAsInFfMSJdLFs1LDIsInFfMiIsMl1d
		\[\begin{tikzcd}[ampersand replacement=\&]
			{C_1} \& {\rins{m}{[[m]]}{D}_1} \& {A_1} \\
			{C_2} \& {\rins{m}{[[m]]}{D}_2} \& {A_2}
			\arrow["{q_1}", from=1-1, to=1-2]
			\arrow["C"', two heads, from=1-1, to=2-1]
			\arrow["{p_1}", two heads, from=1-2, to=1-3]
			\arrow["{\rins{m}{[[m]]}{D}}"', two heads, from=1-2, to=2-2]
			\arrow["A", two heads, from=1-3, to=2-3]
			\arrow["{q_2}"', from=2-1, to=2-2]
			\arrow["{p_2}"', from=2-2, to=2-3]
		\end{tikzcd},\]
		is a morphism of lalis. According to \longref{Proposition}{pro:5.5}, if $u_{C_1}$ is a $C$-universal element of $C_1$, then $p_1 \cdot q_1 \cdot u_{C_1}$ is $A$-universal. Our task is to prove that $q_1 \cdot u_{C_1}$ is  $\rins{m}{[[m]]}{D}$-universal. Consider a commutative diagram as in the solid part of
		% https://q.uiver.app/#q=WzAsNyxbMiwwLCJcXHJpbnN7bX17W1ttXV19e0R9XzEiXSxbMiwxLCJcXHJpbnN7bX17W1ttXV19e0R9XzIiXSxbMSwwLCJcXHBhcnRpYWxcXERlbHRhXm4iXSxbMSwxLCJcXERlbHRhXm4iXSxbMCwwLCJcXERlbHRhXjAiXSxbMywwLCJBXzEiXSxbMywxLCJBXzIiXSxbMCwxLCJcXHJpbnN7bX17W1ttXV19e0R9IiwyLHsic3R5bGUiOnsiaGVhZCI6eyJuYW1lIjoiZXBpIn19fV0sWzIsMCwiXFxnYW1tYV8xIl0sWzMsMSwiXFxnYW1tYV8yIiwyXSxbNCwyLCJbW25dXSJdLFs0LDMsIltbbl1dIiwyXSxbMiwzLCJpXm4iLDIseyJzdHlsZSI6eyJ0YWlsIjp7Im5hbWUiOiJob29rIiwic2lkZSI6InRvcCJ9fX1dLFs0LDAsInFfMSBcXGNkb3QgdV97Q18xfSIsMCx7ImN1cnZlIjotNH1dLFs1LDYsIkEiLDAseyJzdHlsZSI6eyJoZWFkIjp7Im5hbWUiOiJlcGkifX19XSxbMCw1LCJwXzEiXSxbMSw2LCJwXzIiLDJdLFszLDUsIlxcYWxwaGEiLDIseyJsYWJlbF9wb3NpdGlvbiI6ODAsInN0eWxlIjp7ImJvZHkiOnsibmFtZSI6ImRhc2hlZCJ9fX1dXQ==
		\[\begin{tikzcd}[ampersand replacement=\&]
			{\Delta^0} \& {\partial\Delta^n} \& {\rins{m}{[[m]]}{D}_1} \& {A_1} \\
			\& {\Delta^n} \& {\rins{m}{[[m]]}{D}_2} \& {A_2}
			\arrow["{[[n]]}", from=1-1, to=1-2]
			\arrow["{q_1 \cdot u_{C_1}}", bend left, from=1-1, to=1-3]
			\arrow["{[[n]]}"', from=1-1, to=2-2]
			\arrow["{\gamma_1}", from=1-2, to=1-3]
			\arrow["{i^n}"', hook, from=1-2, to=2-2]
			\arrow["{p_1}", from=1-3, to=1-4]
			\arrow["{\rins{m}{[[m]]}{D}}"', two heads, from=1-3, to=2-3]
			\arrow["A", two heads, from=1-4, to=2-4]
			\arrow["\alpha"'{pos=0.8}, dashed, from=2-2, to=1-4]
			\arrow["{\gamma_2}"', from=2-2, to=2-3]
			\arrow["{p_2}"', from=2-3, to=2-4]
		\end{tikzcd},\]
		by the universality of $p_1 \cdot q_1 \cdot u_{C_1}$, there exists a lift $\alpha \colon \Delta^n \to A_1$. We then obtain \longref{Diagram}{diag:new} with $g = \gamma_1$, $a_1 = \alpha$, and $d = \gamma_2$. Therefore, we obtain a lift $\gamma \colon \Delta^n \to \rins{m}{[[m]]}{D}_1$. This proves the universality of $q_1 \cdot u_{C_1}$. Conversely, if both $p$ and $q$ send universal elements to universal elements, their composite would also do.
		
		Finally, the fact that $p$ is actually an isofibration of $\qCat^{\isof}$ follows directly from the stability of isofibrations under pullback in any $\infty$-cosmos.
	\end{proof}
	
	\begin{rk}
		This generalises \cite[Proposition 5.17]{RV:2015}, by taking every codomain of a lali to be the terminal quasi-category $\Delta^0$.
	\end{rk}
	
	In fact, the above result is true for models other than quasi-categories.
	
	\begin{theorem}
		\label{thm:Lali(K)}
		The enhanced simplicial category $\Lali(\calK)$ of lali-isofibrations of an $\infty$-cosmos $\calK$ admits terminally rigged $m$-inserters, and that the projection is an isofibration of $\Lali(\calK)$.

		In this case, the inclusion
		$$\Lali(\calK) \hookrightarrow \calK^{\isof}_\chor$$
		preserves terminally rigged $n$-inserters.
	\end{theorem}
	
	\begin{proof}
		We employ the notations in \longref{Proposition}{pro:Lali_qCat}.
		
		According to \longref{Proposition}{pro:5.5}, for any $\infty$-category $X$ in $\calK$, the maps
		$$
		\begin{aligned}
			&\calK(X ,A) \colon \calK(X, A_1) \twoheadrightarrow \calK(X, A_2),
			\\
			&\calK(X, B^{\partial \Delta^m}) \colon \calK(X, B_1^{\partial \Delta^m}) \twoheadrightarrow \calK(X, B_2^{\partial \Delta^m}),
			\\
			&\calK(X, B^{\Delta^m}) \colon \calK(X, B_1^{\Delta^m}) \twoheadrightarrow \calK(X, B_2^{\Delta^m}),
			\\
			&\calK(X, B) \colon \calK(X, B_1) \twoheadrightarrow \calK(X, B_2),
		\end{aligned}
		$$
		are lali-isofibrations in $\qCat$; in addition, the composite
		$$\calK(X, A) \xrightarrow{{\calK(X, D)}} {\calK(X, B^{\partial \Delta^m})} \xrightarrow{{\calK(X, \ev[[m]])}} \calK(X, B)$$
		is a morphism of lalis in $\qCat$.
		
		By \longref{Proposition}{pro:Lali_qCat}, we conclude that $\calK(X, \rins{m}{[[m]]}{D})$ as in
		% https://q.uiver.app/#q=WzAsNCxbMCwwLCJcXGNhbEsoWCwgXFxyaW5ze259e259e0R9KSJdLFsxLDAsIlxcY2FsSyhYLCBCXntcXERlbHRhXm19KSJdLFswLDEsIlxcY2FsSyhYLCBBKSJdLFsxLDEsIlxcY2FsSyhYLCBCXntcXHBhcnRpYWxcXERlbHRhXm19KSJdLFswLDEsIlxcY2FsSyhYLCBcXHBoaSkiXSxbMiwzLCJcXGNhbEsoWCwgRCkiLDJdLFswLDIsIlxcY2FsSyhYLCBwKSIsMix7InN0eWxlIjp7ImhlYWQiOnsibmFtZSI6ImVwaSJ9fX1dLFsxLDMsIlxcY2FsSyhYLCBCXntcXHBhcnRpYWx9KSIsMCx7InN0eWxlIjp7ImhlYWQiOnsibmFtZSI6ImVwaSJ9fX1dLFswLDUsIiIsMSx7ImxldmVsIjoxLCJzdHlsZSI6eyJuYW1lIjoiY29ybmVyIn19XV0=
		\[\begin{tikzcd}[ampersand replacement=\&]
			{\calK(X, \rins{m}{[[m]]}{D})} \& {\calK(X, B^{\Delta^m})} \\
			{\calK(X, A)} \& {\calK(X, B^{\partial\Delta^m})}
			\arrow["{\calK(X, \phi)}", from=1-1, to=1-2]
			\arrow["{\calK(X, p)}"', two heads, from=1-1, to=2-1]
			\arrow["{\calK(X, B^{\partial})}", two heads, from=1-2, to=2-2]
			\arrow[""{name=0, anchor=center, inner sep=0}, "{\calK(X, D)}"', from=2-1, to=2-2]
			\arrow["\lrcorner"{anchor=center, pos=0.125}, draw=none, from=1-1, to=0]
		\end{tikzcd}\]
		is actually a lali-isofibration of $\qCat$, where the projection $\calK(X, p)$ is a morphism of lalis and reflects morphisms of lalis. It remains to check that for any $\infty$-functor $f \colon Y \to X$ in $\calK$, the commutative square
		% https://q.uiver.app/#q=WzAsNCxbMCwwLCJcXGNhbEsoWCwgXFxyaW5ze259e259e0R9XzEpIl0sWzEsMCwiXFxjYWxLKFksIFxccmluc3tufXtufXtEfV8xKSJdLFswLDEsIlxcY2FsSyhYLCBcXHJpbnN7bn17bn17RH1fMikiXSxbMSwxLCJcXGNhbEsoWSwgXFxyaW5ze259e259e0R9XzIpIl0sWzAsMSwiXFxjYWxLKGYsIFxccmluc3tufXtufXtEfV8xKSJdLFsyLDMsIlxcY2FsSyhmLCBcXHJpbnN7bn17bn17RH1fMikiLDJdLFswLDIsIlxcY2FsSyhYLCBcXHJpbnN7bn17bn17RH0pIiwyLHsic3R5bGUiOnsiaGVhZCI6eyJuYW1lIjoiZXBpIn19fV0sWzEsMywiXFxjYWxLKFksIFxccmluc3tufXtufXtEfSkiLDAseyJzdHlsZSI6eyJoZWFkIjp7Im5hbWUiOiJlcGkifX19XSxbMCw1LCIiLDEseyJsZXZlbCI6MSwic3R5bGUiOnsibmFtZSI6ImNvcm5lciJ9fV1d
		\[\begin{tikzcd}[ampersand replacement=\&, column sep=large]
			{\calK(X, \rins{m}{[[m]]}{D}_1)} \& {\calK(Y, \rins{m}{[[m]]}{D}_1)} \\
			{\calK(X, \rins{m}{[[m]]}{D}_2)} \& {\calK(Y, \rins{m}{[[m]]}{D}_2)}
			\arrow["{\calK(f, \rins{m}{[[m]]}{D}_1)}", from=1-1, to=1-2]
			\arrow["{\calK(X, \rins{m}{[[m]]}{D})}"', two heads, from=1-1, to=2-1]
			\arrow["{\calK(Y, \rins{m}{[[m]]}{D})}", two heads, from=1-2, to=2-2]
			\arrow[""{name=0, anchor=center, inner sep=0}, "{\calK(f, \rins{m}{[[m]]}{D}_2)}"', from=2-1, to=2-2]
			\arrow["\lrcorner"{anchor=center, pos=0.125}, draw=none, from=1-1, to=0]
		\end{tikzcd}\]
		is actually a morphism of lalis.
		
		Note that by the naturality of $\calK(f, \cdot)$, we have
		% https://q.uiver.app/#q=WzAsNCxbMCwwLCJcXGNhbEsoWCwgXFxyaW5ze219e219e0R9KSJdLFsxLDAsIlxcY2FsSyhZLCBcXHJpbnN7bX17bX17RH0pIl0sWzAsMSwiXFxjYWxLKFgsIEEpIl0sWzEsMSwiXFxjYWxLKFksIEEpIl0sWzAsMSwiXFxjYWxLKGYsIFxccmluc3ttfXttfXtEfSkiXSxbMiwzLCJcXGNhbEsoZiwgQSkiLDJdLFswLDIsIlxcY2FsSyhYLCBwKSIsMix7InN0eWxlIjp7ImhlYWQiOnsibmFtZSI6ImVwaSJ9fX1dLFsxLDMsIlxcY2FsSyhZLCBwKSIsMCx7InN0eWxlIjp7ImhlYWQiOnsibmFtZSI6ImVwaSJ9fX1dLFswLDUsIiIsMSx7ImxldmVsIjoxLCJzdHlsZSI6eyJuYW1lIjoiY29ybmVyIn19XV0=
		\[\begin{tikzcd}[ampersand replacement=\&, column sep=large]
			{\calK(X, \rins{m}{[[m]]}{D})} \& {\calK(Y, \rins{m}{[[m]]}{D})} \\
			{\calK(X, A)} \& {\calK(Y, A)}
			\arrow["{\calK(f, \rins{m}{[[m]]}{D})}", from=1-1, to=1-2]
			\arrow["{\calK(X, p)}"', two heads, from=1-1, to=2-1]
			\arrow["{\calK(Y, p)}", two heads, from=1-2, to=2-2]
			\arrow[""{name=0, anchor=center, inner sep=0}, "{\calK(f, A)}"', from=2-1, to=2-2]
			\arrow["\lrcorner"{anchor=center, pos=0.125}, draw=none, from=1-1, to=0]
		\end{tikzcd}\]
		and since $A$ is assumed to be a lali-isofibration of $\calK$, thus $\calK(f, A)$ is a morphism of lalis, and together since $\calK(X, p)$ is known to be a morphism of lalis, we conclude that $\calK(X, p) \cdot \calK(f, \rins{m}{[[m]]}{D}) = \calK(f, A) \cdot \calK(X, p)$ is a morphism of lalis. Now, recall that $\calK(X, p)$ reflects morphisms of lalis, which implies $\calK(f, \rins{m}{[[m]]}{D})$ is a morphism of lalis. Then, by \longref{Proposition}{pro:5.5}, we arrive at the conclusion that $\rins{m}{[[m]]}{D}$ is a lali-isofibration of $\calK$, and that $p$ is a morphism of lalis.
		
		Finally, we show that $p$ reflects morphisms of lalis, and is an isofibration.
		
		By \longref{Proposition}{pro:Lali_qCat}, we know that $\calK(X, p)$ reflects morphisms of lalis, in particular, for a lali-isofibration $C \colon C_1 \twoheadrightarrow C_2$ of $\calK$ and a $0$-arrow $q \colon C \to \rins{m}{[[m]]}{D}$ in $\calK^{\isof}$, we have that $\calK(X, p) \cdot \calK(X, q)$ is a morphism of lalis in $\qCat$ if and only if $\calK(X, q)$ is one in $\qCat$. That means, $p \cdot q$ is a morphism of lalis if and only if $q$ is a morphism of lalis. Lastly, $p$ is an isofibration of $\calK^{\isof}$, because it is the pullback of the isofibration $B^\partial$, and isofibrations are pullback-stable in any $\infty$-cosmos.
	\end{proof}
	
	\begin{coroll}
		\label{cor:EM_Lali}
		The enhanced simplicial category $\Lali(\calK)$ of lali-isofibrations in an $\infty$-cosmos $\calK$ admits Eilenberg-Moore objects over loose monads, and that the forgetful functor is a morphism of lalis and reflects morphisms of lalis.
		
		In this case, by viewing $\calK^{\isof}$ as a chordate enhanced simplicial category, the inclusion
		$$\Lali(\calK) \hookrightarrow \calK^{\isof}_\chor$$
		preserves Eilenberg-Moore objects over loose monads.
	\end{coroll}
	
	\begin{proof}
		This follows directly from \longref{Proposition}{pro:tight_lim},  \longref{Theorem}{thm:Lali(K)}, and \longref{Lemma}{lem:EM}.
	\end{proof}
	
	\subsection{\texorpdfstring{$\Rali(\calK)$}{Rali(K)} }
	We first note that the enhanced simplicial category $\Rali(\calK)$ of rali-isofibrations admit the most basic kind of limits.
	
	\begin{pro}
		\label{pro:tight_lim_Rali}
		Tight cosmological limits exist in the enhanced simplicial category $\Rali(\calK)$ of rali-isofibrations of an $\infty$-cosmos $\calK$, and are preserved by the inclusion
		$$\Rali(\calK) \hookrightarrow \calK^{\isof}_\chor.$$
	\end{pro}
	
	\begin{proof}
		By considering rali-isofibrations, counits, and morphisms of rali-isofibrations, instead, in the proof of \longref{Proposition}{pro:tight_lim}, we obtain a proof for this proposition.
	\end{proof}

	The completeness results above have their dual counterparts for $\Rali(\calK)$, which directly follow from the relationship between lali-isofibrations and rali-isofibrations.
	
	\begin{lemma}
		\label{lem:Rali_Lali}
		$\Lali(\calK) = \Rali(\calK^\co)^\co$.
	\end{lemma}
	
	\begin{proof}
		In \cite[Corollary 5.8]{BL:2023}, it is seen that $\calLali(\calK) = \calRali(\calK^\co)^\co$ as an $\infty$-cosmos. Now, since a loose $0$-arrow is simply a morphism of isofibrations, the equality extends to the enhanced setting.
	\end{proof}

	\begin{coroll}
		\label{cor:Lali_Rali_co}
		For an $\infty$-cosmos $\calK$, $\Lali(\calK)$ admits terminally rigged $n$-inserters if and only if $\Rali(\calK^\co)$ admits initially rigged $n$-inserters.
		
		In addition,  the inclusion
		$$\Lali(\calK) \hookrightarrow \calK^{\isof}_\chor$$
		preserves terminally rigged $n$-inserters precisely if the inclusion
		$$\Rali(\calK) \hookrightarrow \calK^{\isof}_\chor$$
		preserves initially rigged $n$-inserters.
	\end{coroll}
	
	\begin{proof}
		By \longref{Lemma}{lem:Rali_Lali}, we have $\Lali(\calK) = \Rali(\calK^\co)^\co$, so by \longref{Proposition}{pro:terminal_initial}, this is equivalent to say that $\Rali(\calK^\co)$ admits initially rigged $n$-inserters.
		
		The statement for the inclusions follows immediately.
	\end{proof}
	
	We then establish the following dual statement to \longref{Theorem}{thm:Lali(K)}.
	
	\begin{theorem}
		\label{thm:Rali(K)}
		For any $\infty$-cosmos $\calK$, the enhanced simplicial category $\Rali(\calK)$ of rali-isofibrations of $\calK$ admits initially rigged $m$-inserters, and that the projection is an isofibration of $\Rali(\calK)$.
		
		In this case, the inclusion
		$$\Rali(\calK) \hookrightarrow \calK^{\isof}_\chor$$
		preserves initially rigged $n$-inserters.
	\end{theorem}
	
	\begin{proof}
		Note that $\calK^\co$ is also an $\infty$-cosmos. Hence, by \longref{Theorem}{thm:Lali(K)}, we know that $\Lali(\calK^\co)$ admits terminally rigged $m$-inserters. According to \longref{Corollary}{cor:Lali_Rali_co}, we conclude that $\Rali(\calK) = \Lali((\calK^\co)^\co)$ has initially rigged $m$-inserters.
		
		Lastly, an isofibration of $\Lali(\calK^\co)$ is an isofibration of $\calK^{\isof}$ between lali-isofibrations of $\calK^\co$, which is equivalently an isofibration of $\calK^{\isof}$ between rali-isofibrations of $\calK$, as described in the proof of \longref{Lemma}{lem:Rali_Lali}.
	\end{proof}
	
	\begin{coroll}
		$\Rali(\calK)$ admits coEilenberg-Moore objects over loose comonads, and that the forgetful functor is a morphism of ralis and reflects morphisms of ralis.
		
		In this case, the inclusion
		$$\Rali(\calK) \hookrightarrow \calK^{\isof}_\chor$$
		preserves coEilenberg-Moore objects over loose comonads.
	\end{coroll}
	
	\begin{proof}
		This follows immediately from \longref{Corollary}{cor:EM_Lali} and 
		\longref{Proposition}{pro:terminal_initial}.
	\end{proof}
	
	\subsection{From \texorpdfstring{$\Lali(\calK)$}{Lali(K)} or \texorpdfstring{$\Rali(\calK)$}{Rali(K)} to other enhanced simplicial categories}
	From the above results for $\Lali(\calK)$ or $\Rali(\calK)$, we are able to derive completeness results for some other enhanced simplicial categories. All these enhanced simplicial categories share one similarity, that they are indeed obtained as a pullback of $\Lali(\calK)$ or $\Rali(\calK)$ along an isofibration of $\FDeltaCat$.
	
	We recall the notion of isofibrations in enriched category theory.
	
	\begin{defi}
		Let $\V$ be a symmetric monoidal category. Let $\mathcal{E}$ and $\mathcal{B}$ be $\V$-categories. A $\V$-functor $p \colon \mathcal{E} \to \mathcal{B}$ is an \emph{isofibration} precisely when for any object $e$ of $\mathcal{E}$, if there is an isomorphism $p(e) \xrightarrow{\phi} b$ in $\mathcal{B}$, then there exists an isomorphism $e \xrightarrow{\psi} \tilde{b}$ in $E$ such that $p(\psi) = \phi$.
	\end{defi}
	
	\begin{lemma}
		Isofibrations between $\V$-categories are pullback-stable.
	\end{lemma}
	
	\begin{proof}
		Consider a pullback diagram
		% https://q.uiver.app/#q=WzAsNCxbMCwwLCJcXG1hdGhjYWx7UH0iXSxbMSwwLCJcXG1hdGhjYWx7Qn0iXSxbMCwxLCJcXG1hdGhjYWx7QX0iXSxbMSwxLCJcXG1hdGhjYWx7Q30iXSxbMCwxLCJwX1xcbWF0aGNhbHtCfSJdLFsyLDMsIkYiLDJdLFsxLDMsIkciLDAseyJzdHlsZSI6eyJoZWFkIjp7Im5hbWUiOiJlcGkifX19XSxbMCwyLCJwX1xcbWF0aGNhbHtBfSIsMl0sWzAsMywiIiwwLHsic3R5bGUiOnsibmFtZSI6ImNvcm5lciJ9fV1d
		\[\begin{tikzcd}[ampersand replacement=\&]
			{\mathcal{P}} \& {\mathcal{B}} \\
			{\mathcal{A}} \& {\mathcal{C}}
			\arrow["{p_\mathcal{B}}", from=1-1, to=1-2]
			\arrow["{p_\mathcal{A}}"', from=1-1, to=2-1]
			\arrow["\lrcorner"{anchor=center, pos=0.125}, draw=none, from=1-1, to=2-2]
			\arrow["G", two heads, from=1-2, to=2-2]
			\arrow["F"', from=2-1, to=2-2]
		\end{tikzcd}\]
		in $\VCat$, where $G$ is an isofibration. 
		
		Let $Q$ be an object of $\mathcal{P}$, and $\phi \colon P_\twoA(Q) \xrightarrow{\cong} A$ be an isomorphism in $\twoA$. Then, $F(\phi) \colon FP_\twoA(Q) \xrightarrow{\cong} F(A)$ is an isomorphism in $\twoC$. Now, since $G$ is an isofibration, and that $FP_\twoA(Q) = GP_\twoB(Q)$, there exists an isomorphism $\theta \colon P_\twoB(Q) \xrightarrow{\cong} \widetilde{F(A)}$ in $\twoB$, where $G(\theta) = F(\phi)$.
		
		Then, the map $(\phi, \theta) \colon (P_\twoA(Q), P_\twoB(B)) \xrightarrow{\cong} (A, \widetilde{F(A)})$ is an isomorphism in $\mathcal{P}$, where $P_\twoA(\phi, \theta) = \phi$. 
	\end{proof}
	
	\begin{lemma}
		\label{lem:enriched_pullback}
		Let $\V$ be a symmetric monoidal category. Let $\twoA, \twoB,$ and $\twoC$ be $\V$-categories, and $F \colon \twoA \to \twoC$ and $G \colon \twoB \to \twoC$ be $\V$-functors, where $G$ is an isofibration. 
		
		Consider the pullback
		% https://q.uiver.app/#q=WzAsNCxbMCwwLCJcXG1hdGhjYWx7UH0iXSxbMSwwLCJcXG1hdGhjYWx7Qn0iXSxbMCwxLCJcXG1hdGhjYWx7QX0iXSxbMSwxLCJcXG1hdGhjYWx7Q30iXSxbMCwxLCJwX1xcbWF0aGNhbHtCfSJdLFsyLDMsIkYiLDJdLFsxLDMsIkciLDAseyJzdHlsZSI6eyJoZWFkIjp7Im5hbWUiOiJlcGkifX19XSxbMCwyLCJwX1xcbWF0aGNhbHtBfSIsMix7InN0eWxlIjp7ImhlYWQiOnsibmFtZSI6ImVwaSJ9fX1dLFswLDMsIiIsMCx7InN0eWxlIjp7Im5hbWUiOiJjb3JuZXIifX1dXQ==
		\[\begin{tikzcd}[ampersand replacement=\&]
			{\mathcal{P}} \& {\mathcal{B}} \\
			{\mathcal{A}} \& {\mathcal{C}}
			\arrow["{p_\mathcal{B}}", from=1-1, to=1-2]
			\arrow["{p_\mathcal{A}}"', two heads, from=1-1, to=2-1]
			\arrow["\lrcorner"{anchor=center, pos=0.125}, draw=none, from=1-1, to=2-2]
			\arrow["G", two heads, from=1-2, to=2-2]
			\arrow["F"', from=2-1, to=2-2]
		\end{tikzcd}\]
		of $F$ along $G$. Let $W \colon \mathcal{J} \to \V$ be a $\V$-weight, and $H \colon \mathcal{J} \to \mathcal{P}$ be a $\V$-functor. Suppose that $\twoA$ admits the $\V$-limit $\{W, P_\twoA H\}$ and $F$ preserves it, whereas $\twoB$ admits the $\V$-limit $\{W, P_\twoB H\}$ and $G$ preserves it. Then, the pullback $\mathcal{P}$ admits the $\V$-limit $\{W, H\}$, which is preserved by the projections.
	\end{lemma}
	
	\begin{proof}
		Since $F$ and $G$ preserve $\{W, P_\twoA H\}$ and $\{W, P_\twoB H\}$, respectively, we have
		$$F(\{W, P_\twoA H\}) \cong \{W, F P_\twoA H\} = \{W, G P_\twoB H\} \cong G(\{W, P_\twoB H\})$$
		in $\twoC$.
		
		We proceed to construct $\{W, H\}$ in $\mathcal{P}$.
		
		Since $G$ is an isofibration, the above isomorphism $\phi \colon G(\{W, P_\twoB H\}) \xrightarrow{\cong} F(\{W, P_\twoA H\})$ can be lifted to an isomorphism $\psi \colon \{W, P_\twoB H\} \xrightarrow{\cong} L$ in $\twoB$, where $G(L) = F(\{W, P_\twoA H\})$. 
		
		It remains to verify that $(\{W, P_\twoA H\}, L)$ shares the universal property of $\{W, H\}$, i.e., for any $(X, Y) \in \mathcal{P}$,
		$$\mathcal{P}((X, Y), (\{W, P_\twoA H\}, L)) \cong [\mathcal{J}, \V](W, \mathcal{P}((X, Y), H)).$$ 
		Note that in general, we have
		$$\mathcal{P}((X, Y), (M, N)) \cong \twoA(X, M) \times_{\twoC(F(X), F(H)) = \twoC(G(Y), G(N))} \twoB(Y, N)$$
		for any objects $(X, Y), (M, N)$ in $\mathcal{P}$, so, it amounts to ask for
		$$\twoA(X, \{W, P_\twoA H\}) \times_{\twoC(G(Y), G(L))} \twoB(Y, L) \cong [\mathcal{J}, \V](W, \twoA(X, P_\twoA H) \times_{\twoC(G(Y), G P_\twoB H)} \twoB(Y, P_\twoB H)).$$ Since the representables preserve limits, the right hand side is isomorphic to
		$$[\mathcal{J}, \V](W, \twoA(X, P_\twoA H)) \times_{[\mathcal{J}, \V](W, \twoC(G(Y), GP_\twoB H))} [\mathcal{J}, \V](W, \twoB(Y, P_\twoB H)).$$ And, by assumption, we have
		$$
		\begin{aligned}
			\twoA(X, \{W, P_\twoA H\}) &\cong [\mathcal{J}, \V](W, \twoA(X, P_\twoA H)),  
			\\
			\twoB(Y, L) &\cong [\mathcal{J}, \V](W, \twoB(Y, P_\twoB H)),
			\\
			\twoC(G(Y), G(L)) 
			&\cong [\mathcal{J}, \V](W, \twoC(G(Y), G P_\twoB H)),
		\end{aligned}
		$$
		therefore, the desired characterising isomorphism is satisfied.
		
		Finally, it is clear that the projections $P_\twoA$ and $P_\twoB$ both preserve $(\{W, P_\twoA H\}, L)$, by construction.
	\end{proof}
	
	\begin{rk}
		In the above lemma, we specifically assume that $\twoA$ and $\twoB$ admit particular kinds of limits, instead of general $\V$-limits. This allows us to apply the lemma to limits that involve isofibrations, which are not encoded by general $\V$-weights or $\V$-functors.
	\end{rk}
	
	Very often, the interesting enhanced simplicial categories are pullbacks along a certain embedding into an $\infty$-cosmos, which is actually an isofibration of enhanced simplicial categories.
	
	\begin{defi}[{\cite[Chapter 6.3]{book:RV:2022}}]
		Let $\mathcal{L}$ be a simplicial sub-category of an $\infty$-cosmos $\calK$, which is full on $n$-arrows for $n > 0$. Then $\mathcal{L}$ is said to be \emph{replete in $\calK$} if and only if
		\begin{enumerate}
			\item[$\bullet$] every $\infty$-category in $\calK$ is equivalent to an object in $\mathcal{L}$ belongs to $\mathcal{L}$;
			\item[$\bullet$] any equivalence in $\calK$ between objects in $\mathcal{L}$ belongs to $\mathcal{L}$;
			\item[$\bullet$] any $0$-arrow in $\calK$ that is isomorphic in $\calK$ to an $0$-arrow in $\mathcal{L}$ belongs to $\mathcal{L}$.
		\end{enumerate}	
		
		If furthermore $\mathcal{L}$ defines an $\infty$-cosmos with isofibrations, equivalences, trivial fibrations, and simplicial limits created by the inclusion $\mathcal{L} \hookrightarrow \calK$, then $\mathcal{L}$ is called a \emph{cosmologically embedded $\infty$-cosmos of $\calK$}, and the inclusion is called a \emph{cosmological embedding}.
	\end{defi}
	
	\begin{nota}
		For a cosmological embedding, we denote by $\mathcal{L} \hooktwoheadrightarrow \calK$.
		%		
		%		We will be using the same notation when a cosmological embedding extends to an isofibration of enhanced simplicial categories.
		%		
		%		{\hl better use another notation for enhanced}
	\end{nota}
	
	\begin{lemma}
		\label{lem:inclusion_isof}
		Let $\bbA$ be an enhanced simplicial category, whose tight part $\twoA_\tau$ is an $\infty$-cosmos. Let $\twoB$ be an $\infty$-cosmos, which can be seen as a chordate enhanced simplicial category. Suppose that $\twoA_\lambda \hookrightarrow \twoB$ is a full simplicial sub-category, and that $\twoA_\tau \hooktwoheadrightarrow \twoB$ is a cosmological embedding, then, the inclusion $\twoA_\lambda \hookrightarrow \twoB$ extends to an enhanced simplicial functor
		$$U \colon \bbA \twoheadrightarrow \twoB_\chor,$$
		which is an isofibration of $\FDeltaCat$.
	\end{lemma}
	
	\begin{proof}
		When we view $\twoB$ as a chordate enhanced simplicial category, it is clear that $\twoA_\lambda \hookrightarrow \twoB$ becomes an enhanced simplicial functor $U \colon \bbA \hooktwoheadrightarrow \twoB_\chor$, because tightness is always preserved here.
		
		Now, let $a$ be an object of $\bbA$. Assume that there is an isomorphism
		$$U(a) = a \xrightarrow{\phi} b$$ 
		in $\twoB_\chor$, where $b$ is an object of $\twoB$. Since $\twoA_\tau$ is cosmologically embedded into $\twoB$, which means we have $b \in \ob\bbA$, and that $\phi$ actually lies in $\twoA_\tau$.
	\end{proof}
	
	\begin{lemma}
		\label{lem:create_tight}
		Let $\bbA$ and $\bbB$ be enhanced simplicial categories, where $\twoA_\tau$ and $\twoB_\tau$ are both $\infty$-cosmoi. Suppose that 
		$$\bbA \hookrightarrow \bbB$$
		is a full sub-$\F_\Delta$-category, which means both inclusions $\twoA_\tau \hookrightarrow \twoB_\tau$ and $\twoA_\lambda \hookrightarrow \twoB_\lambda$ are full. If the inclusion
		$$\twoA_\tau \hookrightarrow \twoB_\tau$$
		creates cosmological limits, and $\bbB$ admits tight cosmological limits, then the inclusion
		$$\bbA \hookrightarrow \bbB$$
		creates tight cosmological limits.
	\end{lemma}
	
	\begin{proof}
		Certainly, the $\infty$-cosmos $\twoA_\tau$ admits cosmological limits. It suffices to show that the limit projections jointly reflect tightness.
		
		Since the inclusion of $\bbA$ into $\bbB$ is full, a $0$-arrow in $\bbA$ is tight if and only if it is a tight $0$-arrow in $\bbB$. Now, because the limit projections of a cosmological limit in $\bbB$ jointly reflect limits, their preimages in $\bbA$ also do.
	\end{proof}
	
	With these machineries, we are ready to show that terminally rigged $n$-inserters and tight cosmological limits exist in several enhanced simplicial categories.

	\subsubsection{\texorpdfstring{$\JLim(\calK)$}{J-Lim(K)}}
	Recall from \longref{Example}{eg:JLimK} that for an $\infty$-cosmos $\calK$, $\JLim(\calK)$ is the enhanced simplicial category of $\infty$-categories in $\calK$ which admits $J$-shaped limits for a simplicial set $J$.
	
	\begin{theorem}
		\label{thm:trio_JLim}
		The enhanced simplicial category $\JLim(\calK)$ admits tight cosmological limits, and terminally rigged $n$-inserters, and that the projection of a terminally rigged $n$-inserter is always a tight isofibration.
		
		Consequently, $\JLim(\calK)$ admits Eilenberg-Moore objects over monads which do not preserve limits, and that the forgetful functor preserves and reflects $J$-shaped limits.
		
		These $\F_\Delta$-weighted limits are all preserved by the inclusion
		$$\JLim(\calK) \hookrightarrow \calK_\chor.$$
	\end{theorem}
	
	\begin{proof}
		As described in \longref{Example}{eg:JLimK} and the proof of \cite[Proposition 6.3.13]{book:RV:2022}, the $\infty$-cosmos  $\calJLim(\calK)$ of $\infty$-categories of $\calK$ admitting $J$-shaped limits can be formed as the pullback
		% https://q.uiver.app/#q=WzAsNCxbMCwwLCJcXGNhbEpMaW0oXFxjYWxLKSJdLFsxLDAsIlxcY2FsTGFsaShcXGNhbEspIl0sWzAsMSwiXFxjYWxLIl0sWzEsMSwiXFxjYWxLXntcXGlzb2Z9Il0sWzAsMV0sWzEsMywiVSIsMCx7InN0eWxlIjp7InRhaWwiOnsibmFtZSI6Imhvb2siLCJzaWRlIjoidG9wIn0sImhlYWQiOnsibmFtZSI6ImVwaSJ9fX1dLFswLDIsIiIsMCx7InN0eWxlIjp7InRhaWwiOnsibmFtZSI6Imhvb2siLCJzaWRlIjoidG9wIn0sImhlYWQiOnsibmFtZSI6ImVwaSJ9fX1dLFsyLDMsIkZfe0pfXFx0cmlhbmdsZWxlZnR9IiwyXSxbMCwzLCIiLDAseyJzdHlsZSI6eyJuYW1lIjoiY29ybmVyIn19XV0=
		\[\begin{tikzcd}[ampersand replacement=\&]
			{\calJLim(\calK)} \& {\calLali(\calK)} \\
			\calK \& {\calK^{\isof}}
			\arrow[from=1-1, to=1-2]
			\arrow[hook, two heads, from=1-1, to=2-1]
			\arrow["\lrcorner"{anchor=center, pos=0.125}, draw=none, from=1-1, to=2-2]
			\arrow["U", hook, two heads, from=1-2, to=2-2]
			\arrow["{F_{J_\triangleleft}}"', from=2-1, to=2-2]
		\end{tikzcd},\]
		where the inclusion $U \colon \calLali(\calK) \hooktwoheadrightarrow \calK^{\isof}$ is a cosmological embedding, and that
		$$
		\begin{aligned}
			{F_{J_\triangleleft}} \colon \calK &\to \calK^{\isof}
			\\
			A &\mapsto (A^{\triangleleft} \xtwoheadrightarrow{\res_A} A^J)
		\end{aligned}
		$$
		is a cosmological functor sending each $\infty$-category $A$ to the restriction $\res_A \colon A^{J_\triangleleft} \twoheadrightarrow A^J$ described in \longref{Example}{eg:JLimK}, and each $0$-arrow $A \to B$ to the commutative square in \longref{Diagram}{diag:res}. The higher dimensional arrows are mapped accordingly.
		
		Note that the loose part $\calJLim(\calK)_\lambda$ of $\JLim(\calK)$ is actually a full simplicial sub-category of $\calK$, and similarly, the loose part $\calLali(\calK)_\lambda$ of $\Lali(\calK)$ is a full simplicial sub-category of $\Lali(\calK)$, as a consequence, we have a pullback diagram
		% https://q.uiver.app/#q=WzAsNCxbMCwwLCJcXGNhbEpMaW0oXFxjYWxLKV9cXGxhbWJkYSJdLFsxLDAsIlxcY2FsTGFsaShcXGNhbEspX1xcbGFtYmRhIl0sWzAsMSwiXFxjYWxLIl0sWzEsMSwiXFxjYWxLXntcXGlzb2Z9Il0sWzAsMV0sWzEsMywiIiwwLHsic3R5bGUiOnsidGFpbCI6eyJuYW1lIjoiaG9vayIsInNpZGUiOiJ0b3AifX19XSxbMCwyLCIiLDAseyJzdHlsZSI6eyJ0YWlsIjp7Im5hbWUiOiJob29rIiwic2lkZSI6InRvcCJ9fX1dLFsyLDMsIkZfe0pfXFx0cmlhbmdsZWxlZnR9IiwyXSxbMCwzLCIiLDAseyJzdHlsZSI6eyJuYW1lIjoiY29ybmVyIn19XV0=
		\[\begin{tikzcd}[ampersand replacement=\&]
			{\calJLim(\calK)_\lambda} \& {\calLali(\calK)_\lambda} \\
			\calK \& {\calK^{\isof}}
			\arrow[from=1-1, to=1-2]
			\arrow[hook, from=1-1, to=2-1]
			\arrow["\lrcorner"{anchor=center, pos=0.125}, draw=none, from=1-1, to=2-2]
			\arrow[hook, from=1-2, to=2-2]
			\arrow["{F_{J_\triangleleft}}"', from=2-1, to=2-2]
		\end{tikzcd}.\]
		
		Also, by taking $\twoA_\tau = \calLali(\calK)$ in \longref{Lemma}{lem:inclusion_isof}, we deduce that the inclusion $U \colon \calLali(\calK) \hooktwoheadrightarrow \calK^{\isof}$ extends to the inclusion
		$$U \colon \Lali(\calK) \twoheadrightarrow \calK^{\isof}_\chor,$$
		which is an isofibration of enhanced simplicial categories. 
		
		Altogether, the above pullback diagrams combine to a pullback
		% https://q.uiver.app/#q=WzAsNCxbMCwwLCJcXEpMaW0oXFxjYWxLKSJdLFsxLDAsIlxcTGFsaShcXGNhbEspIl0sWzAsMSwiXFxjYWxLIl0sWzEsMSwiXFxjYWxLXntcXGlzb2Z9Il0sWzAsMV0sWzEsMywiVSIsMCx7InN0eWxlIjp7InRhaWwiOnsibmFtZSI6Imhvb2siLCJzaWRlIjoidG9wIn0sImhlYWQiOnsibmFtZSI6ImVwaSJ9fX1dLFswLDIsIiIsMCx7InN0eWxlIjp7InRhaWwiOnsibmFtZSI6Imhvb2siLCJzaWRlIjoidG9wIn0sImhlYWQiOnsibmFtZSI6ImVwaSJ9fX1dLFsyLDMsIkZfe0pfXFx0cmlhbmdsZWxlZnR9IiwyXSxbMCwzLCIiLDAseyJzdHlsZSI6eyJuYW1lIjoiY29ybmVyIn19XV0=
		\begin{equation}
			\label{diag:JLim_pb}
			\begin{tikzcd}[ampersand replacement=\&]
				{\JLim(\calK)} \& {\Lali(\calK)} \\
				\calK_\chor \& {\calK^{\isof}_\chor}
				\arrow[from=1-1, to=1-2]
				\arrow[ two heads, from=1-1, to=2-1]
				\arrow["\lrcorner"{anchor=center, pos=0.125}, draw=none, from=1-1, to=2-2]
				\arrow["U", two heads, from=1-2, to=2-2]
				\arrow["{F_{J_\triangleleft}}"', from=2-1, to=2-2]
			\end{tikzcd}
		\end{equation}
		in $\FDeltaCat$. By \longref{Lemma}{lem:enriched_pullback}, it remains to verify that $U$ and $F_{J_\triangleleft}$ preserve tight cosmological limits and terminally rigged $n$-inserters.
		
		By \longref{Lemma}{lem:create_tight}, $U$ preserves tight cosmological limits; moreover, $F_{J_\triangleleft}$ is cosmological.
		
		Next, according to \longref{Theorem}{thm:Lali(K)}, the inclusion $U \colon \Lali(\calK) \twoheadrightarrow \calK^{\isof}_\chor$ automatically preserves terminally rigged $n$-inserters. Moving on, since $F_{J_\triangleleft}$ is a cosmological functor, by \cite[Proposition 6.2.8]{book:RV:2022}, it preserves any flexible weighted limits. Following \longref{Remark}{rk:term_flexible}, we conclude that $F_{J_\triangleleft}$ preserves terminally rigged $n$-inserters.
		
		By \cite[Proposition 6.3.12]{book:RV:2022}, cosmological embeddings are pullback-stable. Therefore, the projection of a terminally rigged $n$-inserter in $\JLim(\calK)$, which is created by one in $\calK$, is an isofibration.
		
		Finally, by \longref{Lemma}{lem:EM}, $\JLim(\calK)$ has Eilenberg-Moore objects over loose monads, and that the forgetful functor preserves and reflects $J$-shaped limits.
		%		
		%		Finally, since $\JLim(\calK)$ has terminally rigged $n$-inserters whose projections are tight isofibrations, simplicial powers by inchordate enhanced simplicial sets, and limits of a countable chain of tight isofibrations, according to \longref{Lemma}{lem:EM}, $\JLim(\calK)$ has Eilenberg-Moore objects over loose monads, and that the forgetful functor preserves and reflects $J$-shaped limits.
	\end{proof}
	
	\subsubsection{\texorpdfstring{$\JColim(\calK)$}{J-Colim(K)}}
	Next, we consider the case for $\infty$-categories of an $\infty$-cosmos $\calK$ that admit $J$-shaped \emph{colimits} for a simplicial set $J$.
	
	Recall from \longref{Example}{eg:JLimK} that for an $\infty$-cosmos $\calK$, $\JColim(\calK)$ is the enhanced simplicial category of $\infty$-categories in $\calK$ which admits $J$-shaped colimits for a simplicial set $J$. The loose $0$-arrows are $\infty$-functors in $\calK$, whereas the tight $0$-arrows are $\infty$-functors that preserve $J$-shaped colimits.
	
	\begin{theorem}
		\label{thm:trio_JColim}
		The enhanced simplicial category $\JColim(\calK)$ admits tight cosmological limits, and initially rigged $n$-inserters, and that the projection of an initially rigged $n$-inserter is always a tight isofibration.
		
		Consequently, $\JColim(\calK)$ admits coEilenberg-Moore objects over comonads which do not preserve colimits, and that the forgetful functor preserves and reflects $J$-shaped colimits.
		
		These $\F_\Delta$-weighted limits are all preserved by the inclusion
		$$\JColim(\calK) \hookrightarrow \calK_\chor.$$
	\end{theorem}
	
	\begin{proof}
		Similar to the case of $\JLim(\calK)$ in the proof of \longref{Theorem}{thm:trio_JLim}, $\JColim(\calK)$ is a pullback of $\Rali(\calK)$. With similar arguments, we arrive at the conclusion for $\JColim(\calK)$.
	\end{proof}

	\subsubsection{\texorpdfstring{$\La(\calK)$}{La(K)}}
	\label{sec:La}
	Now, we introduce the $\infty$-cosmos $\calLa(\calK)$ for any arbitrary $\infty$-cosmos $\calK$.
	
	\begin{nota}
		Let $\calK$ be an $\infty$-cosmos. Denote by $\calLa(\calK)$ the simplicial sub-category of $\calK^{\isof}$ with objects given by isofibrations of $\calK$ which are \emph{left adjoints}, i.e. it is a left adjoint in the homotopy $2$-category $h\calK$. A morphism of isofibrations
		% https://q.uiver.app/#q=WzAsNCxbMCwwLCJFXzEiXSxbMCwxLCJCXzEiXSxbMSwwLCJFXzIiXSxbMSwxLCJCXzIiXSxbMCwxLCJwXzEiLDIseyJzdHlsZSI6eyJoZWFkIjp7Im5hbWUiOiJlcGkifX19XSxbMiwzLCJwXzIiLDAseyJzdHlsZSI6eyJoZWFkIjp7Im5hbWUiOiJlcGkifX19XSxbMCwyLCJlIl0sWzEsMywiYiIsMV1d
		\[\begin{tikzcd}[ampersand replacement=\&]
			{E_1} \& {E_2} \\
			{B_1} \& {B_2}
			\arrow["e", from=1-1, to=1-2]
			\arrow["{p_1}"', two heads, from=1-1, to=2-1]
			\arrow["{p_2}", two heads, from=1-2, to=2-2]
			\arrow["b"', from=2-1, to=2-2]
		\end{tikzcd}\]
		is a \emph{morphism of left adjoints} precisely if its mate is an isomorphism, i.e. the composite
		% https://q.uiver.app/#q=WzAsNixbMSwwLCJFXzEiXSxbMSwxLCJCXzEiXSxbMiwwLCJFXzIiXSxbMiwxLCJCXzIiXSxbMCwxLCJCXzEiXSxbMywwLCJFXzIiXSxbMCwxLCJwXzEiLDIseyJzdHlsZSI6eyJoZWFkIjp7Im5hbWUiOiJlcGkifX19XSxbMiwzLCJwXzIiLDAseyJzdHlsZSI6eyJoZWFkIjp7Im5hbWUiOiJlcGkifX19XSxbMCwyLCJlIl0sWzEsMywiYiIsMV0sWzQsMSwiIiwxLHsic3R5bGUiOnsiaGVhZCI6eyJuYW1lIjoibm9uZSJ9fX1dLFsyLDUsIiIsMSx7InN0eWxlIjp7ImhlYWQiOnsibmFtZSI6Im5vbmUifX19XSxbNCwwLCJyXzEiXSxbMyw1LCJyXzIiLDJdLFsxMiwxMCwiXFxlcHNpbG9uXzEiLDAseyJzaG9ydGVuIjp7InNvdXJjZSI6MjAsInRhcmdldCI6MjB9fV0sWzExLDEzLCJcXGV0YV8yIiwyLHsic2hvcnRlbiI6eyJzb3VyY2UiOjIwLCJ0YXJnZXQiOjIwfX1dXQ==
		\[\begin{tikzcd}[ampersand replacement=\&]
			\& {E_1} \& {E_2} \& {E_2} \\
			{B_1} \& {B_1} \& {B_2}
			\arrow["e", from=1-2, to=1-3]
			\arrow["{p_1}"', two heads, from=1-2, to=2-2]
			\arrow[""{name=0, anchor=center, inner sep=0}, no head, equal, from=1-3, to=1-4]
			\arrow["{p_2}", two heads, from=1-3, to=2-3]
			\arrow[""{name=1, anchor=center, inner sep=0}, "{r_1}", from=2-1, to=1-2]
			\arrow[""{name=2, anchor=center, inner sep=0}, no head, equal, from=2-1, to=2-2]
			\arrow["b"', from=2-2, to=2-3]
			\arrow[""{name=3, anchor=center, inner sep=0}, "{r_2}"', from=2-3, to=1-4]
			\arrow["{\eta_2}"', shorten <=2pt, shorten >=2pt, Rightarrow, from=0, to=3]
			\arrow["{\epsilon_1}", shorten <=2pt, shorten >=2pt, Rightarrow, from=1, to=2]
		\end{tikzcd}\]
		is invertible as a $2$-morphism in $h\calK$, where $r_i$ denotes the right adjoint for $p_i$, and that $\epsilon_i$ and $\eta_i$ denote the corresponding counit and unit, respectively, for $i = 1, 2$.
		
		In other words, $\calLa(\calK) \subset \calK^{\isof}$ has objects given by left adjoints in $\calK$, $0$-arrows given by morphisms of left adjoints, and it is full on higher dimensional arrows.
	\end{nota}
	
	For convenience, we introduce a name.
	
	\begin{defi}
		Let $\calK$ be an $\infty$-cosmos. An isofibration of $\calK$ is said to  be a \emph{la-isofibration} precisely when it belongs to $\calLa(\calK)$. 
		
		A morphism of isofibrations is said to be \emph{a morphism of la-isofibrations} precisely when it is a $0$-arrow in $\calLa(\calK)$. 
	\end{defi}
	
	Indeed, the simplicial category $\calLa(\calK)$ is an $\infty$-cosmos. For this, we first establish a few supporting lemmata.
	
	\begin{lemma}
		\label{lem:la_as_lali}
		In the $2$-category $\Cat$ of categories, a functor $F \colon A \to B$ is a left adjoint if and only if 
		the canonical limit projection $p_B \colon F \downarrow B \to B$ from the lax limit $F \downarrow B$ of $F$ to $B$ is a lali.
	\end{lemma}
	
	\begin{proof}
		Suppose $F \dashv U$ is an adjunction with unit $H$ and counit $E$. Then,  the $1$-component 
		$$E_b \colon FUb \to b$$
		at an object $b \in \ob B$ is an object of $F \downarrow B$.
		
		We define a functor
		$$
		\begin{aligned}
			u \colon B &\to F \downarrow B
			\\
			b &\mapsto E_b \quad ,
		\end{aligned}
		$$
		and show that $p_B \dashv u$.
		
		We construct a pair of natural transformations
		\begin{center}
			% https://q.uiver.app/#q=WzAsMyxbMCwwLCJGXFxkb3duYXJyb3cgQiJdLFsyLDAsIkZcXGRvd25hcnJvdyBCIl0sWzEsMSwiQiJdLFswLDEsIiIsMix7InN0eWxlIjp7ImhlYWQiOnsibmFtZSI6Im5vbmUifX19XSxbMCwyLCJwX0IiLDJdLFsyLDEsInUiLDJdLFszLDIsIlxcZXRhIiwwLHsic2hvcnRlbiI6eyJzb3VyY2UiOjQwLCJ0YXJnZXQiOjIwfX1dXQ==
			\begin{tikzcd}[ampersand replacement=\&]
				{F\downarrow B} \&\& {F\downarrow B} \\
				\& B
				\arrow[""{name=0, anchor=center, inner sep=0}, equal, from=1-1, to=1-3]
				\arrow["{p_B}"', from=1-1, to=2-2]
				\arrow["u"', from=2-2, to=1-3]
				\arrow["\eta", shorten <=6pt, shorten >=3pt, Rightarrow, from=0, to=2-2]
			\end{tikzcd},
			% https://q.uiver.app/#q=WzAsMyxbMSwwLCJGXFxkb3duYXJyb3cgQiJdLFswLDEsIkIiXSxbMiwxLCJCIl0sWzEsMiwiIiwyLHsic3R5bGUiOnsiaGVhZCI6eyJuYW1lIjoibm9uZSJ9fX1dLFsxLDAsInUiXSxbMCwyLCJwX0IiXSxbMCwzLCJcXGVwc2lsb24iLDAseyJzaG9ydGVuIjp7InNvdXJjZSI6NDAsInRhcmdldCI6MjB9fV1d
			\begin{tikzcd}[ampersand replacement=\&]
				\& {F\downarrow B} \\
				B \&\& B
				\arrow["{p_B}", from=1-2, to=2-3]
				\arrow["u", from=2-1, to=1-2]
				\arrow[""{name=0, anchor=center, inner sep=0}, equal, from=2-1, to=2-3]
				\arrow["\epsilon", shorten <=6pt, shorten >=3pt, Rightarrow, from=1-2, to=0]
			\end{tikzcd}
		\end{center}
		as follows.	It is clear that $p_B u = 1_B$, so we can set $\epsilon = 1$. Besides, for any $Fa \xrightarrow{\phi} b$, there is a corresponding adjunct $a \xrightarrow{\phi^\#}$, which makes the triangle
		% https://q.uiver.app/#q=WzAsMyxbMCwwLCJGYSJdLFsyLDAsIkZVYiJdLFsxLDEsImIiXSxbMCwxLCJGXFxwaGleXFwjIl0sWzAsMiwiXFxwaGkiLDJdLFsxLDIsIkVfYiJdXQ==
		\[\begin{tikzcd}[ampersand replacement=\&]
			Fa \&\& FUb \\
			\& b
			\arrow["{F\phi^\#}", from=1-1, to=1-3]
			\arrow["\phi"', from=1-1, to=2-2]
			\arrow["{E_b}", from=1-3, to=2-2]
		\end{tikzcd}\]
		commute. Thus, we can construct a natural transformation $\eta \colon 1_{F\downarrow B} \Rightarrow up_B$, whose $1$-component at $Fa \xrightarrow{\phi} b$ is given by $F\phi^\#$.
		
		We shall verify the triangle identities. It is obvious that the composite natural transformation
		% https://q.uiver.app/#q=WzAsNCxbMCwwLCJGXFxkb3duYXJyb3cgIEIiXSxbMSwwLCJGXFxkb3duYXJyb3cgIEIiXSxbMCwxLCJCIl0sWzEsMSwiQiJdLFswLDEsIiIsMCx7InN0eWxlIjp7ImhlYWQiOnsibmFtZSI6Im5vbmUifX19XSxbMiwzLCIiLDAseyJzdHlsZSI6eyJoZWFkIjp7Im5hbWUiOiJub25lIn19fV0sWzAsMiwicF9CIiwyXSxbMSwzLCJwX0IiXSxbMiwxLCJ1IiwxXSxbNywzLCJcXGVwc2lsb24iLDIseyJvZmZzZXQiOjUsInNob3J0ZW4iOnsidGFyZ2V0IjoyMH0sInN0eWxlIjp7ImhlYWQiOnsibmFtZSI6Im5vbmUifX19XSxbMCw2LCJcXGV0YSIsMCx7Im9mZnNldCI6LTUsInNob3J0ZW4iOnsidGFyZ2V0IjoyMH19XV0=
		\[\begin{tikzcd}[ampersand replacement=\&]
			{F\downarrow  B} \& {F\downarrow  B} \\
			B \& B
			\arrow[equal, from=1-1, to=1-2]
			\arrow[""{name=0, anchor=center, inner sep=0}, "{p_B}"', from=1-1, to=2-1]
			\arrow[""{name=1, anchor=center, inner sep=0}, "{p_B}", from=1-2, to=2-2]
			\arrow["u"{description}, from=2-1, to=1-2]
			\arrow[equal, from=2-1, to=2-2]
			\arrow["\eta", shift left=5, shorten >=1pt, Rightarrow, from=1-1, to=0]
			\arrow["\epsilon"', shift right=5, shorten >=1pt, equals, from=1, to=2-2]
		\end{tikzcd}\]
		has $1$-component at $Fa \xrightarrow{\phi} b$ given by the identity at $b$. So, this is clearly an identity. Similarly, the composite natural transformation
		% https://q.uiver.app/#q=WzAsNCxbMCwwLCJCIl0sWzEsMCwiQiJdLFswLDEsIkZcXGRvd25hcnJvdyAgQiJdLFsxLDEsIkZcXGRvd25hcnJvdyAgQiJdLFswLDEsIiIsMCx7InN0eWxlIjp7ImhlYWQiOnsibmFtZSI6Im5vbmUifX19XSxbMiwzLCIiLDAseyJzdHlsZSI6eyJoZWFkIjp7Im5hbWUiOiJub25lIn19fV0sWzAsMiwidSIsMl0sWzEsMywidSJdLFsyLDEsInBfQiIsMV0sWzAsNiwiXFxlcHNpbG9uIiwwLHsib2Zmc2V0IjotNSwic2hvcnRlbiI6eyJ0YXJnZXQiOjIwfSwic3R5bGUiOnsiaGVhZCI6eyJuYW1lIjoibm9uZSJ9fX1dLFszLDcsIlxcZXRhIiwwLHsib2Zmc2V0IjotNSwic2hvcnRlbiI6eyJ0YXJnZXQiOjIwfX1dXQ==
		\[\begin{tikzcd}[ampersand replacement=\&]
			B \& B \\
			{F\downarrow  B} \& {F\downarrow  B}
			\arrow[no head, from=1-1, to=1-2]
			\arrow[""{name=0, anchor=center, inner sep=0}, "u"', from=1-1, to=2-1]
			\arrow[""{name=1, anchor=center, inner sep=0}, "u", from=1-2, to=2-2]
			\arrow["{p_B}"{description}, from=2-1, to=1-2]
			\arrow[no head, from=2-1, to=2-2]
			\arrow["\epsilon", shift left=5, shorten >=1pt, equals, from=1-1, to=0]
			\arrow["\eta", shift left=5, shorten >=1pt, Rightarrow, from=2-2, to=1]
		\end{tikzcd}\]
		has $1$-component at $b$ given by the identity at $E_b$, which means it is also an identity.
		
		Conversely, suppose $p_B$ is a lali, with right adjoint $u \colon B \to F\downarrow B$, unit $\eta$, and counit $\epsilon = 1$.
		
		We define a functor
		$$
		\begin{aligned}
			v \colon A &\to F \downarrow B
			\\
			a &\mapsto 1_{Fa} \quad , 
		\end{aligned}
		$$
		which gives us $p_B v (a) = Fa$.
		
		Define $U := p_A u$, where $p_A$ denotes the canonical projection $\F\downarrow B \to A$. We show that $F \dashv U$.
		
		Consider a natural transformation
		% https://q.uiver.app/#q=WzAsMyxbMSwwLCJBIl0sWzAsMSwiRlxcZG93bmFycm93IEIiXSxbMiwxLCJGXFxkb3duYXJyb3cgQiJdLFsxLDIsIiIsMix7InN0eWxlIjp7ImhlYWQiOnsibmFtZSI6Im5vbmUifX19XSxbMSwwLCJwX0EiXSxbMCwyLCJ2Il0sWzAsMywiXFxhbHBoYSIsMCx7InNob3J0ZW4iOnsic291cmNlIjo0MCwidGFyZ2V0IjoyMH19XV0=
		\[\begin{tikzcd}[ampersand replacement=\&]
			\& A \\
			{F\downarrow B} \&\& {F\downarrow B}
			\arrow["v", from=1-2, to=2-3]
			\arrow["{p_A}", from=2-1, to=1-2]
			\arrow[""{name=0, anchor=center, inner sep=0}, equal, from=2-1, to=2-3]
			\arrow["\alpha", shorten <=6pt, shorten >=3pt, Rightarrow, from=1-2, to=0]
		\end{tikzcd}\]
		whose $1$-component $\alpha_\phi$ at $Fa \xrightarrow{\phi} b$ is $(1_{Fa}, \phi)$, as in
		% https://q.uiver.app/#q=WzAsNCxbMCwwLCJGYSJdLFsxLDAsIkZhIl0sWzAsMSwiRmEiXSxbMSwxLCJiIl0sWzAsMSwiIiwyLHsic3R5bGUiOnsiaGVhZCI6eyJuYW1lIjoibm9uZSJ9fX1dLFswLDIsIiIsMCx7InN0eWxlIjp7ImhlYWQiOnsibmFtZSI6Im5vbmUifX19XSxbMSwzLCJcXHBoaSJdLFsyLDMsIlxccGhpIiwyXV0=
		\[\begin{tikzcd}[ampersand replacement=\&]
			Fa \& Fa \\
			Fa \& b
			\arrow[equal, from=1-1, to=1-2]
			\arrow[equal, from=1-1, to=2-1]
			\arrow["\phi", from=1-2, to=2-2]
			\arrow["\phi"', from=2-1, to=2-2]
		\end{tikzcd}.\]
		We set
		\[
		% https://q.uiver.app/#q=WzAsMyxbMCwwLCJBIl0sWzIsMCwiQSJdLFsxLDEsIkIiXSxbMCwxLCIiLDIseyJzdHlsZSI6eyJoZWFkIjp7Im5hbWUiOiJub25lIn19fV0sWzAsMiwiRiIsMl0sWzIsMSwiVSIsMl0sWzMsMiwiSCIsMCx7InNob3J0ZW4iOnsic291cmNlIjo0MCwidGFyZ2V0IjoyMH19XV0=
		\begin{tikzcd}[ampersand replacement=\&]
			A \&\& A \\
			\& B
			\arrow[""{name=0, anchor=center, inner sep=0}, equal, from=1-1, to=1-3]
			\arrow["F"', from=1-1, to=2-2]
			\arrow["U"', from=2-2, to=1-3]
			\arrow["H", shorten <=6pt, shorten >=3pt, Rightarrow, from=0, to=2-2]
		\end{tikzcd}
		\quad := \quad
		% https://q.uiver.app/#q=WzAsNSxbMCwwLCJBIl0sWzIsMCwiQSJdLFswLDEsIkZcXGRvd25hcnJvdyBCIl0sWzIsMSwiRlxcZG93bmFycm93IEIiXSxbMSwyLCJCIl0sWzAsMSwiIiwyLHsic3R5bGUiOnsiaGVhZCI6eyJuYW1lIjoibm9uZSJ9fX1dLFsyLDMsIiIsMCx7InN0eWxlIjp7ImhlYWQiOnsibmFtZSI6Im5vbmUifX19XSxbMCwyLCJ2IiwyXSxbMywxLCJwX0EiLDJdLFsyLDQsInBfQiIsMl0sWzQsMywidSIsMl0sWzYsNCwiXFxldGEiLDAseyJzaG9ydGVuIjp7InNvdXJjZSI6NDAsInRhcmdldCI6MjB9fV1d
		\begin{tikzcd}[ampersand replacement=\&]
			A \&\& A \\
			{F\downarrow B} \&\& {F\downarrow B} \\
			\& B
			\arrow[equal, from=1-1, to=1-3]
			\arrow["v"', from=1-1, to=2-1]
			\arrow[""{name=0, anchor=center, inner sep=0}, equal, from=2-1, to=2-3]
			\arrow["{p_B}"', from=2-1, to=3-2]
			\arrow["{p_A}"', from=2-3, to=1-3]
			\arrow["u"', from=3-2, to=2-3]
			\arrow["\eta", shorten <=6pt, shorten >=3pt, Rightarrow, from=0, to=3-2]
		\end{tikzcd}
		\]
		and
		\[
		% https://q.uiver.app/#q=WzAsMyxbMSwwLCJBIl0sWzAsMSwiQiJdLFsyLDEsIkIiXSxbMSwyLCIiLDIseyJzdHlsZSI6eyJoZWFkIjp7Im5hbWUiOiJub25lIn19fV0sWzEsMCwiVSJdLFswLDIsIkYiXSxbMCwzLCJFIiwwLHsic2hvcnRlbiI6eyJzb3VyY2UiOjQwLCJ0YXJnZXQiOjIwfX1dXQ==
		\begin{tikzcd}[ampersand replacement=\&]
			\& A \\
			B \&\& B
			\arrow["F", from=1-2, to=2-3]
			\arrow["U", from=2-1, to=1-2]
			\arrow[""{name=0, anchor=center, inner sep=0}, equal, from=2-1, to=2-3]
			\arrow["E", shorten <=6pt, shorten >=3pt, Rightarrow, from=1-2, to=0]
		\end{tikzcd}
		\quad := \quad
		% https://q.uiver.app/#q=WzAsNSxbMSwwLCJBIl0sWzAsMSwiRlxcZG93bmFycm93IEIiXSxbMiwxLCJGXFxkb3duYXJyb3cgQiJdLFswLDIsIkIiXSxbMiwyLCJCIl0sWzEsMCwicF9BIl0sWzAsMiwidiJdLFsxLDIsIiIsMix7InN0eWxlIjp7ImhlYWQiOnsibmFtZSI6Im5vbmUifX19XSxbMyw0LCIiLDAseyJzdHlsZSI6eyJoZWFkIjp7Im5hbWUiOiJub25lIn19fV0sWzMsMSwidSJdLFsyLDQsInBfQiJdLFswLDcsIlxcYWxwaGEiLDAseyJzaG9ydGVuIjp7InNvdXJjZSI6NDAsInRhcmdldCI6MjB9fV0sWzcsOCwiXFxlcHNpbG9uIiwwLHsic2hvcnRlbiI6eyJzb3VyY2UiOjQwLCJ0YXJnZXQiOjIwfSwic3R5bGUiOnsiaGVhZCI6eyJuYW1lIjoibm9uZSJ9fX1dXQ==
		\begin{tikzcd}[ampersand replacement=\&]
			\& A \\
			{F\downarrow B} \&\& {F\downarrow B} \\
			B \&\& B
			\arrow["v", from=1-2, to=2-3]
			\arrow["{p_A}", from=2-1, to=1-2]
			\arrow[""{name=0, anchor=center, inner sep=0}, equal, from=2-1, to=2-3]
			\arrow["{p_B}", from=2-3, to=3-3]
			\arrow["u", from=3-1, to=2-1]
			\arrow[""{name=1, anchor=center, inner sep=0}, equal, from=3-1, to=3-3]
			\arrow["\alpha", shorten <=6pt, shorten >=3pt, Rightarrow, from=1-2, to=0]
			\arrow["\epsilon", shorten <=9pt, shorten >=9pt, equals, from=0, to=1]
		\end{tikzcd}.
		\]
		
		It remains to verify the triangle identities. We have
		\[
		% https://q.uiver.app/#q=WzAsNCxbMCwwLCJBIl0sWzIsMCwiQSJdLFswLDIsIkIiXSxbMiwyLCJCIl0sWzAsMSwiIiwwLHsibGV2ZWwiOjIsInN0eWxlIjp7ImhlYWQiOnsibmFtZSI6Im5vbmUifX19XSxbMiwzLCIiLDAseyJsZXZlbCI6Miwic3R5bGUiOnsiaGVhZCI6eyJuYW1lIjoibm9uZSJ9fX1dLFswLDIsIkYiLDJdLFsxLDMsIkYiXSxbMiwxLCJVIiwxXSxbNywzLCJFIiwyLHsib2Zmc2V0Ijo1LCJzaG9ydGVuIjp7InRhcmdldCI6NjB9fV0sWzAsNiwiSCIsMCx7Im9mZnNldCI6LTUsInNob3J0ZW4iOnsic291cmNlIjo0MCwidGFyZ2V0IjoyMH19XV0=
		\begin{tikzcd}[ampersand replacement=\&]
			A \&\& A \\
			\\
			B \&\& B
			\arrow[equals, from=1-1, to=1-3]
			\arrow[""{name=0, anchor=center, inner sep=0}, "F"', from=1-1, to=3-1]
			\arrow[""{name=1, anchor=center, inner sep=0}, "F", from=1-3, to=3-3]
			\arrow["U"{description}, from=3-1, to=1-3]
			\arrow[equals, from=3-1, to=3-3]
			\arrow["H", shift left=5, shorten <=6pt, shorten >=3pt, Rightarrow, from=1-1, to=0]
			\arrow["E"', shift right=5, shorten >=10pt, Rightarrow, from=1, to=3-3]
		\end{tikzcd}
		\quad = \quad
		% https://q.uiver.app/#q=WzAsNyxbMCwwLCJBIl0sWzIsMCwiQSJdLFswLDEsIkZcXGRvd25hcnJvdyBCIl0sWzEsMSwiRlxcZG93bmFycm93IEIiXSxbMCwyLCJCIl0sWzIsMiwiQiJdLFsyLDEsIkZcXGRvd25hcnJvdyBCIl0sWzAsMiwidiIsMl0sWzMsMSwicF9BIl0sWzIsNCwicF9CIiwyXSxbNCwzLCJ1IiwxXSxbNCw1LCIiLDIseyJsZXZlbCI6Miwic3R5bGUiOnsiaGVhZCI6eyJuYW1lIjoibm9uZSJ9fX1dLFsxLDYsInYiXSxbNiw1LCJwX0IiXSxbMCwxLCIiLDEseyJsZXZlbCI6Miwic3R5bGUiOnsiaGVhZCI6eyJuYW1lIjoibm9uZSJ9fX1dLFsyLDMsIiIsMSx7ImxldmVsIjoyLCJzdHlsZSI6eyJoZWFkIjp7Im5hbWUiOiJub25lIn19fV0sWzMsNiwiIiwxLHsibGV2ZWwiOjIsInN0eWxlIjp7ImhlYWQiOnsibmFtZSI6Im5vbmUifX19XSxbMTIsNiwiXFxhbHBoYSIsMix7Im9mZnNldCI6NSwic2hvcnRlbiI6eyJzb3VyY2UiOjIwfX1dLFsyLDksIlxcZXRhIiwwLHsib2Zmc2V0IjotNSwic2hvcnRlbiI6eyJ0YXJnZXQiOjIwfX1dLFszLDExLCJcXGVwc2lsb24iLDAseyJvZmZzZXQiOi01LCJzaG9ydGVuIjp7InNvdXJjZSI6MjAsInRhcmdldCI6NTB9LCJzdHlsZSI6eyJoZWFkIjp7Im5hbWUiOiJub25lIn19fV1d
		\begin{tikzcd}[ampersand replacement=\&]
			A \&\& A \\
			{F\downarrow B} \& {F\downarrow B} \& {F\downarrow B} \\
			B \&\& B
			\arrow[equals, from=1-1, to=1-3]
			\arrow["v"', from=1-1, to=2-1]
			\arrow[""{name=0, anchor=center, inner sep=0}, "v", from=1-3, to=2-3]
			\arrow[equals, from=2-1, to=2-2]
			\arrow[""{name=1, anchor=center, inner sep=0}, "{p_B}"', from=2-1, to=3-1]
			\arrow["{p_A}", from=2-2, to=1-3]
			\arrow[equals, from=2-2, to=2-3]
			\arrow["{p_B}", from=2-3, to=3-3]
			\arrow["u"{description}, from=3-1, to=2-2]
			\arrow[""{name=2, anchor=center, inner sep=0}, equals, from=3-1, to=3-3]
			\arrow["\alpha"', shift right=5, shorten <=1pt, Rightarrow, from=0, to=2-3]
			\arrow["\eta", shift left=5, shorten >=1pt, Rightarrow, from=2-1, to=1]
			\arrow["\epsilon", shift left=5, shorten <=3pt, shorten >=8pt, equals, from=2-2, to=2]
		\end{tikzcd}
		\quad = \quad
		% https://q.uiver.app/#q=WzAsNyxbMCwwLCJBIl0sWzIsMCwiQSJdLFswLDEsIkZcXGRvd25hcnJvdyBCIl0sWzEsMSwiRlxcZG93bmFycm93IEIiXSxbMCwyLCJCIl0sWzIsMiwiQiJdLFsyLDEsIkZcXGRvd25hcnJvdyBCIl0sWzAsMiwidiIsMl0sWzMsMSwicF9BIl0sWzIsNCwicF9CIiwyXSxbNCw1LCIiLDIseyJsZXZlbCI6Miwic3R5bGUiOnsiaGVhZCI6eyJuYW1lIjoibm9uZSJ9fX1dLFsxLDYsInYiXSxbNiw1LCJwX0IiXSxbMCwxLCIiLDEseyJsZXZlbCI6Miwic3R5bGUiOnsiaGVhZCI6eyJuYW1lIjoibm9uZSJ9fX1dLFsyLDMsIiIsMSx7ImxldmVsIjoyLCJzdHlsZSI6eyJoZWFkIjp7Im5hbWUiOiJub25lIn19fV0sWzMsNiwiIiwxLHsibGV2ZWwiOjIsInN0eWxlIjp7ImhlYWQiOnsibmFtZSI6Im5vbmUifX19XSxbMTEsNiwiXFxhbHBoYSIsMix7Im9mZnNldCI6NSwic2hvcnRlbiI6eyJzb3VyY2UiOjIwfX1dXQ==
		\begin{tikzcd}[ampersand replacement=\&]
			A \&\& A \\
			{F\downarrow B} \& {F\downarrow B} \& {F\downarrow B} \\
			B \&\& B
			\arrow[equals, from=1-1, to=1-3]
			\arrow["v"', from=1-1, to=2-1]
			\arrow[""{name=0, anchor=center, inner sep=0}, "v", from=1-3, to=2-3]
			\arrow[equals, from=2-1, to=2-2]
			\arrow["{p_B}"', from=2-1, to=3-1]
			\arrow["{p_A}", from=2-2, to=1-3]
			\arrow[equals, from=2-2, to=2-3]
			\arrow["{p_B}", from=2-3, to=3-3]
			\arrow[equals, from=3-1, to=3-3]
			\arrow["\alpha"', shift right=5, shorten <=1pt, Rightarrow, from=0, to=2-3]
		\end{tikzcd},
		\]
		whose $1$-component at $a \in \ob A$ is $p_B \cdot \alpha \cdot v(a) = p_B \cdot (1_{Fa}, 1_{Fa}) = 1_{Fa}$, we conclude that this is an identity natural transformation. Moreover, we have
		\[
		% https://q.uiver.app/#q=WzAsNCxbMCwwLCJCIl0sWzIsMCwiQiJdLFswLDIsIkEiXSxbMiwyLCJBIl0sWzAsMSwiIiwwLHsibGV2ZWwiOjIsInN0eWxlIjp7ImhlYWQiOnsibmFtZSI6Im5vbmUifX19XSxbMiwzLCIiLDAseyJsZXZlbCI6Miwic3R5bGUiOnsiaGVhZCI6eyJuYW1lIjoibm9uZSJ9fX1dLFswLDIsIlUiLDJdLFsxLDMsIlUiXSxbMiwxLCJGIiwxXSxbNiwwLCJFIiwyLHsib2Zmc2V0Ijo1LCJzaG9ydGVuIjp7InNvdXJjZSI6NDAsInRhcmdldCI6MjB9fV0sWzMsNywiSCIsMCx7Im9mZnNldCI6LTUsInNob3J0ZW4iOnsidGFyZ2V0Ijo2MH19XV0=
		\begin{tikzcd}[ampersand replacement=\&]
			B \&\& B \\
			\\
			A \&\& A
			\arrow[equals, from=1-1, to=1-3]
			\arrow[""{name=0, anchor=center, inner sep=0}, "U"', from=1-1, to=3-1]
			\arrow[""{name=1, anchor=center, inner sep=0}, "U", from=1-3, to=3-3]
			\arrow["F"{description}, from=3-1, to=1-3]
			\arrow[equals, from=3-1, to=3-3]
			\arrow["E"', shift right=5, shorten <=6pt, shorten >=3pt, Rightarrow, from=0, to=1-1]
			\arrow["H", shift left=5, shorten >=10pt, Rightarrow, from=3-3, to=1]
		\end{tikzcd}
		\quad = \quad
		% https://q.uiver.app/#q=WzAsNyxbMCwwLCJCIl0sWzIsMCwiQiJdLFswLDEsIkZcXGRvd25hcnJvdyBCIl0sWzEsMSwiRlxcZG93bmFycm93IEIiXSxbMCwyLCJBIl0sWzIsMiwiQSJdLFsyLDEsIkZcXGRvd25hcnJvdyBCIl0sWzAsMiwidSIsMl0sWzMsMSwicF9CIl0sWzIsNCwicF9BIiwyXSxbNCwzLCJ2IiwyXSxbNCw1LCIiLDIseyJsZXZlbCI6Miwic3R5bGUiOnsiaGVhZCI6eyJuYW1lIjoibm9uZSJ9fX1dLFsxLDYsInUiXSxbNiw1LCJwX0EiXSxbMCwxLCIiLDEseyJsZXZlbCI6Miwic3R5bGUiOnsiaGVhZCI6eyJuYW1lIjoibm9uZSJ9fX1dLFsyLDMsIiIsMSx7ImxldmVsIjoyLCJzdHlsZSI6eyJoZWFkIjp7Im5hbWUiOiJub25lIn19fV0sWzMsNiwiIiwxLHsibGV2ZWwiOjIsInN0eWxlIjp7ImhlYWQiOnsibmFtZSI6Im5vbmUifX19XSxbNiwxMiwiXFxldGEiLDAseyJvZmZzZXQiOi01LCJzaG9ydGVuIjp7InNvdXJjZSI6MjB9fV0sWzksMiwiXFxhbHBoYSIsMix7Im9mZnNldCI6NSwic2hvcnRlbiI6eyJ0YXJnZXQiOjIwfX1dLFsxNCwzLCJcXGVwc2lsb24iLDIseyJvZmZzZXQiOjUsInNob3J0ZW4iOnsic291cmNlIjo0MCwidGFyZ2V0IjoyMH0sInN0eWxlIjp7ImhlYWQiOnsibmFtZSI6Im5vbmUifX19XV0=
		\begin{tikzcd}[ampersand replacement=\&]
			B \&\& B \\
			{F\downarrow B} \& {F\downarrow B} \& {F\downarrow B} \\
			A \&\& A
			\arrow[""{name=0, anchor=center, inner sep=0}, equals, from=1-1, to=1-3]
			\arrow["u"', from=1-1, to=2-1]
			\arrow[""{name=1, anchor=center, inner sep=0}, "u", from=1-3, to=2-3]
			\arrow[equals, from=2-1, to=2-2]
			\arrow[""{name=2, anchor=center, inner sep=0}, "{p_A}"', from=2-1, to=3-1]
			\arrow["{p_B}", from=2-2, to=1-3]
			\arrow[equals, from=2-2, to=2-3]
			\arrow["{p_A}", from=2-3, to=3-3]
			\arrow["v"', from=3-1, to=2-2]
			\arrow[equals, from=3-1, to=3-3]
			\arrow["\epsilon"', shift right=5, shorten <=6pt, shorten >=3pt, equals, from=0, to=2-2]
			\arrow["\alpha"', shift right=5, shorten >=1pt, Rightarrow, from=2, to=2-1]
			\arrow["\eta", shift left=5, shorten <=1pt, Rightarrow, from=2-3, to=1]
		\end{tikzcd}
		\quad = \quad
		% https://q.uiver.app/#q=WzAsNyxbMCwwLCJCIl0sWzIsMCwiQiJdLFswLDEsIkZcXGRvd25hcnJvdyBCIl0sWzEsMSwiRlxcZG93bmFycm93IEIiXSxbMCwyLCJBIl0sWzIsMiwiQSJdLFsyLDEsIkZcXGRvd25hcnJvdyBCIl0sWzAsMiwidSIsMl0sWzIsNCwicF9BIiwyXSxbNCwzLCJ2IiwyXSxbNCw1LCIiLDIseyJsZXZlbCI6Miwic3R5bGUiOnsiaGVhZCI6eyJuYW1lIjoibm9uZSJ9fX1dLFsxLDYsInUiXSxbNiw1LCJwX0EiXSxbMCwxLCIiLDEseyJsZXZlbCI6Miwic3R5bGUiOnsiaGVhZCI6eyJuYW1lIjoibm9uZSJ9fX1dLFsyLDMsIiIsMSx7ImxldmVsIjoyLCJzdHlsZSI6eyJoZWFkIjp7Im5hbWUiOiJub25lIn19fV0sWzMsNiwiIiwxLHsibGV2ZWwiOjIsInN0eWxlIjp7ImhlYWQiOnsibmFtZSI6Im5vbmUifX19XSxbOCwyLCJcXGFscGhhIiwyLHsib2Zmc2V0Ijo1LCJzaG9ydGVuIjp7InRhcmdldCI6MjB9fV1d
		\begin{tikzcd}[ampersand replacement=\&]
			B \&\& B \\
			{F\downarrow B} \& {F\downarrow B} \& {F\downarrow B} \\
			A \&\& A
			\arrow[equals, from=1-1, to=1-3]
			\arrow["u"', from=1-1, to=2-1]
			\arrow["u", from=1-3, to=2-3]
			\arrow[equals, from=2-1, to=2-2]
			\arrow[""{name=0, anchor=center, inner sep=0}, "{p_A}"', from=2-1, to=3-1]
			\arrow[equals, from=2-2, to=2-3]
			\arrow["{p_A}", from=2-3, to=3-3]
			\arrow["v"', from=3-1, to=2-2]
			\arrow[equals, from=3-1, to=3-3]
			\arrow["\alpha"', shift right=5, shorten >=1pt, Rightarrow, from=0, to=2-1]
		\end{tikzcd}
		\]
		which has $1$-component at $b \in \ob B$ given by $p_A \cdot \alpha \cdot u(b) = p_A \cdot (F(1_{Ub}), u(b)) = 1_{Ub}$, therefore, this is also an identity natural transformation.
	\end{proof}
	
	\begin{lemma}
		\label{lem:mor_la_mor_lali}
		Given adjunctions
		% https://q.uiver.app/#q=WzAsMixbMCwwLCJBX2kiXSxbMSwwLCJCX2kiXSxbMCwxLCJGX2kiLDAseyJjdXJ2ZSI6LTF9XSxbMSwwLCJVX2kiLDAseyJjdXJ2ZSI6LTF9XSxbMiwzLCIiLDAseyJsZXZlbCI6MSwic3R5bGUiOnsibmFtZSI6ImFkanVuY3Rpb24ifX1dXQ==
		\[\begin{tikzcd}[ampersand replacement=\&]
			{A_i} \& {B_i}
			\arrow[""{name=0, anchor=center, inner sep=0}, "{F_i}", bend left, from=1-1, to=1-2]
			\arrow[""{name=1, anchor=center, inner sep=0}, "{U_i}", bend left, from=1-2, to=1-1]
			\arrow["\dashv"{anchor=center, rotate=-90}, draw=none, from=0, to=1]
		\end{tikzcd}\]
		and the corresponding lalis as in \longref{Lemma}{lem:la_as_lali}
		% https://q.uiver.app/#q=WzAsMixbMCwwLCJGX2lcXGRvd25hcnJvdyBCX2kiXSxbMiwwLCJCX2kiXSxbMCwxLCJwX3tCX2l9IiwwLHsiY3VydmUiOi0xfV0sWzEsMCwidV9pIiwwLHsiY3VydmUiOi0xfV0sWzIsMywiIiwwLHsibGV2ZWwiOjEsInN0eWxlIjp7Im5hbWUiOiJhZGp1bmN0aW9uIn19XV0=
		\[\begin{tikzcd}[ampersand replacement=\&]
			{F_i\downarrow B_i} \&\& {B_i}
			\arrow[""{name=0, anchor=center, inner sep=0}, "{p_{B_i}}", bend left, from=1-1, to=1-3]
			\arrow[""{name=1, anchor=center, inner sep=0}, "{u_i}", bend left, from=1-3, to=1-1]
			\arrow["\dashv"{anchor=center, rotate=-90}, draw=none, from=0, to=1]
		\end{tikzcd}\]
		in $\Cat$, a pair $(A_1 \xrightarrow{\gamma} A_2, B_1 \xrightarrow{\beta} B_2)$ of functors is a morphism of left adjoints if and only if the corresponding pair $(F_1 \downarrow B_1 \xrightarrow{\overline{\beta}} F_2\downarrow B_2, B_1 \xrightarrow{\beta} B_2)$ is a morphism of lalis, where $\overline{\beta}$ denotes the canonical induced map.
	\end{lemma}
	
	\begin{proof}
		Consider the mate
		\[
		\theta \quad := \quad 
		% https://q.uiver.app/#q=WzAsNixbMSwwLCJBXzEiXSxbMSwxLCJCXzEiXSxbMiwwLCJBXzIiXSxbMiwxLCJCXzIiXSxbMCwxLCJCXzEiXSxbMywwLCJBXzIiXSxbMCwxLCJGXzEiXSxbMiwzLCJGXzIiLDJdLFswLDIsIlxcYWxwaGEiXSxbMSwzLCJcXGJldGEiLDJdLFs0LDEsIiIsMSx7ImxldmVsIjoyLCJzdHlsZSI6eyJoZWFkIjp7Im5hbWUiOiJub25lIn19fV0sWzIsNSwiIiwxLHsibGV2ZWwiOjIsInN0eWxlIjp7ImhlYWQiOnsibmFtZSI6Im5vbmUifX19XSxbNCwwLCJVXzEiXSxbMyw1LCJVXzIiLDJdLFsxMSwxMywiSF8yIiwyLHsic2hvcnRlbiI6eyJzb3VyY2UiOjIwLCJ0YXJnZXQiOjIwfX1dLFsxMiwxMCwiRV8xIiwwLHsic2hvcnRlbiI6eyJzb3VyY2UiOjIwLCJ0YXJnZXQiOjIwfX1dXQ==
		\begin{tikzcd}[ampersand replacement=\&]
			\& {A_1} \& {A_2} \& {A_2} \\
			{B_1} \& {B_1} \& {B_2}
			\arrow["\gamma", from=1-2, to=1-3]
			\arrow["{F_1}", from=1-2, to=2-2]
			\arrow[""{name=0, anchor=center, inner sep=0}, equals, from=1-3, to=1-4]
			\arrow["{F_2}"', from=1-3, to=2-3]
			\arrow[""{name=1, anchor=center, inner sep=0}, "{U_1}", from=2-1, to=1-2]
			\arrow[""{name=2, anchor=center, inner sep=0}, equals, from=2-1, to=2-2]
			\arrow["\beta"', from=2-2, to=2-3]
			\arrow[""{name=3, anchor=center, inner sep=0}, "{U_2}"', from=2-3, to=1-4]
			\arrow["{H_2}"', shorten <=2pt, shorten >=2pt, Rightarrow, from=0, to=3]
			\arrow["{E_1}", shorten <=2pt, shorten >=2pt, Rightarrow, from=1, to=2]
		\end{tikzcd},
		\]
		where $E_i$ and $H_2$ denote the corresponding counits and units, respectively, of $(\alpha, \beta)$. Using the notations in the proof of \longref{Lemma}{lem:la_as_lali}, this is equal to
		% https://q.uiver.app/#q=WzAsMTAsWzAsMCwiQV8xIl0sWzAsMSwiRl8xXFxkb3duYXJyb3cgQl8xIl0sWzEsMSwiRl8xXFxkb3duYXJyb3cgQl8xIl0sWzAsMiwiQl8xIl0sWzIsMiwiQl8xIl0sWzIsMSwiRl8yXFxkb3duYXJyb3cgQl8yIl0sWzMsMSwiRl8yIFxcZG93bmFycm93IEJfMiJdLFsxLDAsIkFfMiJdLFszLDAsIkFfMiJdLFszLDIsIkJfMiJdLFsxLDAsInBfe0FfMX0iXSxbMywxLCJ1XzEiXSxbMyw0LCIiLDIseyJsZXZlbCI6Miwic3R5bGUiOnsiaGVhZCI6eyJuYW1lIjoibm9uZSJ9fX1dLFsxLDIsIiIsMSx7ImxldmVsIjoyLCJzdHlsZSI6eyJoZWFkIjp7Im5hbWUiOiJub25lIn19fV0sWzIsNSwiXFxvdmVybGluZXtcXGJldGF9Il0sWzUsNiwiIiwxLHsibGV2ZWwiOjIsInN0eWxlIjp7ImhlYWQiOnsibmFtZSI6Im5vbmUifX19XSxbMCwyLCJ2XzEiXSxbMCw3LCJcXGdhbW1hIl0sWzcsNSwidl8yIl0sWzcsOCwiIiwxLHsibGV2ZWwiOjIsInN0eWxlIjp7ImhlYWQiOnsibmFtZSI6Im5vbmUifX19XSxbNiw4LCJwX3tBXzJ9IiwyXSxbOSw2LCJ1XzIiLDJdLFs0LDksIlxcYmV0YSIsMl0sWzIsNCwicF97Ql8xfSJdLFs1LDksInBfe0JfMn0iLDJdLFsxMCwxLCJcXGFscGhhXzEiLDAseyJvZmZzZXQiOi01LCJzaG9ydGVuIjp7InNvdXJjZSI6MjB9fV0sWzIsMTIsIlxcZXBzaWxvbl8xIiwyLHsic2hvcnRlbiI6eyJzb3VyY2UiOjIwLCJ0YXJnZXQiOjQwfSwic3R5bGUiOnsiaGVhZCI6eyJuYW1lIjoibm9uZSJ9fX1dLFs2LDIxLCJcXGV0YV8yIiwyLHsib2Zmc2V0Ijo1LCJzaG9ydGVuIjp7InRhcmdldCI6MjB9fV1d
		\[\begin{tikzcd}[ampersand replacement=\&]
			{A_1} \& {A_2} \&\& {A_2} \\
			{F_1\downarrow B_1} \& {F_1\downarrow B_1} \& {F_2\downarrow B_2} \& {F_2 \downarrow B_2} \\
			{B_1} \&\& {B_1} \& {B_2}
			\arrow["\gamma", from=1-1, to=1-2]
			\arrow["{v_1}", from=1-1, to=2-2]
			\arrow[equals, from=1-2, to=1-4]
			\arrow["{v_2}", from=1-2, to=2-3]
			\arrow[""{name=0, anchor=center, inner sep=0}, "{p_{A_1}}", from=2-1, to=1-1]
			\arrow[equals, from=2-1, to=2-2]
			\arrow["{\overline{\beta}}", from=2-2, to=2-3]
			\arrow["{p_{B_1}}", from=2-2, to=3-3]
			\arrow[equals, from=2-3, to=2-4]
			\arrow["{p_{B_2}}"', from=2-3, to=3-4]
			\arrow["{p_{A_2}}"', from=2-4, to=1-4]
			\arrow["{u_1}", from=3-1, to=2-1]
			\arrow[""{name=1, anchor=center, inner sep=0}, equals, from=3-1, to=3-3]
			\arrow["\beta"', from=3-3, to=3-4]
			\arrow[""{name=2, anchor=center, inner sep=0}, "{u_2}"', from=3-4, to=2-4]
			\arrow["{\alpha_1}", shift left=5, shorten <=1pt, Rightarrow, from=0, to=2-1]
			\arrow["{\epsilon_1}"', shorten <=3pt, shorten >=6pt, equals, from=2-2, to=1]
			\arrow["{\eta_2}"', shift right=5, shorten >=1pt, Rightarrow, from=2-4, to=2]
		\end{tikzcd},\]
		where the canonical induced map $\overline{\beta} \colon F_1 \downarrow B_1 \to F_2 \downarrow B_2$ is actually the application of $\beta$.
		
		Note that
		% https://q.uiver.app/#q=WzAsNixbMCwxLCJCXzEiXSxbMSwxLCJGXzFcXGRvd25hcnJvdyBCXzEiXSxbMiwwLCJBXzEiXSxbMywxLCJGXzFcXGRvd25hcnJvdyBCXzEiXSxbNCwxLCJGXzJcXGRvd25hcnJvdyBCXzIiXSxbNSwxLCJBXzIiXSxbMSwzLCIiLDIseyJsZXZlbCI6Miwic3R5bGUiOnsiaGVhZCI6eyJuYW1lIjoibm9uZSJ9fX1dLFswLDEsInVfMSIsMl0sWzEsMiwicF97QV8xfSJdLFsyLDMsInZfMSJdLFszLDQsIlxcb3ZlcmxpbmV7XFxiZXRhfSIsMl0sWzQsNSwicF97QV8yfSIsMl0sWzIsNiwiXFxhbHBoYV8xIiwwLHsic2hvcnRlbiI6eyJzb3VyY2UiOjIwLCJ0YXJnZXQiOjIwfX1dXQ==
		\[\begin{tikzcd}[ampersand replacement=\&]
			\&\& {A_1} \\
			{B_1} \& {F_1\downarrow B_1} \&\& {F_1\downarrow B_1} \& {F_2\downarrow B_2} \& {A_2}
			\arrow["{v_1}", from=1-3, to=2-4]
			\arrow["{u_1}"', from=2-1, to=2-2]
			\arrow["{p_{A_1}}", from=2-2, to=1-3]
			\arrow[""{name=0, anchor=center, inner sep=0}, equals, from=2-2, to=2-4]
			\arrow["{\overline{\beta}}"', from=2-4, to=2-5]
			\arrow["{p_{A_2}}"', from=2-5, to=2-6]
			\arrow["{\alpha_1}", shorten <=3pt, shorten >=3pt, Rightarrow, from=1-3, to=0]
		\end{tikzcd}\]
		is actually an identity natural transformation, because its $1$-component at $b_1 \ob B_1$ is given by the identity on $\gamma U_1 b_1$. Thus, the natural transformation $\theta$ is invertible precisely when the mate
		% https://q.uiver.app/#q=WzAsNixbMSwwLCJGXzEgXFxkb3duYXJyb3cgQl8xIl0sWzEsMSwiQl8xIl0sWzIsMCwiRl8yIFxcZG93bmFycm93IEJfMiJdLFsyLDEsIkJfMiJdLFswLDEsIkJfMSJdLFszLDAsIkZfMiBcXGRvd25hcnJvdyBCXzIiXSxbMCwxLCJwX3tCXzF9Il0sWzIsMywicF97Ql8yfSIsMl0sWzEsMywiXFxiZXRhIiwyXSxbNCwxLCIiLDEseyJsZXZlbCI6Miwic3R5bGUiOnsiaGVhZCI6eyJuYW1lIjoibm9uZSJ9fX1dLFsyLDUsIiIsMSx7ImxldmVsIjoyLCJzdHlsZSI6eyJoZWFkIjp7Im5hbWUiOiJub25lIn19fV0sWzQsMCwidV8xIl0sWzMsNSwidV8yIiwyXSxbMCwyLCJcXG92ZXJsaW5le1xcYmV0YX0iXSxbMiw3LCJcXGV0YV8yIiwwLHsib2Zmc2V0IjotNSwic2hvcnRlbiI6eyJ0YXJnZXQiOjIwfX1dLFs2LDEsIlxcZXBzaWxvbiIsMix7Im9mZnNldCI6NSwic2hvcnRlbiI6eyJzb3VyY2UiOjIwfSwic3R5bGUiOnsiaGVhZCI6eyJuYW1lIjoibm9uZSJ9fX1dXQ==
		\[\begin{tikzcd}[ampersand replacement=\&]
			\& {F_1 \downarrow B_1} \& {F_2 \downarrow B_2} \& {F_2 \downarrow B_2} \\
			{B_1} \& {B_1} \& {B_2}
			\arrow["{\overline{\beta}}", from=1-2, to=1-3]
			\arrow[""{name=0, anchor=center, inner sep=0}, "{p_{B_1}}", from=1-2, to=2-2]
			\arrow[equals, from=1-3, to=1-4]
			\arrow[""{name=1, anchor=center, inner sep=0}, "{p_{B_2}}"', from=1-3, to=2-3]
			\arrow["{u_1}", from=2-1, to=1-2]
			\arrow[equals, from=2-1, to=2-2]
			\arrow["\beta"', from=2-2, to=2-3]
			\arrow["{u_2}"', from=2-3, to=1-4]
			\arrow["\epsilon"', shift right=5, shorten <=1pt, equals, from=0, to=2-2]
			\arrow["{\eta_2}", shift left=5, shorten >=1pt, Rightarrow, from=1-3, to=1]
		\end{tikzcd}\]
		of $(\overline{\beta}, \beta)$ is invertible.
	\end{proof}
	
	In \cite[Lemma 2.4.2]{RV:2015_2}, it is stated that a pair of $1$-morphisms in a $2$-category is an adjunction precisely when it is representably an adjunction in $\Cat$. Making use of this fact, we are able to transfer the above lemmata to an $\infty$-cosmos.
	
	\begin{lemma}
		\label{lem:la_as_lali_K}
		In an $\infty$-cosmos $\calK$, a $0$-arrow $F \colon A \to B$ is a left adjoint if and only if 
		the canonical limit projection $p_B \colon F \downarrow B \to B$ from the lax limit $F \downarrow B$ of $F$, i.e., the comma $\infty$-category of $F$ and $1_B$, to $B$ is a lali.
	\end{lemma}
	
	\begin{proof}
		Any adjunction in an $\infty$-cosmos $\calK$ amounts to one in its homotopy $2$-category $h\calK$.
		
		By \cite[Lemma 2.4.2]{RV:2015_2}, it suffices to show that in $\Cat$, we have, for any category $X$, 
		% https://q.uiver.app/#q=WzAsMixbMCwwLCJoXFxjYWxLKFgsIEEpIl0sWzIsMCwiaFxcY2FsSyhYLCBCKSJdLFswLDEsImhcXGNhbEsoWCwgRikiLDAseyJjdXJ2ZSI6LTF9XSxbMSwwLCJoXFxjYWxLKFgsIFUpIiwwLHsiY3VydmUiOi0xfV0sWzIsMywiIiwwLHsibGV2ZWwiOjEsInN0eWxlIjp7Im5hbWUiOiJhZGp1bmN0aW9uIn19XV0=
		\[\begin{tikzcd}[ampersand replacement=\&]
			{h\calK(X, A)} \&\& {h\calK(X, B)}
			\arrow[""{name=0, anchor=center, inner sep=0}, "{h\calK(X, F)}", bend left, from=1-1, to=1-3]
			\arrow[""{name=1, anchor=center, inner sep=0}, "{h\calK(X, U)}", bend left, from=1-3, to=1-1]
			\arrow["\dashv"{anchor=center, rotate=-90}, draw=none, from=0, to=1]
		\end{tikzcd}\]
		is an adjunction if and only if the corresponding pair
		% https://q.uiver.app/#q=WzAsMixbMCwwLCJoXFxjYWxLKFgsIEZcXGRvd25hcnJvdyBCKSJdLFsyLDAsImhcXGNhbEsoWCwgQikiXSxbMCwxLCJoXFxjYWxLKFgsIHBfQikiLDAseyJjdXJ2ZSI6LTF9XSxbMSwwLCJoXFxjYWxLKFgsIHUpIiwwLHsiY3VydmUiOi0xfV0sWzIsMywiIiwwLHsibGV2ZWwiOjEsInN0eWxlIjp7Im5hbWUiOiJhZGp1bmN0aW9uIn19XV0=
		\[\begin{tikzcd}[ampersand replacement=\&]
			{h\calK(X, F\downarrow B)} \&\& {h\calK(X, B)}
			\arrow[""{name=0, anchor=center, inner sep=0}, "{h\calK(X, p_B)}", bend left, from=1-1, to=1-3]
			\arrow[""{name=1, anchor=center, inner sep=0}, "{h\calK(X, u)}", bend left, from=1-3, to=1-1]
			\arrow["\dashv"{anchor=center, rotate=-90}, draw=none, from=0, to=1]
		\end{tikzcd}\]
		is an adjunction with identity counit. Note that $h\calK(X, F \downarrow B) \cong h\calK(X, F) \downarrow h\calK(X, B)$ and also $h\calK(X, p_B) = p_{h\calK(X, B)}$, because the Yoneda embedding preserves limits. So, by \longref{Lemma}{lem:la_as_lali}, this is true.
	\end{proof}
	
	\begin{lemma}
		\label{lem:mor_la_mor_lali_K}
		In an $\infty$-cosmos $\calK$, given adjunctions
		% https://q.uiver.app/#q=WzAsMixbMCwwLCJBX2kiXSxbMSwwLCJCX2kiXSxbMCwxLCJGX2kiLDAseyJjdXJ2ZSI6LTF9XSxbMSwwLCJVX2kiLDAseyJjdXJ2ZSI6LTF9XSxbMiwzLCIiLDAseyJsZXZlbCI6MSwic3R5bGUiOnsibmFtZSI6ImFkanVuY3Rpb24ifX1dXQ==
		\[\begin{tikzcd}[ampersand replacement=\&]
			{A_i} \& {B_i}
			\arrow[""{name=0, anchor=center, inner sep=0}, "{F_i}", bend left, from=1-1, to=1-2]
			\arrow[""{name=1, anchor=center, inner sep=0}, "{U_i}", bend left, from=1-2, to=1-1]
			\arrow["\dashv"{anchor=center, rotate=-90}, draw=none, from=0, to=1]
		\end{tikzcd}\]
		and the corresponding lalis as in \longref{Lemma}{lem:la_as_lali_K}
		% https://q.uiver.app/#q=WzAsMixbMCwwLCJGX2lcXGRvd25hcnJvdyBCX2kiXSxbMiwwLCJCX2kiXSxbMCwxLCJwX3tCX2l9IiwwLHsiY3VydmUiOi0xfV0sWzEsMCwidV9pIiwwLHsiY3VydmUiOi0xfV0sWzIsMywiIiwwLHsibGV2ZWwiOjEsInN0eWxlIjp7Im5hbWUiOiJhZGp1bmN0aW9uIn19XV0=
		\[\begin{tikzcd}[ampersand replacement=\&]
			{F_i\downarrow B_i} \&\& {B_i}
			\arrow[""{name=0, anchor=center, inner sep=0}, "{p_{B_i}}", bend left, from=1-1, to=1-3]
			\arrow[""{name=1, anchor=center, inner sep=0}, "{u_i}", bend left, from=1-3, to=1-1]
			\arrow["\dashv"{anchor=center, rotate=-90}, draw=none, from=0, to=1]
		\end{tikzcd},\]
		a pair $(A_1 \xrightarrow{\gamma} A_2, B_1 \xrightarrow{\beta} B_2)$ of $0$-arrows is a morphism of left adjoints if and only if the corresponding pair $(F_1 \downarrow B_1 \xrightarrow{\overline{\beta}} F_2\downarrow B_2, B_1 \xrightarrow{\beta} B_2)$ is a morphism of lalis, where $\overline{\beta}$ denotes the canonical induced map.
	\end{lemma}
	
	\begin{proof}
		A $0$-arrow in $\calK$ is an isomorphism, just when it is an isomorphism in the homotopy $2$-category $h\calK$. The statement follows directly from \longref{Lemma}{lem:mor_la_mor_lali}.
	\end{proof}
	
	\begin{pro}
		\label{pro:La_cosmos}
		Let $\calK$ be an $\infty$-cosmos. The simplicial category $\calLa(\calK)$ is an $\infty$-cosmos.
	\end{pro}
	
	\begin{proof}
		From the above series of lemmata, we conclude that the simplicial category $\calLa(\calK)$ of la-isofibrations can be obtained as the pullback
		% https://q.uiver.app/#q=WzAsNCxbMCwwLCJcXGNhbExhKFxcY2FsSykiXSxbMSwwLCJcXGNhbExhbGkoXFxjYWxLKSJdLFswLDEsIlxcY2FsS157XFxpc29mfSJdLFsxLDEsIlxcY2FsS157XFxpc29mfSJdLFswLDFdLFsxLDMsIlUiLDAseyJzdHlsZSI6eyJ0YWlsIjp7Im5hbWUiOiJob29rIiwic2lkZSI6InRvcCJ9LCJoZWFkIjp7Im5hbWUiOiJlcGkifX19XSxbMCwyLCIiLDAseyJzdHlsZSI6eyJ0YWlsIjp7Im5hbWUiOiJob29rIiwic2lkZSI6InRvcCJ9LCJoZWFkIjp7Im5hbWUiOiJlcGkifX19XSxbMiwzLCJMIiwyXSxbMCwzLCIiLDAseyJzdHlsZSI6eyJuYW1lIjoiY29ybmVyIn19XV0=
		\begin{equation}
			\label{diag:pb_la}
			\begin{tikzcd}[ampersand replacement=\&]
				{\calLa(\calK)} \& {\calLali(\calK)} \\
				{\calK^{\isof}} \& {\calK^{\isof}}
				\arrow[from=1-1, to=1-2]
				\arrow[hook, two heads, from=1-1, to=2-1]
				\arrow["\lrcorner"{anchor=center, pos=0.125}, draw=none, from=1-1, to=2-2]
				\arrow["U", hook, two heads, from=1-2, to=2-2]
				\arrow["L"', from=2-1, to=2-2]
			\end{tikzcd}
		\end{equation}
		of $L$ along the inclusion $U \colon \calLali(\calK) \hooktwoheadrightarrow \calK^{\isof}$, where $L$ sends each isofibration $p \colon E \twoheadrightarrow B$ of $\calK$ to the corresponding lali $\pi_{B} \colon p \downarrow B \twoheadrightarrow B$, as in \longref{Lemma}{lem:la_as_lali_K}, and each square
		% https://q.uiver.app/#q=WzAsNCxbMCwwLCJFXzEiXSxbMCwxLCJCXzEiXSxbMSwwLCJFXzIiXSxbMSwxLCJCXzIiXSxbMCwxLCJwXzEiLDIseyJzdHlsZSI6eyJoZWFkIjp7Im5hbWUiOiJlcGkifX19XSxbMiwzLCJwXzIiLDAseyJzdHlsZSI6eyJoZWFkIjp7Im5hbWUiOiJlcGkifX19XSxbMCwyLCJlIl0sWzEsMywiYiIsMl1d
		\[\begin{tikzcd}[ampersand replacement=\&]
			{E_1} \& {E_2} \\
			{B_1} \& {B_2}
			\arrow["e", from=1-1, to=1-2]
			\arrow["{p_1}"', two heads, from=1-1, to=2-1]
			\arrow["{p_2}", two heads, from=1-2, to=2-2]
			\arrow["b"', from=2-1, to=2-2]
		\end{tikzcd}\]
		to the commutative square
		% https://q.uiver.app/#q=WzAsNCxbMCwwLCJwXzEgXFxkb3duYXJyb3cgQl8xIl0sWzAsMSwiQl8xIl0sWzEsMCwicF8yIFxcZG93bmFycm93IEJfMiJdLFsxLDEsIkJfMiJdLFswLDEsIlxccGlfe0JfMX0iLDIseyJzdHlsZSI6eyJoZWFkIjp7Im5hbWUiOiJlcGkifX19XSxbMiwzLCJcXHBpX3tCXzJ9IiwwLHsic3R5bGUiOnsiaGVhZCI6eyJuYW1lIjoiZXBpIn19fV0sWzAsMiwiXFxvdmVybGluZXtifSJdLFsxLDMsImIiLDJdXQ==
		\[\begin{tikzcd}[ampersand replacement=\&]
			{p_1 \downarrow B_1} \& {p_2 \downarrow B_2} \\
			{B_1} \& {B_2}
			\arrow["{\overline{b}}", from=1-1, to=1-2]
			\arrow["{\pi_{B_1}}"', two heads, from=1-1, to=2-1]
			\arrow["{\pi_{B_2}}", two heads, from=1-2, to=2-2]
			\arrow["b"', from=2-1, to=2-2]
		\end{tikzcd}\]
		as in \longref{Lemma}{lem:mor_la_mor_lali_K}.
		
		According to \cite[Proposition 6.3.12]{book:RV:2022}, it suffices to verify that $L$ is a cosmological functor, as we know already that $U$ is replete.
		
		In fact, since $p \downarrow B$ is a cosmological limit, and that limits commute, we conclude that $L$ preserves all cosmological limits.
		
		Next, consider an isofibration $(e, b)$
		% https://q.uiver.app/#q=WzAsNSxbMCwwLCJFXzEiXSxbMiwwLCJFXzIiXSxbMCwyLCJCXzEiXSxbMiwyLCJCXzIiXSxbMSwxLCJQIl0sWzAsMSwiZSIsMCx7InN0eWxlIjp7ImhlYWQiOnsibmFtZSI6ImVwaSJ9fX1dLFsyLDMsImIiLDIseyJzdHlsZSI6eyJoZWFkIjp7Im5hbWUiOiJlcGkifX19XSxbMCwyLCJwXzEiLDIseyJzdHlsZSI6eyJoZWFkIjp7Im5hbWUiOiJlcGkifX19XSxbMSwzLCJwXzIiXSxbNCwxLCIiLDAseyJzdHlsZSI6eyJib2R5Ijp7Im5hbWUiOiJkYXNoZWQifSwiaGVhZCI6eyJuYW1lIjoiZXBpIn19fV0sWzQsMiwiIiwwLHsic3R5bGUiOnsiYm9keSI6eyJuYW1lIjoiZGFzaGVkIn0sImhlYWQiOnsibmFtZSI6ImVwaSJ9fX1dLFs0LDMsIiIsMCx7InN0eWxlIjp7Im5hbWUiOiJjb3JuZXIifX1dLFswLDQsIiIsMCx7InN0eWxlIjp7ImJvZHkiOnsibmFtZSI6ImRhc2hlZCJ9LCJoZWFkIjp7Im5hbWUiOiJlcGkifX19XV0=
		\[\begin{tikzcd}[ampersand replacement=\&]
			{E_1} \&\& {E_2} \\
			\& P \\
			{B_1} \&\& {B_2}
			\arrow["e", two heads, from=1-1, to=1-3]
			\arrow[dashed, two heads, from=1-1, to=2-2]
			\arrow["{p_1}"', two heads, from=1-1, to=3-1]
			\arrow["{p_2}", from=1-3, to=3-3]
			\arrow[dashed, two heads, from=2-2, to=1-3]
			\arrow[dashed, two heads, from=2-2, to=3-1]
			\arrow["\lrcorner"{anchor=center, pos=0.125}, draw=none, from=2-2, to=3-3]
			\arrow["b"', two heads, from=3-1, to=3-3]
		\end{tikzcd}\]
		of $\calK^{\isof}$ from $p_1$ to $p_2$. Its image under $L$ is
		% https://q.uiver.app/#q=WzAsNSxbMCwwLCJwXzFcXGRvd25hcnJvdyBCXzEiXSxbMiwwLCJwXzJcXGRvd25hcnJvdyBCXzIiXSxbMCwyLCJCXzEiXSxbMiwyLCJCXzIiXSxbMSwxLCJMUCJdLFswLDEsIlxcb3ZlcmxpbmV7Yn0iLDAseyJzdHlsZSI6eyJoZWFkIjp7Im5hbWUiOiJlcGkifX19XSxbMiwzLCJiIiwyLHsic3R5bGUiOnsiaGVhZCI6eyJuYW1lIjoiZXBpIn19fV0sWzAsMiwiXFxwaV97Ql8xfSIsMix7InN0eWxlIjp7ImhlYWQiOnsibmFtZSI6ImVwaSJ9fX1dLFsxLDMsIlxccGlfe0JfMn0iXSxbNCwxLCIiLDAseyJzdHlsZSI6eyJib2R5Ijp7Im5hbWUiOiJkYXNoZWQifSwiaGVhZCI6eyJuYW1lIjoiZXBpIn19fV0sWzQsMiwiIiwwLHsic3R5bGUiOnsiYm9keSI6eyJuYW1lIjoiZGFzaGVkIn0sImhlYWQiOnsibmFtZSI6ImVwaSJ9fX1dLFs0LDMsIiIsMCx7InN0eWxlIjp7Im5hbWUiOiJjb3JuZXIifX1dLFswLDQsIiIsMCx7InN0eWxlIjp7ImJvZHkiOnsibmFtZSI6ImRhc2hlZCJ9LCJoZWFkIjp7Im5hbWUiOiJlcGkifX19XV0=
		\[\begin{tikzcd}[ampersand replacement=\&]
			{p_1\downarrow B_1} \&\& {p_2\downarrow B_2} \\
			\& LP \\
			{B_1} \&\& {B_2}
			\arrow["{\overline{b}}", two heads, from=1-1, to=1-3]
			\arrow[dashed, two heads, from=1-1, to=2-2]
			\arrow["{\pi_{B_1}}"', two heads, from=1-1, to=3-1]
			\arrow["{\pi_{B_2}}", from=1-3, to=3-3]
			\arrow[dashed, two heads, from=2-2, to=1-3]
			\arrow[dashed, two heads, from=2-2, to=3-1]
			\arrow["\lrcorner"{anchor=center, pos=0.125}, draw=none, from=2-2, to=3-3]
			\arrow["b"', two heads, from=3-1, to=3-3]
		\end{tikzcd}.\]
		Since comma $\infty$-categories are cosmological limits, by \cite[Proposition 6.2.8]{book:RV:2022}, we see that $p_1 \downarrow B_1 \xrightarrow{\overline{b}} p_2 \downarrow B_2$ and the canonical induced map $p_1 \downarrow B_2 \xrightarrow{k} LP$ are both isofibrations.
		
		As a result,
		$$\calLa(\calK) \hooktwoheadrightarrow \calK^{\isof}$$
		is a cosmological embedding.
	\end{proof}
	
	\begin{nota}
		Denote by $\La(\calK)$ for the enhanced simplicial category of la-isofibrations of an $\infty$-cosmos $\calK$, whose tight part is exactly $\calLa(\calK)$, and whose loose $0$-arrows are given by the morphisms of isofibrations of $\calK$.
	\end{nota}
	
	We can establish several completeness results for $\La(\calK)$ just like before.
	
	\begin{theorem}
		\label{thm:trio_La}
		The enhanced simplicial category $\La(\calK)$ admits tight cosmological limits, and terminally rigged $n$-inserters, and that the projection of a terminally rigged $n$-inserter is always a tight isofibration.
		
		Consequently, $\La(\calK)$ admits Eilenberg-Moore objects over monads which are not morphisms of left adjoints, and that the forgetful functor is a morphism of left adjoints and reflects morphisms of left adjoints.
		
		These $\F_\Delta$-weighted limits are all preserved by the inclusion
		$$\La(\calK) \hookrightarrow \calK^{\isof}_\chor.$$
	\end{theorem}
	
	\begin{proof}
		The pullback in \longref{Diagram}{diag:pb_la} extends to a pullback
		\begin{equation}
			\label{diag:La_pb}
			% https://q.uiver.app/#q=WzAsNCxbMCwwLCJcXExhKFxcY2FsSykiXSxbMSwwLCJcXExhbGkoXFxjYWxLKSJdLFswLDEsIlxcY2FsS157XFxpc29mfSJdLFsxLDEsIlxcY2FsS157XFxpc29mfSJdLFswLDFdLFsxLDMsIlUiLDAseyJzdHlsZSI6eyJoZWFkIjp7Im5hbWUiOiJlcGkifX19XSxbMCwyLCIiLDAseyJzdHlsZSI6eyJoZWFkIjp7Im5hbWUiOiJlcGkifX19XSxbMiwzLCJMIiwyXSxbMCwzLCIiLDAseyJzdHlsZSI6eyJuYW1lIjoiY29ybmVyIn19XV0=
			\begin{tikzcd}[ampersand replacement=\&]
				{\La(\calK)} \& {\Lali(\calK)} \\
				{\calK^{\isof}_\chor} \& {\calK^{\isof}_\chor}
				\arrow[from=1-1, to=1-2]
				\arrow[two heads, from=1-1, to=2-1]
				\arrow["\lrcorner"{anchor=center, pos=0.125}, draw=none, from=1-1, to=2-2]
				\arrow["U", two heads, from=1-2, to=2-2]
				\arrow["L"', from=2-1, to=2-2]
			\end{tikzcd}
		\end{equation}
		in $\FDeltaCat$. 
		%		
		%		Clearly, $U$ preserves flexible weighted limits, particularly terminally rigged $n$-inserters for chordate enhanced simplicial categories.
		
		By \longref{Lemma}{lem:enriched_pullback},  \longref{Lemma}{lem:create_tight}, and also \longref{Lemma}{lem:EM}, we are done.
	\end{proof}

	\subsubsection{\texorpdfstring{$\Ra(\calK)$}{Ra(K)}}
	Next, we introduce the $\infty$-cosmos $\calRa(\calK)$ for any arbitrary $\infty$-cosmos $\calK$.
	
	\begin{nota}
		Let $\calK$ be an $\infty$-cosmos. Denote by $\calRa(\calK)$ the simplicial sub-category of $\calK^{\isof}$ with objects given by isofibrations of $\calK$ which are \emph{right adjoints}, i.e. it is a right adjoint in the homotopy $2$-category $h\calK$. A morphism of isofibrations
		% https://q.uiver.app/#q=WzAsNCxbMCwwLCJFXzEiXSxbMCwxLCJCXzEiXSxbMSwwLCJFXzIiXSxbMSwxLCJCXzIiXSxbMCwxLCJwXzEiLDIseyJzdHlsZSI6eyJoZWFkIjp7Im5hbWUiOiJlcGkifX19XSxbMiwzLCJwXzIiLDAseyJzdHlsZSI6eyJoZWFkIjp7Im5hbWUiOiJlcGkifX19XSxbMCwyLCJlIl0sWzEsMywiYiIsMV1d
		\[\begin{tikzcd}[ampersand replacement=\&]
			{E_1} \& {E_2} \\
			{B_1} \& {B_2}
			\arrow["e", from=1-1, to=1-2]
			\arrow["{p_1}"', two heads, from=1-1, to=2-1]
			\arrow["{p_2}", two heads, from=1-2, to=2-2]
			\arrow["b"', from=2-1, to=2-2]
		\end{tikzcd}\]
		is a \emph{morphism of right adjoints} precisely if its mate is an isomorphism, i.e. the composite
		\[\begin{tikzcd}[ampersand replacement=\&]
			\& {E_1} \& {E_2} \& {E_2} \\
			{B_1} \& {B_1} \& {B_2}
			\arrow["e", from=1-2, to=1-3]
			\arrow["{p_1}"', two heads, from=1-2, to=2-2]
			\arrow[""{name=0, anchor=center, inner sep=0}, equal, from=1-3, to=1-4]
			\arrow["{p_2}", two heads, from=1-3, to=2-3]
			\arrow[""{name=1, anchor=center, inner sep=0}, "{l_1}", from=2-1, to=1-2]
			\arrow[""{name=2, anchor=center, inner sep=0}, equal, from=2-1, to=2-2]
			\arrow["b"', from=2-2, to=2-3]
			\arrow[""{name=3, anchor=center, inner sep=0}, "{l_2}"', from=2-3, to=1-4]
			\arrow["{\eta_1}"', shorten <=2pt, shorten >=2pt, Rightarrow, from=2, to=1]
			\arrow["{\epsilon_2}", shorten <=2pt, shorten >=2pt, Rightarrow, from=3, to=0]
		\end{tikzcd}\]
		is invertible, where $l_i$ denotes the left adjoint for $p_i$, and that $\epsilon_i$ and $\eta_i$ denote the corresponding counit and unit, respectively, for $i = 1, 2$.
		
		In other words, $\calRa(\calK) \subset \calK^{\isof}$ has objects given by right adjoints in $\calK$, $0$-arrows given by morphisms of right adjoints, which is full on higher dimensional arrows.
	\end{nota}
	
	For convenience, we introduce a name.
	
	\begin{defi}
		Let $\calK$ be an $\infty$-cosmos. An isofibration of $\calK$ is said to  be a \emph{ra-isofibration} precisely when it belongs to $\calRa(\calK)$. 
		
		A morphism of isofibrations is said to be \emph{a morphism of ra-isofibrations} precisely when it is a $0$-arrow in $\calRa(\calK)$. 
	\end{defi}

	\begin{pro}
		\label{pro:La_Ra}
		Let $\calK$ be an $\infty$-cosmos. We have $\La(\calK) = \Ra(\calK^\co)^\co$.
		
		As a consequence, the simplicial category $\calRa(\calK)$ is an $\infty$-cosmos.
	\end{pro}
	
	\begin{proof}
		Using similar arguments as in the proof of \longref{Lemma}{lem:Rali_Lali}, we conclude that $\calRa(\calK) = \calLa(\calK^\co)^\co$. By \longref{Proposition}{pro:La_cosmos}, we conclude that $\calRa(\calK)$ is also an $\infty$-cosmos.
	\end{proof}
	
	\begin{nota}
		Denote by $\Ra(\calK)$ for the enhanced simplicial category of ra-isofibrations of an $\infty$-cosmos $\calK$, whose tight part is exactly $\calRa(\calK)$, and whose loose $0$-arrows are given by the morphisms of isofibrations of $\calK$.
	\end{nota}
	
	We can now prove the completeness results dual to those in \longref{Section}{sec:La}.
	
	\begin{theorem}
		\label{thm:trio_Ra}
		The enhanced simplicial category $\Ra(\calK)$ admits tight cosmological limits, and initially rigged $n$-inserters, and that the projection of an initially rigged $n$-inserter is always a tight isofibration.
		
		Consequently, $\Ra(\calK)$ admits coEilenberg-Moore objects over monads which are not morphisms of right adjoints, and that the forgetful functor is a morphism of right adjoints and reflects morphisms of right adjoints.
		
		These $\F_\Delta$-weighted limits are all preserved by the inclusion
		$$\Ra(\calK) \hookrightarrow \calK^{\isof}_\chor.$$
	\end{theorem}
	
	\begin{proof}
		By \longref{Proposition}{pro:K^co_has_dual_lim} and \longref{Proposition}{pro:La_Ra}, we conclude that $\Ra(\calK)$ admits initially rigged $n$-inserters, and that the projection of an initially rigged $n$-inserter is always a tight isofibration. By considering ra-isofibrations, counits, and morphisms of ra-isofibrations, instead, in the proof of \longref{Proposition}{pro:tight_lim}, we see that tight cosmological limits exist in $\Ra(\calK)$. Now, by \longref{Lemma}{lem:EM}, we complete the proof.
	\end{proof}

	\subsubsection{\texorpdfstring{$\Cart(\calK)$}{Cart(K)} and \texorpdfstring{$\Cart(\calK)_{/B}$}{Cart(K)/B}}
	\label{sec:Cart_fixed}
	Recall from \longref{Example}{eg:CartK} that for an $\infty$-cosmos $\calK$, $\Cart(\calK)$ is the enhanced simplicial category of Cartesian fibrations between $\infty$-categories in $\calK$.
	
	\begin{theorem}
		\label{thm:trio_Cart}
		The enhanced simplicial category $\Cart(\calK)$ admits tight cosmological limits, and terminally rigged $n$-inserters, and that the projection of a terminally rigged $n$-inserter is always a tight isofibration.
		
		Consequently, $\Cart(\calK)$ admits Eilenberg-Moore objects over monads which are not Cartesian, and that the forgetful functor is a Cartesian functor, which reflects Cartesian functors.
		
		These $\F_\Delta$-weighted limits are all preserved by the inclusion
		$$\Cart(\calK) \hookrightarrow \calK^{\isof}_\chor.$$
	\end{theorem}
	
	\begin{proof}
		As described in \longref{Example}{eg:CartK} and the proof of \cite[Proposition 6.3.14]{book:RV:2022}, the $\infty$-cosmos  $\calCart(\calK)$ of Cartesian fibrations between $\infty$-categories of $\calK$ can be formed as the pullback
		% https://q.uiver.app/#q=WzAsNCxbMCwwLCJcXGNhbENhcnQoXFxjYWxLKSJdLFsxLDAsIlxcY2FsTGFsaShcXGNhbEspIl0sWzAsMSwiXFxjYWxLXntcXGlzb2Z9Il0sWzEsMSwiXFxjYWxLXntcXGlzb2Z9Il0sWzAsMV0sWzEsMywiVSIsMCx7InN0eWxlIjp7InRhaWwiOnsibmFtZSI6Imhvb2siLCJzaWRlIjoidG9wIn0sImhlYWQiOnsibmFtZSI6ImVwaSJ9fX1dLFswLDIsIiIsMCx7InN0eWxlIjp7InRhaWwiOnsibmFtZSI6Imhvb2siLCJzaWRlIjoidG9wIn0sImhlYWQiOnsibmFtZSI6ImVwaSJ9fX1dLFsyLDMsIktfMSIsMl0sWzAsMywiIiwwLHsic3R5bGUiOnsibmFtZSI6ImNvcm5lciJ9fV1d
		\[\begin{tikzcd}[ampersand replacement=\&]
			{\calCart(\calK)} \& {\calLali(\calK)} \\
			{\calK^{\isof}} \& {\calK^{\isof}}
			\arrow[from=1-1, to=1-2]
			\arrow[hook, two heads, from=1-1, to=2-1]
			\arrow["\lrcorner"{anchor=center, pos=0.125}, draw=none, from=1-1, to=2-2]
			\arrow["U", hook, two heads, from=1-2, to=2-2]
			\arrow["{K_1}"', from=2-1, to=2-2]
		\end{tikzcd},\]
		where 		$$
		\begin{aligned}
			K_1 \colon \calK^{\isof} &\to \calK^{\isof}
			\\
			(E \xtwoheadrightarrow{p} B) &\mapsto (E^{\Delta^1} \xtwoheadrightarrow{i_1 \hat{\pitchfork} p} {B\downarrow p})
		\end{aligned}
		$$
		sends each isofibration $p \colon E \twoheadrightarrow B$ of $\calK$ to its Leibniz exponential $i_1 \hat{\pitchfork} p \colon E^{\Delta^1} \twoheadrightarrow {B \downarrow p}$ with $i_1 \colon \Delta^0 \to \Delta^1$, and each morphism $(e, b) \colon p_1 \to p_2$ of isofibrations to the commutative square in \longref{Diagram}{diag:Leibniz}. 
		
		We can then view $K_1$ as an enhanced simplicial functor between chordate enhanced simplicial categories. Similar to the proof of \longref{Theorem}{thm:trio_JLim}, the above pullback diagram extends to a pullback
		% https://q.uiver.app/#q=WzAsNCxbMCwwLCJcXEpMaW0oXFxjYWxLKSJdLFsxLDAsIlxcTGFsaShcXGNhbEspIl0sWzAsMSwiXFxjYWxLIl0sWzEsMSwiXFxjYWxLXntcXGlzb2Z9Il0sWzAsMV0sWzEsMywiVSIsMCx7InN0eWxlIjp7InRhaWwiOnsibmFtZSI6Imhvb2siLCJzaWRlIjoidG9wIn0sImhlYWQiOnsibmFtZSI6ImVwaSJ9fX1dLFswLDIsIiIsMCx7InN0eWxlIjp7InRhaWwiOnsibmFtZSI6Imhvb2siLCJzaWRlIjoidG9wIn0sImhlYWQiOnsibmFtZSI6ImVwaSJ9fX1dLFsyLDMsIkZfe0pfXFx0cmlhbmdsZWxlZnR9IiwyXSxbMCwzLCIiLDAseyJzdHlsZSI6eyJuYW1lIjoiY29ybmVyIn19XV0=
		\begin{equation}
			\label{diag:Cart_pb}
			% https://q.uiver.app/#q=WzAsNCxbMCwwLCJcXENhcnQoXFxjYWxLKSJdLFsxLDAsIlxcTGFsaShcXGNhbEspIl0sWzAsMSwiXFxjYWxLXntcXGlzb2Z9Il0sWzEsMSwiXFxjYWxLXntcXGlzb2Z9Il0sWzAsMV0sWzEsMywiVSIsMCx7InN0eWxlIjp7ImhlYWQiOnsibmFtZSI6ImVwaSJ9fX1dLFswLDIsIiIsMCx7InN0eWxlIjp7ImhlYWQiOnsibmFtZSI6ImVwaSJ9fX1dLFsyLDMsIktfMSIsMl0sWzAsMywiIiwwLHsic3R5bGUiOnsibmFtZSI6ImNvcm5lciJ9fV1d
			\begin{tikzcd}[ampersand replacement=\&]
				{\Cart(\calK)} \& {\Lali(\calK)} \\
				{\calK^{\isof}_\chor} \& {\calK^{\isof}_\chor}
				\arrow[from=1-1, to=1-2]
				\arrow[two heads, from=1-1, to=2-1]
				\arrow["\lrcorner"{anchor=center, pos=0.125}, draw=none, from=1-1, to=2-2]
				\arrow["U", two heads, from=1-2, to=2-2]
				\arrow["{K_1}"', from=2-1, to=2-2]
			\end{tikzcd}
		\end{equation}
		in $\FDeltaCat$. 
		
		By \longref{Lemma}{lem:enriched_pullback}, it suffices to verify that $K_1$ preserves tight cosmological limits and terminally rigged $n$-inserters. In fact, Leibniz exponentials are also cosmological limits, and therefore commute with tight cosmological limits in $\calK^{\isof}$. Moreover, $K_1$, as a functor between chordate enhanced simplicial categories, preserves flexible weighted limits, including terminally rigged $n$-inserters.
	\end{proof}

	We would also like to establish completeness results also for the case of isofibrations over a fixed base.
	
	\begin{nota}
		Let $B$ be an $\infty$-category in $\calK$. Denote by $\calK^{\isof}_{/B}$ the \emph{$\infty$-cosmos of isofibrations over $B$}. By \cite[Proposition 1.2.22]{book:RV:2022}, $\calK^{\isof}_{/B}$ has
		\begin{enumerate}
			\item[$\bullet$] objects the isofibrations $p \colon E \twoheadrightarrow B$ with codomain $B$ of $\calK$;
			\item[$\bullet$] hom-object from $p_1 \colon E_1 \twoheadrightarrow B$ to $p_2 \colon E_2 \twoheadrightarrow B$ given by the pullback
			% https://q.uiver.app/#q=WzAsNCxbMCwwLCJcXGNhbEtee1xcaXNvZn1fey9CfShwXzEsIHBfMikiXSxbMCwxLCIxIl0sWzEsMCwiXFxjYWxLKEVfMSwgRV8yKSJdLFsxLDEsIlxcY2FsSyhFXzEsIEIpIl0sWzAsMl0sWzEsMywicF8xIiwyXSxbMCwxLCIiLDIseyJzdHlsZSI6eyJoZWFkIjp7Im5hbWUiOiJlcGkifX19XSxbMiwzLCJ7cF8yfV8qIiwwLHsic3R5bGUiOnsiaGVhZCI6eyJuYW1lIjoiZXBpIn19fV0sWzAsMywiIiwxLHsic3R5bGUiOnsibmFtZSI6ImNvcm5lciJ9fV1d
			\[\begin{tikzcd}[ampersand replacement=\&]
				{\calK^{\isof}_{/B}(p_1, p_2)} \& {\calK(E_1, E_2)} \\
				1 \& {\calK(E_1, B)}
				\arrow[from=1-1, to=1-2]
				\arrow[two heads, from=1-1, to=2-1]
				\arrow["\lrcorner"{anchor=center, pos=0.125}, draw=none, from=1-1, to=2-2]
				\arrow["{{p_2}_*}", two heads, from=1-2, to=2-2]
				\arrow["{p_1}"', from=2-1, to=2-2]
			\end{tikzcd},\]
			where $1$ denotes the terminal $\infty$-category in $\calK$;
			\item[$\bullet$] isofibrations given by commutative triangles
			% https://q.uiver.app/#q=WzAsMyxbMCwwLCJFXzEiXSxbMiwwLCJFXzIiXSxbMSwxLCJCIl0sWzAsMSwiZSIsMCx7InN0eWxlIjp7ImhlYWQiOnsibmFtZSI6ImVwaSJ9fX1dLFsxLDIsInBfMiIsMCx7InN0eWxlIjp7ImhlYWQiOnsibmFtZSI6ImVwaSJ9fX1dLFswLDIsInBfMSIsMix7InN0eWxlIjp7ImhlYWQiOnsibmFtZSI6ImVwaSJ9fX1dXQ==
			\[\begin{tikzcd}[ampersand replacement=\&]
				{E_1} \&\& {E_2} \\
				\& B
				\arrow["e", two heads, from=1-1, to=1-3]
				\arrow["{p_1}"', two heads, from=1-1, to=2-2]
				\arrow["{p_2}", two heads, from=1-3, to=2-2]
			\end{tikzcd};\]\
			\item[$\bullet$] terminal object given by the identity $B \twoheadrightarrow$ on $B$;
			\item[$\bullet$] products $\times^B_i p_i$ of $\{p_i \colon E_i \twoheadrightarrow B\}$ are given by the pullback
			% https://q.uiver.app/#q=WzAsNCxbMCwwLCJcXHRpbWVzXkJfaSBFX2kiXSxbMCwxLCJCIl0sWzEsMSwiXFxwcm9kX2kgQiJdLFsxLDAsIlxccHJvZF9pIEVfaSJdLFszLDIsIlxccHJvZF9pIHBfaSIsMCx7InN0eWxlIjp7ImhlYWQiOnsibmFtZSI6ImVwaSJ9fX1dLFsxLDIsIlxcRGVsdGEiLDJdLFswLDEsIlxcdGltZXNeQl9pIHBfaSIsMix7InN0eWxlIjp7ImhlYWQiOnsibmFtZSI6ImVwaSJ9fX1dLFswLDNdLFswLDIsIiIsMSx7InN0eWxlIjp7Im5hbWUiOiJjb3JuZXIifX1dXQ==
			\[\begin{tikzcd}[ampersand replacement=\&]
				{\times^B_i E_i} \& {\prod_i E_i} \\
				B \& {\prod_i B}
				\arrow[from=1-1, to=1-2]
				\arrow["{\times^B_i p_i}"', two heads, from=1-1, to=2-1]
				\arrow["\lrcorner"{anchor=center, pos=0.125}, draw=none, from=1-1, to=2-2]
				\arrow["{\prod_i p_i}", two heads, from=1-2, to=2-2]
				\arrow["\Delta"', from=2-1, to=2-2]
			\end{tikzcd}\]
			of the product $\prod_i p_i$ along the diagonal;
			\item[$\bullet$] simplicial powers $U \pitchfork_B p$ by a simplicial set $U$ given by the pullback
			% https://q.uiver.app/#q=WzAsNCxbMCwwLCJVIFxccGl0Y2hmb3JrX0IgRSJdLFswLDEsIkIiXSxbMSwxLCJCXlUiXSxbMSwwLCJFXlUiXSxbMywyLCJwXlUiLDAseyJzdHlsZSI6eyJoZWFkIjp7Im5hbWUiOiJlcGkifX19XSxbMSwyLCJcXERlbHRhIiwyXSxbMCwxLCJVIFxccGl0Y2hmb3JrX0IgcCIsMix7InN0eWxlIjp7ImhlYWQiOnsibmFtZSI6ImVwaSJ9fX1dLFswLDNdLFswLDIsIiIsMSx7InN0eWxlIjp7Im5hbWUiOiJjb3JuZXIifX1dXQ==
			\[\begin{tikzcd}[ampersand replacement=\&]
				{U \pitchfork_B E} \& {E^U} \\
				B \& {B^U}
				\arrow[from=1-1, to=1-2]
				\arrow["{U \pitchfork_B p}"', two heads, from=1-1, to=2-1]
				\arrow["\lrcorner"{anchor=center, pos=0.125}, draw=none, from=1-1, to=2-2]
				\arrow["{p^U}", two heads, from=1-2, to=2-2]
				\arrow["\Delta"', from=2-1, to=2-2]
			\end{tikzcd}\]
			of $p^U$ along the diagonal;
			\item[$\bullet$] pullbacks and limits of a chain of isofibrations created by the forgetful functor $\calK^{\isof}_{/B} \to \calK$.
		\end{enumerate}
		Furthermore, an equivalence from $p_1$ to $p_2$ in $\calK^{\isof}_{/B}$ is given by a commutative triangle
		\[
		\begin{tikzcd}[ampersand replacement=\&]
			{E_1} \&\& {E_2} \\
			\& B
			\arrow["e", "\simeq"',  from=1-1, to=1-3]
			\arrow["{p_1}"', two heads, from=1-1, to=2-2]
			\arrow["{p_2}", two heads, from=1-3, to=2-2]
		\end{tikzcd}
		\]
		whose top row is an equivalence in $\calK$.
		
		Denote by $\calCart(\calK)_{/B}$ the \emph{$\infty$-cosmos of Cartesian fibrations over $B$}. It is the simplicial full sub-category of $\calK^{\isof}_{/B}$, determined by the Cartesian fibrations.
		
		%    	Denote by $\calDisCart(\calK)_{/B}$ the  \emph{$\infty$-cosmos of discrete Cartesian fibrations over $B$}. It is the simplicial full sub-category of $\calK^{\isof}_{/B}$, determined by the discrete Cartesian fibrations.
		
		We write $\Cart(\calK)_{/B}$ for the enhanced simplicial category of Cartesian fibrations over $B$. Its tight part is given by the $\infty$-cosmos $\calCart(\calK)_{/B}$, and the loose $0$-arrows are simply morphisms of isofibrations, i.e., $0$-arrows in $\calK^{\isof}_{/B}$. 
		%    	Similarly, we write $\DisCart(\calK)_{/B}$ for the enhanced simplicial category of discrete Cartesian fibrations over $B$.
	\end{nota}
	
	Yet, unlike the previous examples, the case for  Cartesian fibrations over a fixed base cannot be done using the same technique. Thie is because the $\infty$-cosmos of Cartesian fibrations over a fixed base is formed as a pullback of a \emph{non-cosmological functor} along a cosmological embedding, which then fails to fulfill the condition of \longref{Lemma}{lem:enriched_pullback}. Instead, we have to argue more directly.
	
	\begin{theorem}
		\label{thm:trio_Cart/B}
		The enhanced simplicial category $\Cart(\calK)_{/B}$ admits tight cosmological limits, and terminally rigged $n$-inserters, and that the projection of a terminally rigged $n$-inserter is always a tight isofibration.
		
		Consequently, $\Cart(\calK)_{/B}$ admits Eilenberg-Moore objects over monads which are not Cartesian functors, and that the forgetful functor is a Cartesian functor, which reflects Cartesian functors.
		
		These $\F_\Delta$-weighted limits are all preserved by the inclusion
		$$\Cart(\calK)_{/B} \hookrightarrow \calK^{\isof}_{/B \;\chor}.$$
	\end{theorem}
	
	\begin{proof}
		Recall in the proof of \cite[Proposition 6.3.14]{book:RV:2022}, the $\infty$-cosmos  $\calCart(\calK)_{/B}$ of Cartesian fibrations over $B$ of an $\infty$-cosmos $\calK$ can be formed as the pullback
		% https://q.uiver.app/#q=WzAsNCxbMCwwLCJcXGNhbENhcnQoXFxjYWxLKV97L0J9Il0sWzEsMCwiXFxjYWxDYXJ0KFxcY2FsSykiXSxbMCwxLCJcXGNhbEtee1xcaXNvZn1fey9CfSJdLFsxLDEsIlxcY2FsS157XFxpc29mfSJdLFswLDFdLFsxLDMsIlUiLDAseyJzdHlsZSI6eyJ0YWlsIjp7Im5hbWUiOiJob29rIiwic2lkZSI6InRvcCJ9LCJoZWFkIjp7Im5hbWUiOiJlcGkifX19XSxbMCwyLCIiLDAseyJzdHlsZSI6eyJ0YWlsIjp7Im5hbWUiOiJob29rIiwic2lkZSI6InRvcCJ9LCJoZWFkIjp7Im5hbWUiOiJlcGkifX19XSxbMiwzLCJJIiwyLHsic3R5bGUiOnsidGFpbCI6eyJuYW1lIjoiaG9vayIsInNpZGUiOiJ0b3AifX19XSxbMCwzLCIiLDAseyJzdHlsZSI6eyJuYW1lIjoiY29ybmVyIn19XV0=
		\[\begin{tikzcd}[ampersand replacement=\&]
			{\calCart(\calK)_{/B}} \& {\calCart(\calK)} \\
			{\calK^{\isof}_{/B}} \& {\calK^{\isof}}
			\arrow[from=1-1, to=1-2]
			\arrow[hook, two heads, from=1-1, to=2-1]
			\arrow["\lrcorner"{anchor=center, pos=0.125}, draw=none, from=1-1, to=2-2]
			\arrow["U", hook, two heads, from=1-2, to=2-2]
			\arrow["I"', hook, from=2-1, to=2-2]
		\end{tikzcd}\]
		where $I$ denotes the inclusion.	We show that the inclusion
		$$\Cart(\calK)_{/B} \hooktwoheadrightarrow \calK^{\isof}_{/B \;\chor}$$
		on the left creates terminally rigged $n$-inserters and tight cosmological limits.
		
		Let $D \colon a \to \partial \Delta^m \pitchfork_B c$ be a diagram in $\calK^{\isof}_{/B}$, where $a \colon A \twoheadrightarrow B$ and $c \colon C \twoheadrightarrow B$ are Cartesian fibrations over $B$. Suppose that the composite $D \cdot \ev[[m]]$
		% https://q.uiver.app/#q=WzAsNixbMCwwLCJBIl0sWzEsMCwiXFxwYXJ0aWFsXFxEZWx0YV5tIFxccGl0Y2hmb3JrX0IgQyJdLFsyLDAsIkMiXSxbMCwxLCJCIl0sWzEsMSwiQiJdLFsyLDEsIkIiXSxbMSwyLCJcXGV2W1ttXV0iXSxbMCwxLCJEIl0sWzAsMywiYSIsMix7InN0eWxlIjp7ImhlYWQiOnsibmFtZSI6ImVwaSJ9fX1dLFsxLDQsIlxccGFydGlhbFxcRGVsdGFebVxccGl0Y2hmb3JrX0IgQyJdLFszLDQsIiIsMix7ImxldmVsIjoyLCJzdHlsZSI6eyJoZWFkIjp7Im5hbWUiOiJub25lIn19fV0sWzQsNSwiIiwyLHsibGV2ZWwiOjIsInN0eWxlIjp7ImhlYWQiOnsibmFtZSI6Im5vbmUifX19XSxbMiw1LCJjIl1d
		\[\begin{tikzcd}[ampersand replacement=\&]
			A \& {\partial\Delta^m \pitchfork_B C} \& C \\
			B \& B \& B
			\arrow["D", from=1-1, to=1-2]
			\arrow["a"', two heads, from=1-1, to=2-1]
			\arrow["{\ev[[m]]}", from=1-2, to=1-3]
			\arrow["{\partial\Delta^m\pitchfork_B c}", from=1-2, to=2-2]
			\arrow["c", from=1-3, to=2-3]
			\arrow[equals, from=2-1, to=2-2]
			\arrow[equals, from=2-2, to=2-3]
		\end{tikzcd}\]
		is a Cartesian functor. By \cite[Proposition 1.2.22, Theorem 5.2.8, Theorem 5.3.4]{book:RV:2022}, the pullback of $D$ along the restriction $\Delta^m\pitchfork C \twoheadrightarrow \partial\Delta^m\pitchfork C$ in $\calK^{\isof}_{/B}$ is equivalent to the commutative cube
		% https://q.uiver.app/#q=WzAsOCxbMCwxLCJBXntcXERlbHRhXjF9Il0sWzIsMSwiKFxccGFydGlhbFxcRGVsdGFebSBcXHBpdGNoZm9ya19CIEMpXntcXERlbHRhXjF9Il0sWzAsMywiQiBcXGRvd25hcnJvdyBhIl0sWzIsMywiQiBcXGRvd25hcnJvdyAoXFxwYXJ0aWFsXFxEZWx0YV5tIFxccGl0Y2hmb3JrX0IgYykiXSxbMSwwLCIoXFxyaW5ze259e259e0R9XzEpXntcXERlbHRhXjF9Il0sWzMsMCwiKFxcRGVsdGFebSBcXHBpdGNoZm9ya19CIEMpXntcXERlbHRhXjF9Il0sWzEsMiwiQlxcZG93bmFycm93IFxccmluc3tufXtufXtEfSJdLFszLDIsIkIgXFxkb3duYXJyb3cgKFxcRGVsdGFebSBcXHBpdGNoZm9ya19CIGMpIl0sWzAsMSwiaV8xIFxcaGF0e1xccGl0Y2hmb3JrfUQiLDFdLFswLDIsImlfMSBcXGhhdHtcXHBpdGNoZm9ya31hIiwyLHsic3R5bGUiOnsiaGVhZCI6eyJuYW1lIjoiZXBpIn19fV0sWzEsMywiXFxzY3JpcHRzdHlsZSB7aV8xIFxcaGF0e1xccGl0Y2hmb3JrfShcXHBhcnRpYWxcXERlbHRhXm1cXHBpdGNoZm9ya19CIGMpfSIsMSx7InN0eWxlIjp7ImhlYWQiOnsibmFtZSI6ImVwaSJ9fX1dLFsyLDMsIkIgXFxkb3duYXJyb3cgRCIsMl0sWzQsMF0sWzQsNV0sWzUsMSwiKFxccGFydGlhbFxccGl0Y2hmb3JrX0IgQylee1xcRGVsdGFeMX0iLDAseyJsYWJlbF9wb3NpdGlvbiI6MzB9XSxbNSw3LCJcXHNjcmlwdHN0eWxlIHtpXzEgXFxoYXR7XFxwaXRjaGZvcmt9KFxcRGVsdGFebVxccGl0Y2hmb3JrX0IgYyl9IiwxLHsic3R5bGUiOnsiaGVhZCI6eyJuYW1lIjoiZXBpIn19fV0sWzQsNiwiIFxccXVhZCIsMSx7InN0eWxlIjp7ImhlYWQiOnsibmFtZSI6ImVwaSJ9fX1dLFs2LDJdLFs2LDcsIlxccXVhZFxccXVhZFxccXVhZFxccXVhZCIsMV0sWzcsMywiQiBcXGRvd25hcnJvdyAoXFxwYXJ0aWFsXFxwaXRjaGZvcmtfQiBjKSIsMCx7ImxhYmVsX3Bvc2l0aW9uIjowfV0sWzYsMTEsIiIsMSx7ImxldmVsIjoxLCJzdHlsZSI6eyJuYW1lIjoiY29ybmVyIn19XSxbNCw4LCIiLDEseyJsYWJlbF9wb3NpdGlvbiI6MCwib2Zmc2V0IjotMiwibGV2ZWwiOjEsInN0eWxlIjp7Im5hbWUiOiJjb3JuZXIifX1dXQ==
		\[\begin{tikzcd}[ampersand replacement=\&]
			\& {(\rins{n}{n}{D}_1)^{\Delta^1}} \&\& {(\Delta^m \pitchfork_B C)^{\Delta^1}} \\
			{A^{\Delta^1}} \&\& {(\partial\Delta^m \pitchfork_B C)^{\Delta^1}} \\
			\& {B\downarrow \rins{n}{n}{D}} \&\& {B \downarrow (\Delta^m \pitchfork_B c)} \\
			{B \downarrow a} \&\& {B \downarrow (\partial\Delta^m \pitchfork_B c)}
			\arrow[from=1-2, to=1-4]
			\arrow[from=1-2, to=2-1]
			\arrow["{ \quad}"{description}, two heads, from=1-2, to=3-2]
			\arrow["{(\partial\pitchfork_B C)^{\Delta^1}}"{pos=0.3}, from=1-4, to=2-3]
			\arrow["{\scriptstyle {i_1 \hat{\pitchfork}(\Delta^m\pitchfork_B c)}}"{description}, two heads, from=1-4, to=3-4]
			\arrow[""{name=0, anchor=center, inner sep=0}, "{i_1 \hat{\pitchfork}D}"{description}, from=2-1, to=2-3]
			\arrow["{i_1 \hat{\pitchfork}a}"', two heads, from=2-1, to=4-1]
			\arrow["{}"{description}, from=3-2, to=3-4]
			\arrow["{\scriptstyle {i_1 \hat{\pitchfork}(\partial\Delta^m\pitchfork_B c)}}"{description}, two heads, from=2-3, to=4-3, crossing over]
			\arrow[from=3-2, to=4-1]
			\arrow["{B \downarrow (\partial\pitchfork_B c)}"{pos=0}, from=3-4, to=4-3]
			\arrow[""{name=1, anchor=center, inner sep=0}, "{B \downarrow D}"', from=4-1, to=4-3]
			\arrow["\lrcorner"{anchor=center, pos=0.125, rotate=-45}, shift left=2, draw=none, from=1-2, to=0]
			\arrow["\lrcorner"{anchor=center, pos=0.125, rotate=-45}, draw=none, from=3-2, to=1]
		\end{tikzcd}\]
		in  $\calK$, where the top and the bottom squares are pullbacks, $i_1 \hat{\pitchfork} a$, ${i_1 \hat{\pitchfork}(\partial\Delta^m\pitchfork_B c)}$ and ${i_1 \hat{\pitchfork}(\Delta^m\pitchfork_B c)}$ are lalis, and that the composite
		% https://q.uiver.app/#q=WzAsNixbMCwwLCJBXntcXERlbHRhXjF9Il0sWzEsMCwiKFxccGFydGlhbFxcRGVsdGFebSBcXHBpdGNoZm9ya19CIEMpXntcXERlbHRhXjF9Il0sWzAsMSwiQiBcXGRvd25hcnJvdyBhIl0sWzEsMSwiQiBcXGRvd25hcnJvdyAoXFxwYXJ0aWFsXFxEZWx0YV5tIFxccGl0Y2hmb3JrX0IgYykiXSxbMiwwLCJDXntcXERlbHRhXjF9Il0sWzIsMSwiQiBcXGRvd25hcnJvdyBjIl0sWzAsMSwiaV8xIFxcaGF0e1xccGl0Y2hmb3JrfUQiXSxbMCwyLCJpXzEgXFxoYXR7XFxwaXRjaGZvcmt9YSIsMix7InN0eWxlIjp7ImhlYWQiOnsibmFtZSI6ImVwaSJ9fX1dLFsxLDMsIlxcc2NyaXB0c3R5bGUge2lfMSBcXGhhdHtcXHBpdGNoZm9ya30oXFxwYXJ0aWFsXFxEZWx0YV5tXFxwaXRjaGZvcmtfQiBjKX0iLDEseyJzdHlsZSI6eyJoZWFkIjp7Im5hbWUiOiJlcGkifX19XSxbMiwzLCJCIFxcZG93bmFycm93IEQiLDJdLFsxLDQsIntcXGV2W1ttXV19XntcXERlbHRhXjF9IiwwLHsic3R5bGUiOnsiaGVhZCI6eyJuYW1lIjoiZXBpIn19fV0sWzMsNSwiXFxldltbbV1dIiwyLHsic3R5bGUiOnsiaGVhZCI6eyJuYW1lIjoiZXBpIn19fV0sWzQsNSwiaV8xIFxcaGF0e1xccGl0Y2hmb3JrfWMiLDAseyJzdHlsZSI6eyJoZWFkIjp7Im5hbWUiOiJlcGkifX19XV0=
		\[\begin{tikzcd}[ampersand replacement=\&]
			{A^{\Delta^1}} \& {(\partial\Delta^m \pitchfork_B C)^{\Delta^1}} \& {C^{\Delta^1}} \\
			{B \downarrow a} \& {B \downarrow (\partial\Delta^m \pitchfork_B c)} \& {B \downarrow c}
			\arrow["{i_1 \hat{\pitchfork}D}", from=1-1, to=1-2]
			\arrow["{i_1 \hat{\pitchfork}a}"', two heads, from=1-1, to=2-1]
			\arrow["{{\ev[[m]]}^{\Delta^1}}", two heads, from=1-2, to=1-3]
			\arrow["{\scriptstyle {i_1 \hat{\pitchfork}(\partial\Delta^m\pitchfork_B c)}}"{description}, two heads, from=1-2, to=2-2]
			\arrow["{i_1 \hat{\pitchfork}c}", two heads, from=1-3, to=2-3]
			\arrow["{B \downarrow D}"', from=2-1, to=2-2]
			\arrow["{\ev[[m]]}"', two heads, from=2-2, to=2-3]
		\end{tikzcd}\]
		is a Cartesian functor.
		
		Now by \longref{Theorem}{thm:Lali(K)}, we deduce that $(\rins{n}{n}{D}_1)^{\Delta^1} \twoheadrightarrow B \downarrow \rins{n}{n}{D}$ is a lali, and the projection, i.e., the left square in the cube, is an isofibration of $\calK^{\isof}$. This means that $\rins{n}{n}{D}$ is a Cartesian fibration, and that the projection  $\rins{n}{n}{D} \twoheadrightarrow a$ is a Cartesian functor, which reflects Cartesian functors.
		
		For tight cosmological limits, it is left to verify that the limit projections jointly reflect Cartesian functors.
		
		Suppose we have a cosmological limit cone, in $\calCart(\calK)_{/B}$, which is also a cosmological limit cone in $\calK^{\isof}_{/B}$, because the inclusion $\Cart(\calK)_{/B} \hooktwoheadrightarrow \calK^{\isof}_{/B \;\chor}$ is a cosmological embedding as it is the pullback of a cosmological embedding. Taking Leibniz exponentials with $i_1$ on the $0$-arrows, the limit cone in $\calK^{\isof}_{/B}$ becomes a limit cone in $\calLali(\calK)$, since Leibniz exponentials commute with cosmological limits. Then, by \longref{Proposition}{pro:tight_lim}, we conclude that the limit projections of this limit cone in $\calLali(\calK)$ jointly reflect morphisms of lalis. According to \cite[Theorem 5.3.4]{book:RV:2022}, the limit projections of the corresponding limit cone in $\calK^{\isof}_{/B}$ jointly reflect Cartesian functors.
	\end{proof}

	\subsubsection{\texorpdfstring{$\CoCart(\calK)$}{CoCart(K)} and \texorpdfstring{$\CoCart(\calK)_{/B}$}{CoCart(K)/B}}
	\label{sec:CoCart_fixed}
	We proceed to \emph{coCartesian} fibrations.
	
	Recall from \longref{Example}{eg:CartK} that for an $\infty$-cosmos $\calK$, $\CoCart(\calK)$ is the enhanced simplicial category of \emph{coCartesian} fibrations between $\infty$-categories in $\calK$.
	
	\begin{theorem}
		\label{thm:trio_CoCart}
		The enhanced simplicial category $\CoCart(\calK)$ admits tight cosmological limits, and initially rigged $n$-inserters, and that the projection of an initially rigged $n$-inserter is always a tight isofibration.
		
		Consequently, $\CoCart(\calK)$ admits coEilenberg-Moore objects over comonads which are not coCartesian functors, and that the forgetful functor is coCartesian, which reflects coCartesian functors.
		
		These $\F_\Delta$-weighted limits are all preserved by the inclusion
		$$\CoCart(\calK) \hookrightarrow \calK^{\isof}_\chor.$$
	\end{theorem}
	
	\begin{proof}
		Similar to the proof of \longref{Theorem}{thm:trio_Cart} for $\Cart(\calK)$, the enhanced simplicial category $\CoCart(\calK)$ can be obtained as a pullback of $\Rali(\calK)$ that satisfies the conditions in \longref{Lemma}{lem:enriched_pullback}.
	\end{proof}

	We now discuss coCartesian fibrations over a fixed base.
	
	\begin{nota}
		Let $B$ be an $\infty$-category in $\calK$. 
		Denote by $\calCoCart(\calK)_{/B}$ the \emph{$\infty$-cosmos of coCartesian fibrations over $B$}. It is the simplicial full sub-category of $\calK^{\isof}_{/B}$, determined by the coCartesian fibrations.
		
		%   		Denote by $\calDisCoCart(\calK)_{/B}$ the  \emph{$\infty$-cosmos of discrete coCartesian fibrations over $B$}. It is the simplicial full sub-category of $\calK^{\isof}_{/B}$, determined by the discrete coCartesian fibrations.
		
		We write $\CoCart(\calK)_{/B}$ for the enhanced simplicial category of coCartesian fibrations over $B$. Its tight part is given by the $\infty$-cosmos $\calCoCart(\calK)_{/B}$, and the loose $0$-arrows are simply morphisms of isofibrations, i.e., $0$-arrows in $\calK^{\isof}_{/B}$. 
		%   		Similarly, we write $\DisCoCart(\calK)_{/B}$ for the enhanced simplicial category of discrete coCartesian fibrations over $B$.
	\end{nota}
	
	%   Again, for the case of a fixed base, we cannot employ the previous way of proving. We will verify the statements straightforwardly,. 

	\begin{theorem}
		\label{thm:trio_CoCart/B}
		The enhanced simplicial category $\CoCart(\calK)_{/B}$ admits tight cosmological limits, and initially rigged $n$-inserters, and that the projection of an initially rigged $n$-inserter is always a tight isofibration.
		
		Consequently, $\CoCart(\calK)_{/B}$ admits coEilenberg-Moore objects over comonads which are not coCartesian functors, and that the forgetful functor is a coCartesian, which reflects coCartesian functors.
		
		These $\F_\Delta$-weighted limits are all preserved by the inclusion
		$$\CoCart(\calK)_{/B} \hookrightarrow \calK^{\isof}_{/B \;\chor}.$$
	\end{theorem}
	
	\begin{proof}
		A straightforward adaptation to the proof of \longref{Theorem}{thm:trio_Cart/B} suffices.
	\end{proof}

	\subsubsection{Analogous results for two-sided fibrations  between $(\infty, 1)$-categories}
	In \cite[Chapter 7]{book:RV:2022}, the notions of two-sided fibrations and modules between $(\infty, 1)$-categories are introduced.
	
	\begin{defi}[{\cite[Chapter 7]{book:RV:2022}}]
		Let $\calK$ be an $\infty$-cosmos, and $A, B$ be $\infty$-categories of $\calK$. A \emph{two-sided isofibration} from $A$ to $B$ is an $\infty$-category in the $\infty$-cosmos $\calK^{\isof}_{/A\times B}$ of isofibrations over $A \times B$.
		
		Equivalently, a two-sided isofibration from $A$ to $B$ is a span $A \xtwoheadleftarrow{p} E \xtwoheadrightarrow{q} B$ in $\calK$ such that $(p, q) \colon E \twoheadrightarrow A \times B$ is an isofibration.
	\end{defi}
	
	\begin{defi}[{\cite[Corollary 7.1.3]{book:RV:2022}}]
		Let $\calK$ be an $\infty$-cosmos, and $A, B$ be $\infty$-categories of $\calK$. A \emph{two-sided fibration} from $A$ to $B$ is a two-sided isofibration $E \twoheadrightarrow A \times B$ of $\calK$, which is {coCartesian on the left} and also {Cartesian on the right}, i.e., the morphism of isofibrations over $A$
		% https://q.uiver.app/#q=WzAsMyxbMCwwLCJFIl0sWzIsMCwiQVxcdGltZXMgQiJdLFsxLDEsIkEiXSxbMCwyLCJwIiwyLHsic3R5bGUiOnsiaGVhZCI6eyJuYW1lIjoiZXBpIn19fV0sWzEsMiwiYSIsMCx7InN0eWxlIjp7ImhlYWQiOnsibmFtZSI6ImVwaSJ9fX1dLFswLDEsIihwLCBxKSIsMCx7InN0eWxlIjp7ImhlYWQiOnsibmFtZSI6ImVwaSJ9fX1dXQ==
		\begin{equation}
			\label{diag:A}
			\begin{tikzcd}[ampersand replacement=\&]
				E \&\& {A\times B} \\
				\& A
				\arrow["{(p, q)}", two heads, from=1-1, to=1-3]
				\arrow["p"', two heads, from=1-1, to=2-2]
				\arrow["a", two heads, from=1-3, to=2-2]
			\end{tikzcd}
		\end{equation}
		belongs to $\calCoCart(\calK)_{/A}$, and defines a Cartesian fibration of $\calK^{\isof}_{/A}$; or, equivalently, the morphism of isofibrations over $B$
		% https://q.uiver.app/#q=WzAsMyxbMCwwLCJFIl0sWzIsMCwiQVxcdGltZXMgQiJdLFsxLDEsIkIiXSxbMCwyLCJxIiwyLHsic3R5bGUiOnsiaGVhZCI6eyJuYW1lIjoiZXBpIn19fV0sWzEsMiwiYiIsMCx7InN0eWxlIjp7ImhlYWQiOnsibmFtZSI6ImVwaSJ9fX1dLFswLDEsIihwLCBxKSIsMCx7InN0eWxlIjp7ImhlYWQiOnsibmFtZSI6ImVwaSJ9fX1dXQ==
		\begin{equation}
			\label{diag:B}
			\begin{tikzcd}[ampersand replacement=\&]
				E \&\& {A\times B} \\
				\& B
				\arrow["{(p, q)}", two heads, from=1-1, to=1-3]
				\arrow["q"', two heads, from=1-1, to=2-2]
				\arrow["b", two heads, from=1-3, to=2-2]
			\end{tikzcd}
		\end{equation}
		belongs to $\calCart(\calK)_{/B}$, and defines a coCartesian fibration of $\calK^{\isof}_{/B}$. Here, $a \colon A \times B \twoheadrightarrow A$ and $b \colon A \times B \twoheadrightarrow B$ denote the canonical projections, such that \longref{Diagram}{diag:A} defines a Cartesian fibration of $\calCoCart(\calK)_{/A}$; or, equivalently, \longref{Diagram}{diag:B} defines a coCartesian fibration of $\calCart(\calK)_{/B}$.
	\end{defi}
	
	\begin{defi}[{\cite[Proposition 7.1.7]{book:RV:2022}}]
		Let $A \xtwoheadleftarrow{p_1} E_1 \xtwoheadrightarrow{q_1} B$ and $A \xtwoheadleftarrow{p_2} E_2 \xtwoheadrightarrow{q_2} B$ be two-sided fibrations from $A$ to $B$ of an $\infty$-cosmos $\calK$. A map of spans
		\begin{equation}
			\label{diag:map_of_spans}
			% https://q.uiver.app/#q=WzAsNCxbMSwwLCJFXzEiXSxbMSwyLCJFXzIiXSxbMCwxLCJBIl0sWzIsMSwiQiJdLFswLDIsInBfMSIsMix7InN0eWxlIjp7ImhlYWQiOnsibmFtZSI6ImVwaSJ9fX1dLFswLDMsInFfMSIsMCx7InN0eWxlIjp7ImhlYWQiOnsibmFtZSI6ImVwaSJ9fX1dLFsxLDIsInBfMiIsMCx7InN0eWxlIjp7ImhlYWQiOnsibmFtZSI6ImVwaSJ9fX1dLFsxLDMsInFfMiIsMix7InN0eWxlIjp7ImhlYWQiOnsibmFtZSI6ImVwaSJ9fX1dLFswLDEsImUiLDFdXQ==
			\begin{tikzcd}[ampersand replacement=\&]
				\& {E_1} \\
				A \&\& B \\
				\& {E_2}
				\arrow["{p_1}"', two heads, from=1-2, to=2-1]
				\arrow["{q_1}", two heads, from=1-2, to=2-3]
				\arrow["e"{description}, from=1-2, to=3-2]
				\arrow["{p_2}", two heads, from=3-2, to=2-1]
				\arrow["{q_2}"', two heads, from=3-2, to=2-3]
			\end{tikzcd}
		\end{equation}
		is a \emph{Cartesian functor} precisely if the equivalent conditions
		\begin{enumerate}
			\item[$\bullet$] $e$ defines a Cartesian functor between Cartesian fibrations of $\calCoCart(\calK)_{/A}$;
			\item[$\bullet$] $e$ defines a coCartesian functor between coCartesian fibrations of $\calCart(\calK)_{/B}$;
			\item[$\bullet$] the mates 
			\begin{equation}
				\label{diag:mates_two_sided}
				% https://q.uiver.app/#q=WzAsNCxbMCwwLCJFXzEiXSxbMSwwLCJFXzIiXSxbMCwxLCJwXzEgXFxkb3duYXJyb3cgQSJdLFsxLDEsInBfMiBcXGRvd25hcnJvdyBBIl0sWzAsMSwiZSJdLFsyLDBdLFszLDFdLFsyLDMsImUgXFxkb3duYXJyb3cgQSIsMl0sWzMsMCwiXFxjb25nIiwwLHsic2hvcnRlbiI6eyJzb3VyY2UiOjMwLCJ0YXJnZXQiOjQwfSwibGV2ZWwiOjJ9XV0=
				\begin{tikzcd}[ampersand replacement=\&]
					{E_1} \& {E_2} \\
					{p_1 \downarrow A} \& {p_2 \downarrow A}
					\arrow["e", from=1-1, to=1-2]
					\arrow[from=2-1, to=1-1]
					\arrow["{e \downarrow A}"', from=2-1, to=2-2]
					\arrow["\cong", shorten <=7pt, shorten >=9pt, Rightarrow, from=2-2, to=1-1]
					\arrow[from=2-2, to=1-2]
				\end{tikzcd}
				\quad \quad
				% https://q.uiver.app/#q=WzAsNCxbMCwwLCJFXzEiXSxbMSwwLCJFXzIiXSxbMCwxLCJCIFxcZG93bmFycm93IHFfMSJdLFsxLDEsIkIgXFxkb3duYXJyb3cgcV8yIl0sWzAsMSwiZSJdLFsyLDBdLFszLDFdLFsyLDMsIkIgXFxkb3duYXJyb3cgZSIsMl0sWzAsMywiXFxjb25nIiwyLHsic2hvcnRlbiI6eyJzb3VyY2UiOjMwLCJ0YXJnZXQiOjQwfSwibGV2ZWwiOjJ9XV0=
				\begin{tikzcd}[ampersand replacement=\&]
					{E_1} \& {E_2} \\
					{B \downarrow q_1} \& {B \downarrow q_2}
					\arrow["e", from=1-1, to=1-2]
					\arrow["\cong"', shorten <=7pt, shorten >=9pt, Rightarrow, from=1-1, to=2-2]
					\arrow[from=2-1, to=1-1]
					\arrow["{B \downarrow e}"', from=2-1, to=2-2]
					\arrow[from=2-2, to=1-2]
				\end{tikzcd}
			\end{equation}
			(of the canonical isomorphisms) are invertible in (the homotopy $2$-category of) $\calK_{/A \times B}$;
		\end{enumerate}
		are satisfied.
	\end{defi}
	
	From the above, we see that two-sided fibrations form an $\infty$-cosmos. This $\infty$-cosmos can be formulated as $\calCart(\calCoCart(\calK)_{/A})_{/a}$, or equivalently, as $\calCoCart(\calCart(\calK)_{/B})_{/b}$, which are isomorphic. Here, $a \colon A \times B \twoheadrightarrow A$ and $b \colon A \times B \twoheadrightarrow B$ denote the canonical projections.
	
	Nevertheless, there are two obvious but different enhanced simplicial categories of two-sided fibrations of an $\infty$-cosmos $\calK$, which are $\Cart(\calCoCart(\calK)_{/A})_{/a}$ and $\CoCart(\calCart(\calK)_{/B})_{/b}$, respectively. Clearly, their tight parts coincide.
	
	\begin{lemma}
		\label{lem:enh_Cart}
		A loose $0$-arrow in the enhanced simplicial category $\Cart(\calCoCart(\calK)_{/A})_{/a}$ is a map of span in \longref{Diagram}{diag:map_of_spans} such that the square on the left of \longref{Diagram}{diag:mates_two_sided} is invertible.
		
		A loose $0$-arrow in the enhanced simplicial category $\CoCart(\calCart(\calK)_{/B})_{/b}$ is a map of span in \longref{Diagram}{diag:map_of_spans} such that the square on the right of \longref{Diagram}{diag:mates_two_sided} is invertible.
	\end{lemma}
	
	\begin{proof}
		Note that a $0$-arrow $\chi_1 \to \chi_2$ in the $\infty$-cosmos $(\calCoCart(\calK)_{/A})^{\isof}_{/a}$ amounts to a commutative triangle
		% https://q.uiver.app/#q=WzAsMyxbMCwwLCJwXzEiXSxbMiwwLCJwXzIiXSxbMSwxLCJhIl0sWzAsMSwiXFxwaSJdLFswLDIsIlxcY2hpXzEiLDIseyJzdHlsZSI6eyJoZWFkIjp7Im5hbWUiOiJlcGkifX19XSxbMSwyLCJcXGNoaV8yIiwwLHsic3R5bGUiOnsiaGVhZCI6eyJuYW1lIjoiZXBpIn19fV1d
		\[\begin{tikzcd}[ampersand replacement=\&]
			{p_1} \&\& {p_2} \\
			\& a
			\arrow["\pi", from=1-1, to=1-3]
			\arrow["{\chi_1}"', two heads, from=1-1, to=2-2]
			\arrow["{\chi_2}", two heads, from=1-3, to=2-2]
		\end{tikzcd}\]
		in $\calCoCart(\calK)_{/A}$, which is equivalent to a commutative triangle
		% https://q.uiver.app/#q=WzAsMyxbMCwwLCJFXzEiXSxbMiwwLCJFXzIiXSxbMSwxLCJBIl0sWzAsMSwiZSJdLFswLDIsInBfMSIsMix7InN0eWxlIjp7ImhlYWQiOnsibmFtZSI6ImVwaSJ9fX1dLFsxLDIsInBfMiIsMCx7InN0eWxlIjp7ImhlYWQiOnsibmFtZSI6ImVwaSJ9fX1dXQ==
		\[\begin{tikzcd}[ampersand replacement=\&]
			{E_1} \&\& {E_2} \\
			\& A
			\arrow["e", from=1-1, to=1-3]
			\arrow["{p_1}"', two heads, from=1-1, to=2-2]
			\arrow["{p_2}", two heads, from=1-3, to=2-2]
		\end{tikzcd}\]
		in $\calK$.
		
		%		 An adjunction
		%		% https://q.uiver.app/#q=WzAsMixbMCwwLCJwXzEiXSxbMiwwLCJhIFxcZG93bmFycm93IFxcY2hpXzEiXSxbMCwxLCJcXERlbHRhX3tcXGNoaV8xfSIsMCx7ImN1cnZlIjotMX1dLFsxLDAsInJfMSIsMCx7ImN1cnZlIjotMX1dLFsyLDMsIiIsMCx7ImxldmVsIjoxLCJzdHlsZSI6eyJuYW1lIjoiYWRqdW5jdGlvbiJ9fV1d
		%		\[\begin{tikzcd}[ampersand replacement=\&]
			%			{p_1} \&\& {a \downarrow \chi_1}
			%			\arrow[""{name=0, anchor=center, inner sep=0}, "{\Delta_{\chi_1}}", bend left, from=1-1, to=1-3]
			%			\arrow[""{name=1, anchor=center, inner sep=0}, "{r_1}", bend left, from=1-3, to=1-1]
			%			\arrow["\dashv"{anchor=center, rotate=-90}, draw=none, from=0, to=1]
			%		\end{tikzcd}\]
		%		in $\calCoCart(\calK)_{/A}$ amounts to a fibred adjunction
		Now, by the internal characterisation of coCartesian functors, asking $\pi$ to be a $0$-arrow in $\calCoCart(\calK)_{/A}$ means asking the square on the left of \longref{Diagram}{diag:mates_two_sided} to be invertible.
		
		A similar argument shows the second statement.		
	\end{proof}
	
	Therefore, depending on how we view the two-sided fibrations from $\infty$-categories $A$ to $B$ of an $\infty$-cosmos as an enhanced simplicial category, we attain different conclusions, by our previous results on fibrations over a fixed base in \longref{Section}{sec:Cart_fixed} and  \longref{Section}{sec:CoCart_fixed}. 
	
	\begin{theorem}
		\label{thm:trio_enh_Cart}
		The enhanced simplicial category $\Cart(\calCoCart(\calK)_{/A})_{/a}$ admits tight cosmological limits, terminally rigged $n$-inserters, and that the projection of a terminally rigged $n$-inserter is an isofibration of $(\calCoCart(\calK)_{/A})^{\isof}$.
		
		Consequently, $\Cart(\calCoCart(\calK)_{/A})_a$ admits Eilenberg-Moore objects over loose monads, and that the forgetful functor is a Cartesian functor of two-sided fibrations from $A$ to $B$, which reflects Cartesian functors.
		
		These $\F_\Delta$-weighted limits are all preserved by the inclusion to $(\calCoCart(\calK)_{/A})^{\isof}_\chor$.
	\end{theorem}
	%	
	%	\begin{theorem}
		%		\label{thm:enh_Cart}
		%		The enhanced simplicial category $\Cart(\calCoCart(\calK)_{/A})_a$ admits terminally rigged $n$-inserters, and that the projection of a terminally rigged $n$-inserter is an isofibration of $(\calCoCart(\calK)_{/A})^{\isof}$.
		%	\end{theorem}
	%	
	%	\begin{pro}
		%		\label{pro:tight_lim_enh_Cart}
		%		Tight cosmological limits exist in the enhanced simplicial category $\Cart(\calCoCart(\calK)_{/A})_a$.
		%	\end{pro}
	%	
	%	\begin{coroll}
		%		The enhanced simplicial category $\Cart(\calCoCart(\calK)_{/A})_a$ admits Eilenberg-Moore objects over loose monads, and that the forgetful functor is a Cartesian functor of two-sided fibrations from $A$ to $B$, which reflects Cartesian functors.
		%	\end{coroll}
	
	\begin{theorem}
		\label{thm:trio_enh_CoCart}
		The enhanced simplicial category $\CoCart(\calCart(\calK)_{/B})_{/b}$ admits tight cosmological limits, initially rigged $n$-inserters, and that the projection of a initially rigged $n$-inserter is an isofibration of $(\calCart(\calK)_{/B})^{\isof}$.
		
		Consequently, $\CoCart(\calCart(\calK)_{/B})_{/b}$ admits coEilenberg-Moore objects over loose comonads, and that the forgetful functor is a Cartesian functor of two-sided fibrations from $A$ to $B$, which reflects Cartesian functors.
		
		These $\F_\Delta$-weighted limits are all preserved by the inclusion to $(\calCart(\calK)_{/B})^{\isof}_\chor$.
	\end{theorem}

	\section{Limits not from \texorpdfstring{$\Lali(\calK)$}{Lali(K)} or \texorpdfstring{$\Rali(\calK)$}{Rali(K)}}
	\label{sec:not_in_Lali}
	We established all the completeness results above by relating each enhanced simplicial category to $\Lali(\calK)$ or $\Rali(\calK)$ of an $\infty$-cosmos $\calK$. A natural guess is that $\Lali(\calK)$ or $\Rali(\calK)$ admits any $\F_\Delta$-weighted limits that these enhanced simplicial categories should possess. However, this might not be the case. In fact, we would show in this section that, particularly, the enhanced simplicial category $\JLim(\calK)$ of $\infty$-categories with $J$-shaped limits admits certain $\F_\Delta$-weighted limits that $\Lali(\calK)$ does not possess.
	
	In order to establish the existence of certain $\F_\Delta$-weighted limits in some specific enhanced simplicial categories, we would be working a lot in the homotopy $2$-category. We call (the equivalence class) of a $1$-morphism in an $\infty$-cosmos an \emph{$\infty$-natural transformation} in the homotopy $2$-category.
	
	\subsection{\texorpdfstring{$\La(\calK)$}{La(K)} and \texorpdfstring{$\Ra(\calK)$}{Ra(K)}}
	The key is actually the left, or respectively, right adjoints.

	\begin{theorem}
		\label{thm:pullback_along_Cartesian_fib_La}
		Let $\calK$ be an $\infty$-cosmos. The enhanced simplicial category $\La(\calK)$ of la-isofibrations of $\calK$ admits pullbacks of a loose $0$-arrow along a tight Cartesian fibration, which are preserved by the inclusion
		$$\La(\calK) \hookrightarrow \calK^{\isof}_\chor.$$
		
		More precisely, let $A, B, C$ be la-isofibrations of $\calK$, $F \colon A \to C$ be a morphism of isofibrations, and $G \colon B \twoheadrightarrow C$ be a morphism of left adjoints, which is also a Cartesian fibration of $\calK^{\isof}$. Then, the pullback
		% https://q.uiver.app/#q=WzAsNCxbMCwwLCJQIl0sWzEsMCwiQiJdLFswLDEsIkEiXSxbMSwxLCJDIl0sWzIsMywiRiIsMl0sWzEsMywiRyIsMCx7InN0eWxlIjp7ImhlYWQiOnsibmFtZSI6ImVwaSJ9fX1dLFswLDIsInBfQSIsMix7InN0eWxlIjp7ImhlYWQiOnsibmFtZSI6ImVwaSJ9fX1dLFswLDEsInBfQiJdLFswLDMsIiIsMSx7InN0eWxlIjp7Im5hbWUiOiJjb3JuZXIifX1dXQ==
		\[\begin{tikzcd}[ampersand replacement=\&]
			P \& B \\
			A \& C
			\arrow["{p_B}", from=1-1, to=1-2]
			\arrow["{p_A}"', two heads, from=1-1, to=2-1]
			\arrow["\lrcorner"{anchor=center, pos=0.125}, draw=none, from=1-1, to=2-2]
			\arrow["G", two heads, from=1-2, to=2-2]
			\arrow["F"', from=2-1, to=2-2]
		\end{tikzcd}\]
		of $F$ along $G$ in $\calK^{\isof}$ is also a la-isofibration, and that the projection $p_A$ is a morphism of left adjoints, which also reflects morphisms of left adjoints.
	\end{theorem}
	
	\begin{proof}
		Write $A \colon a_1 \twoheadrightarrow a_2$, $B \colon b_1 \twoheadrightarrow b_2$, and $C \colon c_1 \twoheadrightarrow c_2$. Denote by $r_A \colon a_2 \to a_1$, $r_B \colon b_2 \to b_1$, and $r_C \colon c_2 \to c_1$, the right adjoints to $A$, $B$, and $C$, respectively. Then $F$ and $G$ amount to  commutative squares
		\[
		% https://q.uiver.app/#q=WzAsNCxbMSwwLCJhXzIiXSxbMCwxLCJjXzEiXSxbMSwxLCJjXzIiXSxbMCwwLCJhXzEiXSxbMSwyLCJDIiwyLHsic3R5bGUiOnsiaGVhZCI6eyJuYW1lIjoiZXBpIn19fV0sWzAsMiwiRl8yIl0sWzMsMSwiRl8xIiwyXSxbMywwLCJBIiwwLHsic3R5bGUiOnsiaGVhZCI6eyJuYW1lIjoiZXBpIn19fV1d
		\begin{tikzcd}[ampersand replacement=\&]
			{a_1} \& {a_2} \\
			{c_1} \& {c_2}
			\arrow["A", two heads, from=1-1, to=1-2]
			\arrow["{F_1}"', from=1-1, to=2-1]
			\arrow["{F_2}", from=1-2, to=2-2]
			\arrow["C"', two heads, from=2-1, to=2-2]
		\end{tikzcd},
		\quad \quad
		% https://q.uiver.app/#q=WzAsNCxbMSwwLCJiXzIiXSxbMCwxLCJjXzEiXSxbMSwxLCJjXzIiXSxbMCwwLCJiXzEiXSxbMSwyLCJDIiwyLHsic3R5bGUiOnsiaGVhZCI6eyJuYW1lIjoiZXBpIn19fV0sWzAsMiwiR18yIiwwLHsic3R5bGUiOnsiaGVhZCI6eyJuYW1lIjoiZXBpIn19fV0sWzMsMSwiR18xIiwyLHsic3R5bGUiOnsiaGVhZCI6eyJuYW1lIjoiZXBpIn19fV0sWzMsMCwiQiIsMCx7InN0eWxlIjp7ImhlYWQiOnsibmFtZSI6ImVwaSJ9fX1dXQ==
		\begin{tikzcd}[ampersand replacement=\&]
			{b_1} \& {b_2} \\
			{c_1} \& {c_2}
			\arrow["B", two heads, from=1-1, to=1-2]
			\arrow["{G_1}"', two heads, from=1-1, to=2-1]
			\arrow["{G_2}", two heads, from=1-2, to=2-2]
			\arrow["C"', two heads, from=2-1, to=2-2]
		\end{tikzcd}\]
		in $\calK$, respectively. Here, both $G_1$ and $G_2$ are Cartesian fibrations of $\calK$, because the (co)domain projections are cosmological functors. Since $G$ is a morphism of left adjoints, its mate
		% https://q.uiver.app/#q=WzAsNixbMSwwLCJiXzEiXSxbMSwxLCJiXzIiXSxbMiwwLCJjXzEiXSxbMiwxLCJjXzIiXSxbMCwxLCJiXzIiXSxbMywwLCJjXzEiXSxbMCwxLCJCIiwyLHsic3R5bGUiOnsiaGVhZCI6eyJuYW1lIjoiZXBpIn19fV0sWzIsMywiQyIsMCx7InN0eWxlIjp7ImhlYWQiOnsibmFtZSI6ImVwaSJ9fX1dLFswLDIsIkdfMSIsMCx7InN0eWxlIjp7ImhlYWQiOnsibmFtZSI6ImVwaSJ9fX1dLFsxLDMsIkdfMiIsMix7InN0eWxlIjp7ImhlYWQiOnsibmFtZSI6ImVwaSJ9fX1dLFs0LDEsIiIsMSx7ImxldmVsIjoyLCJzdHlsZSI6eyJoZWFkIjp7Im5hbWUiOiJub25lIn19fV0sWzIsNSwiIiwxLHsibGV2ZWwiOjIsInN0eWxlIjp7ImhlYWQiOnsibmFtZSI6Im5vbmUifX19XSxbNCwwLCJyX0IiXSxbMyw1LCJyX0MiLDJdLFsxMiwxMCwiXFxlcHNpbG9uX0IiLDAseyJzaG9ydGVuIjp7InNvdXJjZSI6MjAsInRhcmdldCI6MjB9fV0sWzExLDEzLCJcXGV0YV9DIiwyLHsic2hvcnRlbiI6eyJzb3VyY2UiOjIwLCJ0YXJnZXQiOjIwfX1dXQ==
		\[\begin{tikzcd}[ampersand replacement=\&]
			\& {b_1} \& {c_1} \& {c_1} \\
			{b_2} \& {b_2} \& {c_2}
			\arrow["{G_1}", two heads, from=1-2, to=1-3]
			\arrow["B"', two heads, from=1-2, to=2-2]
			\arrow[""{name=0, anchor=center, inner sep=0}, equals, from=1-3, to=1-4]
			\arrow["C", two heads, from=1-3, to=2-3]
			\arrow[""{name=1, anchor=center, inner sep=0}, "{r_B}", from=2-1, to=1-2]
			\arrow[""{name=2, anchor=center, inner sep=0}, equals, from=2-1, to=2-2]
			\arrow["{G_2}"', two heads, from=2-2, to=2-3]
			\arrow[""{name=3, anchor=center, inner sep=0}, "{r_C}"', from=2-3, to=1-4]
			\arrow["{\eta_C}"', shorten <=2pt, shorten >=2pt, Rightarrow, from=0, to=3]
			\arrow["{\epsilon_B}", shorten <=2pt, shorten >=2pt, Rightarrow, from=1, to=2]
		\end{tikzcd}\]
		which we denote by $\lambda_G$, is an invertible. Here, $\epsilon_B$ is the counit for the left adjoint $B$, whereas $\eta_C$ is the unit for the left adjoint $C$. Similarly, we denote by $\lambda_F$ the mate of $F$.
		
		We first construct a right adjoint for the isofibration $P \colon p_1 \twoheadrightarrow p_2$.
		
		Consider the composite
		% https://q.uiver.app/#q=WzAsNyxbMCwwLCJwXzIiXSxbMSwwLCJiXzIiXSxbMiwwLCJiXzEiXSxbMSwxLCJjXzIiXSxbMCwxLCJhXzIiXSxbMiwyLCJjXzEiXSxbMCwyLCJhXzEiXSxbNiw1LCJGXzEiLDJdLFsyLDUsIkdfMSIsMCx7InN0eWxlIjp7ImhlYWQiOnsibmFtZSI6ImVwaSJ9fX1dLFswLDEsIntwX0J9XzIiXSxbMCw0LCJ7cF9BfV8yIiwyLHsic3R5bGUiOnsiaGVhZCI6eyJuYW1lIjoiZXBpIn19fV0sWzEsMiwicl9CIl0sWzEsMywiR18yIiwyLHsic3R5bGUiOnsiaGVhZCI6eyJuYW1lIjoiZXBpIn19fV0sWzQsMywiRl8yIl0sWzAsMywiIiwxLHsic3R5bGUiOnsibmFtZSI6ImNvcm5lciJ9fV0sWzQsNiwicl9BIiwyXSxbMyw1LCJyX0MiLDJdLFs2LDMsIlxcbGFtYmRhX0YiLDIseyJzaG9ydGVuIjp7InNvdXJjZSI6MzAsInRhcmdldCI6MzB9LCJsZXZlbCI6Mn1dLFszLDIsIlxcbGFtYmRhX0deey0xfSIsMix7InNob3J0ZW4iOnsic291cmNlIjozMCwidGFyZ2V0IjozMH0sImxldmVsIjoyfV1d
		\[\begin{tikzcd}[ampersand replacement=\&]
			{p_2} \& {b_2} \& {b_1} \\
			{a_2} \& {c_2} \\
			{a_1} \&\& {c_1}
			\arrow["{{p_B}_2}", from=1-1, to=1-2]
			\arrow["{{p_A}_2}"', two heads, from=1-1, to=2-1]
			\arrow["\lrcorner"{anchor=center, pos=0.125}, draw=none, from=1-1, to=2-2]
			\arrow["{r_B}", from=1-2, to=1-3]
			\arrow["{G_2}"', two heads, from=1-2, to=2-2]
			\arrow["{G_1}", two heads, from=1-3, to=3-3]
			\arrow["{F_2}", from=2-1, to=2-2]
			\arrow["{r_A}"', from=2-1, to=3-1]
			\arrow["{\lambda_G^{-1}}"', shorten <=6pt, shorten >=6pt, Rightarrow, from=2-2, to=1-3]
			\arrow["{r_C}"', from=2-2, to=3-3]
			\arrow["{\lambda_F}"', shorten <=6pt, shorten >=6pt, Rightarrow, from=3-1, to=2-2]
			\arrow["{F_1}"', from=3-1, to=3-3]
		\end{tikzcd}\]
		in the homotopy $2$-category $h\calK$. Since $G_1$ is a Cartesian fibration of $\calK$, by \longref{Lemma}{lem:isofib_in_isofib} and \cite[Definition 5.2.1]{book:RV:2022}, there exists a $G_1$-Cartesian lift $\chi \colon x \Rightarrow r_B \cdot {p_B}_2$ in $h\calK$ such that 
		% https://q.uiver.app/#q=WzAsNyxbMCwwLCJwXzIiXSxbMSwwLCJiXzIiXSxbMiwwLCJiXzEiXSxbMSwxLCJjXzIiXSxbMCwxLCJhXzIiXSxbMiwyLCJjXzEiXSxbMCwyLCJhXzEiXSxbNiw1LCJGXzEiLDJdLFsyLDUsIkdfMSIsMCx7InN0eWxlIjp7ImhlYWQiOnsibmFtZSI6ImVwaSJ9fX1dLFswLDEsIntwX0J9XzIiXSxbMCw0LCJ7cF9BfV8yIiwyLHsic3R5bGUiOnsiaGVhZCI6eyJuYW1lIjoiZXBpIn19fV0sWzEsMiwicl9CIl0sWzEsMywiR18yIiwyLHsic3R5bGUiOnsiaGVhZCI6eyJuYW1lIjoiZXBpIn19fV0sWzQsMywiRl8yIl0sWzAsMywiIiwxLHsic3R5bGUiOnsibmFtZSI6ImNvcm5lciJ9fV0sWzQsNiwicl9BIiwyXSxbMyw1LCJyX0MiLDJdLFs2LDMsIlxcbGFtYmRhX0YiLDIseyJzaG9ydGVuIjp7InNvdXJjZSI6MzAsInRhcmdldCI6MzB9LCJsZXZlbCI6Mn1dLFszLDIsIlxcbGFtYmRhX0deey0xfSIsMix7InNob3J0ZW4iOnsic291cmNlIjozMCwidGFyZ2V0IjozMH0sImxldmVsIjoyfV1d
		\[\begin{tikzcd}[ampersand replacement=\&]
			{p_2} \& {b_2} \& {b_1} \\
			{a_2} \& {c_2} \\
			{a_1} \&\& {c_1}
			\arrow["{{p_B}_2}", from=1-1, to=1-2]
			\arrow["{{p_A}_2}"', two heads, from=1-1, to=2-1]
			\arrow["\lrcorner"{anchor=center, pos=0.125}, draw=none, from=1-1, to=2-2]
			\arrow["{r_B}", from=1-2, to=1-3]
			\arrow["{G_2}"', two heads, from=1-2, to=2-2]
			\arrow["{G_1}", two heads, from=1-3, to=3-3]
			\arrow["{F_2}", from=2-1, to=2-2]
			\arrow["{r_A}"', from=2-1, to=3-1]
			\arrow["{\lambda_G^{-1}}"', shorten <=6pt, shorten >=6pt, Rightarrow, from=2-2, to=1-3]
			\arrow["{r_C}"', from=2-2, to=3-3]
			\arrow["{\lambda_F}"', shorten <=6pt, shorten >=6pt, Rightarrow, from=3-1, to=2-2]
			\arrow["{F_1}"', from=3-1, to=3-3]
		\end{tikzcd}
		\quad = \quad
		% https://q.uiver.app/#q=WzAsNSxbMCwwLCJwXzIiXSxbMiwwLCJiXzEiXSxbMCwxLCJhXzIiXSxbMiwyLCJjXzEiXSxbMCwyLCJhXzEiXSxbNCwzLCJGXzEiLDJdLFsxLDMsIkdfMSIsMCx7InN0eWxlIjp7ImhlYWQiOnsibmFtZSI6ImVwaSJ9fX1dLFswLDIsIntwX0F9XzIiLDJdLFsyLDQsInJfQSIsMl0sWzAsMSwiIiwwLHsiY3VydmUiOjJ9XSxbMCwxLCJyX0IgXFxjZG90IHtwX0J9XzIiLDAseyJjdXJ2ZSI6LTJ9XSxbOSwxMCwieCIsMix7InNob3J0ZW4iOnsic291cmNlIjoyMCwidGFyZ2V0IjoyMH19XV0=
		\begin{tikzcd}[ampersand replacement=\&]
			{p_2} \&\& {b_1} \\
			{a_2} \\
			{a_1} \&\& {c_1}
			\arrow[""{name=0, anchor=center, inner sep=0}, bend right, from=1-1, to=1-3]
			\arrow[""{name=1, anchor=center, inner sep=0}, "{r_B \cdot {p_B}_2}", bend left, from=1-1, to=1-3]
			\arrow["{{p_A}_2}"', from=1-1, to=2-1]
			\arrow["{G_1}", two heads, from=1-3, to=3-3]
			\arrow["{r_A}"', from=2-1, to=3-1]
			\arrow["{F_1}"', from=3-1, to=3-3]
			\arrow["x"', shorten <=3pt, shorten >=3pt, Rightarrow, from=0, to=1]
		\end{tikzcd}.
		\]
		By the weak universal property as shown in \cite[Proposition 3.3.1]{book:RV:2022} of pullbacks, there is a unique $0$-arrow $r \colon p_2 \to p_1$ in $\calK$ such that ${p_B}_1 \cdot r = x$ and ${p_A}_1 \cdot r = r_A \cdot {p_A}_2$. 
		%		the diagram
		%		% https://q.uiver.app/#q=WzAsNSxbMSwxLCJwXzEiXSxbMiwxLCJiXzEiXSxbMSwyLCJhXzEiXSxbMiwyLCJjXzEiXSxbMCwwLCJwXzIiXSxbMSwzLCJHXzEiLDAseyJzdHlsZSI6eyJoZWFkIjp7Im5hbWUiOiJlcGkifX19XSxbMiwzLCJGXzEiLDJdLFswLDEsIntwX0J9XzEiXSxbMCwyLCJ7cF9BfV8xIiwyLHsic3R5bGUiOnsiaGVhZCI6eyJuYW1lIjoiZXBpIn19fV0sWzAsMywiIiwxLHsic3R5bGUiOnsibmFtZSI6ImNvcm5lciJ9fV0sWzQsMiwicl9BIFxcY2RvdCB7cF9BfV8yIiwyLHsiY3VydmUiOjN9XSxbNCwxLCJ4IiwwLHsiY3VydmUiOi0zfV0sWzQsMCwiXFxleGlzdHMhIHIiLDAseyJzdHlsZSI6eyJib2R5Ijp7Im5hbWUiOiJkYXNoZWQifX19XV0=
		%		\[\begin{tikzcd}[ampersand replacement=\&]
			%			{p_2} \\
			%			\& {p_1} \& {b_1} \\
			%			\& {a_1} \& {c_1}
			%			\arrow["{\exists! r}", dashed, from=1-1, to=2-2]
			%			\arrow["x", bend left, from=1-1, to=2-3]
			%			\arrow["{r_A \cdot {p_A}_2}"', bend right, from=1-1, to=3-2]
			%			\arrow["{{p_B}_1}", from=2-2, to=2-3]
			%			\arrow["{{p_A}_1}"', two heads, from=2-2, to=3-2]
			%			\arrow["\lrcorner"{anchor=center, pos=0.125}, draw=none, from=2-2, to=3-3]
			%			\arrow["{G_1}", two heads, from=2-3, to=3-3]
			%			\arrow["{F_1}"', from=3-2, to=3-3]
			%		\end{tikzcd}\]
		%		commutes. 
		
		It remains to show that $P \dashv r$. Our goal is to construct a pair of $\infty$-natural transformations
		\[
		% https://q.uiver.app/#q=WzAsMyxbMCwwLCJwXzEiXSxbMiwwLCJwXzEiXSxbMSwxLCJwXzIiXSxbMCwxLCIiLDIseyJsZXZlbCI6Miwic3R5bGUiOnsiaGVhZCI6eyJuYW1lIjoibm9uZSJ9fX1dLFswLDIsIlAiLDIseyJzdHlsZSI6eyJoZWFkIjp7Im5hbWUiOiJlcGkifX19XSxbMiwxLCJyIiwyXSxbMywyLCJcXGV0YSIsMCx7InNob3J0ZW4iOnsic291cmNlIjozMCwidGFyZ2V0IjozMH19XV0=
		\begin{tikzcd}[ampersand replacement=\&]
			{p_1} \&\& {p_1} \\
			\& {p_2}
			\arrow[""{name=0, anchor=center, inner sep=0}, equals, from=1-1, to=1-3]
			\arrow["P"', two heads, from=1-1, to=2-2]
			\arrow["r"', from=2-2, to=1-3]
			\arrow["\eta", shorten <=5pt, shorten >=5pt, Rightarrow, from=0, to=2-2]
		\end{tikzcd}
		\quad\quad
		% https://q.uiver.app/#q=WzAsMyxbMCwxLCJwXzIiXSxbMiwxLCJwXzIiXSxbMSwwLCJwXzEiXSxbMCwxLCIiLDIseyJsZXZlbCI6Miwic3R5bGUiOnsiaGVhZCI6eyJuYW1lIjoibm9uZSJ9fX1dLFswLDIsInIiXSxbMiwxLCJQIiwwLHsic3R5bGUiOnsiaGVhZCI6eyJuYW1lIjoiZXBpIn19fV0sWzIsMywiXFxlcHNpbG9uIiwwLHsic2hvcnRlbiI6eyJzb3VyY2UiOjMwLCJ0YXJnZXQiOjMwfX1dXQ==
		\begin{tikzcd}[ampersand replacement=\&]
			\& {p_1} \\
			{p_2} \&\& {p_2}
			\arrow["P", two heads, from=1-2, to=2-3]
			\arrow["r", from=2-1, to=1-2]
			\arrow[""{name=0, anchor=center, inner sep=0}, equals, from=2-1, to=2-3]
			\arrow["\epsilon", shorten <=5pt, shorten >=5pt, Rightarrow, from=1-2, to=0]
		\end{tikzcd}
		\]
		which would serve as the unit and counit, respectively.
		
		To obtain $\eta$ from the weak universal property of pullbacks, we need to have $\infty$-natural transformations ${p_B}_1 \Rightarrow {p_B}_1 \cdot r \cdot P$ and ${p_A}_1 \Rightarrow {p_A}_1 \cdot r \cdot P$ satisfying the required condition. Our construction of the former $\infty$-natural transformation rely on \cite[Proposition 5.2.11]{book:RV:2022}, which makes use of the fact that $G_1$ is a Cartesian fibration.
		
		We first verify that
		\begin{equation}
			\label{eqt:5.2.11}
			% https://q.uiver.app/#q=WzAsNSxbMSwwLCJiXzEiXSxbMywwLCJiXzEiXSxbMiwxLCJiXzIiXSxbMCwwLCJwXzEiXSxbNCwwLCJjXzEiXSxbMCwxLCIiLDIseyJsZXZlbCI6Miwic3R5bGUiOnsiaGVhZCI6eyJuYW1lIjoibm9uZSJ9fX1dLFswLDIsIkIiLDIseyJzdHlsZSI6eyJoZWFkIjp7Im5hbWUiOiJlcGkifX19XSxbMiwxLCJyX0IiLDJdLFszLDAsIntwX0J9XzEiLDJdLFsxLDQsIkdfMSIsMix7InN0eWxlIjp7ImhlYWQiOnsibmFtZSI6ImVwaSJ9fX1dLFs1LDIsIlxcZXRhX0IiLDAseyJzaG9ydGVuIjp7InNvdXJjZSI6MzAsInRhcmdldCI6MzB9fV1d
			\begin{tikzcd}[ampersand replacement=\&]
				{p_1} \& {b_1} \&\& {b_1} \& {c_1} \\
				\&\& {b_2}
				\arrow["{{p_B}_1}"', from=1-1, to=1-2]
				\arrow[""{name=0, anchor=center, inner sep=0}, equals, from=1-2, to=1-4]
				\arrow["B"', two heads, from=1-2, to=2-3]
				\arrow["{G_1}"', two heads, from=1-4, to=1-5]
				\arrow["{r_B}"', from=2-3, to=1-4]
				\arrow["{\eta_B}", shorten <=5pt, shorten >=5pt, Rightarrow, from=0, to=2-3]
			\end{tikzcd}
			\quad = \quad
			% https://q.uiver.app/#q=WzAsOSxbMCwwLCJwXzEiXSxbMSwwLCJhXzEiXSxbMiwxLCJhXzIiXSxbMywwLCJhXzEiXSxbMCwxLCJwXzIiXSxbMCwyLCJiXzIiXSxbMiwyLCJjXzIiXSxbMywzLCJjXzEiXSxbMCwzLCJiXzEiXSxbMSwzLCIiLDIseyJsZXZlbCI6Miwic3R5bGUiOnsiaGVhZCI6eyJuYW1lIjoibm9uZSJ9fX1dLFswLDEsIntwX0F9XzEiLDAseyJzdHlsZSI6eyJoZWFkIjp7Im5hbWUiOiJlcGkifX19XSxbMSwyLCJBIiwyLHsic3R5bGUiOnsiaGVhZCI6eyJuYW1lIjoiZXBpIn19fV0sWzIsMywicl9BIiwyXSxbMCw0LCJQIiwyLHsic3R5bGUiOnsiaGVhZCI6eyJuYW1lIjoiZXBpIn19fV0sWzQsMiwie3BfQX1fMiIsMCx7InN0eWxlIjp7ImhlYWQiOnsibmFtZSI6ImVwaSJ9fX1dLFs0LDUsIntwX0J9XzIiLDJdLFsyLDYsIkZfMiIsMl0sWzUsNiwiR18yIiwwLHsic3R5bGUiOnsiaGVhZCI6eyJuYW1lIjoiZXBpIn19fV0sWzMsNywiRl8xIl0sWzUsOCwicl9CIiwyXSxbOCw3LCJHXzEiLDIseyJzdHlsZSI6eyJoZWFkIjp7Im5hbWUiOiJlcGkifX19XSxbNiw3LCJyX0MiXSxbNiw4LCJcXGxhbWJkYV9HXnstMX0iLDAseyJzaG9ydGVuIjp7InNvdXJjZSI6MzAsInRhcmdldCI6NDB9LCJsZXZlbCI6Mn1dLFszLDYsIlxcbGFtYmRhX0YiLDAseyJzaG9ydGVuIjp7InNvdXJjZSI6NTAsInRhcmdldCI6MzB9LCJsZXZlbCI6Mn1dLFs5LDIsIlxcZXRhX0EiLDAseyJzaG9ydGVuIjp7InNvdXJjZSI6MzAsInRhcmdldCI6MzB9fV0sWzQsMTcsIiIsMCx7ImxldmVsIjoxLCJzdHlsZSI6eyJuYW1lIjoiY29ybmVyIn19XV0=
			\begin{tikzcd}[ampersand replacement=\&]
				{p_1} \& {a_1} \&\& {a_1} \\
				{p_2} \&\& {a_2} \\
				{b_2} \&\& {c_2} \\
				{b_1} \&\&\& {c_1}
				\arrow["{{p_A}_1}", two heads, from=1-1, to=1-2]
				\arrow["P"', two heads, from=1-1, to=2-1]
				\arrow[""{name=0, anchor=center, inner sep=0}, equals, from=1-2, to=1-4]
				\arrow["A"', two heads, from=1-2, to=2-3]
				\arrow["{\lambda_F}", shorten <=19pt, shorten >=11pt, Rightarrow, from=1-4, to=3-3]
				\arrow["{F_1}", from=1-4, to=4-4]
				\arrow["{{p_A}_2}", two heads, from=2-1, to=2-3]
				\arrow["{{p_B}_2}"', from=2-1, to=3-1]
				\arrow["{r_A}"', from=2-3, to=1-4]
				\arrow["{F_2}"', from=2-3, to=3-3]
				\arrow[""{name=1, anchor=center, inner sep=0}, "{G_2}", two heads, from=3-1, to=3-3]
				\arrow["{r_B}"', from=3-1, to=4-1]
				\arrow["{\lambda_G^{-1}}", shorten <=14pt, shorten >=19pt, Rightarrow, from=3-3, to=4-1]
				\arrow["{r_C}", from=3-3, to=4-4]
				\arrow["{G_1}"', two heads, from=4-1, to=4-4]
				\arrow["{\eta_A}", shorten <=5pt, shorten >=5pt, Rightarrow, from=0, to=2-3]
				\arrow["\lrcorner"{anchor=center, pos=0.125}, draw=none, from=2-1, to=1]
			\end{tikzcd}.
		\end{equation}
		Indeed, we have
		{\allowdisplaybreaks
			\begin{align*}
				&
				% https://q.uiver.app/#q=WzAsOCxbMCwwLCJwXzEiXSxbMSwwLCJhXzEiXSxbMSwxLCJhXzIiXSxbMiwwLCJhXzEiXSxbMCwxLCJwXzIiXSxbMCwyLCJiXzIiXSxbMSwyLCJjXzIiXSxbMiwyLCJjXzEiXSxbMSwzLCIiLDIseyJsZXZlbCI6Miwic3R5bGUiOnsiaGVhZCI6eyJuYW1lIjoibm9uZSJ9fX1dLFswLDEsIntwX0F9XzEiXSxbMSwyLCJBIiwyLHsic3R5bGUiOnsiaGVhZCI6eyJuYW1lIjoiZXBpIn19fV0sWzIsMywicl9BIiwyXSxbMCw0LCJQIiwyLHsic3R5bGUiOnsiaGVhZCI6eyJuYW1lIjoiZXBpIn19fV0sWzQsMiwie3BfQX1fMiJdLFs0LDUsIntwX0J9XzIiLDJdLFsyLDYsIkZfMiIsMl0sWzUsNiwiR18yIiwyLHsic3R5bGUiOnsiaGVhZCI6eyJuYW1lIjoiZXBpIn19fV0sWzMsNywiRl8xIl0sWzYsNywicl9DIiwyXSxbMyw2LCJcXGxhbWJkYV9GIiwwLHsic2hvcnRlbiI6eyJzb3VyY2UiOjUwLCJ0YXJnZXQiOjMwfSwibGV2ZWwiOjJ9XSxbOCwxMSwiXFxldGFfQSIsMix7InNob3J0ZW4iOnsic291cmNlIjozMCwidGFyZ2V0IjozMH19XSxbNCwxNiwiIiwwLHsibGV2ZWwiOjEsInN0eWxlIjp7Im5hbWUiOiJjb3JuZXIifX1dXQ==
				\begin{tikzcd}[ampersand replacement=\&, scale cd = 0.9]
					{p_1} \& {a_1} \& {a_1} \\
					{p_2} \& {a_2} \\
					{b_2} \& {c_2} \& {c_1}
					\arrow["{{p_A}_1}", from=1-1, to=1-2]
					\arrow["P"', two heads, from=1-1, to=2-1]
					\arrow[""{name=0, anchor=center, inner sep=0}, equals, from=1-2, to=1-3]
					\arrow["A"', two heads, from=1-2, to=2-2]
					\arrow["{\lambda_F}", shorten <=19pt, shorten >=11pt, Rightarrow, from=1-3, to=3-2]
					\arrow["{F_1}", from=1-3, to=3-3]
					\arrow["{{p_A}_2}", from=2-1, to=2-2]
					\arrow["{{p_B}_2}"', from=2-1, to=3-1]
					\arrow[""{name=1, anchor=center, inner sep=0}, "{r_A}"', from=2-2, to=1-3]
					\arrow["{F_2}"', from=2-2, to=3-2]
					\arrow[""{name=2, anchor=center, inner sep=0}, "{G_2}"', two heads, from=3-1, to=3-2]
					\arrow["{r_C}"', from=3-2, to=3-3]
					\arrow["{\eta_A}"', shorten <=3pt, shorten >=3pt, Rightarrow, from=0, to=1]
					\arrow["\lrcorner"{anchor=center, pos=0.125}, draw=none, from=2-1, to=2]
				\end{tikzcd}
				\; = \;
				\begin{tikzcd}[ampersand replacement=\&, scale cd = 0.9]
					{p_1} \& {a_1} \&\& {a_1} \\
					{p_2} \& {a_2} \& {a_2} \\
					{b_2} \& {c_2} \&\& {c_1} \\
					\& {b_1} \\
					{c_2} \&\&\& {c_1}
					\arrow["{{p_A}_1}", two heads, from=1-1, to=1-2]
					\arrow["P"', two heads, from=1-1, to=2-1]
					\arrow[""{name=0, anchor=center, inner sep=0}, equals, from=1-2, to=1-4]
					\arrow["A"', two heads, from=1-2, to=2-2]
					\arrow["A", two heads, from=1-4, to=2-3]
					\arrow["{F_1}", from=1-4, to=3-4]
					\arrow["{{p_A}_2}", two heads, from=2-1, to=2-2]
					\arrow["{{p_B}_2}"', from=2-1, to=3-1]
					\arrow[""{name=1, anchor=center, inner sep=0}, "{r_A}"{description}, from=2-2, to=1-4]
					\arrow[equals, from=2-2, to=2-3]
					\arrow["{F_2}"', from=2-2, to=3-2]
					\arrow["{F_2}", from=2-3, to=3-2]
					\arrow[""{name=2, anchor=center, inner sep=0}, "{G_2}", two heads, from=3-1, to=3-2]
					\arrow["{r_B}"', from=3-1, to=4-2]
					\arrow["{G_2}"', two heads, from=3-1, to=5-1]
					\arrow["{\lambda_G^{-1}}", shorten <=2pt, shorten >=2pt, Rightarrow, from=3-2, to=4-2]
					\arrow["{r_C}"{description}, from=3-2, to=5-4]
					\arrow["{r_C}"', two heads, from=3-4, to=3-2]
					\arrow[""{name=3, anchor=center, inner sep=0}, equals, from=3-4, to=5-4]
					\arrow["{\lambda_G}", shorten <=4pt, shorten >=6pt, Rightarrow, from=4-2, to=5-1]
					\arrow["{G_1}"', two heads, from=4-2, to=5-4]
					\arrow["{r_C}"', from=5-1, to=5-4]
					\arrow["{\eta_A}"', shorten <=3pt, shorten >=3pt, Rightarrow, from=0, to=1]
					\arrow["\lrcorner"{anchor=center, pos=0.125}, draw=none, from=2-1, to=2]
					\arrow["{\epsilon_A}"'{pos=0.9}, shorten <=1pt, Rightarrow, from=1, to=2-3]
					\arrow["{\eta_C}"', shorten <=17pt, shorten >=28pt, Rightarrow, from=3, to=3-2]
				\end{tikzcd}
				\; = \;
				% https://q.uiver.app/#q=WzAsMTEsWzAsMCwicF8xIl0sWzEsMCwiYV8xIl0sWzMsMCwiYV8xIl0sWzAsMSwicF8yIl0sWzAsMiwiYl8yIl0sWzEsMiwiY18yIl0sWzMsMiwiY18xIl0sWzIsMSwiYV8yIl0sWzMsNCwiY18xIl0sWzAsNCwiY18yIl0sWzEsMywiYl8xIl0sWzEsMiwiIiwyLHsibGV2ZWwiOjIsInN0eWxlIjp7ImhlYWQiOnsibmFtZSI6Im5vbmUifX19XSxbMCwxLCJ7cF9BfV8xIiwwLHsic3R5bGUiOnsiaGVhZCI6eyJuYW1lIjoiZXBpIn19fV0sWzAsMywiUCIsMix7InN0eWxlIjp7ImhlYWQiOnsibmFtZSI6ImVwaSJ9fX1dLFszLDQsIntwX0J9XzIiLDJdLFs0LDUsIkdfMiIsMCx7InN0eWxlIjp7ImhlYWQiOnsibmFtZSI6ImVwaSJ9fX1dLFsyLDYsIkZfMSJdLFs2LDUsInJfQyIsMix7InN0eWxlIjp7ImhlYWQiOnsibmFtZSI6ImVwaSJ9fX1dLFsyLDcsIkEiLDAseyJzdHlsZSI6eyJoZWFkIjp7Im5hbWUiOiJlcGkifX19XSxbNyw1LCJGXzIiXSxbNCw5LCJHXzIiLDIseyJzdHlsZSI6eyJoZWFkIjp7Im5hbWUiOiJlcGkifX19XSxbNiw4LCIiLDAseyJsZXZlbCI6Miwic3R5bGUiOnsiaGVhZCI6eyJuYW1lIjoibm9uZSJ9fX1dLFs5LDgsInJfQyIsMl0sWzQsMTAsInJfQiIsMl0sWzEwLDgsIkdfMSIsMix7InN0eWxlIjp7ImhlYWQiOnsibmFtZSI6ImVwaSJ9fX1dLFs1LDgsInJfQyIsMV0sWzUsMTAsIlxcbGFtYmRhX0deey0xfSIsMCx7InNob3J0ZW4iOnsic291cmNlIjoyMCwidGFyZ2V0IjoyMH0sImxldmVsIjoyfV0sWzEwLDksIlxcbGFtYmRhX0ciLDAseyJzaG9ydGVuIjp7InNvdXJjZSI6MjAsInRhcmdldCI6MzB9LCJsZXZlbCI6Mn1dLFszLDE1LCIiLDAseyJsZXZlbCI6MSwic3R5bGUiOnsibmFtZSI6ImNvcm5lciJ9fV0sWzIxLDUsIlxcZXRhX0MiLDIseyJzaG9ydGVuIjp7InNvdXJjZSI6MzAsInRhcmdldCI6NTB9fV1d
				\begin{tikzcd}[ampersand replacement=\&, scale cd = 0.9]
					{p_1} \& {a_1} \&\& {a_1} \\
					{p_2} \&\& {a_2} \\
					{b_2} \& {c_2} \&\& {c_1} \\
					\& {b_1} \\
					{c_2} \&\&\& {c_1}
					\arrow["{{p_A}_1}", two heads, from=1-1, to=1-2]
					\arrow["P"', two heads, from=1-1, to=2-1]
					\arrow[equals, from=1-2, to=1-4]
					\arrow["A", two heads, from=1-4, to=2-3]
					\arrow["{F_1}", from=1-4, to=3-4]
					\arrow["{{p_B}_2}"', from=2-1, to=3-1]
					\arrow["{F_2}", from=2-3, to=3-2]
					\arrow[""{name=0, anchor=center, inner sep=0}, "{G_2}", two heads, from=3-1, to=3-2]
					\arrow["{r_B}"', from=3-1, to=4-2]
					\arrow["{G_2}"', two heads, from=3-1, to=5-1]
					\arrow["{\lambda_G^{-1}}", shorten <=2pt, shorten >=2pt, Rightarrow, from=3-2, to=4-2]
					\arrow["{r_C}"{description}, from=3-2, to=5-4]
					\arrow["{r_C}"', two heads, from=3-4, to=3-2]
					\arrow[""{name=1, anchor=center, inner sep=0}, equals, from=3-4, to=5-4]
					\arrow["{\lambda_G}", shorten <=4pt, shorten >=6pt, Rightarrow, from=4-2, to=5-1]
					\arrow["{G_1}"', two heads, from=4-2, to=5-4]
					\arrow["{r_C}"', from=5-1, to=5-4]
					\arrow["\lrcorner"{anchor=center, pos=0.125}, draw=none, from=2-1, to=0]
					\arrow["{\eta_C}"', shorten <=17pt, shorten >=28pt, Rightarrow, from=1, to=3-2]
				\end{tikzcd}
				\\
				&= \quad 
				% https://q.uiver.app/#q=WzAsOCxbMCwwLCJwXzEiXSxbMSwwLCJiXzEiXSxbMiwwLCJiXzEiXSxbMSwxLCJiXzIiXSxbMiwxLCJiXzIiXSxbMywxLCJjXzIiXSxbMywwLCJjXzEiXSxbNCwwLCJjXzEiXSxbMSwyLCIiLDAseyJsZXZlbCI6Miwic3R5bGUiOnsiaGVhZCI6eyJuYW1lIjoibm9uZSJ9fX1dLFsyLDYsIkdfMSIsMCx7InN0eWxlIjp7ImhlYWQiOnsibmFtZSI6ImVwaSJ9fX1dLFs2LDcsIiIsMCx7ImxldmVsIjoyLCJzdHlsZSI6eyJoZWFkIjp7Im5hbWUiOiJub25lIn19fV0sWzMsNCwiIiwwLHsibGV2ZWwiOjIsInN0eWxlIjp7ImhlYWQiOnsibmFtZSI6Im5vbmUifX19XSxbNSw3LCJyX0MiLDJdLFs2LDUsIkMiLDIseyJzdHlsZSI6eyJoZWFkIjp7Im5hbWUiOiJlcGkifX19XSxbMiw0LCJCIiwwLHsic3R5bGUiOnsiaGVhZCI6eyJuYW1lIjoiZXBpIn19fV0sWzQsNSwiR18yIiwyLHsic3R5bGUiOnsiaGVhZCI6eyJuYW1lIjoiZXBpIn19fV0sWzEsMywiQiIsMix7InN0eWxlIjp7ImhlYWQiOnsibmFtZSI6ImVwaSJ9fX1dLFszLDIsInJfQiIsMV0sWzAsMSwie3BfQn1fMSJdLFs4LDE3LCJcXGV0YV9CIiwyLHsic2hvcnRlbiI6eyJzb3VyY2UiOjIwLCJ0YXJnZXQiOjIwfX1dLFsxNywxMSwiXFxlcHNpbG9uX0MiLDAseyJzaG9ydGVuIjp7InNvdXJjZSI6MjAsInRhcmdldCI6MjB9fV0sWzEwLDEyLCJcXGV0YV9DIiwyLHsic2hvcnRlbiI6eyJzb3VyY2UiOjIwLCJ0YXJnZXQiOjIwfX1dXQ==
				\begin{tikzcd}[ampersand replacement=\&]
					{p_1} \& {b_1} \& {b_1} \& {c_1} \& {c_1} \\
					\& {b_2} \& {b_2} \& {c_2}
					\arrow["{{p_B}_1}", from=1-1, to=1-2]
					\arrow[""{name=0, anchor=center, inner sep=0}, equals, from=1-2, to=1-3]
					\arrow["B"', two heads, from=1-2, to=2-2]
					\arrow["{G_1}", two heads, from=1-3, to=1-4]
					\arrow["B", two heads, from=1-3, to=2-3]
					\arrow[""{name=1, anchor=center, inner sep=0}, equals, from=1-4, to=1-5]
					\arrow["C"', two heads, from=1-4, to=2-4]
					\arrow[""{name=2, anchor=center, inner sep=0}, "{r_B}"{description}, from=2-2, to=1-3]
					\arrow[""{name=3, anchor=center, inner sep=0}, equals, from=2-2, to=2-3]
					\arrow["{G_2}"', two heads, from=2-3, to=2-4]
					\arrow[""{name=4, anchor=center, inner sep=0}, "{r_C}"', from=2-4, to=1-5]
					\arrow["{\eta_B}"', shorten <=2pt, shorten >=2pt, Rightarrow, from=0, to=2]
					\arrow["{\eta_C}"', shorten <=2pt, shorten >=2pt, Rightarrow, from=1, to=4]
					\arrow["{\epsilon_C}", shorten <=2pt, shorten >=2pt, Rightarrow, from=2, to=3]
				\end{tikzcd}
				\quad = \quad
				% https://q.uiver.app/#q=WzAsOCxbMCwwLCJwXzEiXSxbMSwwLCJiXzEiXSxbMiwwLCJiXzEiXSxbMSwxLCJiXzIiXSxbMiwxLCJiXzIiXSxbMywxLCJjXzIiXSxbMywwLCJjXzEiXSxbNCwwLCJjXzEiXSxbMSwyLCIiLDAseyJsZXZlbCI6Miwic3R5bGUiOnsiaGVhZCI6eyJuYW1lIjoibm9uZSJ9fX1dLFsyLDYsIkdfMSIsMCx7InN0eWxlIjp7ImhlYWQiOnsibmFtZSI6ImVwaSJ9fX1dLFs2LDcsIiIsMCx7ImxldmVsIjoyLCJzdHlsZSI6eyJoZWFkIjp7Im5hbWUiOiJub25lIn19fV0sWzMsNCwiIiwwLHsibGV2ZWwiOjIsInN0eWxlIjp7ImhlYWQiOnsibmFtZSI6Im5vbmUifX19XSxbNSw3LCJyX0MiLDJdLFs0LDUsIkdfMiIsMix7InN0eWxlIjp7ImhlYWQiOnsibmFtZSI6ImVwaSJ9fX1dLFsxLDMsIkIiLDIseyJzdHlsZSI6eyJoZWFkIjp7Im5hbWUiOiJlcGkifX19XSxbMywyLCJyX0IiLDFdLFswLDEsIntwX0J9XzEiXSxbMiw1LCJcXGxhbWJkYV9HIiwwLHsic2hvcnRlbiI6eyJzb3VyY2UiOjIwLCJ0YXJnZXQiOjIwfSwibGV2ZWwiOjJ9XSxbOCwxNSwiXFxldGFfQiIsMix7InNob3J0ZW4iOnsic291cmNlIjoyMCwidGFyZ2V0IjoyMH19XV0=
				\begin{tikzcd}[ampersand replacement=\&]
					{p_1} \& {b_1} \& {b_1} \& {c_1} \& {c_1} \\
					\& {b_2} \& {b_2} \& {c_2}
					\arrow["{{p_B}_1}", from=1-1, to=1-2]
					\arrow[""{name=0, anchor=center, inner sep=0}, equals, from=1-2, to=1-3]
					\arrow["B"', two heads, from=1-2, to=2-2]
					\arrow["{G_1}", two heads, from=1-3, to=1-4]
					\arrow["{\lambda_G}", shorten <=4pt, shorten >=4pt, Rightarrow, from=1-3, to=2-4]
					\arrow[equals, from=1-4, to=1-5]
					\arrow[""{name=1, anchor=center, inner sep=0}, "{r_B}"{description}, from=2-2, to=1-3]
					\arrow[equals, from=2-2, to=2-3]
					\arrow["{G_2}"', two heads, from=2-3, to=2-4]
					\arrow["{r_C}"', from=2-4, to=1-5]
					\arrow["{\eta_B}"', shorten <=2pt, shorten >=2pt, Rightarrow, from=0, to=1]
				\end{tikzcd}.
			\end{align*}
		}
		Now, by pasting $\lambda_G^{-1}$ on both sides, we attain \longref{Equation}{eqt:5.2.11}. By \cite[Proposition 5.2.11]{book:RV:2022}, there exists a lift 
		% https://q.uiver.app/#q=WzAsMyxbMCwwLCJwXzEiXSxbMSwxLCJwXzIiXSxbMiwwLCJiXzEiXSxbMCwyLCJ7cF9CfV8xIl0sWzAsMSwiUCIsMix7InN0eWxlIjp7ImhlYWQiOnsibmFtZSI6ImVwaSJ9fX1dLFsxLDIsIngiLDJdLFszLDEsIlxcb3ZlcmxpbmV7XFxnYW1tYX0iLDAseyJzaG9ydGVuIjp7InNvdXJjZSI6MzAsInRhcmdldCI6MjB9fV1d
		\[\begin{tikzcd}[ampersand replacement=\&]
			{p_1} \&\& {b_1} \\
			\& {p_2}
			\arrow[""{name=0, anchor=center, inner sep=0}, "{{p_B}_1}", from=1-1, to=1-3]
			\arrow["P"', two heads, from=1-1, to=2-2]
			\arrow["x"', from=2-2, to=1-3]
			\arrow["{\overline{\gamma}}", shorten <=5pt, shorten >=3pt, Rightarrow, from=0, to=2-2]
		\end{tikzcd}\]
		of $F_1 \cdot \eta_A \cdot {p_A}_1$ satisfying
		\begin{equation*}
			% https://q.uiver.app/#q=WzAsNCxbMCwwLCJwXzEiXSxbMSwwLCJiXzEiXSxbMywwLCJiXzEiXSxbMiwxLCJiXzIiXSxbMSwyLCIiLDIseyJsZXZlbCI6Miwic3R5bGUiOnsiaGVhZCI6eyJuYW1lIjoibm9uZSJ9fX1dLFswLDEsIntwX0J9XzEiXSxbMSwzLCJCIiwyLHsic3R5bGUiOnsiaGVhZCI6eyJuYW1lIjoiZXBpIn19fV0sWzMsMiwicl9CIiwyXSxbNCwzLCJcXGV0YV9CIiwwLHsic2hvcnRlbiI6eyJzb3VyY2UiOjMwLCJ0YXJnZXQiOjMwfX1dXQ==
			\begin{tikzcd}[ampersand replacement=\&]
				{p_1} \& {b_1} \&\& {b_1} \\
				\&\& {b_2}
				\arrow["{{p_B}_1}", from=1-1, to=1-2]
				\arrow[""{name=0, anchor=center, inner sep=0}, equals, from=1-2, to=1-4]
				\arrow["B"', two heads, from=1-2, to=2-3]
				\arrow["{r_B}"', from=2-3, to=1-4]
				\arrow["{\eta_B}", shorten <=5pt, shorten >=5pt, Rightarrow, from=0, to=2-3]
			\end{tikzcd}
			\quad = \quad
			% https://q.uiver.app/#q=WzAsMyxbMCwwLCJwXzEiXSxbMSwxLCJwXzIiXSxbMiwwLCJiXzEiXSxbMCwyLCJ7cF9CfV8xIl0sWzAsMSwiUCIsMix7InN0eWxlIjp7ImhlYWQiOnsibmFtZSI6ImVwaSJ9fX1dLFsxLDIsIngiLDFdLFsxLDIsInJfQiBcXGNkb3Qge3BfQn1fMiIsMix7ImN1cnZlIjozfV0sWzMsMSwiXFxvdmVybGluZXtcXGdhbW1hfSIsMCx7InNob3J0ZW4iOnsic291cmNlIjozMCwidGFyZ2V0IjoyMH19XSxbNSw2LCJcXGNoaSIsMCx7InNob3J0ZW4iOnsic291cmNlIjoyMCwidGFyZ2V0IjoyMH19XV0=
			\begin{tikzcd}[ampersand replacement=\&]
				{p_1} \&\& {b_1} \\
				\& {p_2}
				\arrow[""{name=0, anchor=center, inner sep=0}, "{{p_B}_1}", from=1-1, to=1-3]
				\arrow["P"', two heads, from=1-1, to=2-2]
				\arrow[""{name=1, anchor=center, inner sep=0}, "x"{description}, from=2-2, to=1-3]
				\arrow[""{name=2, anchor=center, inner sep=0}, "{r_B \cdot {p_B}_2}"', bend right, from=2-2, to=1-3]
				\arrow["{\overline{\gamma}}", shorten <=5pt, shorten >=3pt, Rightarrow, from=0, to=2-2]
				\arrow["\chi", shorten <=3pt, shorten >=3pt, Rightarrow, from=1, to=2]
			\end{tikzcd}.
		\end{equation*}
		This means that we have an $\infty$-natural transformation
		% https://q.uiver.app/#q=WzAsNCxbMCwwLCJwXzEiXSxbMSwwLCJhXzEiXSxbMywwLCJhXzEiXSxbMiwxLCJhXzIiXSxbMSwyLCIiLDIseyJsZXZlbCI6Miwic3R5bGUiOnsiaGVhZCI6eyJuYW1lIjoibm9uZSJ9fX1dLFswLDEsIntwX0F9XzEiXSxbMSwzLCJBIiwyLHsic3R5bGUiOnsiaGVhZCI6eyJuYW1lIjoiZXBpIn19fV0sWzMsMiwicl9BIiwyXSxbNCwzLCJcXGV0YV9BIiwwLHsic2hvcnRlbiI6eyJzb3VyY2UiOjMwLCJ0YXJnZXQiOjMwfX1dXQ==
		\[\begin{tikzcd}[ampersand replacement=\&]
			{p_1} \& {a_1} \&\& {a_1} \\
			\&\& {a_2}
			\arrow["{{p_A}_1}", from=1-1, to=1-2]
			\arrow[""{name=0, anchor=center, inner sep=0}, equals, from=1-2, to=1-4]
			\arrow["A"', two heads, from=1-2, to=2-3]
			\arrow["{r_A}"', from=2-3, to=1-4]
			\arrow["{\eta_A}", shorten <=5pt, shorten >=5pt, Rightarrow, from=0, to=2-3]
		\end{tikzcd}\]
		in $h\calK$ satisfying
		\begin{equation}
			\label{eqt:eta_equals_to}
			% https://q.uiver.app/#q=WzAsNCxbMCwwLCJwXzEiXSxbMSwxLCJwXzIiXSxbMiwwLCJiXzEiXSxbMywwLCJjXzEiXSxbMCwyLCJ7cF9CfV8xIl0sWzAsMSwiUCIsMix7InN0eWxlIjp7ImhlYWQiOnsibmFtZSI6ImVwaSJ9fX1dLFsxLDIsIngiLDJdLFsyLDMsIkdfMSIsMCx7InN0eWxlIjp7ImhlYWQiOnsibmFtZSI6ImVwaSJ9fX1dLFs0LDEsIlxcb3ZlcmxpbmV7XFxnYW1tYX0iLDAseyJzaG9ydGVuIjp7InNvdXJjZSI6MzAsInRhcmdldCI6MjB9fV1d
			\begin{tikzcd}[ampersand replacement=\&]
				{p_1} \&\& {b_1} \& {c_1} \\
				\& {p_2}
				\arrow[""{name=0, anchor=center, inner sep=0}, "{{p_B}_1}", from=1-1, to=1-3]
				\arrow["P"', two heads, from=1-1, to=2-2]
				\arrow["{G_1}", two heads, from=1-3, to=1-4]
				\arrow["x"', from=2-2, to=1-3]
				\arrow["{\overline{\gamma}}", shorten <=5pt, shorten >=3pt, Rightarrow, from=0, to=2-2]
			\end{tikzcd}
			\quad = \quad
			% https://q.uiver.app/#q=WzAsNSxbMCwwLCJwXzEiXSxbMSwwLCJhXzEiXSxbMywwLCJhXzEiXSxbMiwxLCJhXzIiXSxbNCwwLCJjXzEiXSxbMSwyLCIiLDIseyJsZXZlbCI6Miwic3R5bGUiOnsiaGVhZCI6eyJuYW1lIjoibm9uZSJ9fX1dLFswLDEsIntwX0F9XzEiXSxbMSwzLCJBIiwyLHsic3R5bGUiOnsiaGVhZCI6eyJuYW1lIjoiZXBpIn19fV0sWzMsMiwicl9BIiwyXSxbMiw0LCJGXzEiXSxbNSwzLCJcXGV0YV9BIiwwLHsic2hvcnRlbiI6eyJzb3VyY2UiOjMwLCJ0YXJnZXQiOjMwfX1dXQ==
			\begin{tikzcd}[ampersand replacement=\&]
				{p_1} \& {a_1} \&\& {a_1} \& {c_1} \\
				\&\& {a_2}
				\arrow["{{p_A}_1}", from=1-1, to=1-2]
				\arrow[""{name=0, anchor=center, inner sep=0}, equals, from=1-2, to=1-4]
				\arrow["A"', two heads, from=1-2, to=2-3]
				\arrow["{F_1}", from=1-4, to=1-5]
				\arrow["{r_A}"', from=2-3, to=1-4]
				\arrow["{\eta_A}", shorten <=5pt, shorten >=5pt, Rightarrow, from=0, to=2-3]
			\end{tikzcd}.
		\end{equation}
		Now by the weak universal property of pullbacks, there exists an $\infty$-natural transformation
		% https://q.uiver.app/#q=WzAsMyxbMCwwLCJwXzEiXSxbMiwwLCJwXzEiXSxbMSwxLCJwXzIiXSxbMCwxLCIiLDIseyJsZXZlbCI6Miwic3R5bGUiOnsiaGVhZCI6eyJuYW1lIjoibm9uZSJ9fX1dLFswLDIsIlAiLDIseyJzdHlsZSI6eyJoZWFkIjp7Im5hbWUiOiJlcGkifX19XSxbMiwxLCJyIiwyXSxbMywyLCJcXGV0YSIsMCx7InNob3J0ZW4iOnsic291cmNlIjozMCwidGFyZ2V0IjozMH19XV0=
		\[\begin{tikzcd}[ampersand replacement=\&]
			{p_1} \&\& {p_1} \\
			\& {p_2}
			\arrow[""{name=0, anchor=center, inner sep=0}, equals, from=1-1, to=1-3]
			\arrow["P"', two heads, from=1-1, to=2-2]
			\arrow["r"', from=2-2, to=1-3]
			\arrow["\eta", shorten <=5pt, shorten >=5pt, Rightarrow, from=0, to=2-2]
		\end{tikzcd}\]
		such that
		\[
		% https://q.uiver.app/#q=WzAsNCxbMCwwLCJwXzEiXSxbMiwwLCJwXzEiXSxbMSwxLCJwXzIiXSxbMywwLCJiXzEiXSxbMCwxLCIiLDIseyJsZXZlbCI6Miwic3R5bGUiOnsiaGVhZCI6eyJuYW1lIjoibm9uZSJ9fX1dLFswLDIsIlAiLDIseyJzdHlsZSI6eyJoZWFkIjp7Im5hbWUiOiJlcGkifX19XSxbMiwxLCJyIiwyXSxbMSwzLCJ7cF9CfV8xIl0sWzQsMiwiXFxldGEiLDAseyJzaG9ydGVuIjp7InNvdXJjZSI6MzAsInRhcmdldCI6MzB9fV1d
		\begin{tikzcd}[ampersand replacement=\&]
			{p_1} \&\& {p_1} \& {b_1} \\
			\& {p_2}
			\arrow[""{name=0, anchor=center, inner sep=0}, equals, from=1-1, to=1-3]
			\arrow["P"', two heads, from=1-1, to=2-2]
			\arrow["{{p_B}_1}", from=1-3, to=1-4]
			\arrow["r"', from=2-2, to=1-3]
			\arrow["\eta", shorten <=5pt, shorten >=5pt, Rightarrow, from=0, to=2-2]
		\end{tikzcd}
		\quad = \quad
		% https://q.uiver.app/#q=WzAsMyxbMCwwLCJwXzEiXSxbMSwxLCJwXzIiXSxbMiwwLCJiXzEiXSxbMCwyLCJ7cF9CfV8xIl0sWzAsMSwiUCIsMix7InN0eWxlIjp7ImhlYWQiOnsibmFtZSI6ImVwaSJ9fX1dLFsxLDIsIngiLDJdLFszLDEsIlxcb3ZlcmxpbmV7XFxnYW1tYX0iLDAseyJzaG9ydGVuIjp7InNvdXJjZSI6MzAsInRhcmdldCI6MjB9fV1d
		\begin{tikzcd}[ampersand replacement=\&]
			{p_1} \&\& {b_1} \\
			\& {p_2}
			\arrow[""{name=0, anchor=center, inner sep=0}, "{{p_B}_1}", from=1-1, to=1-3]
			\arrow["P"', two heads, from=1-1, to=2-2]
			\arrow["x"', from=2-2, to=1-3]
			\arrow["{\overline{\gamma}}", shorten <=5pt, shorten >=3pt, Rightarrow, from=0, to=2-2]
		\end{tikzcd},
		\]
		and
		\[
		% https://q.uiver.app/#q=WzAsNCxbMCwwLCJwXzEiXSxbMiwwLCJwXzEiXSxbMSwxLCJwXzIiXSxbMywwLCJhXzEiXSxbMCwxLCIiLDIseyJsZXZlbCI6Miwic3R5bGUiOnsiaGVhZCI6eyJuYW1lIjoibm9uZSJ9fX1dLFswLDIsIlAiLDIseyJzdHlsZSI6eyJoZWFkIjp7Im5hbWUiOiJlcGkifX19XSxbMiwxLCJyIiwyXSxbMSwzLCJ7cF9BfV8xIiwwLHsic3R5bGUiOnsiaGVhZCI6eyJuYW1lIjoiZXBpIn19fV0sWzQsMiwiXFxldGEiLDAseyJzaG9ydGVuIjp7InNvdXJjZSI6MzAsInRhcmdldCI6MzB9fV1d
		\begin{tikzcd}[ampersand replacement=\&]
			{p_1} \&\& {p_1} \& {a_1} \\
			\& {p_2}
			\arrow[""{name=0, anchor=center, inner sep=0}, equals, from=1-1, to=1-3]
			\arrow["P"', two heads, from=1-1, to=2-2]
			\arrow["{{p_A}_1}", two heads, from=1-3, to=1-4]
			\arrow["r"', from=2-2, to=1-3]
			\arrow["\eta", shorten <=5pt, shorten >=5pt, Rightarrow, from=0, to=2-2]
		\end{tikzcd}
		\quad = \quad
		% https://q.uiver.app/#q=WzAsNCxbMSwwLCJhXzEiXSxbMywwLCJhXzEiXSxbMiwxLCJhXzIiXSxbMCwwLCJwXzEiXSxbMCwxLCIiLDIseyJsZXZlbCI6Miwic3R5bGUiOnsiaGVhZCI6eyJuYW1lIjoibm9uZSJ9fX1dLFswLDIsIkEiLDIseyJzdHlsZSI6eyJoZWFkIjp7Im5hbWUiOiJlcGkifX19XSxbMiwxLCJyX0EiLDJdLFszLDAsIntwX0F9XzEiLDAseyJzdHlsZSI6eyJoZWFkIjp7Im5hbWUiOiJlcGkifX19XSxbNCwyLCJcXGV0YSIsMCx7InNob3J0ZW4iOnsic291cmNlIjozMCwidGFyZ2V0IjozMH19XV0=
		\begin{tikzcd}[ampersand replacement=\&]
			{p_1} \& {a_1} \&\& {a_1} \\
			\&\& {a_2}
			\arrow["{{p_A}_1}", two heads, from=1-1, to=1-2]
			\arrow[""{name=0, anchor=center, inner sep=0}, equals, from=1-2, to=1-4]
			\arrow["A"', two heads, from=1-2, to=2-3]
			\arrow["{r_A}"', from=2-3, to=1-4]
			\arrow["\eta", shorten <=5pt, shorten >=5pt, Rightarrow, from=0, to=2-3]
		\end{tikzcd}.
		\]
		
		Similarly, to obtain $\epsilon$, we need $\infty$-natural transformations ${p_A}_2 \cdot P \cdot r \Rightarrow {p_A}_2$ and ${p_B}_2 \cdot P \cdot r \Rightarrow {p_B}_2$ that satisfy the required condition, so as to apply the weak universal property.
		
		Consider the compositea
		\[
		% https://q.uiver.app/#q=WzAsNCxbMCwxLCJwXzIiXSxbMiwxLCJiXzIiXSxbNCwxLCJiXzIiXSxbMywwLCJiXzEiXSxbMSwzLCJyX0IiXSxbMywyLCJCIiwwLHsic3R5bGUiOnsiaGVhZCI6eyJuYW1lIjoiZXBpIn19fV0sWzEsMiwiIiwwLHsibGV2ZWwiOjIsInN0eWxlIjp7ImhlYWQiOnsibmFtZSI6Im5vbmUifX19XSxbMCwxLCJ7cF9CfV8yIiwyXSxbMCwzLCJ4Il0sWzMsNiwiXFxlcHNpbG9uX0IiLDAseyJzaG9ydGVuIjp7InNvdXJjZSI6MzAsInRhcmdldCI6MzB9fV0sWzgsMSwiXFxjaGkiLDIseyJzaG9ydGVuIjp7InNvdXJjZSI6MjB9fV1d
		\begin{tikzcd}[ampersand replacement=\&]
			\&\&\& {b_1} \\
			{p_2} \&\& {b_2} \&\& {b_2}
			\arrow["B", two heads, from=1-4, to=2-5]
			\arrow[""{name=0, anchor=center, inner sep=0}, "x", from=2-1, to=1-4]
			\arrow["{{p_B}_2}"', from=2-1, to=2-3]
			\arrow["{r_B}", from=2-3, to=1-4]
			\arrow[""{name=1, anchor=center, inner sep=0}, equals, from=2-3, to=2-5]
			\arrow["{\epsilon_B}", shorten <=5pt, shorten >=5pt, Rightarrow, from=1-4, to=1]
			\arrow["\chi"', shorten <=2pt, Rightarrow, from=0, to=2-3]
		\end{tikzcd}, \quad \quad
		% https://q.uiver.app/#q=WzAsNCxbMSwxLCJhXzIiXSxbMywxLCJhXzIiXSxbMiwwLCJhXzEiXSxbMCwxLCJwXzEiXSxbMCwxLCIiLDIseyJsZXZlbCI6Miwic3R5bGUiOnsiaGVhZCI6eyJuYW1lIjoibm9uZSJ9fX1dLFswLDIsInJfQSJdLFsyLDEsIkEiLDAseyJzdHlsZSI6eyJoZWFkIjp7Im5hbWUiOiJlcGkifX19XSxbMywwLCJ7cF9BfV8yIiwwLHsic3R5bGUiOnsiaGVhZCI6eyJuYW1lIjoiZXBpIn19fV0sWzIsNCwiXFxlcHNpbG9uX0EiLDIseyJzaG9ydGVuIjp7InNvdXJjZSI6MzAsInRhcmdldCI6MzB9fV1d
		\begin{tikzcd}[ampersand replacement=\&]
			\&\& {a_1} \\
			{p_1} \& {a_2} \&\& {a_2}
			\arrow["A", two heads, from=1-3, to=2-4]
			\arrow["{{p_A}_2}", two heads, from=2-1, to=2-2]
			\arrow["{r_A}", from=2-2, to=1-3]
			\arrow[""{name=0, anchor=center, inner sep=0}, equals, from=2-2, to=2-4]
			\arrow["{\epsilon_A}"', shorten <=5pt, shorten >=5pt, Rightarrow, from=1-3, to=0]
		\end{tikzcd}.
		\]
		We would like to verify that
		\begin{equation}
			\label{eqt:epsilon}
			% https://q.uiver.app/#q=WzAsNSxbMCwxLCJwXzIiXSxbMiwxLCJiXzIiXSxbNCwxLCJiXzIiXSxbMywwLCJiXzEiXSxbNSwxLCJjXzIiXSxbMSwzLCJyX0IiXSxbMywyLCJCIiwwLHsic3R5bGUiOnsiaGVhZCI6eyJuYW1lIjoiZXBpIn19fV0sWzEsMiwiIiwwLHsibGV2ZWwiOjIsInN0eWxlIjp7ImhlYWQiOnsibmFtZSI6Im5vbmUifX19XSxbMCwxLCJ7cF9CfV8yIiwyXSxbMCwzLCJ4Il0sWzIsNCwiR18yIiwwLHsic3R5bGUiOnsiaGVhZCI6eyJuYW1lIjoiZXBpIn19fV0sWzMsNywiXFxlcHNpbG9uX0IiLDAseyJzaG9ydGVuIjp7InNvdXJjZSI6MzAsInRhcmdldCI6MzB9fV0sWzksMSwiXFxjaGkiLDIseyJzaG9ydGVuIjp7InNvdXJjZSI6MjB9fV1d
			\begin{tikzcd}[ampersand replacement=\&]
				\&\&\& {b_1} \\
				{p_2} \&\& {b_2} \&\& {b_2} \& {c_2}
				\arrow["B", two heads, from=1-4, to=2-5]
				\arrow[""{name=0, anchor=center, inner sep=0}, "x", from=2-1, to=1-4]
				\arrow["{{p_B}_2}"', from=2-1, to=2-3]
				\arrow["{r_B}", from=2-3, to=1-4]
				\arrow[""{name=1, anchor=center, inner sep=0}, equals, from=2-3, to=2-5]
				\arrow["{G_2}", two heads, from=2-5, to=2-6]
				\arrow["{\epsilon_B}", shorten <=5pt, shorten >=5pt, Rightarrow, from=1-4, to=1]
				\arrow["\chi"', shorten <=2pt, Rightarrow, from=0, to=2-3]
			\end{tikzcd}
			\quad = \quad
			% https://q.uiver.app/#q=WzAsNSxbMSwxLCJhXzIiXSxbMywxLCJhXzIiXSxbMiwwLCJhXzEiXSxbMCwxLCJwXzEiXSxbNCwxLCJjXzIiXSxbMCwxLCIiLDIseyJsZXZlbCI6Miwic3R5bGUiOnsiaGVhZCI6eyJuYW1lIjoibm9uZSJ9fX1dLFswLDIsInJfQSJdLFsyLDEsIkEiLDAseyJzdHlsZSI6eyJoZWFkIjp7Im5hbWUiOiJlcGkifX19XSxbMywwLCJ7cF9BfV8yIiwwLHsic3R5bGUiOnsiaGVhZCI6eyJuYW1lIjoiZXBpIn19fV0sWzEsNCwiRl8yIl0sWzIsNSwiXFxlcHNpbG9uX0EiLDIseyJzaG9ydGVuIjp7InNvdXJjZSI6MzAsInRhcmdldCI6MzB9fV1d
			\begin{tikzcd}[ampersand replacement=\&]
				\&\& {a_1} \\
				{p_1} \& {a_2} \&\& {a_2} \& {c_2}
				\arrow["A", two heads, from=1-3, to=2-4]
				\arrow["{{p_A}_2}", two heads, from=2-1, to=2-2]
				\arrow["{r_A}", from=2-2, to=1-3]
				\arrow[""{name=0, anchor=center, inner sep=0}, equals, from=2-2, to=2-4]
				\arrow["{F_2}", from=2-4, to=2-5]
				\arrow["{\epsilon_A}"', shorten <=5pt, shorten >=5pt, Rightarrow, from=1-3, to=0]
			\end{tikzcd}.
		\end{equation}
		From the construction of $\chi$, it is easy to see that
		\[
		% https://q.uiver.app/#q=WzAsOCxbMCwwLCJwXzIiXSxbMSwwLCJiXzIiXSxbMywwLCJjXzIiXSxbMSwxLCJjXzIiXSxbMiwxLCJiXzEiXSxbMywyLCJjXzEiXSxbMCwxLCJhXzIiXSxbMCwyLCJhXzEiXSxbMCwxLCJ7cF9CfV8yIl0sWzAsNiwie3BfQX1fMiIsMix7InN0eWxlIjp7ImhlYWQiOnsibmFtZSI6ImVwaSJ9fX1dLFsxLDMsIkdfMiIsMix7InN0eWxlIjp7ImhlYWQiOnsibmFtZSI6ImVwaSJ9fX1dLFs2LDMsIkZfMiJdLFswLDMsIiIsMSx7InN0eWxlIjp7Im5hbWUiOiJjb3JuZXIifX1dLFsxLDIsIkdfMiIsMCx7InN0eWxlIjp7ImhlYWQiOnsibmFtZSI6ImVwaSJ9fX1dLFsyLDUsInJfQyJdLFs0LDUsIkdfMSIsMCx7InN0eWxlIjp7ImhlYWQiOnsibmFtZSI6ImVwaSJ9fX1dLFsxLDQsInJfQiJdLFszLDUsInJfQyIsMl0sWzYsNywicl9BIiwyXSxbNyw1LCJGXzEiLDJdLFs3LDMsIlxcbGFtYmRhX0YiLDIseyJzaG9ydGVuIjp7InNvdXJjZSI6MzAsInRhcmdldCI6MzB9LCJsZXZlbCI6Mn1dLFszLDQsIlxcbGFtYmRhX0dee1xcaXNvZn0iLDAseyJzaG9ydGVuIjp7InNvdXJjZSI6MjAsInRhcmdldCI6MjB9LCJsZXZlbCI6Mn1dLFs0LDIsIlxcbGFtYmRhX0ciLDIseyJzaG9ydGVuIjp7InNvdXJjZSI6MzAsInRhcmdldCI6MzB9LCJsZXZlbCI6Mn1dXQ==
		\begin{tikzcd}[ampersand replacement=\&]
			{p_2} \& {b_2} \&\& {c_2} \\
			{a_2} \& {c_2} \& {b_1} \\
			{a_1} \&\&\& {c_1}
			\arrow["{{p_B}_2}", from=1-1, to=1-2]
			\arrow["{{p_A}_2}"', two heads, from=1-1, to=2-1]
			\arrow["\lrcorner"{anchor=center, pos=0.125}, draw=none, from=1-1, to=2-2]
			\arrow["{G_2}", two heads, from=1-2, to=1-4]
			\arrow["{G_2}"', two heads, from=1-2, to=2-2]
			\arrow["{r_B}", from=1-2, to=2-3]
			\arrow["{r_C}", from=1-4, to=3-4]
			\arrow["{F_2}", from=2-1, to=2-2]
			\arrow["{r_A}"', from=2-1, to=3-1]
			\arrow["{\lambda_G^{\isof}}", shorten <=3pt, shorten >=3pt, Rightarrow, from=2-2, to=2-3]
			\arrow["{r_C}"', from=2-2, to=3-4]
			\arrow["{\lambda_G}"', shorten <=6pt, shorten >=6pt, Rightarrow, from=2-3, to=1-4]
			\arrow["{G_1}", two heads, from=2-3, to=3-4]
			\arrow["{\lambda_F}"', shorten <=6pt, shorten >=6pt, Rightarrow, from=3-1, to=2-2]
			\arrow["{F_1}"', from=3-1, to=3-4]
		\end{tikzcd}
		\quad = \quad
		% https://q.uiver.app/#q=WzAsNyxbMCwwLCJwXzIiXSxbMiwwLCJiXzIiXSxbMywwLCJjXzIiXSxbMiwxLCJiXzEiXSxbMywyLCJjXzEiXSxbMCwxLCJhXzIiXSxbMCwyLCJhXzEiXSxbMCwxLCJ7cF9CfV8yIl0sWzAsNSwie3BfQX1fMiIsMix7InN0eWxlIjp7ImhlYWQiOnsibmFtZSI6ImVwaSJ9fX1dLFsxLDIsIkdfMiIsMCx7InN0eWxlIjp7ImhlYWQiOnsibmFtZSI6ImVwaSJ9fX1dLFsyLDQsInJfQyJdLFszLDQsIkdfMSIsMCx7InN0eWxlIjp7ImhlYWQiOnsibmFtZSI6ImVwaSJ9fX1dLFs1LDYsInJfQSIsMl0sWzYsNCwiRl8xIiwyXSxbMywyLCJcXGxhbWJkYV9HIiwyLHsic2hvcnRlbiI6eyJzb3VyY2UiOjMwLCJ0YXJnZXQiOjMwfSwibGV2ZWwiOjJ9XSxbMCwzLCJ4IiwyXSxbMSwzLCJyX0IiXSxbMTUsMSwiXFxjaGkiLDIseyJzaG9ydGVuIjp7InNvdXJjZSI6MzAsInRhcmdldCI6MzB9fV1d
		\begin{tikzcd}[ampersand replacement=\&]
			{p_2} \&\& {b_2} \& {c_2} \\
			{a_2} \&\& {b_1} \\
			{a_1} \&\&\& {c_1}
			\arrow["{{p_B}_2}", from=1-1, to=1-3]
			\arrow["{{p_A}_2}"', two heads, from=1-1, to=2-1]
			\arrow[""{name=0, anchor=center, inner sep=0}, "x"', from=1-1, to=2-3]
			\arrow["{G_2}", two heads, from=1-3, to=1-4]
			\arrow["{r_B}", from=1-3, to=2-3]
			\arrow["{r_C}", from=1-4, to=3-4]
			\arrow["{r_A}"', from=2-1, to=3-1]
			\arrow["{\lambda_G}"', shorten <=6pt, shorten >=6pt, Rightarrow, from=2-3, to=1-4]
			\arrow["{G_1}", two heads, from=2-3, to=3-4]
			\arrow["{F_1}"', from=3-1, to=3-4]
			\arrow["\chi"', shorten <=7pt, shorten >=7pt, Rightarrow, from=0, to=1-3]
		\end{tikzcd}.
		\]
		Therefore, we have
		{\allowdisplaybreaks
			\begin{align*}
				&
				\begin{tikzcd}[ampersand replacement=\&]
					\&\& {a_1} \\
					{p_1} \& {a_2} \&\& {a_2} \& {c_2}
					\arrow["A", two heads, from=1-3, to=2-4]
					\arrow["{{p_A}_2}", two heads, from=2-1, to=2-2]
					\arrow["{r_A}", from=2-2, to=1-3]
					\arrow[""{name=0, anchor=center, inner sep=0}, equals, from=2-2, to=2-4]
					\arrow["{F_2}", from=2-4, to=2-5]
					\arrow["{\epsilon_A}"', shorten <=5pt, shorten >=5pt, Rightarrow, from=1-3, to=0]
				\end{tikzcd}
				\quad = \quad
				% https://q.uiver.app/#q=WzAsOSxbMCwwLCJwXzIiXSxbMywwLCJiXzIiXSxbMywxLCJjXzIiXSxbMywyLCJjXzEiXSxbMCwxLCJhXzIiXSxbMCwyLCJhXzEiXSxbMSwxLCJhXzIiXSxbMSwyLCJjXzEiXSxbMiwyLCJjXzEiXSxbMCwxLCJ7cF9CfV8yIl0sWzAsNCwie3BfQX1fMiIsMix7InN0eWxlIjp7ImhlYWQiOnsibmFtZSI6ImVwaSJ9fX1dLFsxLDIsIkdfMiIsMCx7InN0eWxlIjp7ImhlYWQiOnsibmFtZSI6ImVwaSJ9fX1dLFsyLDMsInJfQyJdLFs0LDUsInJfQSIsMl0sWzYsMiwiRl8yIl0sWzQsNiwiIiwyLHsibGV2ZWwiOjIsInN0eWxlIjp7ImhlYWQiOnsibmFtZSI6Im5vbmUifX19XSxbNSw2LCJBIiwyLHsic3R5bGUiOnsiaGVhZCI6eyJuYW1lIjoiZXBpIn19fV0sWzAsNiwiIiwyLHsic3R5bGUiOnsibmFtZSI6ImNvcm5lciJ9fV0sWzcsOCwiIiwyLHsibGV2ZWwiOjIsInN0eWxlIjp7ImhlYWQiOnsibmFtZSI6Im5vbmUifX19XSxbNSw3LCJGXzEiLDJdLFs4LDMsIkMiLDIseyJzdHlsZSI6eyJoZWFkIjp7Im5hbWUiOiJlcGkifX19XSxbNywyLCJDIiwwLHsic3R5bGUiOnsiaGVhZCI6eyJuYW1lIjoiZXBpIn19fV0sWzIsOF0sWzE2LDE1LCJcXGVwc2lsb25fQSIsMCx7InNob3J0ZW4iOnsic291cmNlIjoyMCwidGFyZ2V0IjoyMH19XSxbMTgsMjEsIlxcZXRhX0IiLDIseyJsYWJlbF9wb3NpdGlvbiI6ODAsInNob3J0ZW4iOnsic291cmNlIjozMCwidGFyZ2V0IjozMH19XSxbOCwxMiwiXFxlcHNpbG9uX0IiLDIseyJsYWJlbF9wb3NpdGlvbiI6NzAsIm9mZnNldCI6LTIsInNob3J0ZW4iOnsic291cmNlIjozMCwidGFyZ2V0IjozMH19XV0=
				\begin{tikzcd}[ampersand replacement=\&]
					{p_2} \&\&\& {b_2} \\
					{a_2} \& {a_2} \&\& {c_2} \\
					{a_1} \& {c_1} \& {c_1} \& {c_1}
					\arrow["{{p_B}_2}", from=1-1, to=1-4]
					\arrow["{{p_A}_2}"', two heads, from=1-1, to=2-1]
					\arrow["\lrcorner"{anchor=center, pos=0.125}, draw=none, from=1-1, to=2-2]
					\arrow["{G_2}", two heads, from=1-4, to=2-4]
					\arrow[""{name=0, anchor=center, inner sep=0}, equals, from=2-1, to=2-2]
					\arrow["{r_A}"', from=2-1, to=3-1]
					\arrow["{F_2}", from=2-2, to=2-4]
					\arrow[from=2-4, to=3-3]
					\arrow[""{name=1, anchor=center, inner sep=0}, "{r_C}", from=2-4, to=3-4]
					\arrow[""{name=2, anchor=center, inner sep=0}, "A"', two heads, from=3-1, to=2-2]
					\arrow["{F_1}"', from=3-1, to=3-2]
					\arrow[""{name=3, anchor=center, inner sep=0}, "C", two heads, from=3-2, to=2-4]
					\arrow[""{name=4, anchor=center, inner sep=0}, equals, from=3-2, to=3-3]
					\arrow["C"', two heads, from=3-3, to=3-4]
					\arrow["{\epsilon_A}", shorten <=2pt, shorten >=2pt, Rightarrow, from=2, to=0]
					\arrow["{\eta_B}"'{pos=0.8}, shorten <=5pt, shorten >=5pt, Rightarrow, from=4, to=3]
					\arrow["{\epsilon_B}"'{pos=0.7}, shift left=2, shorten <=7pt, shorten >=7pt, Rightarrow, from=3-3, to=1]
				\end{tikzcd}
				\\
				&  = \quad
				% https://q.uiver.app/#q=WzAsOCxbMCwwLCJwXzIiXSxbMywwLCJiXzIiXSxbMywxLCJjXzIiXSxbMywyLCJjXzIiXSxbMCwxLCJhXzIiXSxbMCwyLCJhXzEiXSxbMSwyLCJjXzEiXSxbMiwyLCJjXzEiXSxbMCwxLCJ7cF9CfV8yIl0sWzAsNCwie3BfQX1fMiIsMix7InN0eWxlIjp7ImhlYWQiOnsibmFtZSI6ImVwaSJ9fX1dLFsxLDIsIkdfMiIsMCx7InN0eWxlIjp7ImhlYWQiOnsibmFtZSI6ImVwaSJ9fX1dLFsyLDMsIiIsMCx7ImxldmVsIjoyLCJzdHlsZSI6eyJoZWFkIjp7Im5hbWUiOiJub25lIn19fV0sWzQsNSwicl9BIiwyXSxbNiw3LCIiLDIseyJsZXZlbCI6Miwic3R5bGUiOnsiaGVhZCI6eyJuYW1lIjoibm9uZSJ9fX1dLFs1LDYsIkZfMSIsMl0sWzcsMywiQyIsMix7InN0eWxlIjp7ImhlYWQiOnsibmFtZSI6ImVwaSJ9fX1dLFsyLDddLFs0LDIsIkZfMiJdLFs1LDIsIlxcbGFtYmRhX0YiLDAseyJzaG9ydGVuIjp7InNvdXJjZSI6NDAsInRhcmdldCI6NDB9LCJsZXZlbCI6Mn1dLFs3LDExLCJcXGVwc2lsb25fQyIsMix7ImxhYmVsX3Bvc2l0aW9uIjo3MCwib2Zmc2V0IjotMiwic2hvcnRlbiI6eyJzb3VyY2UiOjMwLCJ0YXJnZXQiOjMwfX1dXQ==
				\begin{tikzcd}[ampersand replacement=\&]
					{p_2} \&\&\& {b_2} \\
					{a_2} \&\&\& {c_2} \\
					{a_1} \& {c_1} \& {c_1} \& {c_2}
					\arrow["{{p_B}_2}", from=1-1, to=1-4]
					\arrow["{{p_A}_2}"', two heads, from=1-1, to=2-1]
					\arrow["{G_2}", two heads, from=1-4, to=2-4]
					\arrow["{F_2}", from=2-1, to=2-4]
					\arrow["{r_A}"', from=2-1, to=3-1]
					\arrow[from=2-4, to=3-3]
					\arrow[""{name=0, anchor=center, inner sep=0}, equals, from=2-4, to=3-4]
					\arrow["{\lambda_F}", shorten <=32pt, shorten >=32pt, Rightarrow, from=3-1, to=2-4]
					\arrow["{F_1}"', from=3-1, to=3-2]
					\arrow[equals, from=3-2, to=3-3]
					\arrow["C"', two heads, from=3-3, to=3-4]
					\arrow["{\epsilon_C}"'{pos=0.7}, shift left=2, shorten <=7pt, shorten >=7pt, Rightarrow, from=3-3, to=0]
				\end{tikzcd}
				\quad = \quad
				% https://q.uiver.app/#q=WzAsMTAsWzAsMCwicF8yIl0sWzMsMCwiYl8yIl0sWzMsMSwiY18yIl0sWzMsMywiY18yIl0sWzAsMSwiYV8yIl0sWzAsMywiYV8xIl0sWzIsMywiY18xIl0sWzIsMiwiY18xIl0sWzIsMCwiYl8yIl0sWzIsMSwiYl8xIl0sWzAsNCwie3BfQX1fMiIsMix7InN0eWxlIjp7ImhlYWQiOnsibmFtZSI6ImVwaSJ9fX1dLFsxLDIsIkdfMiIsMCx7InN0eWxlIjp7ImhlYWQiOnsibmFtZSI6ImVwaSJ9fX1dLFsyLDMsIiIsMCx7ImxldmVsIjoyLCJzdHlsZSI6eyJoZWFkIjp7Im5hbWUiOiJub25lIn19fV0sWzQsNSwicl9BIiwyXSxbNiw3LCIiLDIseyJsZXZlbCI6Miwic3R5bGUiOnsiaGVhZCI6eyJuYW1lIjoibm9uZSJ9fX1dLFs1LDYsIkZfMSIsMl0sWzcsMiwiQyIsMCx7InN0eWxlIjp7ImhlYWQiOnsibmFtZSI6ImVwaSJ9fX1dLFs4LDEsIiIsMCx7ImxldmVsIjoyLCJzdHlsZSI6eyJoZWFkIjp7Im5hbWUiOiJub25lIn19fV0sWzgsOSwicl9CIiwyXSxbOSwxLCJCIiwyLHsic3R5bGUiOnsiaGVhZCI6eyJuYW1lIjoiZXBpIn19fV0sWzAsOCwie3BfQn1fMiJdLFswLDksIngiLDJdLFs2LDMsIkMiLDIseyJzdHlsZSI6eyJoZWFkIjp7Im5hbWUiOiJlcGkifX19XSxbMiw2XSxbOSw3LCJHXzEiLDIseyJzdHlsZSI6eyJoZWFkIjp7Im5hbWUiOiJlcGkifX19XSxbMjEsOCwiXFxjaGkiLDAseyJzaG9ydGVuIjp7InNvdXJjZSI6MjAsInRhcmdldCI6MzB9fV0sWzIzLDEyLCJcXGVwc2lsb25fQyIsMix7ImxhYmVsX3Bvc2l0aW9uIjo3MCwib2Zmc2V0IjotMiwic2hvcnRlbiI6eyJzb3VyY2UiOjMwLCJ0YXJnZXQiOjMwfX1dLFsxNCwyMywiXFxldGFfQyIsMCx7ImxhYmVsX3Bvc2l0aW9uIjo5MCwic2hvcnRlbiI6eyJzb3VyY2UiOjMwLCJ0YXJnZXQiOjMwfX1dLFsxOSwxNywiXFxlcHNpbG9uX0IiLDAseyJzaG9ydGVuIjp7InNvdXJjZSI6MjAsInRhcmdldCI6MjB9fV1d
				\begin{tikzcd}[ampersand replacement=\&]
					{p_2} \&\& {b_2} \& {b_2} \\
					{a_2} \&\& {b_1} \& {c_2} \\
					\&\& {c_1} \\
					{a_1} \&\& {c_1} \& {c_2}
					\arrow["{{p_B}_2}", from=1-1, to=1-3]
					\arrow["{{p_A}_2}"', two heads, from=1-1, to=2-1]
					\arrow[""{name=0, anchor=center, inner sep=0}, "x"', from=1-1, to=2-3]
					\arrow[""{name=1, anchor=center, inner sep=0}, equals, from=1-3, to=1-4]
					\arrow["{r_B}"', from=1-3, to=2-3]
					\arrow["{G_2}", two heads, from=1-4, to=2-4]
					\arrow["{r_A}"', from=2-1, to=4-1]
					\arrow[""{name=2, anchor=center, inner sep=0}, "B"', two heads, from=2-3, to=1-4]
					\arrow["{G_1}"', two heads, from=2-3, to=3-3]
					\arrow[""{name=3, anchor=center, inner sep=0}, from=2-4, to=4-3]
					\arrow[""{name=4, anchor=center, inner sep=0}, equals, from=2-4, to=4-4]
					\arrow["C", two heads, from=3-3, to=2-4]
					\arrow["{F_1}"', from=4-1, to=4-3]
					\arrow[""{name=5, anchor=center, inner sep=0}, equals, from=4-3, to=3-3]
					\arrow["C"', two heads, from=4-3, to=4-4]
					\arrow["\chi", shorten <=5pt, shorten >=7pt, Rightarrow, from=0, to=1-3]
					\arrow["{\epsilon_B}", shorten <=2pt, shorten >=2pt, Rightarrow, from=2, to=1]
					\arrow["{\epsilon_C}"'{pos=0.7}, shift left=2, shorten <=5pt, shorten >=5pt, Rightarrow, from=3, to=4]
					\arrow["{\eta_C}"{pos=0.9}, shorten <=5pt, shorten >=5pt, Rightarrow, from=5, to=3]
				\end{tikzcd}
				\\
				& = \quad
				\begin{tikzcd}[ampersand replacement=\&]
					\&\&\& {b_1} \\
					{p_2} \&\& {b_2} \&\& {b_2} \& {c_2}
					\arrow["B", two heads, from=1-4, to=2-5]
					\arrow[""{name=0, anchor=center, inner sep=0}, "x", from=2-1, to=1-4]
					\arrow["{{p_B}_2}"', from=2-1, to=2-3]
					\arrow["{r_B}", from=2-3, to=1-4]
					\arrow[""{name=1, anchor=center, inner sep=0}, equals, from=2-3, to=2-5]
					\arrow["{G_2}", two heads, from=2-5, to=2-6]
					\arrow["{\epsilon_B}", shorten <=5pt, shorten >=5pt, Rightarrow, from=1-4, to=1]
					\arrow["\chi"', shorten <=2pt, Rightarrow, from=0, to=2-3]
				\end{tikzcd},
			\end{align*}
		}
		which gives \longref{Equation}{eqt:epsilon}.
		By the weak universal property, there is an $\infty$-natural transformation
		% https://q.uiver.app/#q=WzAsMyxbMCwxLCJwXzIiXSxbMiwxLCJwXzIiXSxbMSwwLCJwXzEiXSxbMCwxLCIiLDIseyJsZXZlbCI6Miwic3R5bGUiOnsiaGVhZCI6eyJuYW1lIjoibm9uZSJ9fX1dLFswLDIsInIiXSxbMiwxLCJQIiwwLHsic3R5bGUiOnsiaGVhZCI6eyJuYW1lIjoiZXBpIn19fV0sWzIsMywiXFxlcHNpbG9uIiwyLHsic2hvcnRlbiI6eyJzb3VyY2UiOjMwLCJ0YXJnZXQiOjMwfX1dXQ==
		\[\begin{tikzcd}[ampersand replacement=\&]
			\& {p_1} \\
			{p_2} \&\& {p_2}
			\arrow["P", two heads, from=1-2, to=2-3]
			\arrow["r", from=2-1, to=1-2]
			\arrow[""{name=0, anchor=center, inner sep=0}, equals, from=2-1, to=2-3]
			\arrow["\epsilon"', shorten <=5pt, shorten >=5pt, Rightarrow, from=1-2, to=0]
		\end{tikzcd}\]
		such that
		\[
		% https://q.uiver.app/#q=WzAsNCxbMCwxLCJwXzIiXSxbMiwxLCJwXzIiXSxbMSwwLCJwXzEiXSxbMywxLCJiXzIiXSxbMCwxLCIiLDIseyJsZXZlbCI6Miwic3R5bGUiOnsiaGVhZCI6eyJuYW1lIjoibm9uZSJ9fX1dLFswLDIsInIiXSxbMiwxLCJQIiwwLHsic3R5bGUiOnsiaGVhZCI6eyJuYW1lIjoiZXBpIn19fV0sWzEsMywie3BfQn1fMiJdLFsyLDQsIlxcZXBzaWxvbiIsMix7InNob3J0ZW4iOnsic291cmNlIjozMCwidGFyZ2V0IjozMH19XV0=
		\begin{tikzcd}[ampersand replacement=\&]
			\& {p_1} \\
			{p_2} \&\& {p_2} \& {b_2}
			\arrow["P", two heads, from=1-2, to=2-3]
			\arrow["r", from=2-1, to=1-2]
			\arrow[""{name=0, anchor=center, inner sep=0}, equals, from=2-1, to=2-3]
			\arrow["{{p_B}_2}", from=2-3, to=2-4]
			\arrow["\epsilon"', shorten <=5pt, shorten >=5pt, Rightarrow, from=1-2, to=0]
		\end{tikzcd}
		\quad = \quad
		\begin{tikzcd}[ampersand replacement=\&]
			\&\&\& {b_1} \\
			{p_2} \&\& {b_2} \&\& {b_2}
			\arrow["B", two heads, from=1-4, to=2-5]
			\arrow[""{name=0, anchor=center, inner sep=0}, "x", from=2-1, to=1-4]
			\arrow["{{p_B}_2}"', from=2-1, to=2-3]
			\arrow["{r_B}", from=2-3, to=1-4]
			\arrow[""{name=1, anchor=center, inner sep=0}, equals, from=2-3, to=2-5]
			\arrow["{\epsilon_B}", shorten <=5pt, shorten >=5pt, Rightarrow, from=1-4, to=1]
			\arrow["\chi"', shorten <=2pt, Rightarrow, from=0, to=2-3]
		\end{tikzcd},
		\]
		and
		\[
		% https://q.uiver.app/#q=WzAsNCxbMCwxLCJwXzIiXSxbMiwxLCJwXzIiXSxbMSwwLCJwXzEiXSxbMywxLCJhXzIiXSxbMCwxLCIiLDIseyJsZXZlbCI6Miwic3R5bGUiOnsiaGVhZCI6eyJuYW1lIjoibm9uZSJ9fX1dLFswLDIsInIiXSxbMiwxLCJQIiwwLHsic3R5bGUiOnsiaGVhZCI6eyJuYW1lIjoiZXBpIn19fV0sWzEsMywie3BfQX1fMiJdLFsyLDQsIlxcZXBzaWxvbiIsMix7InNob3J0ZW4iOnsic291cmNlIjozMCwidGFyZ2V0IjozMH19XV0=
		\begin{tikzcd}[ampersand replacement=\&]
			\& {p_1} \\
			{p_2} \&\& {p_2} \& {a_2}
			\arrow["P", two heads, from=1-2, to=2-3]
			\arrow["r", from=2-1, to=1-2]
			\arrow[""{name=0, anchor=center, inner sep=0}, equals, from=2-1, to=2-3]
			\arrow["{{p_A}_2}", from=2-3, to=2-4]
			\arrow["\epsilon"', shorten <=5pt, shorten >=5pt, Rightarrow, from=1-2, to=0]
		\end{tikzcd}
		\quad = \quad
		\begin{tikzcd}[ampersand replacement=\&]
			\&\& {a_1} \\
			{p_1} \& {a_2} \&\& {a_2}
			\arrow["A", two heads, from=1-3, to=2-4]
			\arrow["{{p_A}_2}", two heads, from=2-1, to=2-2]
			\arrow["{r_A}", from=2-2, to=1-3]
			\arrow[""{name=0, anchor=center, inner sep=0}, equals, from=2-2, to=2-4]
			\arrow["{\epsilon_A}"', shorten <=5pt, shorten >=5pt, Rightarrow, from=1-3, to=0]
		\end{tikzcd}.
		\]
		
		It remains to show that $\eta$ and $\epsilon$ satisfy the triangle identities for unit and counit.
		
		Indeed, we have
		{\allowdisplaybreaks
			\begin{align*}
				&
				% https://q.uiver.app/#q=WzAsNixbMCwwLCJwXzIiXSxbMSwwLCJwXzIiXSxbMCwyLCJwXzEiXSxbMywwLCJiXzIiXSxbMSwyLCJwXzEiXSxbMywyLCJiXzEiXSxbMCwxLCIiLDIseyJsZXZlbCI6Miwic3R5bGUiOnsiaGVhZCI6eyJuYW1lIjoibm9uZSJ9fX1dLFswLDIsInIiLDJdLFsyLDEsIlAiLDEseyJzdHlsZSI6eyJoZWFkIjp7Im5hbWUiOiJlcGkifX19XSxbMSwzLCJ7cF9CfV8yIl0sWzEsNCwiciJdLFs0LDUsIntwX0J9XzEiLDJdLFszLDUsInJfQiJdLFs0LDMsIlxcY2hpIiwyLHsic2hvcnRlbiI6eyJzb3VyY2UiOjQwLCJ0YXJnZXQiOjQwfSwibGV2ZWwiOjJ9XSxbMiw0LCIiLDAseyJsZXZlbCI6Miwic3R5bGUiOnsiaGVhZCI6eyJuYW1lIjoibm9uZSJ9fX1dLFs4LDYsIlxcZXBzaWxvbiIsMix7InNob3J0ZW4iOnsic291cmNlIjozMCwidGFyZ2V0IjozMH19XSxbMTQsOCwiXFxldGEiLDIseyJzaG9ydGVuIjp7InNvdXJjZSI6MzAsInRhcmdldCI6MzB9fV1d
				\begin{tikzcd}[ampersand replacement=\&]
					{p_2} \& {p_2} \&\& {b_2} \\
					\\
					{p_1} \& {p_1} \&\& {b_1}
					\arrow[""{name=0, anchor=center, inner sep=0}, equals, from=1-1, to=1-2]
					\arrow["r"', from=1-1, to=3-1]
					\arrow["{{p_B}_2}", from=1-2, to=1-4]
					\arrow["r", from=1-2, to=3-2]
					\arrow["{r_B}", from=1-4, to=3-4]
					\arrow[""{name=1, anchor=center, inner sep=0}, "P"{description}, two heads, from=3-1, to=1-2]
					\arrow[""{name=2, anchor=center, inner sep=0}, equals, from=3-1, to=3-2]
					\arrow["\chi"', shorten <=22pt, shorten >=22pt, Rightarrow, from=3-2, to=1-4]
					\arrow["{{p_B}_1}"', from=3-2, to=3-4]
					\arrow["\epsilon"', shorten <=6pt, shorten >=6pt, Rightarrow, from=1, to=0]
					\arrow["\eta"', shorten <=6pt, shorten >=6pt, Rightarrow, from=2, to=1]
				\end{tikzcd}
				\; = \;
				% https://q.uiver.app/#q=WzAsNixbMCwwLCJwXzIiXSxbMSwwLCJwXzIiXSxbMCwyLCJwXzEiXSxbMywwLCJiXzIiXSxbMiwxLCJwXzEiXSxbMywyLCJiXzEiXSxbMCwxLCIiLDIseyJsZXZlbCI6Miwic3R5bGUiOnsiaGVhZCI6eyJuYW1lIjoibm9uZSJ9fX1dLFswLDIsInIiLDJdLFsyLDEsIlAiLDEseyJzdHlsZSI6eyJoZWFkIjp7Im5hbWUiOiJlcGkifX19XSxbMSwzLCJ7cF9CfV8yIl0sWzEsNCwiciJdLFs0LDUsIntwX0J9XzEiLDJdLFszLDUsInJfQiJdLFs0LDMsIlxcY2hpIiwyLHsic2hvcnRlbiI6eyJzb3VyY2UiOjQwLCJ0YXJnZXQiOjQwfSwibGV2ZWwiOjJ9XSxbMiw1LCJ7cF9CfV8xIiwyXSxbOCw2LCJcXGVwc2lsb24iLDIseyJzaG9ydGVuIjp7InNvdXJjZSI6MzAsInRhcmdldCI6MzB9fV0sWzE0LDEwLCJcXG92ZXJsaW5le1xcZ2FtbWF9IiwwLHsic2hvcnRlbiI6eyJzb3VyY2UiOjMwLCJ0YXJnZXQiOjQwfX1dXQ==
				\begin{tikzcd}[ampersand replacement=\&]
					{p_2} \& {p_2} \&\& {b_2} \\
					\&\& {p_1} \\
					{p_1} \&\&\& {b_1}
					\arrow[""{name=0, anchor=center, inner sep=0}, equals, from=1-1, to=1-2]
					\arrow["r"', from=1-1, to=3-1]
					\arrow["{{p_B}_2}", from=1-2, to=1-4]
					\arrow[""{name=1, anchor=center, inner sep=0}, "r", from=1-2, to=2-3]
					\arrow["{r_B}", from=1-4, to=3-4]
					\arrow["\chi"', shorten <=8pt, shorten >=8pt, Rightarrow, from=2-3, to=1-4]
					\arrow["{{p_B}_1}"', from=2-3, to=3-4]
					\arrow[""{name=2, anchor=center, inner sep=0}, "P"{description}, two heads, from=3-1, to=1-2]
					\arrow[""{name=3, anchor=center, inner sep=0}, "{{p_B}_1}"', from=3-1, to=3-4]
					\arrow["\epsilon"', shorten <=6pt, shorten >=6pt, Rightarrow, from=2, to=0]
					\arrow["{\overline{\gamma}}", shorten <=10pt, shorten >=13pt, Rightarrow, from=3, to=1]
				\end{tikzcd}
				\; = \;
				% https://q.uiver.app/#q=WzAsNixbMCwwLCJwXzIiXSxbMSwwLCJwXzIiXSxbMCwyLCJwXzEiXSxbMywwLCJiXzIiXSxbMywyLCJiXzEiXSxbMiwyLCJiXzEiXSxbMCwxLCIiLDIseyJsZXZlbCI6Miwic3R5bGUiOnsiaGVhZCI6eyJuYW1lIjoibm9uZSJ9fX1dLFswLDIsInIiLDJdLFsyLDEsIlAiLDIseyJzdHlsZSI6eyJoZWFkIjp7Im5hbWUiOiJlcGkifX19XSxbMSwzLCJ7cF9CfV8yIl0sWzMsNCwicl9CIl0sWzUsNCwiIiwyLHsibGV2ZWwiOjIsInN0eWxlIjp7ImhlYWQiOnsibmFtZSI6Im5vbmUifX19XSxbNSwzLCJCIiwwLHsic3R5bGUiOnsiaGVhZCI6eyJuYW1lIjoiZXBpIn19fV0sWzIsNSwie3BfQn1fMSIsMl0sWzgsNiwiXFxlcHNpbG9uIiwyLHsic2hvcnRlbiI6eyJzb3VyY2UiOjMwLCJ0YXJnZXQiOjMwfX1dLFsxMSwxMiwiXFxldGFfQiIsMix7InNob3J0ZW4iOnsic291cmNlIjozMCwidGFyZ2V0IjozMH19XV0=
				\begin{tikzcd}[ampersand replacement=\&]
					{p_2} \& {p_2} \&\& {b_2} \\
					\\
					{p_1} \&\& {b_1} \& {b_1}
					\arrow[""{name=0, anchor=center, inner sep=0}, equals, from=1-1, to=1-2]
					\arrow["r"', from=1-1, to=3-1]
					\arrow["{{p_B}_2}", from=1-2, to=1-4]
					\arrow["{r_B}", from=1-4, to=3-4]
					\arrow[""{name=1, anchor=center, inner sep=0}, "P"', two heads, from=3-1, to=1-2]
					\arrow["{{p_B}_1}"', from=3-1, to=3-3]
					\arrow[""{name=2, anchor=center, inner sep=0}, "B", two heads, from=3-3, to=1-4]
					\arrow[""{name=3, anchor=center, inner sep=0}, equals, from=3-3, to=3-4]
					\arrow["\epsilon"', shorten <=6pt, shorten >=6pt, Rightarrow, from=1, to=0]
					\arrow["{\eta_B}"', shorten <=6pt, shorten >=6pt, Rightarrow, from=3, to=2]
				\end{tikzcd}
				\\
				& = \quad
				% https://q.uiver.app/#q=WzAsNixbMCwwLCJwXzIiXSxbMCwyLCJwXzEiXSxbMywwLCJiXzIiXSxbMywyLCJiXzEiXSxbMiwyLCJiXzEiXSxbMiwwLCJiXzIiXSxbMCwxLCJyIiwyXSxbMiwzLCJyX0IiXSxbNCwzLCIiLDIseyJsZXZlbCI6Miwic3R5bGUiOnsiaGVhZCI6eyJuYW1lIjoibm9uZSJ9fX1dLFs0LDIsIkIiLDEseyJzdHlsZSI6eyJoZWFkIjp7Im5hbWUiOiJlcGkifX19XSxbMSw0LCJ7cF9CfV8xIiwyXSxbNSwyLCIiLDIseyJsZXZlbCI6Miwic3R5bGUiOnsiaGVhZCI6eyJuYW1lIjoibm9uZSJ9fX1dLFswLDUsIntwX0J9XzIiXSxbNSw0LCJyX0IiLDJdLFswLDQsIngiLDJdLFs4LDksIlxcZXRhX0IiLDIseyJzaG9ydGVuIjp7InNvdXJjZSI6MzAsInRhcmdldCI6MzB9fV0sWzksMTEsIlxcZXBzaWxvbl9CIiwwLHsic2hvcnRlbiI6eyJzb3VyY2UiOjMwLCJ0YXJnZXQiOjMwfX1dLFsxNCwxMiwiXFxjaGkiLDIseyJzaG9ydGVuIjp7InNvdXJjZSI6MzAsInRhcmdldCI6MzB9fV1d
				\begin{tikzcd}[ampersand replacement=\&]
					{p_2} \&\& {b_2} \& {b_2} \\
					\\
					{p_1} \&\& {b_1} \& {b_1}
					\arrow[""{name=0, anchor=center, inner sep=0}, "{{p_B}_2}", from=1-1, to=1-3]
					\arrow["r"', from=1-1, to=3-1]
					\arrow[""{name=1, anchor=center, inner sep=0}, "x"', from=1-1, to=3-3]
					\arrow[""{name=2, anchor=center, inner sep=0}, equals, from=1-3, to=1-4]
					\arrow["{r_B}"', from=1-3, to=3-3]
					\arrow["{r_B}", from=1-4, to=3-4]
					\arrow["{{p_B}_1}"', from=3-1, to=3-3]
					\arrow[""{name=3, anchor=center, inner sep=0}, "B"{description}, two heads, from=3-3, to=1-4]
					\arrow[""{name=4, anchor=center, inner sep=0}, equals, from=3-3, to=3-4]
					\arrow["\chi"', shorten <=6pt, shorten >=6pt, Rightarrow, from=1, to=0]
					\arrow["{\eta_B}"', shorten <=6pt, shorten >=6pt, Rightarrow, from=4, to=3]
					\arrow["{\epsilon_B}", shorten <=6pt, shorten >=6pt, Rightarrow, from=3, to=2]
				\end{tikzcd}
				\quad = \quad
				% https://q.uiver.app/#q=WzAsMyxbMCwxLCJwXzIiXSxbMiwxLCJiXzEiXSxbMSwwLCJiXzIiXSxbMCwyLCJ7cF9CfV8yIl0sWzIsMSwicl9CIl0sWzAsMSwieCIsMl0sWzUsMiwiXFxjaGkiLDIseyJzaG9ydGVuIjp7InNvdXJjZSI6MzAsInRhcmdldCI6MzB9fV1d
				\begin{tikzcd}[ampersand replacement=\&]
					\& {b_2} \\
					{p_2} \&\& {b_1}
					\arrow["{r_B}", from=1-2, to=2-3]
					\arrow["{{p_B}_2}", from=2-1, to=1-2]
					\arrow[""{name=0, anchor=center, inner sep=0}, "x"', from=2-1, to=2-3]
					\arrow["\chi"', shorten <=5pt, shorten >=5pt, Rightarrow, from=0, to=1-2]
				\end{tikzcd}.
			\end{align*}
		}
		Also, we have
		{\allowdisplaybreaks
			\begin{align*}
				% https://q.uiver.app/#q=WzAsNixbMCwwLCJwXzIiXSxbMSwwLCJwXzIiXSxbMCwyLCJwXzEiXSxbMSwyLCJwXzEiXSxbMiwyLCJiXzEiXSxbMywyLCJjXzEiXSxbMCwxLCIiLDIseyJsZXZlbCI6Miwic3R5bGUiOnsiaGVhZCI6eyJuYW1lIjoibm9uZSJ9fX1dLFswLDIsInIiLDJdLFsyLDEsIlAiLDEseyJzdHlsZSI6eyJoZWFkIjp7Im5hbWUiOiJlcGkifX19XSxbMyw0LCJ7cF9CfV8xIiwyXSxbMiwzLCIiLDAseyJsZXZlbCI6Miwic3R5bGUiOnsiaGVhZCI6eyJuYW1lIjoibm9uZSJ9fX1dLFs0LDUsIkdfMSIsMix7InN0eWxlIjp7ImhlYWQiOnsibmFtZSI6ImVwaSJ9fX1dLFsxLDMsInIiXSxbOCw2LCJcXGVwc2lsb24iLDIseyJzaG9ydGVuIjp7InNvdXJjZSI6MzAsInRhcmdldCI6MzB9fV0sWzEwLDgsIlxcZXRhIiwyLHsic2hvcnRlbiI6eyJzb3VyY2UiOjMwLCJ0YXJnZXQiOjMwfX1dXQ==
				\begin{tikzcd}[ampersand replacement=\&]
					{p_2} \& {p_2} \\
					\\
					{p_1} \& {p_1} \& {b_1} \& {c_1}
					\arrow[""{name=0, anchor=center, inner sep=0}, equals, from=1-1, to=1-2]
					\arrow["r"', from=1-1, to=3-1]
					\arrow["r", from=1-2, to=3-2]
					\arrow[""{name=1, anchor=center, inner sep=0}, "P"{description}, two heads, from=3-1, to=1-2]
					\arrow[""{name=2, anchor=center, inner sep=0}, equals, from=3-1, to=3-2]
					\arrow["{{p_B}_1}"', from=3-2, to=3-3]
					\arrow["{G_1}"', two heads, from=3-3, to=3-4]
					\arrow["\epsilon"', shorten <=6pt, shorten >=6pt, Rightarrow, from=1, to=0]
					\arrow["\eta"', shorten <=6pt, shorten >=6pt, Rightarrow, from=2, to=1]
				\end{tikzcd}
				\; &= \quad
				% https://q.uiver.app/#q=WzAsNixbMCwwLCJwXzIiXSxbMSwwLCJwXzIiXSxbMCwyLCJwXzEiXSxbMiwyLCJiXzEiXSxbMywyLCJjXzEiXSxbMiwwLCJwXzEiXSxbMCwxLCIiLDIseyJsZXZlbCI6Miwic3R5bGUiOnsiaGVhZCI6eyJuYW1lIjoibm9uZSJ9fX1dLFswLDIsInIiLDJdLFsyLDEsIlAiLDEseyJzdHlsZSI6eyJoZWFkIjp7Im5hbWUiOiJlcGkifX19XSxbMyw0LCJHXzEiLDIseyJzdHlsZSI6eyJoZWFkIjp7Im5hbWUiOiJlcGkifX19XSxbMSw1LCJyIl0sWzUsMywie3BfQn1fMSJdLFsyLDMsIntwX0J9XzEiLDJdLFs4LDYsIlxcZXBzaWxvbiIsMix7InNob3J0ZW4iOnsic291cmNlIjozMCwidGFyZ2V0IjozMH19XSxbMTIsMSwiXFxvdmVybGluZXtcXGdhbW1hfSIsMix7InNob3J0ZW4iOnsic291cmNlIjo0MCwidGFyZ2V0Ijo0MH19XV0=
				\begin{tikzcd}[ampersand replacement=\&]
					{p_2} \& {p_2} \& {p_1} \\
					\\
					{p_1} \&\& {b_1} \& {c_1}
					\arrow[""{name=0, anchor=center, inner sep=0}, equals, from=1-1, to=1-2]
					\arrow["r"', from=1-1, to=3-1]
					\arrow["r", from=1-2, to=1-3]
					\arrow["{{p_B}_1}", from=1-3, to=3-3]
					\arrow[""{name=1, anchor=center, inner sep=0}, "P"{description}, two heads, from=3-1, to=1-2]
					\arrow[""{name=2, anchor=center, inner sep=0}, "{{p_B}_1}"', from=3-1, to=3-3]
					\arrow["{G_1}"', two heads, from=3-3, to=3-4]
					\arrow["\epsilon"', shorten <=6pt, shorten >=6pt, Rightarrow, from=1, to=0]
					\arrow["{\overline{\gamma}}"', shorten <=15pt, shorten >=15pt, Rightarrow, from=2, to=1-2]
				\end{tikzcd}
				\\			
				\; = \quad
				% https://q.uiver.app/#q=WzAsNyxbMCwwLCJwXzIiXSxbMSwwLCJwXzIiXSxbMCwyLCJwXzEiXSxbMiwyLCJhXzEiXSxbMywyLCJjXzEiXSxbMiwwLCJhXzIiXSxbMSwyLCJhXzEiXSxbMCwxLCIiLDIseyJsZXZlbCI6Miwic3R5bGUiOnsiaGVhZCI6eyJuYW1lIjoibm9uZSJ9fX1dLFswLDIsInIiLDJdLFsyLDEsIlAiLDIseyJzdHlsZSI6eyJoZWFkIjp7Im5hbWUiOiJlcGkifX19XSxbMyw0LCJGXzEiLDJdLFsxLDUsIntwX0F9XzIiLDAseyJzdHlsZSI6eyJoZWFkIjp7Im5hbWUiOiJlcGkifX19XSxbNSwzLCJyX0EiXSxbNiw1LCJBIiwwLHsic3R5bGUiOnsiaGVhZCI6eyJuYW1lIjoiZXBpIn19fV0sWzYsMywiIiwyLHsibGV2ZWwiOjIsInN0eWxlIjp7ImhlYWQiOnsibmFtZSI6Im5vbmUifX19XSxbMiw2LCJ7cF9BfV8xIiwyLHsic3R5bGUiOnsiaGVhZCI6eyJuYW1lIjoiZXBpIn19fV0sWzksNywiXFxlcHNpbG9uIiwyLHsic2hvcnRlbiI6eyJzb3VyY2UiOjMwLCJ0YXJnZXQiOjMwfX1dLFsxNCwxMywiXFxldGFfQSIsMix7InNob3J0ZW4iOnsic291cmNlIjozMCwidGFyZ2V0IjozMH19XV0=
				\begin{tikzcd}[ampersand replacement=\&]
					{p_2} \& {p_2} \& {a_2} \\
					\\
					{p_1} \& {a_1} \& {a_1} \& {c_1}
					\arrow[""{name=0, anchor=center, inner sep=0}, equals, from=1-1, to=1-2]
					\arrow["r"', from=1-1, to=3-1]
					\arrow["{{p_A}_2}", two heads, from=1-2, to=1-3]
					\arrow["{r_A}", from=1-3, to=3-3]
					\arrow[""{name=1, anchor=center, inner sep=0}, "P"', two heads, from=3-1, to=1-2]
					\arrow["{{p_A}_1}"', two heads, from=3-1, to=3-2]
					\arrow[""{name=2, anchor=center, inner sep=0}, "A", two heads, from=3-2, to=1-3]
					\arrow[""{name=3, anchor=center, inner sep=0}, equals, from=3-2, to=3-3]
					\arrow["{F_1}"', from=3-3, to=3-4]
					\arrow["\epsilon"', shorten <=6pt, shorten >=6pt, Rightarrow, from=1, to=0]
					\arrow["{\eta_A}"', shorten <=6pt, shorten >=6pt, Rightarrow, from=3, to=2]
				\end{tikzcd}
				\; &= \quad
				% https://q.uiver.app/#q=WzAsNixbMCwwLCJwXzIiXSxbMSwwLCJhXzIiXSxbMiwyLCJhXzEiXSxbMywyLCJjXzEiXSxbMiwwLCJhXzIiXSxbMSwyLCJhXzEiXSxbMCwxLCJ7cF9BfV8yIiwwLHsic3R5bGUiOnsiaGVhZCI6eyJuYW1lIjoiZXBpIn19fV0sWzIsMywiRl8xIiwyXSxbMSw0LCIiLDAseyJsZXZlbCI6Miwic3R5bGUiOnsiaGVhZCI6eyJuYW1lIjoibm9uZSJ9fX1dLFs0LDIsInJfQSJdLFs1LDQsIkEiLDEseyJzdHlsZSI6eyJoZWFkIjp7Im5hbWUiOiJlcGkifX19XSxbNSwyLCIiLDIseyJsZXZlbCI6Miwic3R5bGUiOnsiaGVhZCI6eyJuYW1lIjoibm9uZSJ9fX1dLFsxLDUsInJfQSIsMl0sWzExLDEwLCJcXGV0YV9BIiwyLHsic2hvcnRlbiI6eyJzb3VyY2UiOjMwLCJ0YXJnZXQiOjMwfX1dLFsxMCw4LCJcXGVwc2lsb25fQSIsMCx7InNob3J0ZW4iOnsic291cmNlIjozMCwidGFyZ2V0IjozMH19XV0=
				\begin{tikzcd}[ampersand replacement=\&]
					{p_2} \& {a_2} \& {a_2} \\
					\\
					\& {a_1} \& {a_1} \& {c_1}
					\arrow["{{p_A}_2}", two heads, from=1-1, to=1-2]
					\arrow[""{name=0, anchor=center, inner sep=0}, equals, from=1-2, to=1-3]
					\arrow["{r_A}"', from=1-2, to=3-2]
					\arrow["{r_A}", from=1-3, to=3-3]
					\arrow[""{name=1, anchor=center, inner sep=0}, "A"{description}, two heads, from=3-2, to=1-3]
					\arrow[""{name=2, anchor=center, inner sep=0}, equals, from=3-2, to=3-3]
					\arrow["{F_1}"', from=3-3, to=3-4]
					\arrow["{\eta_A}"', shorten <=6pt, shorten >=6pt, Rightarrow, from=2, to=1]
					\arrow["{\epsilon_A}", shorten <=6pt, shorten >=6pt, Rightarrow, from=1, to=0]
				\end{tikzcd}
			\end{align*}
		}	
		\noindent which is an identity $\infty$-natural transformation at $F_1 \cdot r_A \cdot {p_A}_2 = G_1 \cdot x$. By \cite[Proposition 5.2.11]{book:RV:2022}, the composite
		% https://q.uiver.app/#q=WzAsNSxbMCwwLCJwXzIiXSxbMSwwLCJwXzIiXSxbMCwxLCJwXzEiXSxbMSwxLCJwXzEiXSxbMiwxLCJiXzEiXSxbMCwxLCIiLDIseyJsZXZlbCI6Miwic3R5bGUiOnsiaGVhZCI6eyJuYW1lIjoibm9uZSJ9fX1dLFswLDIsInIiLDJdLFsyLDEsIlAiLDEseyJzdHlsZSI6eyJoZWFkIjp7Im5hbWUiOiJlcGkifX19XSxbMyw0LCJ7cF9CfV8xIiwyXSxbMiwzLCIiLDAseyJsZXZlbCI6Miwic3R5bGUiOnsiaGVhZCI6eyJuYW1lIjoibm9uZSJ9fX1dLFsxLDMsInIiXSxbNyw1LCJcXGVwc2lsb24iLDIseyJzaG9ydGVuIjp7InNvdXJjZSI6MzAsInRhcmdldCI6MzB9fV0sWzksNywiXFxldGEiLDIseyJzaG9ydGVuIjp7InNvdXJjZSI6MzAsInRhcmdldCI6MzB9fV1d
		\[\begin{tikzcd}[ampersand replacement=\&]
			{p_2} \& {p_2} \\
			{p_1} \& {p_1} \& {b_1}
			\arrow[""{name=0, anchor=center, inner sep=0}, equals, from=1-1, to=1-2]
			\arrow["r"', from=1-1, to=2-1]
			\arrow["r", from=1-2, to=2-2]
			\arrow[""{name=1, anchor=center, inner sep=0}, "P"{description}, two heads, from=2-1, to=1-2]
			\arrow[""{name=2, anchor=center, inner sep=0}, equals, from=2-1, to=2-2]
			\arrow["{{p_B}_1}"', from=2-2, to=2-3]
			\arrow["\epsilon"', shorten <=3pt, shorten >=3pt, Rightarrow, from=1, to=0]
			\arrow["\eta"', shorten <=3pt, shorten >=3pt, Rightarrow, from=2, to=1]
		\end{tikzcd}\]
		is invertible. Furthermore, we have
		\[
		\begin{aligned}
			% https://q.uiver.app/#q=WzAsNSxbMCwwLCJwXzIiXSxbMSwwLCJwXzIiXSxbMCwxLCJwXzEiXSxbMSwxLCJwXzEiXSxbMiwxLCJhXzEiXSxbMCwxLCIiLDIseyJsZXZlbCI6Miwic3R5bGUiOnsiaGVhZCI6eyJuYW1lIjoibm9uZSJ9fX1dLFswLDIsInIiLDJdLFsyLDEsIlAiLDEseyJzdHlsZSI6eyJoZWFkIjp7Im5hbWUiOiJlcGkifX19XSxbMyw0LCJ7cF9BfV8xIiwyLHsic3R5bGUiOnsiaGVhZCI6eyJuYW1lIjoiZXBpIn19fV0sWzIsMywiIiwwLHsibGV2ZWwiOjIsInN0eWxlIjp7ImhlYWQiOnsibmFtZSI6Im5vbmUifX19XSxbMSwzLCJyIl0sWzcsNSwiXFxlcHNpbG9uIiwyLHsic2hvcnRlbiI6eyJzb3VyY2UiOjMwLCJ0YXJnZXQiOjMwfX1dLFs5LDcsIlxcZXRhIiwyLHsic2hvcnRlbiI6eyJzb3VyY2UiOjMwLCJ0YXJnZXQiOjMwfX1dXQ==
			\begin{tikzcd}[ampersand replacement=\&]
				{p_2} \& {p_2} \\
				{p_1} \& {p_1} \& {a_1}
				\arrow[""{name=0, anchor=center, inner sep=0}, equals, from=1-1, to=1-2]
				\arrow["r"', from=1-1, to=2-1]
				\arrow["r", from=1-2, to=2-2]
				\arrow[""{name=1, anchor=center, inner sep=0}, "P"{description}, two heads, from=2-1, to=1-2]
				\arrow[""{name=2, anchor=center, inner sep=0}, equals, from=2-1, to=2-2]
				\arrow["{{p_A}_1}"', two heads, from=2-2, to=2-3]
				\arrow["\epsilon"', shorten <=3pt, shorten >=3pt, Rightarrow, from=1, to=0]
				\arrow["\eta"', shorten <=3pt, shorten >=3pt, Rightarrow, from=2, to=1]
			\end{tikzcd}
			& \quad = \quad
			% https://q.uiver.app/#q=WzAsNSxbMCwwLCJwXzIiXSxbMSwwLCJhXzIiXSxbMSwxLCJhXzEiXSxbMiwxLCJhXzEiXSxbMiwwLCJhXzIiXSxbMiwzLCIiLDIseyJsZXZlbCI6Miwic3R5bGUiOnsiaGVhZCI6eyJuYW1lIjoibm9uZSJ9fX1dLFsxLDIsInJfQSIsMl0sWzAsMSwie3BfQX1fMiIsMCx7InN0eWxlIjp7ImhlYWQiOnsibmFtZSI6ImVwaSJ9fX1dLFs0LDMsInJfQSJdLFsxLDQsIiIsMCx7ImxldmVsIjoyLCJzdHlsZSI6eyJoZWFkIjp7Im5hbWUiOiJub25lIn19fV0sWzIsNCwiQSIsMSx7InN0eWxlIjp7ImhlYWQiOnsibmFtZSI6ImVwaSJ9fX1dLFsxMCw1LCJcXGV0YV9BIiwwLHsic2hvcnRlbiI6eyJzb3VyY2UiOjMwLCJ0YXJnZXQiOjMwfX1dLFs5LDEwLCJcXGVwc2lsb25fQSIsMix7InNob3J0ZW4iOnsic291cmNlIjozMCwidGFyZ2V0IjozMH19XV0=
			\begin{tikzcd}[ampersand replacement=\&]
				{p_2} \& {a_2} \& {a_2} \\
				\& {a_1} \& {a_1}
				\arrow["{{p_A}_2}", two heads, from=1-1, to=1-2]
				\arrow[""{name=0, anchor=center, inner sep=0}, equals, from=1-2, to=1-3]
				\arrow["{r_A}"', from=1-2, to=2-2]
				\arrow["{r_A}", from=1-3, to=2-3]
				\arrow[""{name=1, anchor=center, inner sep=0}, "A"{description}, two heads, from=2-2, to=1-3]
				\arrow[""{name=2, anchor=center, inner sep=0}, equals, from=2-2, to=2-3]
				\arrow["{\epsilon_A}"', shorten <=3pt, shorten >=3pt, Rightarrow, from=0, to=1]
				\arrow["{\eta_A}", shorten <=3pt, shorten >=3pt, Rightarrow, from=1, to=2]
			\end{tikzcd},
		\end{aligned}
		\]
		which is clearly an identity. Altogether, by the weak universal property,  we conclude that
		% https://q.uiver.app/#q=WzAsNCxbMCwwLCJwXzIiXSxbMSwwLCJwXzIiXSxbMCwxLCJwXzEiXSxbMSwxLCJwXzEiXSxbMCwxLCIiLDIseyJsZXZlbCI6Miwic3R5bGUiOnsiaGVhZCI6eyJuYW1lIjoibm9uZSJ9fX1dLFswLDIsInIiLDJdLFsyLDEsIlAiLDEseyJzdHlsZSI6eyJoZWFkIjp7Im5hbWUiOiJlcGkifX19XSxbMiwzLCIiLDAseyJsZXZlbCI6Miwic3R5bGUiOnsiaGVhZCI6eyJuYW1lIjoibm9uZSJ9fX1dLFsxLDMsInIiXSxbNiw0LCJcXGVwc2lsb24iLDIseyJzaG9ydGVuIjp7InNvdXJjZSI6MzAsInRhcmdldCI6MzB9fV0sWzcsNiwiXFxldGEiLDIseyJzaG9ydGVuIjp7InNvdXJjZSI6MzAsInRhcmdldCI6MzB9fV1d
		\[\begin{tikzcd}[ampersand replacement=\&]
			{p_2} \& {p_2} \\
			{p_1} \& {p_1}
			\arrow[""{name=0, anchor=center, inner sep=0}, equals, from=1-1, to=1-2]
			\arrow["r"', from=1-1, to=2-1]
			\arrow["r", from=1-2, to=2-2]
			\arrow[""{name=1, anchor=center, inner sep=0}, "P"{description}, two heads, from=2-1, to=1-2]
			\arrow[""{name=2, anchor=center, inner sep=0}, equals, from=2-1, to=2-2]
			\arrow["\epsilon"', shorten <=3pt, shorten >=3pt, Rightarrow, from=1, to=0]
			\arrow["\eta"', shorten <=3pt, shorten >=3pt, Rightarrow, from=2, to=1]
		\end{tikzcd}\]
		is invertible.
		
		Next,
		\[
		\begin{aligned}
			&
			% https://q.uiver.app/#q=WzAsNSxbMCwxLCJwXzIiXSxbMSwxLCJwXzIiXSxbMCwwLCJwXzEiXSxbMSwwLCJwXzEiXSxbMywxLCJiXzIiXSxbMCwxLCIiLDIseyJsZXZlbCI6Miwic3R5bGUiOnsiaGVhZCI6eyJuYW1lIjoibm9uZSJ9fX1dLFsyLDMsIiIsMCx7ImxldmVsIjoyLCJzdHlsZSI6eyJoZWFkIjp7Im5hbWUiOiJub25lIn19fV0sWzIsMCwiUCIsMix7InN0eWxlIjp7ImhlYWQiOnsibmFtZSI6ImVwaSJ9fX1dLFswLDMsInIiLDFdLFsxLDQsIntwX0J9XzIiLDJdLFszLDEsIlAiLDAseyJzdHlsZSI6eyJoZWFkIjp7Im5hbWUiOiJlcGkifX19XSxbNiw4LCJcXGV0YSIsMix7InNob3J0ZW4iOnsic291cmNlIjozMCwidGFyZ2V0IjozMH19XSxbOCw1LCJcXGVwc2lsb24iLDAseyJzaG9ydGVuIjp7InNvdXJjZSI6MzAsInRhcmdldCI6MzB9fV1d
			\begin{tikzcd}[ampersand replacement=\&]
				{p_1} \& {p_1} \\
				{p_2} \& {p_2} \&\& {b_2}
				\arrow[""{name=0, anchor=center, inner sep=0}, equals, from=1-1, to=1-2]
				\arrow["P"', two heads, from=1-1, to=2-1]
				\arrow["P", two heads, from=1-2, to=2-2]
				\arrow[""{name=1, anchor=center, inner sep=0}, "r"{description}, from=2-1, to=1-2]
				\arrow[""{name=2, anchor=center, inner sep=0}, equals, from=2-1, to=2-2]
				\arrow["{{p_B}_2}"', from=2-2, to=2-4]
				\arrow["\eta"', shorten <=3pt, shorten >=3pt, Rightarrow, from=0, to=1]
				\arrow["\epsilon", shorten <=3pt, shorten >=3pt, Rightarrow, from=1, to=2]
			\end{tikzcd}
			\quad = \quad
			% https://q.uiver.app/#q=WzAsNixbMCwxLCJwXzIiXSxbMSwxLCJiXzIiXSxbMCwwLCJwXzEiXSxbMSwwLCJwXzEiXSxbMywxLCJiXzIiXSxbMiwwLCJiXzEiXSxbMCwxLCJ7cF9CfV8yIiwyXSxbMiwzLCIiLDAseyJsZXZlbCI6Miwic3R5bGUiOnsiaGVhZCI6eyJuYW1lIjoibm9uZSJ9fX1dLFsyLDAsIlAiLDIseyJzdHlsZSI6eyJoZWFkIjp7Im5hbWUiOiJlcGkifX19XSxbMCwzLCJyIiwxXSxbMSw0LCJ7cF9CfV8yIiwyLHsibGV2ZWwiOjIsInN0eWxlIjp7ImhlYWQiOnsibmFtZSI6Im5vbmUifX19XSxbMSw1LCJyX0IiLDFdLFs1LDQsIkIiLDAseyJzdHlsZSI6eyJoZWFkIjp7Im5hbWUiOiJlcGkifX19XSxbMyw1LCJ7cF9CfV8xIl0sWzMsMSwiXFxjaGkiLDAseyJzaG9ydGVuIjp7InNvdXJjZSI6MjAsInRhcmdldCI6MjB9LCJsZXZlbCI6Mn1dLFs3LDksIlxcZXRhIiwyLHsic2hvcnRlbiI6eyJzb3VyY2UiOjMwLCJ0YXJnZXQiOjMwfX1dLFs1LDEwLCJcXGVwc2lsb25fQiIsMCx7InNob3J0ZW4iOnsic291cmNlIjozMCwidGFyZ2V0IjozMH19XV0=
			\begin{tikzcd}[ampersand replacement=\&]
				{p_1} \& {p_1} \& {b_1} \\
				{p_2} \& {b_2} \&\& {b_2}
				\arrow[""{name=0, anchor=center, inner sep=0}, equals, from=1-1, to=1-2]
				\arrow["P"', two heads, from=1-1, to=2-1]
				\arrow["{{p_B}_1}", from=1-2, to=1-3]
				\arrow["\chi", shorten <=2pt, shorten >=2pt, Rightarrow, from=1-2, to=2-2]
				\arrow["B", two heads, from=1-3, to=2-4]
				\arrow[""{name=1, anchor=center, inner sep=0}, "r"{description}, from=2-1, to=1-2]
				\arrow["{{p_B}_2}"', from=2-1, to=2-2]
				\arrow["{r_B}"{description}, from=2-2, to=1-3]
				\arrow[""{name=2, anchor=center, inner sep=0}, "{{p_B}_2}"', equals, from=2-2, to=2-4]
				\arrow["\eta"', shorten <=3pt, shorten >=3pt, Rightarrow, from=0, to=1]
				\arrow["{\epsilon_B}", shorten <=5pt, shorten >=5pt, Rightarrow, from=1-3, to=2]
			\end{tikzcd}
			\\
			\quad &= \quad
			% https://q.uiver.app/#q=WzAsNixbMCwyLCJwXzIiXSxbMiwyLCJiXzIiXSxbMCwwLCJwXzEiXSxbMSwxLCJwXzEiXSxbMywyLCJiXzIiXSxbMiwwLCJiXzEiXSxbMCwxLCJ7cF9CfV8yIiwyXSxbMiwwLCJQIiwyLHsic3R5bGUiOnsiaGVhZCI6eyJuYW1lIjoiZXBpIn19fV0sWzAsMywiciIsMV0sWzEsNCwie3BfQn1fMiIsMix7ImxldmVsIjoyLCJzdHlsZSI6eyJoZWFkIjp7Im5hbWUiOiJub25lIn19fV0sWzEsNSwicl9CIiwxXSxbNSw0LCJCIiwwLHsic3R5bGUiOnsiaGVhZCI6eyJuYW1lIjoiZXBpIn19fV0sWzMsNSwie3BfQn1fMSJdLFszLDEsIlxcY2hpIiwwLHsic2hvcnRlbiI6eyJzb3VyY2UiOjIwLCJ0YXJnZXQiOjIwfSwibGV2ZWwiOjJ9XSxbMiw1LCJ7cF9CfV8xIl0sWzExLDksIlxcZXBzaWxvbl9CIiwwLHsic2hvcnRlbiI6eyJzb3VyY2UiOjMwLCJ0YXJnZXQiOjMwfX1dLFsxNCwzLCJcXG92ZXJsaW5le1xcZ2FtbWF9IiwyLHsic2hvcnRlbiI6eyJzb3VyY2UiOjMwLCJ0YXJnZXQiOjIwfX1dXQ==
			\begin{tikzcd}[ampersand replacement=\&]
				{p_1} \&\& {b_1} \\
				\& {p_1} \\
				{p_2} \&\& {b_2} \& {b_2}
				\arrow[""{name=0, anchor=center, inner sep=0}, "{{p_B}_1}", from=1-1, to=1-3]
				\arrow["P"', two heads, from=1-1, to=3-1]
				\arrow[""{name=1, anchor=center, inner sep=0}, "B", two heads, from=1-3, to=3-4]
				\arrow["{{p_B}_1}", from=2-2, to=1-3]
				\arrow["\chi", shorten <=4pt, shorten >=4pt, Rightarrow, from=2-2, to=3-3]
				\arrow["r"{description}, from=3-1, to=2-2]
				\arrow["{{p_B}_2}"', from=3-1, to=3-3]
				\arrow["{r_B}"{description}, from=3-3, to=1-3]
				\arrow[""{name=2, anchor=center, inner sep=0}, "{{p_B}_2}"', equals, from=3-3, to=3-4]
				\arrow["{\overline{\gamma}}"', shorten <=5pt, shorten >=3pt, Rightarrow, from=0, to=2-2]
				\arrow["{\epsilon_B}", shorten <=6pt, shorten >=6pt, Rightarrow, from=1, to=2]
			\end{tikzcd}
			\quad = \quad
			% https://q.uiver.app/#q=WzAsNSxbMiwxLCJiXzIiXSxbMCwwLCJwXzEiXSxbMywxLCJiXzIiXSxbMywwLCJiXzEiXSxbMiwwLCJiXzEiXSxbMCwyLCIiLDIseyJsZXZlbCI6Miwic3R5bGUiOnsiaGVhZCI6eyJuYW1lIjoibm9uZSJ9fX1dLFswLDMsInJfQiIsMV0sWzMsMiwiQiIsMCx7InN0eWxlIjp7ImhlYWQiOnsibmFtZSI6ImVwaSJ9fX1dLFs0LDMsIiIsMix7ImxldmVsIjoyLCJzdHlsZSI6eyJoZWFkIjp7Im5hbWUiOiJub25lIn19fV0sWzQsMCwiQiIsMix7InN0eWxlIjp7ImhlYWQiOnsibmFtZSI6ImVwaSJ9fX1dLFsxLDQsIntwX0J9XzEiXSxbNiw1LCJcXGVwc2lsb25fQiIsMCx7InNob3J0ZW4iOnsic291cmNlIjozMCwidGFyZ2V0IjozMH19XSxbOCw2LCJcXGV0YV9CIiwyLHsic2hvcnRlbiI6eyJzb3VyY2UiOjMwLCJ0YXJnZXQiOjMwfX1dXQ==
			\begin{tikzcd}[ampersand replacement=\&]
				{p_1} \&\& {b_1} \& {b_1} \\
				\&\& {b_2} \& {b_2}
				\arrow["{{p_B}_1}", from=1-1, to=1-3]
				\arrow[""{name=0, anchor=center, inner sep=0}, equals, from=1-3, to=1-4]
				\arrow["B"', two heads, from=1-3, to=2-3]
				\arrow["B", two heads, from=1-4, to=2-4]
				\arrow[""{name=1, anchor=center, inner sep=0}, "{r_B}"{description}, from=2-3, to=1-4]
				\arrow[""{name=2, anchor=center, inner sep=0}, equals, from=2-3, to=2-4]
				\arrow["{\eta_B}"', shorten <=3pt, shorten >=3pt, Rightarrow, from=0, to=1]
				\arrow["{\epsilon_B}", shorten <=3pt, shorten >=3pt, Rightarrow, from=1, to=2]
			\end{tikzcd},
		\end{aligned}
		\]
		which gives the identity. And, similarly,
		\[
		\begin{aligned}
			% https://q.uiver.app/#q=WzAsNSxbMCwxLCJwXzIiXSxbMSwxLCJwXzIiXSxbMCwwLCJwXzEiXSxbMSwwLCJwXzEiXSxbMywxLCJhXzIiXSxbMCwxLCIiLDIseyJsZXZlbCI6Miwic3R5bGUiOnsiaGVhZCI6eyJuYW1lIjoibm9uZSJ9fX1dLFsyLDMsIiIsMCx7ImxldmVsIjoyLCJzdHlsZSI6eyJoZWFkIjp7Im5hbWUiOiJub25lIn19fV0sWzIsMCwiUCIsMix7InN0eWxlIjp7ImhlYWQiOnsibmFtZSI6ImVwaSJ9fX1dLFswLDMsInIiLDFdLFsxLDQsIntwX0F9XzIiLDJdLFszLDEsIlAiLDAseyJzdHlsZSI6eyJoZWFkIjp7Im5hbWUiOiJlcGkifX19XSxbNiw4LCJcXGV0YSIsMix7InNob3J0ZW4iOnsic291cmNlIjozMCwidGFyZ2V0IjozMH19XSxbOCw1LCJcXGVwc2lsb24iLDAseyJzaG9ydGVuIjp7InNvdXJjZSI6MzAsInRhcmdldCI6MzB9fV1d
			\begin{tikzcd}[ampersand replacement=\&]
				{p_1} \& {p_1} \\
				{p_2} \& {p_2} \&\& {a_2}
				\arrow[""{name=0, anchor=center, inner sep=0}, equals, from=1-1, to=1-2]
				\arrow["P"', two heads, from=1-1, to=2-1]
				\arrow["P", two heads, from=1-2, to=2-2]
				\arrow[""{name=1, anchor=center, inner sep=0}, "r"{description}, from=2-1, to=1-2]
				\arrow[""{name=2, anchor=center, inner sep=0}, equals, from=2-1, to=2-2]
				\arrow["{{p_A}_2}"', from=2-2, to=2-4]
				\arrow["\eta"', shorten <=3pt, shorten >=3pt, Rightarrow, from=0, to=1]
				\arrow["\epsilon", shorten <=3pt, shorten >=3pt, Rightarrow, from=1, to=2]
			\end{tikzcd}
			& \quad = \quad
			% https://q.uiver.app/#q=WzAsNSxbMiwxLCJhXzIiXSxbMCwwLCJwXzEiXSxbMywxLCJhXzIiXSxbMywwLCJhXzEiXSxbMiwwLCJhXzEiXSxbMCwyLCIiLDIseyJsZXZlbCI6Miwic3R5bGUiOnsiaGVhZCI6eyJuYW1lIjoibm9uZSJ9fX1dLFszLDIsIkEiLDAseyJzdHlsZSI6eyJoZWFkIjp7Im5hbWUiOiJlcGkifX19XSxbNCwzLCIiLDIseyJsZXZlbCI6Miwic3R5bGUiOnsiaGVhZCI6eyJuYW1lIjoibm9uZSJ9fX1dLFs0LDAsIkEiLDIseyJzdHlsZSI6eyJoZWFkIjp7Im5hbWUiOiJlcGkifX19XSxbMSw0LCJ7cF9BfV8xIiwwLHsic3R5bGUiOnsiaGVhZCI6eyJuYW1lIjoiZXBpIn19fV0sWzAsMywicl9BIiwxXSxbNywxMCwiXFxldGFfQkEiLDIseyJzaG9ydGVuIjp7InNvdXJjZSI6MzAsInRhcmdldCI6MzB9fV0sWzEwLDUsIlxcZXBzaWxvbl9BIiwwLHsic2hvcnRlbiI6eyJzb3VyY2UiOjMwLCJ0YXJnZXQiOjMwfX1dXQ==
			\begin{tikzcd}[ampersand replacement=\&]
				{p_1} \&\& {a_1} \& {a_1} \\
				\&\& {a_2} \& {a_2}
				\arrow["{{p_A}_1}", two heads, from=1-1, to=1-3]
				\arrow[""{name=0, anchor=center, inner sep=0}, equals, from=1-3, to=1-4]
				\arrow["A"', two heads, from=1-3, to=2-3]
				\arrow["A", two heads, from=1-4, to=2-4]
				\arrow[""{name=1, anchor=center, inner sep=0}, "{r_A}"{description}, from=2-3, to=1-4]
				\arrow[""{name=2, anchor=center, inner sep=0}, equals, from=2-3, to=2-4]
				\arrow["{\eta_BA}"', shorten <=3pt, shorten >=3pt, Rightarrow, from=0, to=1]
				\arrow["{\epsilon_A}", shorten <=3pt, shorten >=3pt, Rightarrow, from=1, to=2]
			\end{tikzcd}
		\end{aligned}
		\]
		equals the identity. Therefore, by the weak universal property of pullbacks, we conclude that is invertible.	By \cite[Lemma 2.1.11]{book:RV:2022}, the triangle identities hold, so we are done.
	\end{proof}
	
	Clearly, the proof above could be dualised to obtain a dual version of the theorem.
	
	\begin{theorem}
		\label{thm:pullback_along_Cartesian_fib_Ra}
		Let $\calK$ be an $\infty$-cosmos. The enhanced simplicial category $\Ra(\calK)$ of ra-isofibrations of $\calK$ admits pullbacks of a loose $0$-arrow along a tight coCartesian fibration, which are preserved by the inclusion
		$$\Ra(\calK) \hookrightarrow \calK^{\isof}_\chor.$$ The projection is a morphism of right adjoints, and reflects morphisms of right adjoints.
	\end{theorem}
	
	However, the enhanced simplicial category $\Lali(\calK)$ might not admit pullbacks of a loose $0$-arrow along a tight Cartesian fibration. Even if we assume further that the counits $\epsilon_A$, $\epsilon_B$, $\epsilon_C$ in the proof of \longref{Theorem}{thm:pullback_along_Cartesian_fib_La}  are invertible, we are still unable to force the constructed $\epsilon$ to be invertible. The issue here is caused by the non-uniqueness of a Cartesian lift. More precisely, we know that
	\[
	\begin{tikzcd}[ampersand replacement=\&]
		\&\&\& {b_1} \\
		{p_2} \&\& {b_2} \&\& {b_2} \& {c_2}
		\arrow["B", two heads, from=1-4, to=2-5]
		\arrow[""{name=0, anchor=center, inner sep=0}, "x", from=2-1, to=1-4]
		\arrow["{{p_B}_2}"', from=2-1, to=2-3]
		\arrow["{r_B}", from=2-3, to=1-4]
		\arrow[""{name=1, anchor=center, inner sep=0}, equals, from=2-3, to=2-5]
		\arrow["{G_2}", two heads, from=2-5, to=2-6]
		\arrow["{\epsilon_B}", shorten <=5pt, shorten >=5pt, Rightarrow, from=1-4, to=1]
		\arrow["\chi"', shorten <=2pt, Rightarrow, from=0, to=2-3]
	\end{tikzcd}
	\]
	is invertible, because $\epsilon_A$ is assumed to be invertible. Then $G_2$, as an isofibration, should provide an invertible lift, but there is no guaranteed uniqueness of lifts, and thus, we cannot conclude that
	\[
	\begin{tikzcd}[ampersand replacement=\&]
		\&\&\& {b_1} \\
		{p_2} \&\& {b_2} \&\& {b_2}
		\arrow["B", two heads, from=1-4, to=2-5]
		\arrow[""{name=0, anchor=center, inner sep=0}, "x", from=2-1, to=1-4]
		\arrow["{{p_B}_2}"', from=2-1, to=2-3]
		\arrow["{r_B}", from=2-3, to=1-4]
		\arrow[""{name=1, anchor=center, inner sep=0}, equals, from=2-3, to=2-5]
		\arrow["{\epsilon_B}", shorten <=5pt, shorten >=5pt, Rightarrow, from=1-4, to=1]
		\arrow["\chi"', shorten <=2pt, Rightarrow, from=0, to=2-3]
	\end{tikzcd}
	\]
	must be invertible. As a consequence, we are unable to apply the weak universal property to get an invertible $\epsilon$, so the constructed pullback cannot be a lali. This same issue happens for $\Rali(\calK)$ as well.
	
	\subsection{\texorpdfstring{$\JLim(\calK)$}{J-Lim(K)} and \texorpdfstring{$\JColim(\calK)$}{J-Coim(K)}}
	Indeed, the enhanced simplicial category $\JLim(\calK)$ of $\infty$-categories of an $\infty$-cosmos $\calK$ that have $J$-shaped limits for a simplicial set $J$ can be formulated as the pullback of $\La(\calK)$, other than  $\Lali(\calK)$.
	
	To see this, we recall the definition of $\infty$-categories with limits, and a proposition on characterising $\infty$-categories with limits using lalis in the book by Riehl and Verity.
	
	\begin{defi}[{\cite[Definition 2.3.2]{book:RV:2022}}]
		An $\infty$-category $A$ of an $\infty$-cosmos $\calK$ \emph{admits all $J$-shaped limits} if and only if the constant diagram functor $\triangle \colon A \to A^J$ admits a right adjoint.
	\end{defi}
	
	\begin{pro}[{\cite[Proposition 4.3.4]{book:RV:2022}}]
		An $\infty$-category $A$ admits a limit of a family of diagrams $d \colon D \to A^J$ indexed by a simplicial set $J$ if and only if there is an absolute right lifting of $d$ 
		% https://q.uiver.app/#q=WzAsMyxbMCwxLCJEIl0sWzEsMSwiQV5KIl0sWzEsMCwiQV57Sl9cXHRyaWFuZ2xlbGVmdH0iXSxbMiwxLCIiLDIseyJzdHlsZSI6eyJoZWFkIjp7Im5hbWUiOiJlcGkifX19XSxbMCwxLCJkIiwyXSxbMCwyXSxbNSw0LCJcXGVwc2lsb24iLDAseyJzaG9ydGVuIjp7InNvdXJjZSI6MzAsInRhcmdldCI6MzB9fV1d
		\[\begin{tikzcd}[ampersand replacement=\&]
			\& {A^{J_\triangleleft}} \\
			D \& {A^J}
			\arrow["\res", two heads, from=1-2, to=2-2]
			\arrow[""{name=0, anchor=center, inner sep=0}, from=2-1, to=1-2]
			\arrow[""{name=1, anchor=center, inner sep=0}, "d"', from=2-1, to=2-2]
			\arrow["\epsilon", shorten <=3pt, shorten >=3pt, Rightarrow, from=0, to=1]
		\end{tikzcd}\]
		through the restriction $\res \colon A^{J_\triangleleft} \twoheadrightarrow A^J$. 
		
		When these equivalent conditions hold, $\epsilon$ is necessarily an isomorphism and may be chosen to be the identity.
	\end{pro}
	
	In particular, if $A$ has \emph{all} $J$-shaped limits, then the above proposition says that there is an absolute right lifting
	% https://q.uiver.app/#q=WzAsMyxbMCwxLCJBXkoiXSxbMSwxLCJBXkoiXSxbMSwwLCJBXntKX1xcdHJpYW5nbGVsZWZ0fSJdLFswLDEsIiIsMix7ImxldmVsIjoyLCJzdHlsZSI6eyJoZWFkIjp7Im5hbWUiOiJub25lIn19fV0sWzAsMl0sWzIsMSwiXFxyZXMiLDAseyJzdHlsZSI6eyJoZWFkIjp7Im5hbWUiOiJlcGkifX19XSxbNCwzLCJcXGVwc2lsb24iLDAseyJzaG9ydGVuIjp7InNvdXJjZSI6MzAsInRhcmdldCI6MzB9fV1d
	\[\begin{tikzcd}[ampersand replacement=\&]
		\& {A^{J_\triangleleft}} \\
		{A^J} \& {A^J}
		\arrow["\res", two heads, from=1-2, to=2-2]
		\arrow[""{name=0, anchor=center, inner sep=0}, from=2-1, to=1-2]
		\arrow[""{name=1, anchor=center, inner sep=0}, equals, from=2-1, to=2-2]
		\arrow["\epsilon", shorten <=3pt, shorten >=3pt, Rightarrow, from=0, to=1]
	\end{tikzcd}\]
	of the identity at $A^J$ through the restriction, which is \emph{necessarily invertible}. By {\cite[Lemma 2.3.7]{book:RV:2022}}, $\epsilon$  defines the counit of an adjunction, with $\res \colon A^{J_\triangleleft} \twoheadrightarrow A^J$ being the left adjoint.
	
	All-in-all, the left adjoint $\res \colon A^{J_\triangleleft} \twoheadrightarrow A^J$ is \emph{automatically} a lali. Consequently, the $\infty$-cosmos $\calJLim(\calK)$ of $\infty$-categories with $J$-shaped limits can actually be formed as a pullback
	% https://q.uiver.app/#q=WzAsNCxbMCwwLCJcXGNhbEpMaW0oXFxjYWxLKSJdLFsxLDAsIlxcY2FsTGEoXFxjYWxLKSJdLFswLDEsIlxcY2FsSyJdLFsxLDEsIlxcY2FsS157XFxpc29mfSJdLFswLDFdLFsxLDMsIlUiLDAseyJzdHlsZSI6eyJ0YWlsIjp7Im5hbWUiOiJob29rIiwic2lkZSI6InRvcCJ9LCJoZWFkIjp7Im5hbWUiOiJlcGkifX19XSxbMCwyLCIiLDAseyJzdHlsZSI6eyJ0YWlsIjp7Im5hbWUiOiJob29rIiwic2lkZSI6InRvcCJ9LCJoZWFkIjp7Im5hbWUiOiJlcGkifX19XSxbMiwzLCJGX3tKX1xcdHJpYW5nbGVsZWZ0fSIsMl0sWzAsMywiIiwwLHsic3R5bGUiOnsibmFtZSI6ImNvcm5lciJ9fV1d
	\[\begin{tikzcd}[ampersand replacement=\&]
		{\calJLim(\calK)} \& {\calLa(\calK)} \\
		\calK \& {\calK^{\isof}}
		\arrow[from=1-1, to=1-2]
		\arrow[hook, two heads, from=1-1, to=2-1]
		\arrow["\lrcorner"{anchor=center, pos=0.125}, draw=none, from=1-1, to=2-2]
		\arrow["U", hook, two heads, from=1-2, to=2-2]
		\arrow["{F_{J_\triangleleft}}"', from=2-1, to=2-2]
	\end{tikzcd}\]
	of the $\infty$-cosmos $\calLa(\calK)$, as in the proof of \longref{Theorem}{thm:trio_JLim}, where $U \colon \calLa(\calK) \hookrightarrow \calK^{\isof}$ is a cosmological embedding, from \longref{Proposition}{pro:La_cosmos}.
	
	Using the previous arguments, we are able to show the following.
	
	\begin{theorem}
		\label{thm:pullback_along_Cartesian_fib_JLim}
		Let $\calK$ be an $\infty$-cosmos. The enhanced simplicial category $\JLim(\calK)$ of $\infty$-categories of $\calK$ that have $J$-shaped limits for a simplicial set $J$ admits pullbacks of a loose $0$-arrow along a tight Cartesian fibration, and that the projection preserves and reflects limits.
	\end{theorem}
	
	\begin{proof}
		Note that Cartesian fibrations in $\calLa(\calK)$ are simply Cartesian fibrations in $\calK^{\isof}$. Besides, by \cite[Corollary 5.3.5]{book:RV:2022}, the cosmological functor $F_{J_{\triangleleft}}$ preserves Cartesian fibrations.
		
		Following the proof of \longref{Theorem}{thm:trio_JLim}, we obtain a pullback diagram
		% https://q.uiver.app/#q=WzAsNCxbMCwwLCJcXEpMaW0oXFxjYWxLKSJdLFsxLDAsIlxcTGEoXFxjYWxLKSJdLFswLDEsIlxcY2FsSyJdLFsxLDEsIlxcY2FsS157XFxpc29mfSJdLFswLDFdLFsxLDMsIlUiLDAseyJzdHlsZSI6eyJoZWFkIjp7Im5hbWUiOiJlcGkifX19XSxbMCwyLCIiLDAseyJzdHlsZSI6eyJoZWFkIjp7Im5hbWUiOiJlcGkifX19XSxbMiwzLCJGX3tKX1xcdHJpYW5nbGVsZWZ0fSIsMl0sWzAsMywiIiwwLHsic3R5bGUiOnsibmFtZSI6ImNvcm5lciJ9fV1d
		\[\begin{tikzcd}[ampersand replacement=\&]
			{\JLim(\calK)} \& {\La(\calK)} \\
			\calK_\chor \& {\calK^{\isof}_\chor}
			\arrow[from=1-1, to=1-2]
			\arrow[two heads, from=1-1, to=2-1]
			\arrow["\lrcorner"{anchor=center, pos=0.125}, draw=none, from=1-1, to=2-2]
			\arrow["U", two heads, from=1-2, to=2-2]
			\arrow["{F_{J_\triangleleft}}"', from=2-1, to=2-2]
		\end{tikzcd}\]
		in $\FDeltaCat$, in which both ${F_{J_\triangleleft}}$ and $U$ preserve pullbacks of a loose $0$-arrow along a tight Cartesian fibration, and that $U$ is an isofibration of enhanced simplicial categories, by \longref{Lemma}{lem:inclusion_isof}.
		
		Now by  \longref{Theorem}{thm:pullback_along_Cartesian_fib_La} and \longref{Lemma}{lem:enriched_pullback}, we achieve the desired statement.
	\end{proof}
	
	Similarly, we have a dual version of the above theorem.
	
	\begin{theorem}
		\label{thm:pullback_along_Cartesian_fib_JColim}
		Let $\calK$ be an $\infty$-cosmos. The enhanced simplicial category $\JColim(\calK)$ of $\infty$-categories of $\calK$ that have $J$-shaped colimits for a simplicial set $J$ admits pullbacks of a loose $0$-arrow along a tight Cartesian fibration, and that the projection preserves and reflects colimits.
	\end{theorem}
	%	
	%	\begin{proof}
		%		Arguing as in the proof of \longref{Theorem}{thm:trio_JColim} and \longref{Theorem}{thm:pullback_along_Cartesian_fib_JLim}, we have a pullback diagram
		%		% https://q.uiver.app/#q=WzAsNCxbMCwwLCJcXEpDb2xpbShcXGNhbEspIl0sWzEsMCwiXFxSYShcXGNhbEspIl0sWzAsMSwiXFxjYWxLIl0sWzEsMSwiXFxjYWxLXntcXGlzb2Z9Il0sWzAsMV0sWzEsMywiVSIsMCx7InN0eWxlIjp7ImhlYWQiOnsibmFtZSI6ImVwaSJ9fX1dLFswLDIsIiIsMCx7InN0eWxlIjp7ImhlYWQiOnsibmFtZSI6ImVwaSJ9fX1dLFsyLDMsIkZfe0pfXFx0cmlhbmdsZXJpZ2h0fSIsMl0sWzAsMywiIiwwLHsic3R5bGUiOnsibmFtZSI6ImNvcm5lciJ9fV1d
		%		\[\begin{tikzcd}[ampersand replacement=\&]
			%			{\JColim(\calK)} \& {\Ra(\calK)} \\
			%			\calK \& {\calK^{\isof}}
			%			\arrow[from=1-1, to=1-2]
			%			\arrow[two heads, from=1-1, to=2-1]
			%			\arrow["\lrcorner"{anchor=center, pos=0.125}, draw=none, from=1-1, to=2-2]
			%			\arrow["U", two heads, from=1-2, to=2-2]
			%			\arrow["{F_{J_\triangleright}}"', from=2-1, to=2-2]
			%		\end{tikzcd}\]
		%		in $\FDeltaCat$, which satisfies all the conditions in \longref{Lemma}{lem:enriched_pullback}. Hence, the result follows directly from  \longref{Theorem}{thm:pullback_along_Cartesian_fib_Ra}.
		%	\end{proof}
	
	%%		\bibliography{../joannaref}
	
%	\bibliographystyle{alpha}
%	\bibliography{joannaref}

\end{document}